\documentclass[reqno,centertags,12pt,a4paper]{amsart}
\usepackage{dsfont}
\usepackage{amscd}
\usepackage{amsmath}
\usepackage{tikz}
\usetikzlibrary{positioning}
\usepackage{amssymb}
\usepackage{amsfonts}
\usepackage{hyperref}
\usepackage{enumerate}
\usepackage{latexsym}
\usepackage{cases}
\usepackage{amsthm}
\usepackage{graphicx}
\usepackage{verbatim}
\usepackage{mathrsfs}
\usepackage{color}
\usepackage{wasysym}
\usepackage{makecell}
\usepackage{mathtools}
\usepackage{enumitem}
\usepackage{tikz-cd}
\usepackage{bm}
\usepackage{overpic}

\usepackage{cprotect}
\newcommand{\circled}[1]{%
  \tikz[baseline=(char.base)]{%
    \node[
      circle,
      draw,
      inner sep=0.1pt,      % 内边距（根据字体微调）
      minimum size=1.10em,    % 固定直径（关键参数！）
    ] (char) {#1};%
  }%
}

\newcommand{\C}{\mathbb{C}}
\newcommand{\R}{\mathbb{R}}
\newcommand{\N}{\mathbb{N}}
\newcommand{\I}{\mathbb{I}}
\newcommand{\J}{\mathbb{J}}

\newcommand{\DD}{\mathcal{D}}
\newcommand{\TT}{\mathcal{T}}
\newcommand{\Z}{\mathbb{Z}}
\newcommand{\F}{\mathbb{F}}

\newcommand{\1}{{\bm 1}}
\newcommand{\nn}{{\vec{k}}}
\newcommand{\rank}{\mathrm{rank}}
\newcommand{\supp}{\mathrm{supp}}

\newtheorem{theorem}{Theorem}[section]
\newtheorem{lemma}{Lemma}[section]
\newtheorem{proposition}[lemma]{Proposition}
\newtheorem{corollary}[lemma]{Corollary}
\newtheorem{definition}[lemma]{Definition}
\newtheorem{conjecture}[lemma]{Conjecture}
\newtheorem{remark}[lemma]{Remark}
\newtheorem{question}[lemma]{Question}

\theoremstyle{definition}
\numberwithin{equation}{section}

\newcommand{\norm}[1]{\left\|#1\right\|}

\newcommand{\abs}[1]{\left|#1\right|}

\allowdisplaybreaks[2] 

\begin{document}
	
	\title[Weighted $L^2$ restriction and comparison of nondegeneracy conditions]{Weighted $L^2$ restriction and comparison of nondegeneracy conditions for quadratic manifolds of arbitrary codimensions}
	\author[Zhenbin~Cao,~Jingyue~Li,~Changxing ~Miao~and~Yixuan~Pang]{Zhenbin~Cao,~Jingyue~Li,~Changxing ~Miao~and~Yixuan~Pang}
	
	\date{\today}
	
	\address{Institute of Mathematics, Henan Academy of Sciences, Zhengzhou 450046, China}
	\email{caozhenbin@hnas.ac.cn}
	
	\address{School of Mathematics and Statistics, Zhengzhou University, Zhengzhou 450001, China}
	\email{m\_lijingyue@163.com}
	
	\address{Institute of Applied Physics and Computational Mathematics, Beijing 100088, China}
	\email{miao\_{}changxing@iapcm.ac.cn}
	
	\address{Department of Mathematics, University of Pennsylvania, Philadelphia, PA 19104}
	\email{pyixuan@sas.upenn.edu}
	
	
	\begin{abstract}
        We systematically study weighted $L^2$ restriction for quadratic manifolds of arbitrary codimensions by sharp uniform Fourier decay estimates and a refinement of the Du-Zhang method in \cite{dz19}. Comparison with prior results is also discussed. In addition, we obtain an almost complete relation diagram for all existing nondegeneracy conditions for quadratic manifolds of arbitrary codimensions. These conditions come from various topics in harmonic analysis related to ``curvature'': Fourier restriction, decoupling, Fourier decay, Fourier dimension, weighted restriction, and Radon-like transforms. The diagram has many implications, such as ``best possible Stein-Tomas implies best possible $\ell^pL^p$ decoupling''. The proof of the diagram requires a combination of ideas from Fourier analysis, complex analysis, convex geometry, geometric invariant theory, combinatorics, and matrix analysis.
	\end{abstract}

	\maketitle

    \tableofcontents
	
	\section{Introduction}\label{section1}
	
	\subsection{Weighted restriction}\label{Weighted restriction estimates1}\phantom{x}

	Let $d,n \geq 1$. We denote by $\mathbf{Q}(\xi)=(Q_1(\xi),...,Q_n(\xi))$ an $n$-tuple of real quadratic forms in $d$ variables. The graph of such a tuple, $S_{\mathbf{Q}}\coloneqq\{ (\xi,\mathbf{Q}(\xi)): \xi \in [0,1]^d \}$, is a quadratic manifold of dimension $d$ and codimension $n$ in $\mathbb{R}^{d+n}$. Let $\mu_\mathbf{Q}$ be the pushforward of the Lebesgue measure on $\R^d$ via $\xi\mapsto (\xi,\mathbf{Q}(\xi))$. Define the Fourier extension operator $E^\mathbf{Q}$ associated with $S_\mathbf{Q}$ as
    \begin{equation}\label{s1e0}
        E^{\mathbf{Q}}f(x):= \int_{[0,1]^d} e^{ix\cdot(\xi,\mathbf{Q}(\xi))}f(\xi)d\xi,\quad \quad x\in \mathbb{R}^{d+n}.
    \end{equation}

    Let $0< \alpha \leq d+n$. For Lebesgue measurable functions $H:\mathbb{R}^{d+n} \rightarrow [0,1]$, we define
	\begin{equation}\label{def weight}
		A_\alpha(H):=\sup_{x'\in \mathbb{R}^{d+n},r\geq1} \frac{\int_{B(x',r)} H(y)dy}{r^\alpha}.     
	\end{equation}
	We say $H$ is an $\alpha$-dimensional weight if $A_\alpha(H)<\infty$. In this paper, we mainly focus on   weighted $L^2$ restriction estimates of the form
	\begin{equation}\label{s1e1}
		\|E^{\mathbf{Q}}f\|_{L^2(B_R,H)}\lesssim_H R^{s}\|f\|_{L^2([0,1]^d)}.
	\end{equation}
	We denote by $s(\alpha,\mathbf{Q})$ the infimum of $s$ for which (\ref{s1e1}) holds for any $f\in L^2([0,1]^d)$, $R\geq1$, and $\alpha$-dimensional weight $H$.

    \subsubsection{History of weighted restriction for quadratic manifolds}\phantom{x}
    
	The case $n=1$, in which $S_\mathbf{Q}$ is a quadratic hypersurface (paraboloid/hyperbolic paraboloid), has been fully researched. When $S_\mathbf{Q}$ is a paraboloid, i.e.,
    $$  \mathbf{Q}(\xi)=(\xi_1^2+\cdots+\xi_d^2),    $$
    the best upper bounds of $s(\alpha,\mathbf{Q})$ yet are 
	\begin{align}\label{his1}
		~s(\alpha,\mathbf{Q})\leq 
		\begin{cases}
			~	0, \quad \quad\quad \quad    & 0<  \alpha \leq \frac{d}{2},  \\
			~	\frac{2\alpha-d}{4}, \quad \quad \quad\quad 
			&\frac{d}{2}<  \alpha< \frac{d+1}{2},  \\
			~	\frac{\alpha}{2(d+1)}, \quad \quad \quad\quad &\frac{d+1}{2}\leq   \alpha\leq d+1. 
		\end{cases}
	\end{align}
	Besides, (\ref{his1}) is sharp for all $\alpha$ when $d=1$ (see  \cite{mattila87,wolff99}), and sharp for $0<\alpha\leq d/2$ and $d\leq \alpha\leq d+1$ when $d\geq 2$ (see \cite{d20}). The first and second bounds of (\ref{his1}) were derived by Shayya \cite{shayya21} via  global weighted restriction estimates. The last bound in (\ref{his1}) was derived by Du and Zhang \cite{dz19}, where they proved one class of new estimates, named the fractal $L^2$ restriction estimate. (\ref{his1}) were used to obtain the essentially sharp range for pointwise convergence of solutions to the free Schr\"odinger equation  \cite{dgl17,dz19}, and to make new progress on Falconer's distance set conjecture in geometric measure theory \cite{dz19}. When $S_\mathbf{Q}$ is a hyperbolic paraboloid, i.e.,
	$$\mathbf{Q}(\xi)=(\xi_1^2+\cdots+\xi_{d-m}^2-\xi_{d-m+1}^2-\cdots-\xi_d^2)$$
	for some $0<m\leq d/2$, Barron, Erdo\u{g}an, and Harris \cite{beh21} obtained
	\begin{align}\label{his2}
		~s(\alpha,\mathbf{Q})\leq 
		\begin{cases}
			~	0, \quad \quad\quad \quad    & 0<  \alpha \leq \frac{d}{2},   \\
			~	\frac{2\alpha-d}{4}, \quad \quad \quad\quad 
			&\frac{d}{2}<  \alpha< d+1 -m,   \\
			~	\frac{1}{2}, \quad \quad \quad\quad &d+1 -m\leq   \alpha\leq d+1. 
		\end{cases}
	\end{align}
	Moreover, for $m>1$ and $j\in [1,m-1]$, there is
	\begin{align}\label{his3}
		~s(\alpha,\mathbf{Q})\leq 
		\begin{cases}
			~	\frac{\alpha}{2(d+1-j)}, \quad \quad\quad \quad    & \frac{(j-1)(d+1-j)}{m-1}  \leq  \alpha <\frac{j(d+1-j)}{m},   \\
			~	\frac{\alpha-j}{2(d+1-j-m)}, \quad \quad \quad\quad 
			&\frac{j(d+1-j)}{m} \leq   \alpha< \frac{j(d-j)  }{m-1}.
		\end{cases}
	\end{align}
	If $m<d/2$, they also got
	\begin{equation}\label{his4}
		s(\alpha,\mathbf{Q})\leq \frac{\alpha+1}{4}-\frac{d+1}{8},\quad \quad \frac{d}{2}<\alpha <\frac{d+3}{2}.
	\end{equation}
	By considering one variant of the parabolic example of \cite{d20}, they showed that (\ref{his2})-(\ref{his4}) can give sharp bounds for all $\alpha$ for the cases: even $d\geq 2$ and $m=\frac{d}{2}$, odd $d\geq 3$ and $m=\frac{d-1}{2}$. In particular, it also indicates that $s(\alpha,\mathbf{Q})$ is completely determined when $d=2$ and $d=3$ for all $m>0$. Their argument includes two new effects compared to the paraboloid. The first is that the third bound $1/2$ in (\ref{his2}) is sharp in the range $d+1 -m\leq   \alpha\leq d+1$. The second is, in order to obtain (\ref{his3}) and (\ref{his4}), they considered  one variant of the Du-Zhang method, which can explore the lower-dimensional curvature information more fully for the hyperbolic paraboloid.

	%Their proof includes one variant of the Du-Zhang method and applications on the bilinear weighted restriction estimates with more adaptive transverse condition for the hyperbolic paraboloid. 
    
    However, the research on the case $n\geq 2$ is still inadequate. When $d=3$ and $n=2$, Guo and Oh \cite{go22}  gave a preliminary classification of all irreducible quadratic forms $\mathbf{Q}$. Based on this classification, Wang, the first and third authors \cite{cmw24} applied several methods, such as the Du-Zhang method adapted to quadratic manifolds and sharp Stein-Tomas-type inequalities, to study weighted restriction estimates for each individual case. 
    
    In this paper, we will make use of new ingredients and provide systematic results on weighted restriction for quadratic manifolds of arbitrary codimensions. When $d=3$ and $n=2$, our results also improve upon that in \cite{cmw24} in certain ranges.

	%To be specific, Shayya used annulus decomposition, which is in a  similar spirit to Bourgain \cite{bourgain91}, to study the case $0<  \alpha \leq d/2$. When $ d/2<  \alpha< (d+1)/2$, Shayya built a weighted H\"older-type inequality, and achieved the bootstrapping between different $\alpha$. 
	
    \subsubsection{Main result on weighted $L^2$ restriction}\phantom{x}
    
	Before stating our main result, we shall first introduce some notations. For any $N \in \R^{n \times n'}$ ($0 \leq n' \leq n$) and $M \in \R^{d \times d}$, define\footnote{Here we adopt the convention that $R_{N,M}\mathbf{Q} \coloneqq (0)$ if $n'=0$.}
    \begin{align}\label{def:R_{M',M}}
        R_{N,M}\mathbf{Q} \coloneqq (\mathbf{Q} \circ M) \cdot N,
    \end{align}
    where $\mathbf{Q}\circ M$ is the composition of $\mathbf{Q}$ and $M$. We say that $\mathbf{Q}$ and $\mathbf{Q}'$ are equivalent and write $\mathbf{Q}\equiv \mathbf{Q}'$ if $\mathbf{Q}' = R_{N,M}\mathbf{Q}$ for some $(N,M) \in {\rm GL}(n,\R) \times {\rm GL}(d,\R)$.
    \begin{definition}[\cite{gozzk23}]\label{s1d1}
		Given an $n$-tuple of real quadratic forms $\mathbf{Q}=(Q_1,...,Q_n)$ in $d$ variables, we define
		$$  {\rm NV}(\mathbf{Q}) := \# \{    1 \leq d' \leq d : \partial_{\xi_{d'}} Q_{n'} \not\equiv 0 {\rm~for~some~} 1\leq n'\leq n \}      .  $$
		For any $0 \leq d' \leq d$ and $0\leq n'\leq n$, we define
		\begin{equation}\label{s2def2e1}
			\mathfrak{d}_{d',n'}(\mathbf{Q}):= \inf_{\substack{   M \in \mathbb{R}^{d\times d} \\ {\rm rank}(M)=d'   }}  \inf_{\substack{   N \in \mathbb{R}^{n\times n'} \\ {\rm rank}(N)=n'   }} {\rm NV}(R_{N,M}\mathbf{Q}).
	\end{equation}
    \end{definition}
	
	\begin{definition}[\cite{gozzk23}]\label{decouping def for high codim}
		Let $q,p \geq 2$. For any dyadic $K\geq 1$, let $\{\tau\}$ be the partition of $[0,1]^d$ into $K^{-1}$-cubes. Define $\Gamma_{q,p}^d(\mathbf{Q})$ to be the smallest number $\Gamma$ such that
		\begin{equation}\label{eq:decoupling_extension}
			\|E^{\mathbf{Q}}f\|_{L^p(\R^{d+n})} \lesssim_\epsilon K^{\Gamma+\epsilon} \Big(   \sum_{\tau} \|E^{\mathbf{Q}}f_\tau\|_{L^p(\R^{d+n})}^q  \Big)^{\frac{1}{q}}
		\end{equation}
		holds for any measurable function $f$ supported on $[0,1]^d$, all dyadic $K\geq 1$, and all $\epsilon>0$, where $f_\tau \coloneqq f\chi_\tau$ for any $\tau $.
	\end{definition}

	Guo, Oh, Zhang and Zorin-Kranich \cite{gozzk23} used the algebraic quantities $\mathfrak{d}_{d',n'}(\mathbf{Q})$ to measure ``curvature" properties of the quadratic manifold $S_{\mathbf{Q}}$, which allowed  them to give a complete and sharp description of the decoupling exponents $\Gamma_{q,p}^d(\mathbf{Q})$ for all $q,p\geq 2$ and $\mathbf{Q}$. More precisely, they proved: For $2 \leq q \leq p <\infty$, there is
	\begin{equation}\label{dec th1 00}
		\Gamma_{q,p}^d(\textbf{Q})= \max_{0\leq d'\leq d} \max_{0\leq n' \leq n} \Big\{  d'\Big(1-\frac{1}{p}-\frac{1}{q}\Big)- \mathfrak{d}_{d',n'}(\textbf{Q})\Big(\frac{1}{2}-\frac{1}{p}\Big) -\frac{2(n-n')}{p}    \Big\}.
	\end{equation}
	And for $2\leq p<q \leq \infty$, there is
	$$  \Gamma_{q,p}^d(\textbf{Q})=\Gamma_{p,p}^d(\textbf{Q})+ d\Big(\frac{1}{p}-\frac{1}{q}\Big) .   $$
	Since we will fix $q=2$ in the rest of the paper, we abbreviate $\Gamma_{2,p}^d(\mathbf{Q})$ to $\Gamma_p^d(\mathbf{Q})$ for convenience. Besides, for $2\leq k\leq d+1$, we will also need to use the quantity $\Gamma_{p}^{k-2}(\textbf{Q})$, defined by replacing $d$ with $k-2$ in (\ref{dec th1 00}). 
	
    Another quantity was introduced by Gan, Guth and Oh \cite{ggo23} to derive general $k$-linear restriction estimates for quadratic manifolds of arbitrary codimensions. 
    \begin{definition}[\cite{ggo23}]\label{s1d2}
        Suppose that $\mathbf{Q}=(Q_1,...,Q_n)$ is an $n$-tuple of real quadratic forms in $d$ variables. Let $2\leq k\leq d+1$ and $0\leq m\leq d+n$ be integers. Define $X(\mathbf{Q},k,m)$ to be the biggest integer $X\leq d+1$ such that  
        \begin{equation}\label{s1d2e1}
            \sup_{\dim V=m} \dim \{ \xi\in \mathbb{R}^d : \dim (\pi_{V_\xi} (V))<X   \}\leq k-2,    
        \end{equation}
        where $V_\xi$ denotes the tangent space of the graph $S_{\mathbf{Q}}$ at the point $(\xi,\mathbf{Q}(\xi))$, and $\pi_{V_\xi}$ denotes the orthogonal projection from $\mathbb{R}^{d+n}$ onto $V_\xi$. 
    \end{definition}

    Our first main result is as follows. 
    \begin{theorem}\label{th1}
        Let $d,n\geq 1$ and $0<\alpha\leq d+n$. Suppose that $\mathbf{Q}=(Q_1,...,Q_n)$ is an $n$-tuple of real quadratic forms in $d$ variables. Denote by $\tilde{d}$ the largest number $d'$ such that $\mathfrak{d}_{d',n}(\mathbf{Q})=0$. Then we have
        \begin{align}\label{t1e1}
            ~s(\alpha,\mathbf{Q})\leq 
            \begin{cases}
                ~	0, \quad \quad\quad \quad    & 0<  \alpha \leq \frac{\mathfrak{d}_{d,1}(\mathbf{Q})}{2},   \\
                ~	\frac{2\alpha-\mathfrak{d}_{d,1}(\mathbf{Q})}{4}, \quad \quad \quad\quad 
                &\frac{\mathfrak{d}_{d,1}(\mathbf{Q})}{2}<  \alpha< d-\tilde{d}+n,   \\
                ~	\frac{n}{2}, \quad \quad \quad\quad &d-\tilde{d}+n\leq   \alpha\leq d+n,  
            \end{cases}
        \end{align}
        where the first and third estimates and associated ranges in {\rm (\ref{t1e1})} are sharp. Additionally, when $\frac{\mathfrak{d}_{d,1}(\mathbf{Q})}{2}<  \alpha< d-\tilde{d}+n$, we have
        \begin{equation}\label{t1e2}
            s(\alpha,\mathbf{Q})\leq \min_{ \substack{ 2\leq k\leq d+1 \\ p\geq 2 } } \max\left\{   \max_{1\leq m\leq d+n} \frac{m-X(\mathbf{Q},k,m)}{2m}\cdot\alpha ,~\frac{\Gamma_p^{k-2}(\mathbf{Q})}{2}+\frac{d+n-\alpha}{2p}+\frac{\alpha+n-d}{4} \right\}.
        \end{equation}        
    \end{theorem}
    \begin{remark}
        Theorem~{\rm\ref{th1}} is a very general result that applies to all possible $\mathbf{Q}$. But in practice, we do not need to compute the upper bounds for $s(\alpha,\mathbf{Q})$ for each individual $\mathbf{Q}$: It suffices to pick one $\mathbf{Q}$ from each equivalence class, by observing that equivalent quadratic forms enjoy the same weighted restriction theory.
    \end{remark}

    In general, finding the exact values of $X(\mathbf{Q},k,m)$ and $\Gamma_p^{k-2}(\mathbf{Q})$ in (\ref{t1e2}) is not an easy task. Though $\Gamma_p^{k-2}(\mathbf{Q})$ is completely determined by $\mathfrak{d}_{d',n'}(\mathbf{Q})$ in view of (\ref{dec th1 00}), computing $\mathfrak{d}_{d',n'}(\mathbf{Q})$ for all $0\leq d'\leq d$ and $0\leq n'\leq n$ can still be hard. In Appendix ~A~and~B of \cite{ggo23}, Gan, Guth and Oh provided algorithms to compute $\mathfrak{d}_{d',n'}(\mathbf{Q})$ and $X(\mathbf{Q},k,m)$ by using tools from real algebraic geometry. However, their algorithms can be very cumbersome to work out by hand, and in some cases it is the most convenient (although still not easy) to directly compute the quantities via ad hoc methods. For these reasons, we will show how to calculate the upper bounds in (\ref{t1e1}) and (\ref{t1e2}) for a host of special cases, which can also be regarded as applications of Theorem~\ref{th1}: 
    \begin{itemize}
        \item When $n=1$ and $S_\mathbf{Q}$ is a paraboloid or a hyperbolic paraboloid, Theorem \ref{th1} can recover (\ref{his1}) and (\ref{his2})-(\ref{his3}) respectively (see Corollary~\ref{cor1}~and~\ref{cor2}).
        \item When $d=3$ and $n=2$, Theorem \ref{th1} can further improve upon Theorem 1.3 in \cite{cmw24}. Not only that, we will also give a complete classification of all quadratic forms $\mathbf{Q}$ when $d+n\leq 5$ (see Lemma \ref{addth2}) and obtain the corresponding weighted restriction estimates (see Corollary~\ref{th3}). Our classification refines that of Guo and Oh \cite{go22}.
        \item When $n=2$, we obtain weighted restriction estimates for a small but interesting class of manifolds with a good curvature property introduced in \cite{ggo23}, called good manifolds (see Corollary~\ref{cor:good_mfd}). The calculation also relies on our study of the relations between various nondegeneracy conditions for quadratic manifolds, which will be discussed in Subsection~\ref{subsec:relations}.
    \end{itemize}

    \subsubsection{Application in geometric measure theory}\phantom{x}
    
    A direct application of Theorem \ref{th1} is the average Fourier decay rates of fractal measures along quadratic manifolds. As mentioned before, the $n=1$ case can be further applied to Falconer's distance set conjecture. 
	
	For any set $A\subset\mathbb{R}^{d+n}$, we denote by $\mathcal{M}(A)$ the collection of all nonnegative finite Borel measures with compact support contained in $ A$. 
	For $\alpha\in(0,d+n]$, we define the $\alpha$-dimensional energy of $\mu\in\mathcal{M}(B^{d+n}(0,1))$ as
	$$I_{\alpha}(\mu):=\int_{\R^{d+n}}\int_{\R^{d+n}}\frac{d\mu(x)d\mu(y)}{|x-y|^{\alpha}}=C_{\alpha}\int_{\R^{d+n}} |\xi|^{\alpha-d-n}|\widehat{\mu}(\xi)|^2d\xi.$$
	Let $\beta(\alpha,\mathbf{Q})$ denote the supremum of the number $\beta>0$ for which
	\begin{align}\label{WZJ1}
		\int_{ S_{\mathbf{Q}}}|\widehat{\mu}(R\xi)|^2d\sigma(\xi)\lesssim R^{-\beta},\quad\quad\forall ~R>1,
	\end{align}
	for all $\mu\in\mathcal{M}(B^{d+n}(0,1))$ with $I_{\alpha}(\mu)<\infty$. 
	
	When $n=1$ and $S_\mathbf{Q}$ is a hyperbolic paraboloid, Barron, Erdo\u{g}an, and Harris \cite{beh21} established the following relation: 
	\begin{align}\label{W1}
		\beta(\alpha,\mathbf{Q})=\alpha-2s(\alpha,\mathbf{Q}).
	\end{align}
	In fact, one can easily check that their proof carries over to all quadratic forms  $\mathbf{Q}$. Now by combining Theorem \ref{th1} with (\ref{W1}), we get the following corollary.
    
    \begin{corollary}\label{th00000}
		Let $d,n\geq 1$ and $0<\alpha\leq d+n$. Suppose that $\mathbf{Q}=(Q_1,...,Q_n)$ is an $n$-tuple of real quadratic forms in $d$ variables. Then we have
		\begin{align}\label{t1000e1}
			~\beta(\alpha,\mathbf{Q})\geq 
			\begin{cases}
				~	\alpha, \quad \quad\quad \quad    & 0<  \alpha \leq \frac{\mathfrak{d}_{d,1}(\mathbf{Q})}{2},   \\
				~	\frac{\mathfrak{d}_{d,1}(\mathbf{Q})}{2}, \quad \quad \quad\quad 
				&\frac{\mathfrak{d}_{d,1}(\mathbf{Q})}{2}<  \alpha< d-\tilde{d}+n,   \\
				~	\alpha-n, \quad \quad \quad\quad &d-\tilde{d}+n\leq   \alpha\leq d+n,  
			\end{cases}
		\end{align}
		where the first and third estimates and associated ranges in {\rm (\ref{t1000e1})} are sharp. Additionally, when $\frac{\mathfrak{d}_{d,1}(\mathbf{Q})}{2}<  \alpha< d-\tilde{d}+n$, we have
		\begin{equation}\label{t1000e2}
			\beta(\alpha,\mathbf{Q})\geq \max_{ \substack{ 2\leq k\leq d+1 \\ p\geq 2 } } \min\left\{   \min_{1\leq m\leq d+n} \frac{X(\mathbf{Q},k,m)}{m}\cdot\alpha ,~\frac{d+\alpha-n}{2}  - \frac{d+n-\alpha}{p} - \Gamma_p^{k-2}(\mathbf{Q})\right\}.
		\end{equation}
	\end{corollary}

    \subsubsection{Proof strategy for Theorem~\ref{th1}}\phantom{x}

    Among the three bounds in (\ref{t1e1}), the last one is trivial by the classical Agmon-H\"ormander inequality:
    \begin{equation}\label{AH ineq}
        \|E^{\mathbf{Q}}f\|_{L^2(B_R)}\lesssim R^{\frac{n}{2}}\|f\|_{L^2([0,1]^d)}. 
    \end{equation}
    And the second one immediately follows from the first one by a weighted H\"older-type inequality due to Shayya \cite{shayya21}. So it remains to consider the first bound in (\ref{t1e1}).
    
    We will  basically follow the strategy adopted by Shayya \cite{shayya21}, which combines the layer cake representation and Parseval's relation to reduce everything to the uniform Fourier decay of the surface measure (see (\ref{uniform decay e3}) below). Technically speaking, such decay is exploited by using an isotropic annulus decomposition, and fully determines the related range of $\alpha$.
    
    Our main contribution is obtaining a complete algebraic characterization of the sharp uniform decay rate for an arbitrary quadratic manifold of higher codimensions, which we now elaborate on. To begin with, we shall always consider smoothly truncated surface measures. This is important for us, because rough truncation may cause an unnecessary loss of Fourier decay along $(\xi,0)$. Thankfully, such a ``smoothing trick'' does not affect Shayya's argument.
    
	Let $\zeta=(\xi,\eta)$ with $\xi \in \mathbb{R}^d$ and $\eta \in \mathbb{R}^n$. When $n=1$ and $S_\mathbf{Q}$ is a paraboloid or a hyperbolic paraboloid, we have the following classical result
	\begin{equation}\label{unifrom decay e1}
		|\widehat{ \varphi d\mu_{\mathbf{Q}} }(\zeta)|\lesssim (1+|\zeta|)^{-\frac{d}{2}},  
	\end{equation}
	where $\varphi \in C_c^\infty([0,1]^d)$. However, the case for $n\geq 2$ becomes more complex. 
	Banner \cite{banner02} applied the stationary phase method to get the (nonuniform) Fourier decay estimate
	\begin{equation}\label{uniform decay e2}
		|\widehat{\varphi d\mu_{\mathbf{Q}}}(\zeta)| \lesssim \Pi_{j=1}^d(1+|\mu_j(\theta)||\zeta|)^{-\frac{1}{2}} , 
	\end{equation}
	where $\theta = \eta/|\eta| \in \mathbb{S}^{n-1}$, and $\{\mu_j(\theta)\}_{j=1}^d$ denote the eigenvalues of the matrix 
	$$\overline{Q}(\theta) \coloneqq \sum_{j=1}^n \theta_j\nabla_\xi^2 Q_j.\footnote{$\nabla_\xi^2 Q_j$ denotes the Hessian matrix of second derivatives of $Q_j$.}$$
	In particular, if $n=1$ and $S_\mathbf{Q}$ is a paraboloid or a hyperbolic paraboloid, we easily obtain $|\mu_j(\theta)|\sim 1$ for any $j=1,...,d$, and (\ref{uniform decay e2}) turns into (\ref{unifrom decay e1}) immediately. 
    
    However, Banner's estimate (\ref{uniform decay e2}) also has a few drawbacks. The major problem is the nonuniformity in $\theta$: In most higher-codimensional cases, $|\mu_j(\theta)|$ may vanish for some $j$ and $\theta$. In other words, when $n\geq 2$, the Fourier decay can be anisotropic along different normal directions. So the main difficulty is to quantify such directional singularity due to degenerate eigenvalues. We completely solved this problem by proving the optimal uniform Fourier decay estimate for quadratic manifolds of arbitrary codimensions (see Corollary~\ref{cor:recover_banner}):
	\begin{equation}\label{uniform decay e3}
		|\widehat{\varphi d\mu_{\mathbf{Q}}}(\zeta) | \lesssim (1+|\zeta|)^{-\frac{\mathfrak{d}_{d,1}(\mathbf{Q})}{2}} . 
	\end{equation}
    Our proof relies on two key steps that convert (\ref{uniform decay e2}) into (\ref{uniform decay e3}). Let $m(\theta)$ be the number of non-zero eigenvalues $\mu_j(\theta)$. The first step is to utilize the eigenvalue perturbation theory in matrix analysis to show the local stability of $m(\theta)$, which ensures the uniformity of the decay estimate and connects the uniform decay rate to $\min_{\theta} m(\theta)$. And the second step is to give a complete algebraic characterization of $\min_{\theta} m(\theta)$ by using $\mathfrak{d}_{d,1}(\mathbf{Q})$ (see Lemma~\ref{lem:alg_char_m}).

    Another awkward thing is that Banner's proof of (\ref{uniform decay e2}) relies on some black boxes in the theory of oscillatory integrals, whose details were skipped. More precisely, we need to follow the proof of Proposition 6 in Chapter VIII of Stein~\cite{stein93} instead of directly using any existing results. This motivates us to provide an elementary and self-contained alternative proof of Lemma 2.1.1 of \cite{banner02}, which does not depend on deep results for oscillatory integrals. The key point is to build one slightly stronger estimate for the extension operator smoothed by Gaussian functions (see Theorem~\ref{thm:uniform_decay}), which allows us to fully exploit the fact that $\mathbf{Q}$ is a quadratic form, and everything will come down to direct computations. Such a ``Gaussian smoothing trick'' was also used by Stein~\cite{stein93} and Mockenhaupt~\cite{mockenhaupt96}. \begin{remark}
        The results and techniques that we develop for proving {\rm (\ref{uniform decay e3})} will also play an important role in the study of relationships between nondegeneracy conditions {\rm(}see Theorem~\ref{thm:relation_diagram} in the next subsection{\rm)}, which constitutes the other main part of this paper.
    \end{remark}

	For the bound in (\ref{t1e2}), we will basically apply the strategy from Du and Zhang \cite{dz19}. When $n=1$ and $S_\mathbf{Q}$ is a paraboloid, they established a new type of estimate, which is called the fractal $L^2$ restriction estimate:
	\begin{equation}\label{fractal l2 restriction}
		\|E^{\mathbf{Q}}f\|_{L^2(X)}  \lesssim \gamma^{\frac{1}{d+1}} R^{\frac{\alpha}{2(d+1)}+\epsilon}\|f\|_{L^2([0,1]^d)}.
	\end{equation}
	Here $X$ is a union of lattice unit cubes in $B_R$, and $\gamma$ is a refined parameter that can measure the sparsity of $X$. Via the locally constant property, this estimate implies (\ref{s1e1}). The main method of Du and Zhang is to divide the original $(d+1)$-dimensional problem into the broad case and the narrow case ($d$-dimensional case). For the broad case, the multilinear restriction estimate can offer a good bound. For the narrow case, $d$-dimensional $\ell^2L^p$ decoupling and incidence argument can be of use to close the induction on scales. 
    
    Although considering such the broad-narrow analysis is enough in the parabolic case, when $n=1$ and $S_\mathbf{Q}$ is a hyperbolic paraboloid, things become more complicated. In fact, Barron, Erdo\u{g}an, and Harris \cite{beh21} considered the broad-narrow analysis with a general dividing dimension $k$ ($2\leq k\leq d+1$),  and used $k$-linear restriction estimates to obtain new weighted restriction results for intermediate $\alpha$'s.

	When $n \geq 2$, more difficulties prevent us from further promoting the argument in \cite{beh21}, such as the lack of $k$-linear restriction estimates for quadratic manifolds. In \cite{cmw24}, Wang, the first and third authors generalized the classical transversality condition for $n=1$ to higher-codimensional cases in a natural way (Definition~5.3 in \cite{cmw24}), and built the corresponding $k$-linear restriction estimates (Theorem~5.6 in \cite{cmw24}). However, such results have two shortcomings: Firstly, they does not apply when $d<n$; secondly, they require the specific expression of $\mathbf{Q}$ to determine the dividing dimension in the broad-narrow analysis. These drawbacks confine the authors of \cite{cmw24} to the case $d=3$ and $n=2$.
    
    Recently, more flexible $k$-linear restriction estimates ($2\leq k\leq d+1$) under another more robust type of transversality condition called the $\theta$-uniform condition (see  Definition~\ref{def:theta_uniform}) were built in \cite{gozzk23} when $k=d+1$ and \cite{ggo23} when $2\leq k\leq d+1$ independently. In this paper, we will perform the broad-narrow analysis with a general dividing dimension $k$ and adopt their $k$-linear restriction estimates to prove (\ref{t1e2}). Our argument can recover both the parabolic case in \cite{dz19} and the hyperbolic case in \cite{beh21}.

    \subsubsection{Sharpness of several estimates in Theorem \ref{th1}}\phantom{x}

    Now we briefly discuss the sharpness of the first and third bounds and associated ranges in (\ref{t1e1}).
    The first bound and its associated range in (\ref{t1e1}) are derived via the uniform Fourier decay estimate (\ref{uniform decay e3}), in which we only care about those directions with the worst Fourier decay rate. Though $s(\alpha,\mathbf{Q})=0$ is automatically sharp once established, an interesting question is: Is it possible to exploit more delicate anisotropic Fourier decay estimates or use a more refined argument to give a larger range of $\alpha$ for which $s(\alpha,\mathbf{Q})=0$? The reason why one might expect an affirmative answer is that there have been successful precedents for taking advantage of such anisotropy in the research on the Fourier restriction. For example, when $n=2$ and $\mathbf{Q}$ satisfies the (CM) condition (see Definition~\ref{def:CM}), Christ \cite{christ82} and Mockenhaupt \cite{mockenhaupt96} applied the interpolation of analytic families of operators to study $L^2$-based restriction estimates. In particular, they considered the family of analytic operators $T_z$ with the kernels 
    $$  K_z(\zeta):= \frac{1}{\Gamma(2+dz)} |\det(\overline{Q}(\theta))|^z \widehat{ \varphi d\mu_{\mathbf{Q}} }(\zeta),  $$
	where $\Gamma$ denotes the gamma function, and $z$ satisfies $  -2/d \leq \Re z \leq 1/2 $\footnote{$\Re z$ denotes the real part of $z$. }. Note that when $\Re z = 1/2 $,  the second factor in the kernel can exactly compensate for the directional singularity in (\ref{uniform decay e2}) caused by degenerate eigenvalues, which allows us to ignore those degenerate eigenvalues and regard $\widehat{\varphi d\mu_{\mathbf{Q}} }$ as owning uniform decay of order $d/2$ (which is the best possible) in the proof. 
    
    However, in the current setting of weighted restriction, we will take one example to demonstrate that $s(\alpha,\mathbf{Q})>0$ whenever $\alpha>\frac{\mathfrak{d}_{d,1}(\mathbf{Q})}{2} $, which means that the range $0<  \alpha \leq \frac{\mathfrak{d}_{d,1}(\mathbf{Q})}{2}$ of the first bound in (\ref{t1e1}) is optimal. This might be surprising at first glance, but should be acceptable with a little thinking. The main reason is that there is no presumed angular information in the definition of weights (\ref{def weight}), so we may not capture or eliminate the directional singularity in (\ref{uniform decay e2}) even if it only happens on a null set.

    As for the third bound and its associated range in (\ref{t1e1}), first note that the second bound in (\ref{t1e1}) will also reach $n/2$ when $\alpha=\frac{\mathfrak{d}_{d,1}(\mathbf{Q})}{2}+n$, which implies that $s(\alpha,\mathbf{Q})\leq n/2$ for $\frac{\mathfrak{d}_{d,1}(\mathbf{Q})}{2}+n \leq \alpha\leq d+n$. Unfortunately, this estimate is not optimal in many cases, such as when $S_\mathbf{Q}$ is a paraboloid. In fact, we will take one example to show the sharpness of the trivial bound $n/2$ in (\ref{t1e1}) in the range $d-\tilde{d}+n\leq   \alpha\leq d+n$. This generalizes the sharpness of the third bound in (\ref{his2}) where $S_\mathbf{Q}$ is a hyperbolic paraboloid. Furthermore, we will apply the Du-Zhang method to show that $s(\alpha,\mathbf{Q})< n/2$ whenever $\alpha< d-\tilde{d}+n$, which means that the range $d-\tilde{d}+n\leq \alpha\leq d+n$ of the third bound in (\ref{t1e1}) is the maximal possible.

	\subsection{Comparison of nondegeneracy conditions}\label{subsec:relations}\phantom{x}

    Recall that our Theorem~\ref{th1} relies on both uniform Fourier decay estimates (closely related to $\mathfrak{d}_{d,1}(\mathbf{Q})$) and $\ell^2L^p$ decoupling inequalities, while Corollary~\ref{th3} even relies on Stein-Tomas-type inequalities in some ranges. In short, in different ranges, different tools may dominate. A better understanding of how these tools are related to each other would be helpful in the study of weighted restriction, especially in higher-codimensional cases. 
    
    More precisely, when $n\geq2$, different tools usually correspond to different nondegeneracy conditions, and this is in contrast with  the $n=1$ case, where the Gaussian curvature almost decides everything. There have been various nondegeneracy conditions for quadratic manifolds in the literature, and different notions are introduced to describe different phenomena from different perspectives, including algebra ($\mathfrak{d}_{d',n'}(\mathbf{Q})$), geometry (Oberlin affine curvature) and analysis (Fourier restriction/decoupling). However, most of them have only been introduced and studied in the recent years, and it was unclear how they are related to each other. Sorting out these issues systematically is a natural question and will shed light on where we should go for future research. In other words, many intrinsic structures will be hidden if we only look at the $n=1$ case, so it is particularly interesting and even necessary to study various nondegeneracy conditions when $n\geq 2$. In fact, even the problem of finding the right nondegeneracy condition associated with a given tool itself is nontrivial. 
    
    In the second main theorem of this paper (Theorem~\ref{thm:relation_diagram} below), we will provide an almost complete diagram of the relationships between different kinds of ``maximal'' nondegeneracy  for quadratic manifolds, which is the most interesting case. Aside from our motivation for studying weighted restriction, this diagram also has its independent value and is instructive for many other topics in harmonic analysis related to curvature, which will be discussed in full detail later on.

    \subsubsection{Nondegeneracy conditions}\phantom{x}
    
    Before stating the theorem, we shall first introduce a bunch of nondegeneracy conditions, each of which owns a special place in the literature.
	
    \begin{definition}[\cite{christ82}, \cite{mockenhaupt96}]\label{def:CM}
        We say that $\mathbf{Q}$ satisfies the ${\rm (CM)}$ condition\footnote{If $n=1$, then by interpreting $d\sigma$ as the point mass on $\mathbb{S}^0 = \{-1,+1\}$, we can simply identify the (CM) condition as $\det(Q_1)\neq 0$, i.e., nonvanishing Gaussian curvature.} if
        \begin{equation}\label{eq:CM}
            \int_{\mathbb{S}^{n-1}}  |\det(\overline{Q}(\theta))|^{-\gamma}d\sigma(\theta) <\infty,   \quad \quad \forall\, 0<\gamma<\frac{n}{d},    
        \end{equation}
        and that $\mathbf{Q}$ satisfies the endpoint ${\rm (CM)}$ condition if 
        \begin{equation}\label{eq:CM_endpoint}
            \int_{\mathbb{S}^{n-1}}  |\det(\overline{Q}(\theta))|^{-\frac{n}{d}}d\sigma(\theta) <\infty,    
        \end{equation}
         where $\sigma$ is the standard surface measure on $\mathbb{S}^{n-1}$.
    \end{definition}
	
    \begin{definition}[\cite{gozzk23}]\label{def:str_non_degen}
        We say that $\mathbf{Q}$ is \textit{strongly nondegenerate} if
        \begin{equation}
            \mathfrak{d}_{d-m,n'}(\mathbf{Q}) \geq \frac{n'd}{n}-m,   
        \end{equation}
        for every $m$ with $0\leq m\leq d$ and every $n'$ with $0\leq n' \leq n$.
    \end{definition}
	
    \begin{definition}[\cite{gozzk23}]\label{def:non_degen}
        We say that $\mathbf{Q}$ is \textit{nondegenerate} if 
        \begin{align}
            \mathfrak{d}_{d-m,n'}(\mathbf{Q}) \geq \frac{n'd}{n} - 2m,
        \end{align}
        for every $m$ with $0 \leq m <d/2$ and every $n'$ with $0 \leq n' \leq n$.
    \end{definition}

    \begin{definition}[\cite{gozzk23}]\label{def:best_l2Lp_dec}
        We say that $\mathbf{Q}$ satisfies \textit{the best $\ell^2L^p$ decoupling} if
        \begin{equation}
            \Gamma_p^d(\mathbf{Q}) = \max\Big\{0, d\Big(\frac{1}{2} - \frac{1}{p}\Big) - \frac{2n}{p}\Big\}   
        \end{equation}
        for $2 \leq p < \infty$.
    \end{definition}

    \begin{definition}[\cite{gozzk23}]\label{def:best_lpLp_dec}
        We say that $\mathbf{Q}$ satisfies \textit{the best $\ell^pL^p$ decoupling} if
        \begin{equation}
            \Gamma_{p,p}^d(\mathbf{Q}) = \max\Big\{ d\Big(\frac{1}{2} - \frac{1}{p}\Big), 2d\Big(\frac{1}{2} - \frac{1}{p}\Big) - \frac{2n}{p} \Big\}   
        \end{equation}
        for $2 \leq p < \infty$.
    \end{definition}

    \begin{definition}[\cite{ggo23}]\label{def:good_mfd}
        We say 
        $$   \mathbf{Q}(\xi)=(a_1\xi_1^2 + \cdots + a_d\xi_d^2,b_1\xi_1^2 + \cdots + b_d\xi_d^2)  $$
        is ``\textit{good}'' if $a_i$'s are all positive and every two by two minor of the following matrix 
        \begin{equation}
            \begin{pmatrix}
                a_1 & a_2 &... &a_n \\
                b_1 & b_2 &... &b_n
            \end{pmatrix}    
        \end{equation}
        has rank two. In this case, we may also call $S_\mathbf{Q}$ a good manifold.
    \end{definition}

    Another key nondegeneracy notion is ``well-curved'', which was introduced by Gressman \cite{gressman19} and enjoys a rich background from geometric invariant theory. The precise definition, which is quite technical, will be left to Subsection~\ref{obe_aff_cur_con_sec2_1}. Here we only outline the key points to get the reader a rough sense of what it is about. Intuitively, we say $\mathbf{Q}$ is well-curved if a special measure associated with $S_\mathbf{Q}$, called the ``Oberlin affine measure'', is everywhere nonvanishing. The construction of this measure relies on the Kempf-Ness minimum vector calculations in geometric invariant theory (see Section~2 in \cite{gressman19}). It is worth mentioning that well-curvedness is closely related to the ``Oberlin affine curvature condition'' (or ``Oberlin condition'' for short, see  Definition~\ref{def:oberlin_condition}), which is kind of like a Knapp-type testing condition.

    We also record the definition of ``Salem'' here for the reader's convenience, which is closely related to the optimal uniform Fourier decay of the surface measure.
    \begin{definition}\label{def:salem}
        The Fourier dimension of a set $A \subset \R^{d+n}$ is defined to be
	\begin{align*}
		\dim_F A \coloneqq \sup\Big\{ s\leq d+n: \exists ~\mu \in \mathcal{M}(A) {\rm~ such ~that~ } |\widehat{\mu}(x)| \lesssim |x|^{-s/2} ,~\forall\, x\in\R^{d+n} \Big\}.
	\end{align*}
        We say that a set $A$ is a Salem set if $\dim_F A  = \dim_H A$ {\rm(}Hausdorff dimension{\rm)}, and $\mathbf{Q}$ is Salem if $S_\mathbf{Q}$ is a Salem set.
    \end{definition}

    \subsubsection{Main result on nondegeneracy conditions}\phantom{x}
    
    Our second main result is as follows.
    \begin{theorem}\label{thm:relation_diagram}
    For quadratic forms $\mathbf{Q}$, the following relation diagram holds:

    \begin{overpic}[width=1\textwidth]{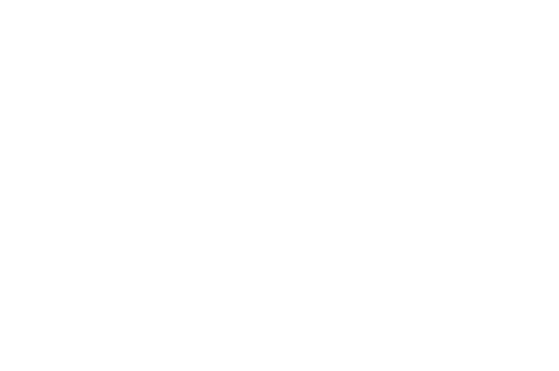} % 图片路径
    \put(6,33){
    $\begin{array}{ccccc}
        \rm{Salem}\phantom{xxx}  &   &   &   &  {\rm best~}\ell^2 L^p{~\rm decoupling} \\
        \displaystyle\left\Updownarrow\vphantom{\int}\right. \circled{\rm 11} &   &   &   & \displaystyle\left\Updownarrow\vphantom{\int}\right. \circled{\rm 4} \\
        \mathfrak{d}_{d,1}(\mathbf{Q})=d \phantom{xx} &   &  \rm{good}   &  \xRightarrow[d {\rm ~even,~} n=2]{\circled{\rm 3}} & {\rm strongly~ nondegenerate}\\
        \displaystyle\left\Downarrow\vphantom{\int}\right. \circled{\rm 10} &   &   &   & \displaystyle\left\Downarrow\vphantom{\int}\right. \circled{\rm 5} \\
        \widetilde{E}^{\mathbf{Q}}1 \in L^{\frac{2(d+n)}{d}+} & \xLeftrightarrow{\circled{\rm 1}} & {\rm (CM)} & \xRightarrow[\phantom{d {\rm ~even, ~} n=2}]{\circled{\rm 2}} & E^\mathbf{Q}: L^2 \rightarrow L^{\frac{2(d+2n)}{d}+} \\
        &   & \phantom{\displaystyle\left\Downarrow\vphantom{\int}\right.} & \fbox{\footnotesize	{$n=2$:\,\rm{equiv.}}} & \displaystyle\left\Downarrow\vphantom{\int}\right.\circled{\rm 6} \\
        &   & \circled{\rm 7}\phantom{x} \phantom{\Longrightarrow} & \xRightarrow{\phantom{d {\rm~ even,~ } n=2}} &  \text{\rm well-curved} \phantom{xx}\\
        &   &   &   & \displaystyle\left\Downarrow\vphantom{\int}\right. \circled{\rm 8} \\
        &   &   &   & {\rm nondegenerate} \phantom{xx} \\
        &   &   &   & \displaystyle\left\Updownarrow\vphantom{\int}\right. \circled{\rm 9} \\
        &   &   &   & {\rm best~} \ell^p L^p{~\rm decoupling}  
    \end{array}$
    }
      
    \put(38,25.16){    \tikzset{every picture/.style={line width=0.5pt}}    
    
    \begin{tikzpicture}[x=0.75pt,y=0.75pt,yscale=-1,xscale=1]
    \draw    (293.72,81.07) -- (293.85,137.03) ;
    \draw    (297.67,81.07) -- (297.67,134.18) ;
    \draw    (297.67,134.18) -- (339,134.18) ;
    \draw    (293.85,137.03) -- (338.76,137.03) ;
    
    \end{tikzpicture}
    }
    \end{overpic}

        Here $\widetilde{E}^\mathbf{Q}$ denotes the smoothed Fourier extension operator \begin{align}\label{eq:gaussian_integral}            
            \widetilde{E}^\mathbf{Q}f(x) \coloneqq \int_{\R^d} e^{i x\cdot(\xi,\mathbf{Q}(\xi))} f(\xi) \cdot e^{-|\xi|^2/2} d\xi, \qquad x\in\R^{d+n},
        \end{align}
        and by ``$L^{p+}$'' we mean that a bound holds with $L^{p+\epsilon}$ in place of $L^{p+}$ for any $\epsilon>0$. Moreover, 
        $\circled{{\rm 2}}\,\circled{\rm 6}\,\circled{\rm 7}$ are all ``$\Longleftrightarrow$'' when $n=2$, and all the one-sided implication relations above are strict {\rm(}i.e., the reverse implication is not true in general{\rm)} except possibly for $\circled{\rm 6}$. The thresholds $\frac{2(d+n)}{d}$ and $\frac{2(d+2n)}{d}$ are the best possible, in the sense that if they are replaced by something smaller, then the estimates cannot hold for any $\mathbf{Q}$.
    \end{theorem}

	Before proceeding, we provide a host of clarifications and remarks on the theorem itself.
    \begin{itemize}
        \item In the diagram, $\circled{2}$ is essentially\footnote{Here we say ``essentially'' because in the literature people only proved the desired bound for smoothed versions of $E^\mathbf{Q}$ instead of $E^\mathbf{Q}$ itself, so there are still minor technicalities that we need to deal with.} due to Mockenhaupt \cite{mockenhaupt96}, while $\circled{4}$ and $\circled{9}$ are due to Guo-Oh-Zhang-Zorin-Kranich \cite{gozzk23}. So our main contribution lies in the others\footnote{The $n=2$ case of $\circled{7}$ is due to Dendrinos-Mustata-Vitturi \cite{dmv22}.}: $\circled{1}\,\circled{5}\,\circled{10}\,\circled{11}$ rely on results or techniques from Section~\ref{sec2} (uniform Fourier decay); $\circled{6}\,\circled{7}\,\circled{8}$ rely on results or techniques from \cite{gressman19} (geometric invariant theory); while $\circled{3}$ basically comes down to combinatorics (the intermediate value principle).
        % \item Here we do not touch the issue of whether endpoint estimates should hold or not.
        \item All the conditions (except for ``good'') are the most nondegenerate in their own sense, although we did not write down all the underlying meanings in the diagram: $\mathfrak{d}_{d,1}(\mathbf{Q})=d$ represents the best uniform Fourier decay estimate\footnote{In fact, the condition $\mathfrak{d}_{d,1}(\mathbf{Q})=d$ enjoys much richer background, for which one may consult Appendix~\ref{appendix:par_d,1}. There we also record many interesting properties of $\mathfrak{d}_{d,1}(\mathbf{Q})$.}, $\mathbf{Q}$ Salem represents the largest Fourier dimension of $S_\mathbf{Q}$, $\widetilde{E}^{\mathbf{Q}}1 \in L^{\frac{2(d+n)}{d}+}$ represents the best $L^p$ bound of the Fourier transform of the surface measure, $E^\mathbf{Q}: L^2 \rightarrow L^{\frac{2(d+2n)}{d}+}$ represents the best Stein-Tomas-type inequality (or ``best Stein-Tomas'' for short), and well-curvedness represents the best Oberlin condition.
        \item Whenever we say ``best'', what we really mean is ``best possible'', which should be differentiated from ``sharp'': A quadratic form always satisfies  a sharp inequality/condition of the given form, which nevertheless is not necessarily ``best'': by ``best'' we require the exponent (once achieved) to be optimal among all quadratic forms. However, we should point out a caveat that for given $d$ and $n$, there is no guarantee that all nondegeneracy conditions in the diagram can always be achieved by some $\mathbf{Q}$:
        \begin{itemize}
            \item When $d\geq 3$ is odd, there is no $\mathbf{Q}$ with $\mathfrak{d}_{d,1}(\mathbf{Q})=d$ by the fundamental theorem of algebra, which by $\circled{11}$ implies that there is no Salem $\mathbf{Q}$;
            \item when $d=3$ and $n=2$, by (the proof of) Lemma~\ref{addth2}, there is no strongly nondegenerate $\mathbf{Q}$ (all $\mathbf{Q}$'s have $\mathfrak{d}_{2,1}(\mathbf{Q}) = 0$), which by $\circled{4}$ implies that there is no $\mathbf{Q}$ satisfying the best $\ell^2L^p$ decoupling;
            \item when $d=2$ and $n=3$, by Lemma~\ref{addth2}, the only quadratic form is $\mathbf{Q}(\xi) = (\xi_1^2,\xi_1\xi_2,\xi_2^2)$, which clearly does not satisfy the (CM) condition, and by $\circled{1}$ also fails to satisfy $\widetilde{E}^\mathbf{Q}1 \in L^{5+}$.
        \end{itemize}   On the other hand, by (the proof of) (3) of Theorem~\ref{thm:gressman_thm2} in \cite{gressman19} together with $\circled{8}$ and $\circled{9}$, we know that for any $d$ and $n$, there always exists $\mathbf{Q}$ that is well-curved, nondegenerate, and satisfies the best $\ell^pL^p$ decoupling. Also, one can easily check that for any $d$ and $n=2$, there always exists $\mathbf{Q}$ that is good.
        \item By temporarily ignoring the good condition, which is very special and only defined when $n=2$, we can provide a very nice illustration of how the whole diagram changes as $n$ grows. The $n=1$ case (when $\mathbf{Q}(\xi) = (Q(\xi))$) of Theorem~\ref{thm:relation_diagram} is well-known. In fact, all nondegeneracy conditions are equivalent to nonvanishing Gaussian curvature (i.e., $\det(\nabla_\xi^2 Q)\neq 0$), or ``$Q(\xi) \equiv$ paraboloids/hyperbolic paraboloids'', except for ``strongly nondegenerate'', which is equivalent to ``$Q(\xi) \equiv$ paraboloids''. When $n=2$, the cycle formed by $\circled{2}\,\circled{6}\,\circled{7}$ is still equivalent, but $\circled{8}$ and $\circled{10}$ become strict. And when $n \geq 3$, not only $\circled{8}$ and $\circled{10}$ but also $\circled{2}$ and $\circled{7}$ become strict.\footnote{Along the proof of Theorem~\ref{thm:relation_diagram} in Section~\ref{appendix:relationships}, we will explicitly construct all counterexamples.} This somehow indicates that the larger $n$ is, the more these nondegeneracy conditions diverge from each other. Such a phenomenon certifies the particularity of higher-codimensional cases and the necessity of introducing many different nondegeneracy conditions. Without loss of generality, we may only focus on the $n\geq 2$ case of Theorem~\ref{thm:relation_diagram}.
        \item Well-curvedness is the only condition in the diagram that already enjoys a well-established theory for general higher-codimensional submanifolds of degree $\geq 3$.
    \end{itemize}

    \subsubsection{Interesting corollaries}\phantom{x}
    
    Although we motivated Theorem~\ref{thm:relation_diagram} by the study of weighted restriction at the beginning, it also has its independent value and can lead to interesting corollaries. Here we only exhibit two of them.
    
    By tracing through $\circled{7}\,\circled{8}\,\circled{9}$ in Theorem~\ref{thm:relation_diagram}, we get the following corollary:
    \begin{corollary}\label{cor:(CM)_decoupling}
        If $\mathbf{Q}$ satisfies the {\rm (CM)} condition, then it also satisfies the best $\ell^pL^p$ decoupling.
    \end{corollary}
    This may not be obvious if we try to prove it directly without resorting to the theory of well-curvedness established by Gressman \cite{gressman19}. We care about this because in practice the (CM) condition is the most elegant criterion in the diagram and naturally arises in many PDE problems, while decoupling inequalities are very convenient to use.
    
    Besides, by tracing through $\circled{10}\,\circled{1}\,\circled{7}\,\circled{8}$ in Theorem~\ref{thm:relation_diagram}, we get another interesting corollary:
    \begin{corollary}\label{cor:algebra_partial}
        If a quadratic form $\mathbf{Q}$ satisfies $\mathfrak{d}_{d,1}(\mathbf{Q}) = d$, then it must also satisfy $\mathfrak{d}_{d',n'}(\mathbf{Q}) \geq \frac{n'}{n}d - 2(d-d')$ for any $\frac{d}{2} \leq d' \leq d$ and $0\leq n' \leq n$ {\rm(}nondegenerate{\rm)}.
    \end{corollary}
    It is noteworthy that Corollary~\ref{cor:algebra_partial} is purely algebraic in itself, but our proof (if one unfolds all the ingredients) requires a combination of ideas from Fourier analysis, complex analysis, convex geometry, geometric invariant theory, and algebra in an essential way. We do not know of any direct proof that purely relies on algebra. 
    
    In general, finding properties of a single $\mathfrak{d}_{d',n'}(\mathbf{Q})$ is already a very hard algebraic problem, even if we only care about $\mathfrak{d}_{d,1}(\mathbf{Q})$ (see Appendix~\ref{appendix:par_d,1}). Conceivably, there are even fewer existing results on relationships between different $\mathfrak{d}_{d',n'}(\mathbf{Q})$, and probably the most well-known one is Claim~3.5 in \cite{gozzk23}. However, that claim only provides relationships when $n'=1$: it says that $\mathfrak{d}_{d,1}(\mathbf{Q}) - \mathfrak{d}_{d',1}(\mathbf{Q}) \leq 2(d-d')$ for any $\frac{d}{2} \leq d' \leq d$. So it is remarkable that our Corollary~\ref{cor:algebra_partial} allows us to deduce the information of $\mathfrak{d}_{d',n'}(\mathbf{Q})$ for all $d'$ and $n'$ simply from that of the single $\mathfrak{d}_{d,1}(\mathbf{Q})$. As we will see in Appendix~\ref{appendix:par_d,1}, the condition $\mathfrak{d}_{d,1}(\mathbf{Q}) = d$ is quite strong in the sense that for ``almost all'' pairs of $(d,n)$, there is no $\mathbf{Q}$ satisfying it. So the punchline of Corollary~\ref{cor:algebra_partial} is that if one imposes a strong constraint on some (or even a single) $\mathfrak{d}_{d',n'}(\mathbf{Q})$, then it is possible that many other $\mathfrak{d}_{d',n'}(\mathbf{Q})$ will automatically satisfy weak constraints (which, although may not be so strong, are still good enough in many applications, such as $\ell^pL^p$ decoupling). In this way, one can see that different $\mathfrak{d}_{d',n'}(\mathbf{Q})$ are indeed closely correlated with each other in a deep and unusual way.

    \subsubsection{The significance of Theorem~\ref{thm:relation_diagram}}\phantom{x}
    
    In the end, let us talk more about the importance of Theorem~\ref{thm:relation_diagram}, especially how it may facilitate the development of many other topics.

    Firstly, it might be helpful in the study of Fourier restriction for quadratic manifolds of arbitrary codimensions. For example, one can see Subsection~\ref{subsec:final_remarks} for how Theorem~\ref{thm:relation_diagram} may guide our future research of higher-codimensional Fourier restriction theory.
        
	Secondly, it might be helpful in clarifying the relationship between Fourier restriction and decoupling, especially in higher-codimensional cases. Although decoupling has long been a crucial tool for Fourier restriction, people never get tired of developing new interesting perspectives on how they are related to each other. Intuitively, the final column in Theorem~\ref{thm:relation_diagram} tells us that ``best Stein-Tomas lies strictly in between best $\ell^2L^p$ and best $\ell^pL^p$ decoupling''.
	
	Thirdly, it might be helpful in connecting Fourier restriction and Radon-like transforms, which are the two main topics in harmonic analysis related to curvature. One major difficulty, to start with, is the following discrepancies:
	\begin{itemize}
		\item The forms of operators: Fourier restriction, under wave packet decomposition, is a combination of both oscillation and geometric interference; while Radon-like transform, as a positive operator, only involves geometric interference, and the oscillation is purely implicit.
		\item The dominating conditions: The $L^p \rightarrow L^q$ range of Fourier extension is determined by both $L^p$ bound of $\widetilde{E}^\mathbf{Q}1$ and Knapp-type examples, while the $L^p \rightarrow L^q$ range of Radon-like transforms is purely determined by Knapp-type examples.
		\item The existing techniques: Fourier restriction relies on decoupling inequalities (see \cite{ww24}), while $L^p \rightarrow L^q$ improving properties of Radon-like transform relies on sublevel set estimates (see \cite{gressman24}).
	\end{itemize}
	Our Theorem~\ref{thm:relation_diagram} may help to understand the last two discrepancies, because Knapp-type examples correspond to the Oberlin condition, and sublevel set estimates are also closely related to ideas and techniques developed in the study of this condition. Intuitively, the final column in Theorem~\ref{thm:relation_diagram} tells us that ``well-curvedness/best Oberlin condition lies strictly in between the best $\ell^2L^p$ and the best $\ell^pL^p$ decoupling''.
	
	Finally, it has the potential of inspiring new problems. 
	For example, one might ask whether or not $\circled{6}$ is strict in general, i.e., whether or not well-curvedness implies best Stein-Tomas. Although we do not have a complete answer to this question, by noting that these two conditions are both determined by Knapp-type examples, we are naturally led to propose the following conjecture, for which we do not know of any counterexample:
	\begin{conjecture}\label{conj:S-T}
		$\mathbf{Q}$ satisfies the best Stein-Tomas if and only if it is well-curved.
	\end{conjecture}
	Note that $\circled{6}$ in Theorem~\ref{thm:relation_diagram} proves the necessity part of Conjecture~\ref{conj:S-T} for all $n \geq 2$ as well as the sufficiency part for $n=2$. This conjecture, once verified for all $n \geq 3$, would provide a complete answer to a question implicitly raised in Remark~2.19 in \cite{mockenhaupt96}, where it is pointed out that even when $n=3$, one cannot expect a simple condition in terms of derivatives on $\det(\overline{Q}(\theta))$ ($\theta \in \mathbb{S}^{n-1}$) which characterizes the best Stein-Tomas. In contrast, the notion of well-curvedness in Conjecture~\ref{conj:S-T} appeals to a richer and deeper family of algebraic operations than merely the determinant.

	\vskip0.3cm
	
	\noindent \textbf{Outline of the paper.} In Section~\ref{secadd2}, we introduce some notations and basic tools. In Section~\ref{sec2}, we prove sharp uniform Fourier decay estimates. In Section~\ref{sec3}, we prove weighted restriction estimates in Theorem~\ref{th1}. In Section~\ref{sec4}, we give some examples and arguments showing the sharpness part of Theorem~\ref{th1}. In Section~\ref{sec5}, we apply Theorem~\ref{th1} to various examples of $\mathbf{Q}$. In Section~\ref{appendix:relationships}, we prove Theorem~\ref{thm:relation_diagram} and offer additional remarks on the higher-codimensional Fourier restriction theory. In Appendix~\ref{appendix:par_d,1}, we provide more historical background of $\mathfrak{d}_{d,1}(\mathbf{Q})$ and systematically record its properties.
	
	\vskip0.3cm
	
	\noindent \textbf{Notations.} If $X$ is a finite set, we use $\# X$ to denote its cardinality. If $X$ is a measurable set, we use $|X|$ to denote its Lebesgue measure. We use $B^N(c,r)$ to represent a closed ball centered at $c$ with radius $r$ in $\mathbb{R}^N$. We abbreviate $B^{N}(c,r)$ to $B(c,r)$ if $\mathbb{R}^{N}$ is clear in the content, and abbreviate $B(0,R)$ to $B_R$. We write $A\lesssim_\epsilon B$ to mean that there exists a constant $C$ depending on $\epsilon$ such that $A\leq CB$. Moreover, $A \sim B$ means $A \lesssim B$ and $A\gtrsim B$. Define $e(b):=e^{ i b}$ for each $b \in \mathbb{R}$. Let $p'$ denote the dual exponent of $p$, i.e., $1/p+1/p'=1$. For any $m_1,m_2\in\N^+$, let $\R^{m_1\times m_2}$ denote the space of all real $m_1\times m_2$ matrices. For any $m \in \N^+$, let ${\rm O}(m,\R)$ denote the real orthogonal group of order $m$, ${\rm GL}(m,\R)$ denote the real general linear group of order $m$, ${\rm SL} (m,\R)$ denote the real special linear group of order $m$, ${\rm S}(m,\R)$ denote the space of all real symmetric matrices of order $m$. Let $I$ be the identity matrix whose order is always adapted to the context.

    Suppose that $S_\psi \coloneqq \{ (\xi,\psi(\xi)): \xi \in U \}$ is a graph in $\mathbb{R}^{d+n}$, where $U$ is a bounded open/closed region of $\R^d$ and $\psi:U\rightarrow\R^n$ is a smooth function. Then we define 
    $$N_{\delta}(S_\psi|_U):=\big\{(\xi,\psi(\xi) + t): \xi\in U , ~|t|<\delta\big\}.$$
	And we abbreviated $N_{\delta}(S_\psi|_U)$ as $N_{\delta}(S_\psi)$ if $U=[0,1]^d$.
	
	We say that a set $Z \in \mathbb{R}^d$ is called a semi-algebraic set if it can be written as a finite union of sets of the form
	\begin{equation}\label{sg1}
		\{  x\in \mathbb{R}^d:~P_1(x) = 0,...,P_{l}(x)=0,P_{l+1}(x)>0,...,P_{l+k}(x)>0   \},  
	\end{equation}
	where $P_1$,..., $P_{l+k}$ are polynomials. Define the complexity of $Z$ to be the smallest sum
	of the degrees of the polynomials appearing in all possible descriptions (\ref{sg1}). Define the dimension of $Z$ to be the Hausdorff dimension of the set $Z$.

	For any ball $B=B(c_B,K)$ with $K\geq 1$, we define $w_B$ as
    \begin{equation}\label{notation_wb}
        w_B(x):=\Big(   1+ \frac{|x-c_B|}{K} \Big)^{-100(d+n)}.
    \end{equation}

    Let $\dot{P}_d^\kappa$ be the space of all real homogeneous polynomials of degree $\kappa$ in $d$ variables.

    Define the surface measure of $S_\mathbf{Q}$ to be $\mu_\mathbf{Q}^* \coloneqq \mu_\mathbf{Q}\chi_{[0,1]^d}$, i.e., $$d\mu_\mathbf{Q}^*(\xi,\eta) \coloneqq \chi_{[0,1]^d}(\xi) d\mu_\mathbf{Q}(\xi,\eta), \qquad \forall\,(\xi,\eta)\in\R^d\times\R^n.$$

	\section{Preliminaries}\label{secadd2}
	
	In this section, we will introduce some notations and basic tools. Readers only interested in Theorem~\ref{th1} may safely skip Subsection~\ref{obe_aff_cur_con_sec2_1}.

    \subsection{Several versions of decoupling}\label{dec_qua_man_sec2_2}\phantom{x}
	
	For $p\geq 2$, by Definition \ref{decouping def for high codim}, we have\footnote{For simplicity, we will suppress technicalities like ``for all $f$ supported on $[0,1]^d$, all dyadic $K\geq1$, and all $\epsilon>0$''.}
    \begin{equation}\label{eq:decoupling_extension_2}
        \|E^{\mathbf{Q}}f\|_{L^p(\R^{d+n})} \lesssim_\epsilon K^{\Gamma_{p}^d(\mathbf{Q})+\epsilon} \Big(   \sum_{\tau} \|E^{\mathbf{Q}}f_\tau\|_{L^p(\R^{d+n})}^2  \Big)^{\frac{1}{2}},
    \end{equation}
    where $\{\tau\}$ is the partition of $[0,1]^d$ into $K^{-1}$-cubes, and $f=\sum_\tau f_\tau =\sum_\tau f\chi_\tau$. 
    
    It is a kind of folklore that (\ref{eq:decoupling_extension_2}) has several essentially equivalent formulations, such as the local version:
    \begin{equation}\label{eq:decoupling_extension_3}
        \|E^{\mathbf{Q}}f\|_{L^p(B_{K^{2}})} \lesssim_\epsilon K^{\Gamma_{p}^d(\mathbf{Q})+\epsilon} \Big(   \sum_\tau \|E^{\mathbf{Q}}f_\tau\|_{L^p(w_{B_{K^{2}}})}^2  \Big)^{\frac{1}{2}},
    \end{equation}
    and the neighborhood version:
    \begin{align}\label{eq:decoupling_neighborhood}
        \| \widehat{F}\|_{L^p(\R^{d+n})} \lesssim_\epsilon K^{\Gamma_{p}^d(\mathbf{Q})+\epsilon} \Big(   \sum_\tau \big\|\widehat{F_\tau}\big\|_{L^p(\R^{d+n})}^2  \Big)^{\frac{1}{2}},
    \end{align}
    where $F$ is supported on $N_{K^{-2}}(S_{\mathbf{Q}})$, $F=\sum_\tau F_\tau$, and each $F_\tau$ is supported on $N_{K^{-2}}(S_{\mathbf{Q}}|_\tau)$. 
    Readers can see Proposition~9.15 in \cite{demeter2020} on the equivalence of the three versions.

	For future applications, we also need to introduce the lower-dimensional version of (\ref{eq:decoupling_extension_2}). To begin with, Gan, Guth and Oh \cite{ggo23} introduced the following transversality condition designed for the study of higher-codimensional Fourier restriction theory:
    
    \begin{definition}[\cite{ggo23}, $\theta$-uniform condition]\label{def:theta_uniform}
        Let $\theta \in (0,1]$, and $M,K\geq 1$. Let $  \{ X_m \}_{m=1}^{d+n} $ be nonnegative integers and $\{\tau_j\}_{j=1}^M$ be a collection of $K^{-1}$-cubes in $[0,1]^d$. We say that $\{\tau_j\}_{j=1}^M$ is $\theta$-uniform with the controlling sequence $  \{ X_m \}_{m=1}^{d+n} $ if for each $1\leq m\leq d+n$ and any subspace $V\subset \mathbb{R}^{d+n}$ with $\dim V=m$, there are at most $\theta M$ many $\tau_j$ intersecting 
        \begin{equation}\label{s3t1e1}
            \left\{  \xi \in \mathbb{R}^d:\dim (\pi_{V_\xi}(V)) <X_m   \right\}.
        \end{equation}
    \end{definition}
    
    This condition can lead to adaptive lower-dimensional $\ell^2 L^p$ decoupling:
    
    \begin{lemma}[\cite{ggo23}, Theorem 3.2]\label{dec low var}
        Let $0\leq k\leq d-1$, and $p \geq 2$. Given $E> 100$ and $\epsilon>0$, there exist constants $L=L(d,E)$, $\mu_0=\mu_0(d,E,\epsilon)$ and $c=c(d,E,\epsilon)$, such that the following holds true. Let $Z\subset \mathbb{R}^d$ be a $k$-dimensional semi-algebraic set  with complexity $\leq E$. For every $0<K^{-1}<\mu_0$, there exist	
        \begin{equation}\label{sec3 rel0}
            K^c \leq K_1 \leq K_2 \leq \cdots\leq K_L \leq K^{\frac{1}{2}}
        \end{equation}
        and collections of pairwise disjoint $1/K_j$-cubes $\mathcal{W}_j$ for  $j=1,2,...,L,$ such that 
        \begin{equation}\label{sec3 rel1}
            N_{1/K}(Z) \cap [0,1]^d \subset  \bigcup_{j=1}^L \bigcup_{W \in \mathcal{W}_j} W
        \end{equation}
        and
        \begin{equation}\label{dec low var e1}
            \Big\|\sum_{W \in \mathcal{W}_j} E^{\mathbf{Q}}f_W\Big\|_{L^p(\R^{d+n})} \leq C_{\epsilon,p,E} \cdot K_j^{\Gamma_{p}^k(\mathbf{Q})+\epsilon} \Big(   \sum_{W \in \mathcal{W}_j}  \|E^{\mathbf{Q}}f_W\|_{L^p(\R^{d+n})}^2\Big)^{\frac{1}{2}}.
        \end{equation}
    \end{lemma}
    
    Since the formulation of Lemma~\ref{dec low var} is not exactly the same as that of Theorem~3.2 in \cite{ggo23}, we will provide some clarifications below.
    
    Firstly, Theorem~3.2 of \cite{ggo23} is $\ell^pL^p$ decoupling, rather than $\ell^2L^p$ decoupling. However, the same proof for the current version also works (even holds for all $\ell^qL^p$ decoupling with $q\leq p$), and we do not repeat the arguments here. 
    
    Secondly, in Theorem~3.2 of \cite{ggo23}, the power of $K_j$ in $\text{(\ref{dec low var e1})}$ is actually ${\rm D}_{p}(\mathbf{Q}|_{L_k})$, which denotes the optimal power of $\ell^2L^p$ decoupling constant for all functions $f$ supported on $N_{K_j^{-2}}(L_k)$ with $L_k$ being a $k$-dimensional linear subspace. To bridge this discordance, note that $S_{\mathbf{Q}}|_{L_k}$ is a $k$-dimensional manifold of codimension $n$, so we can repeat the proof of Theorem 2.2 in \cite{gzk20} (or Corollary 5.4 in \cite{gozzk23}) to obtain
    $$   \text{D}_{p}(\mathbf{Q}|_{L_k}) \leq \sup_{L_k} \Gamma_{p}^k(\mathbf{Q}|_{L_k})+\epsilon = \Gamma_{p}^k(\mathbf{Q})+\epsilon, $$
    where the second equality can be shown by combining (\ref{dec th1 00}) with the observation that
    $$   \mathfrak{d}_{d',n'}(\mathbf{Q}) = \inf_{L_k} \mathfrak{d}_{d',n'}(\mathbf{Q}|_{L_k})   $$
    for any $d'\leq k$ and $n'\leq n$. Thus we get Lemma~\ref{dec low var}. 
    
    Finally, we point out that (\ref{dec low var e1}) is also equivalent to its local version 
    \begin{equation}\label{dec low var e2}
        \Big\|\sum_{W \in \mathcal{W}_j} E^{\mathbf{Q}}f_W\Big\|_{L^p(B_{K_j^2})} \leq C_{\epsilon,p,E} \cdot K_j^{\Gamma_{p}^k(\mathbf{Q})+\epsilon} \Big(   \sum_{W \in \mathcal{W}_j}  \|E^{\mathbf{Q}}f_W\|_{L^p(w_{B_{K_j^2}})}^2\Big)^{\frac{1}{2}}.
    \end{equation}

    \subsection{Well-curvedness}\label{obe_aff_cur_con_sec2_1}\phantom{x}

    As mentioned before, a core condition in Theorem~\ref{thm:relation_diagram} that we shall elaborate on is well-curvedness. However, the scope of this subsection is well beyond explaining the concepts: Firstly, it includes a brief exposition of the main tools from \cite{gressman19} that will be used in the proof of $\circled{6}$, $\circled{7}$, $\circled{8}$ in Theorem~\ref{thm:relation_diagram}; secondly, it also contains a bunch of new results that were not proved in \cite{gressman19}, because we need to remould some of the tools for our own purposes.

    \subsubsection{Well-curvedness and the Oberlin condition}\phantom{x}
    
    To state the definition of well-curvedness, we need to first introduce the Oberlin condition. It was introduced by Oberlin \cite{oberlin00} for hypersurfaces and systematically developed by Gressman \cite{gressman19} for higher-codimensional submanifolds.
	
    \begin{definition}[\cite{gressman19}, the Oberlin condition]\label{def:oberlin_condition}
        A Borel measure $\mu$ in $\R^{N}$ will be said to satisfy the Oberlin condition with exponent $\alpha \geq 0$\footnote{When $\alpha=0$, the condition simply tells us $\mu$ is finite. We include this trivial case purely for some uniformity of exposition.} when there exists a finite positive constant $C$ such that, for any convex body\footnote{By ``convex body'' we mean any compact convex set with nonempty interior. In \cite{gressman19} it is only assumed that $K$ is compact and convex, but this really does not matter: Any convex set $K$ with empty interior must be contained in a hyperplane, so when $\alpha>0$, (\ref{eq: oberlin_condition}) simply tells us that $\mu(K)=0$; however, this can also be derived by applying (\ref{eq: oberlin_condition}) to a sequence of convex bodies $K_n$ shrinking to $K$.} $K$ in $\R^{N}$, we have
        \begin{align}\label{eq: oberlin_condition}
            \mu(K) \leq C|K|^\alpha.
        \end{align}
    \end{definition}

    This is a Knapp-type testing condition and corresponds to the Frostman condition when $K$ is restricted to balls. However, by extending from balls to 
    convex bodies, we are able to capture not only dimension but also curvature, which will become clear later on. In fact, we can alternatively define the Oberlin condition by testing the measure $\mu$ on a much smaller category of geometric objects than convex bodies. This is the content of the following lemma, which should be a folklore but not formally presented anywhere in the literature:
	\begin{lemma}\label{lem:Oberlin_equiv}
		Let $\alpha\geq 0$, and $\mu$ be any Borel measure in $\R^{N}$. Then the following are equivalent to each other:

        \noindent $(i)$ There exists a finite positive constant $C$ such that, for any convex body $K$ in $\R^N$, we have $\mu(K) \leq C|K|^\alpha$. {\rm(}Definition~{\rm\ref{def:oberlin_condition})}

        \noindent $(ii)$ There exists a finite positive constant $C$ such that, for any ellipsoid\footnote{By ``ellipsoid'' we mean the image of $B(0,1)$ under a nonsingular affine transformation.} $J$ in $\R^N$, we have $\mu(J) \leq C|J|^\alpha$.

        \noindent $(iii)$ There exists a finite positive constant $C$ such that, for any rectangular box\footnote{By ``rectangular box'' we mean the image of $\prod_{i=1}^N [-c_i,c_i]$ ($c_i>0$) under an isometry.} $T$ in $\R^N$, we have $\mu(T) \leq C|T|^\alpha$.
	\end{lemma}
	\begin{proof}
		$(i) \Rightarrow (ii)$: This is trivial because any ellipsoid is a convex body.
		
		$(ii) \Rightarrow (i)$: By John’s ellipsoid theorem (Theorem III in \cite{john2014extremum}), for any convex body $K$ in $\R^N$, there is an ellipsoid $J$ so that if $c$ is the center of $J$, then $c+N^{-1}(J-c)\footnote{This is the dilation of $J$ by a factor of $N^{-1}$ with center $c$, i.e., $\{c+N^{-1}(x-c): x\in J\}$.} \subset K \subset J$. Note that $$|J| = N^N |c+N^{-1}(J-c)| \leq N^N|K|,$$ so we have $$\mu(K) \leq \mu(J) \leq C|J|^\alpha \leq CN^{N\alpha} |K|^\alpha,$$ as desired. 
		
		$(ii) \Rightarrow (iii)$: For any rectangular box $T$, there exists a nonsingular affine transformation $L$ of $\R^N$ such that $L([-\frac{1}{\sqrt{N}}, \frac{1}{\sqrt{N}}]^N) = T$. Let $J \coloneqq L(B(0,1))$, then $[-\frac{1}{\sqrt{N}}, \frac{1}{\sqrt{N}}]^N \subset B(0,1)$ implies $T \subset J$, and $|J| = C_N|T|$ for some constant $C_N = |B(0,1)|/(\frac{2}{\sqrt{N}})^N$. Therefore, we have $$\mu(T) \leq \mu(J) \leq C|J|^\alpha \leq CC_N^\alpha |T|^\alpha,$$ as desired.
		
		$(iii) \Rightarrow (ii)$: For any ellipsoid $J$, there exists a nonsingular affine transformation $L$ of $\R^N$ such that $L(B(0,1)) = J$. Let $L(x) = A\cdot x + b$, where $A \in {\rm GL}(N, \R)$ is the linear part of $L$ and $b \in \R^N$ is the translation part of $L$. By the singular value decomposition, there exists $D \in {\rm GL}(N,\R)$ diagonal with positive entries $\{c_i\}_{i=1}^N$, $O_1,O_2 \in {\rm O}(N,\R)$ such that $A = O_1 D O_2$. Let $X \coloneqq O_2^{-1}([-1,1]^N)$, then $$L(X) = A(X) + b = O_1 D O_2(X) + b = O_1 D([-1,1]^N) +b$$ is a rectangular box, since it is the image of $D([-1,1]^N)$ under an isometry. Also, $B(0,1) \subset X$ implies $J \subset T$, and $|T| = C_N|J|$ for some constant $C_N = |X|/|B(0,1)| = 2^N/|B(0,1)|$. Therefore, we have $$\mu(J) \leq \mu(T) \leq C|T|^\alpha \leq CC_N^\alpha |J|^\alpha,$$ as desired.
	\end{proof}
    Sometimes, the equivalent formulations of the Oberlin condition are more convenient to use, such as in the proof of Proposition~\ref{prop:convex_test}. Also, all statements in Lemma~\ref{lem:Oberlin_equiv} trivially implies that $\mu$ is $\sigma$-finite.

	Our Definition~\ref{def:oberlin_condition} is actually more general than that in \cite{gressman19}, where $\mu$ is supported on a $d$-dimensional immersed submanifold of $\R^{d+n}$\footnote{This will also be our main focus, but the definition itself surely makes sense for general Borel measures.}. Under this additional assumption, Gressman \cite{gressman19} examined the Oberlin condition in the case of maximal nondegeneracy by combining aspects of geometric invariant theory, convex geometry, and frame theory. In particular, he canonically constructed an equiaffine-invariant measure which is essentially the largest possible measure satisfying the Oberlin condition for the largest nontrivial choice of $\alpha$. To say that $\alpha$ is nontrivial means that there is a nonzero measure $\mu$ satisfying (\ref{eq: oberlin_condition}) for this $\alpha$ on some immersed
	submanifold of dimension $d$ and codimension $n$. 
    
    To state Gressman's main results, we also need the concept of \textit{homogeneous dimension} $D = D(d,n)$, which is defined to be the smallest positive integer which equals the sum of the degrees of some collection of $d+n$ distinct, nonconstant monomials in $d$ variables. For example, for nontrivial quadratic manifolds, we always have $n \leq \frac{d(d+1)}{2}$, which implies $D = d+2n$.
	
    \begin{theorem}[{\cite{gressman19}, Theorem 1}]\label{thm:gressman_thm1}
        Suppose that $\Sigma$ is an immersed $d$-dimensional submanifold of $\R^{d+n}$. To any such $\Sigma$, one may associate a nonnegative Borel measure $\mu_{\mathcal{A}}$ on $\R^{d+n}$, constructed in Section {\rm2} of \cite{gressman19}, which is supported on $\Sigma$. Then the following are true:
    
        \noindent {\rm (1)} If $\mu$ is any nonnegative Borel measure supported on $\Sigma$ which satisfies {\rm(\ref{eq: oberlin_condition})} with exponent $\alpha > d/D$, then $\mu$ is the zero measure.
    
        \noindent {\rm (2)} If $\mu$ is any nonnegative Borel measure supported on $\Sigma$ which satisfies {\rm(\ref{eq: oberlin_condition})} with exponent $\alpha = d/D$, then $\mu \lesssim \mu_{\mathcal{A}}$.
    
        \noindent {\rm (3)} If $\Sigma$ is the image of an immersion $f: \Omega \rightarrow \R^{d+n}$, where $\Omega \subset \R^d$ is open with compact closure $\overline{\Omega}$ and $f$ extends to be real analytic on a neighborhood of $\overline{\Omega}$, then the measure $\mu_{\mathcal{A}}$ satisfies {\rm(\ref{eq: oberlin_condition})} with exponent $\alpha = d/D$ and is consequently the largest such measure, up to a multiplicative constant.
        
        We call $\mu_{\mathcal{A}}$ the {\rm(}Oberlin{\rm)} affine measure of $\Sigma$. 
        Furthermore, the correspondence sending $\Sigma$ to $\mu_{\mathcal{A}}$ is equiaffine invariant, meaning that if $\Sigma$ and $\Sigma_0$ are immersed submanifolds such that $\Sigma_0$ is the image of $\Sigma$ under some equiaffine transformation\footnote{Affine transformations that preserve volumes.} $T$ of $\R^{d+n}$, then the affine measure $\mu'_{\mathcal{A}}$ of $\Sigma_0$ and the affine measure $\mu_{\mathcal{A}}$ of $\Sigma$ satisfy $\mu'_{\mathcal{A}}(T(E)) = \mu_{\mathcal{A}}(E)$ for all Borel sets $E$.
    \end{theorem}
	
	Suppose $\Sigma$ is parametrized by $f: \Omega \rightarrow \R^{d+n}$ ($\Omega \subset \R^d$) with $f(t) = (t, f_1(t),\dots, f_n(t))$, then the affine measure $\mu_{\mathcal{A}}$ is defined as the pushforward via $f$ of the measure $\nu_{\mathcal{A}}$ on $\R^d$ whose density is generated by certain ``affine curvature tensor''. One can see (14) in \cite{gressman19} for the precise definition of $\nu_{\mathcal{A}}$, which is inspired by geometric invariant theory. From this construction, the uniqueness of $\mu_{\mathcal{A}}$ and $\nu_{\mathcal{A}}$ follows immediately. However, we do not know beforehand how regular $\mu_{\mathcal{A}}$ or $\nu_{\mathcal{A}}$ is.
	
	To certify the nontriviality of Theorem~\ref{thm:gressman_thm1}, i.e., to show that $\mu_{\mathcal{A}}$ is not simply the zero measure for all possible choices of $\Sigma$, Gressman \cite{gressman19} considered $f$ of the form
	\begin{align}\label{canonical_poly_form}
		f(t) \coloneqq ((t^\gamma)_{1\leq|\gamma|<\kappa}, p_1(t), \dots, p_m(t)),
	\end{align}
	where $p_1,\dots,p_m$ are linearly independent elements in $\dot{P}_d^\kappa$, and $m$ and $\kappa$ are chosen (as functions of $n$ and $d$ only) so that the right-hand side of (\ref{canonical_poly_form}) is an element of $\R^{d+n}$. Let $P_{j\ell}(t) \coloneqq \partial_j p_\ell(t)$ for $j = 1,\dots,d$, $\ell = 1,\dots,m$. Such an embedding $f$ will be called a \textit{model form} when there exist real numbers $c$ and $c'$ such that
	\begin{align}\label{model_poly_form}
		\sum_{j=1}^d P_{j\ell}(\partial) P_{j\ell'}(t)|_{t=0} = c\delta_{\ell,\ell'}  \quad \text{and} \quad \sum_{\ell=1}^m P_{j\ell}(\partial) P_{j'\ell}(t)|_{t=0} = c'\delta_{j,j'}
	\end{align}
	where $\delta$ is the Kronecker delta. The main result for model forms in \cite{gressman19} is:
    \begin{theorem}[\cite{gressman19}, Theorem 2]\label{thm:gressman_thm2}
        The following are true for embeddings $f$ of the form {\rm(\ref{canonical_poly_form})}.
    
        \noindent {\rm (1)} The closure of the orbit $\{Nf(M^Tt)\}_{N\in {\rm SL}(d+n,\R), M \in {\rm SL}(d,\R)}$ in the space of $(d+n)$-tuples of polynomials of degree at most $\kappa$ always contains an embedding of the form {\rm(\ref{canonical_poly_form})} which is a model form {\rm(\ref{model_poly_form})}. If any embedding in the closure of the orbit is degenerate, then all are degenerate {\rm(}i.e., $\mu_{\mathcal{A}} = 0$ for each embedding in the orbit closure or $\mu_{\mathcal{A}} \neq 0$ for each embedding{\rm)}.
    
        \noindent {\rm (2)} For any $p \coloneqq (p_1,\dots,p_m)$ satisfying {\rm(\ref{model_poly_form})}, the affine measure $\mu_{\mathcal{A}}$ of the submanifold of $\R^{d+n}$ parameterized by {\rm(\ref{canonical_poly_form})} is a nonzero constant times the pushforward of Lebesgue measure via $f$ if and only if $c$ and $c'$ are nonzero.
    
        \noindent {\rm (3)} For any pair $(d,n)$, there is some $p \coloneqq (p_1,\dots,p_m)$ satisfying {\rm(\ref{model_poly_form})} for nonzero $c$ and $c'$. Consequently, in any dimension and codimension, there is a submanifold $\Sigma$ whose affine measure $\mu_{\mathcal{A}}$ is everywhere nonzero on $\Sigma$. In this case, we say that $\Sigma$ is ``well-curved''.
    \end{theorem}
    For simplicity, we say that a quadratic form $\mathbf{Q}$ is well-curved if $S_\mathbf{Q}$ is well-curved.

    The proof of (2) in \cite{gressman19} relies on a direct computation of $\nu_{\mathcal{A}}$ (Lemma 4 there) via the so-called first fundamental theorem of invariant theory (Lemma 5 there), which yields an explicit expression of $\nu_{\mathcal{A}}$. However, if we only care about the ``constant'' property, then there is a much simpler proof, which we now present. We choose to do so not only because it is elegant and yields more general results (Proposition~\ref{prop:TDI_affine_const}), but also because it contains geometric insights that also appear in the proof of $\circled{11}$ in Theorem~\ref{thm:relation_diagram}.
    \begin{definition}\label{def:TDI_mfd}
        We say that $\Sigma$ is a {\rm TDI} manifold {\rm(}translation-dilation-invariant manifold{\rm)} if each $f_j$ in the parametrization of $\Sigma$ is a homogeneous polynomial whose all partial derivatives are linear combinations of $f_1,\dots,f_n$.
    \end{definition}
    Historically, TDI manifolds play an important role in decoupling theory, since they enjoy a large group of symmetries that resembles the classical ``parabolic rescaling'' for paraboloids. For example, all quadratic manifolds are TDI. In fact, one can easily check that any parametrization (\ref{canonical_poly_form}) induces a TDI manifold, so the following proposition recovers the ``constant'' property in (2) of Theorem~\ref{thm:gressman_thm2} from a very different perspective.
    \begin{proposition}\label{prop:TDI_affine_const}
        Suppose $\Sigma$ is a {\rm TDI} manifold, then $\nu_{\mathcal{A}}$ is a constant {\rm(}might be zero{\rm)} times the Lebesgue measure on $\R^d$.
    \end{proposition}
    \begin{proof}
        The key point is the equiaffine invariance of $\mu_{\mathcal{A}}$. Note that for any Borel set $E \subset \R^d$ and $a\in\R^d$, there exists some $T_a \in {\rm SL}(d+n,\R)$ such that $\Sigma$ is invariant under $T_a$ and $T_a(f(E)) = f(E+a)$ (``parabolic rescaling'' for TDI manifolds). By taking $\Sigma_0 = \Sigma$ and $T=T_a$ in Theorem~\ref{thm:gressman_thm1}, we have $\mu'_{\mathcal{A}}=\mu_{\mathcal{A}}$ by the uniqueness of $\mu_{\mathcal{A}}$, and 
        $$\nu_{\mathcal{A}}(E) = \mu_{\mathcal{A}}(f(E)) = \mu_{\mathcal{A}}(T_a(f(E))) = \mu_{\mathcal{A}}(f(E+a)) = \nu_{\mathcal{A}}(E+a).$$
        In other words, $\nu_{\mathcal{A}}$ is translation invariant, so for any $\varphi \in C_c^\infty(\R^d)$ with $\int \varphi = 1$ and $\varphi\geq 0$, we have $\nu_{\mathcal{A}} * \varphi = \nu_{\mathcal{A}}$. This is because for any Borel set $E \subset \R^d$, by Tonelli's theorem we have
        \begin{align*}
            \nu_{\mathcal{A}} * \varphi (E) &= \int_{\R^d}\int_{\R^d} \varphi(\xi-\eta)d\nu_{\mathcal{A}}(\eta) \chi_E(\xi) d\xi\\
            &= \int_{\R^d}\int_{\R^d} \varphi(\xi) \chi_E(\xi+\eta) d\xi d\nu_{\mathcal{A}}(\eta)\\
            &= \int_{\R^d}\int_{\R^d} \chi_{E-\xi}(\eta) d\nu_{\mathcal{A}}(\eta) \varphi(\xi) d\xi\\
            &= \int_{\R^d} \nu_{\mathcal{A}}(E) \varphi(\xi) d\xi = \nu_{\mathcal{A}}(E).
        \end{align*}
        Thus $\nu_{\mathcal{A}} = \nu_{\mathcal{A}} * \varphi$ must be smooth. But any translation invariant smooth function must be a constant, so we are done.
    \end{proof}

    \subsubsection{Newton-type polyhedra characterizations}\phantom{x}
	
	Recall that the original definition of the Oberlin condition (Definition~\ref{def:oberlin_condition}) only involves concepts from convex geometry. So it is remarkable that Gressman \cite{gressman19} obtained many other characterizations in different flavors besides convex geometry, such as algebra, geometric invariant theory, sublevel set estimates, and Newton-type polyhedra. However, for our purposes, we shall only need the Newton-type polyhedra characterization, which will be presented below.
    
    Let us first fix some notations. For $p=(p_1,\dots,p_m)$ being an $m$-tuple of real homogeneous polynomials in $d$ variables, we define $R_{N,M} p \coloneqq (p\circ M)\cdot N$\footnote{We rewrite the definition of $R_{N,M}$ in \cite{gressman19} to better match the definition of $\mathfrak{d}_{d',n'}(\mathbf{Q})$ in the literature, which will be convenient for later discussions. For example, when $p = \mathbf{Q}$, this agrees with our previous definition of $R_{N,M}$ in (\ref{def:R_{M',M}}). Such a modification does not change the content of the theorems in \cite{gressman19}.} for any $(N,M) \in {\rm GL}(m,\R) \times {\rm GL}(d,\R)$. Also, let $\1_u \coloneqq (1,\dots,1) \in [0,\infty)^u$ for any $u \in \N^+$. Then we have:
    
    \begin{theorem}[\cite{gressman19}, Theorem 6]\label{thm:gressman_thm6}
        Let $\{e_1,\dots,e_m\}$ be the standard basis of $\R^m$. For any $p = (p_1,\dots,p_m)$ being an $m$-tuple of real homogeneous polynomials of degree $\kappa$ in $d$ variables, let $\mathcal{N}(p) \subset [0,\infty)^m \times [0,\infty)^d$ be the convex hull of all points $(e_j, \gamma)$ such that 
        \begin{align*}
            \partial^\gamma p_j(t)|_{t=0} \neq 0.
        \end{align*}
        Then the associated submanifold defined by {\rm(\ref{canonical_poly_form})} is well-curved if and only if 
        \begin{align}\label{eq:newton_convex_hull}
            \left( \frac{1}{m} \1_m, \frac{\kappa}{d} \1_d \right) \in \mathcal{N}(R_{O_1,O_2} p)
        \end{align}
        for all $(O_1,O_2) \in {\rm O}(m, \R) \times {\rm O}(d, \R)$.
    \end{theorem}
	
	Although we have introduced the full versions of the theorems in \cite{gressman19} for completeness and for the reader's convenience, in this paper we will only care about their applications to quadratic manifolds. An unusual benefit that we enjoy by restricting ourselves to this special (but still quite general compared with paraboloids/hyperbolic paraboloids) class of manifolds is that now the form (\ref{canonical_poly_form}) comes for free, so Theorem~\ref{thm:gressman_thm2} and \ref{thm:gressman_thm6} are directly applicable (take $m=n$, $\kappa=2$, $p = \mathbf{Q}$, and $f(\xi) = (\xi, \mathbf{Q}(\xi))$). Theorem~\ref{thm:gressman_thm6} is vital in the proof of $\circled{6}$ in Theorem~\ref{thm:relation_diagram}.
    % \begin{corollary}\label{cor:gressman_thm6}
    %     Let $\{e_1,\dots,e_n\}$ be the standard basis of $\R^n$. For any $\mathbf{Q} = (Q_1,\dots,Q_n)$ being an $n$-tuple of quadratic forms in $d$ variables, let $\mathcal{N}(\mathbf{Q}) \subset [0,\infty)^n \times [0,\infty)^d$ be the convex hull of all points $(e_j, \gamma)$ such that 
    %     \begin{align*}
    %         \partial^\gamma Q_j(\xi)|_{\xi=0} \neq 0.
    %     \end{align*}
    %     Then $\mathbf{Q}$ is well-curved if and only if 
    %     \begin{align}\label{cond:Newton_poly_quad}
    %         \left( \frac{1}{n} \1_n, \frac{2}{d} \1_d \right) \in \mathcal{N}(R_{O_1,O_2} \mathbf{Q})
    %     \end{align}
    %     for all $(O_1,O_2) \in {\rm O}(n, \R) \times {\rm O}(d, \R)$.
    % \end{corollary}

    On the other hand, in the proof of $\circled{8}$ in Theorem~\ref{thm:relation_diagram}, we need to connect well-curvedness with algebraic quantities $\mathfrak{d}_{d',n'}(\mathbf{Q})$. For this purpose, we will prove the following variant of Theorem~\ref{thm:gressman_thm6}:
    \begin{theorem}\label{thm:gressman_thm6_variant}
        Let $\{e_1,\dots,e_m\}$ be the standard basis of $\R^m$. For any $p = (p_1,\dots,p_m)$ being an $m$-tuple of real homogeneous polynomials of degree $\kappa$ in $d$ variables, let $\mathcal{N}(p) \subset [0,\infty)^m \times [0,\infty)^d$ be the convex hull of all points $(e_j, \gamma)$ such that 
        \begin{align*}
            \partial^\gamma p_j(t)|_{t=0} \neq 0.
        \end{align*}
        Then the associated submanifold defined by {\rm(\ref{canonical_poly_form})} is well-curved if and only if 
        \begin{align*}
            \left( \frac{1}{m} \1_m, \frac{\kappa}{d} \1_d \right) \in \mathcal{N}(R_{N,M} p)
        \end{align*}
        for all $(N,M) \in {\rm GL}(m, \R) \times {\rm GL}(d, \R)$.
    \end{theorem}

    Compared with Theorem~\ref{thm:gressman_thm6}, which only tests (\ref{eq:newton_convex_hull}) over $(N,M) \in {\rm O}(n,\R) \times {\rm O}(d,\R)$, here in Theorem~\ref{thm:gressman_thm6_variant} we are testing over $(N,M) \in {\rm GL}(n,\R) \times {\rm GL}(d,\R)$, which is much larger. Therefore, if we want to prove the well-curvedness of some $\mathbf{Q}$ via the Newton-type polyhedra, then there is no benefit of resorting to Theorem~\ref{thm:gressman_thm6_variant}. However, if we aim to show that some $\mathbf{Q}$ is not well-curved, then Theorem~\ref{thm:gressman_thm6_variant} is much more convenient to use and automatically matches the definition of $\mathfrak{d}_{d',n'}(\mathbf{Q})$. This is the punchline behind the proof of $\circled{8}$ in Theorem~\ref{thm:relation_diagram}.
    
    Note that the action of $\{R_{N,M}\}_{(N,M) \in {\rm GL}(n,\R) \times {\rm GL}(d,\R)}$ forms a group, so we have the following by-product of Theorem~\ref{thm:gressman_thm6_variant} (take $m=n$, $\kappa=2$, $p = \mathbf{Q}$, and $f(\xi) = (\xi, \mathbf{Q}(\xi))$), which is a reflection of the intuition that the ${\rm GL}$-action does not change the ``shape'' of a surface:
    \begin{corollary}\label{cor:quad_equiv_curved}
        For $\mathbf{Q} \equiv \mathbf{Q}'$, $\mathbf{Q}$ is well-curved if and only if $\mathbf{Q}'$ is well-curved.
    \end{corollary}
	
	Now we turn to the proof of Theorem~\ref{thm:gressman_thm6_variant}, which is very similar to that of (1) in Theorem~\ref{thm:gressman_thm2}. The key ingredient is Lemma 4 in \cite{gressman19}, for which we shall further introduce some notations. We equip $\dot{P}_d^\kappa$ with the inner product $\langle\cdot,\cdot\rangle_\kappa$ given by $\langle q, r \rangle_\kappa = \kappa! q(\partial)r(t)|_{t=0}$ for any $q, r \in \dot{P}_d^\kappa$, and let $\norm{\cdot}_\kappa$ be the norm induced by $\langle\cdot,\cdot\rangle_\kappa$. Then we have:
    
    \begin{lemma}[\cite{gressman19}, Lemma 4]\label{lem:nu_simple_express}
        Let $f$ have the form {\rm(\ref{canonical_poly_form})}. Let $p \coloneqq (p_1,\dots,p_m) \in (\dot{P}_d^\kappa)^m$ be the $m$-tuple of highest order parts of $f$. Then $\nu_{\mathcal{A}}$ is a constant times Lebesgue measure on $\R^d$, and
        \begin{align}\label{eq:nu_simple_express}
            \frac{d \nu_{\mathcal{A}}}{d t} = C \Big[\inf_{\substack{N\in {\rm SL}(m,\R)\\ M\in {\rm SL}(d,\R)}} \sum_{j=1}^m \norm{(R_{N,M}p)_j}_\kappa^2\Big]^{\frac{md}{2D}}
        \end{align}
        for some constant $C = C(d,n)$. Here $D = D(d,n)$ is the homogeneous dimension.
    \end{lemma}
    \begin{proof}[Proof of Theorem~\ref{thm:gressman_thm6_variant}]
        The ``if'' part is trivial: If the Newton-type polyhedra condition holds for all $(N,M) \in {\rm GL}(n,\R) \times {\rm GL}(d,\R)$, then it in particular holds for all $(N,M) \in {\rm O}(n,\R) \times {\rm O}(d,\R)$, and by Theorem~\ref{thm:gressman_thm6}, we get well-curvedness. So it remains to prove the ``only if'' part.
        
        By Lemma~\ref{lem:nu_simple_express}, we have
        \begin{align*}
            \frac{d \nu_{\mathcal{A}}}{d t} = C \Big[\inf_{\substack{N\in {\rm SL}(m,\R)\\ M\in {\rm SL}(d,\R)}} \sum_{j=1}^m \norm{(R_{N,M}p)_j}_\kappa^2\Big]^{\frac{md}{2D}}.
        \end{align*}
        By noting that interchanging columns of $N$ or rows of $M$ does not affect the value of the expression, we can replace the infimum above by $\inf_{|\det(N)|=|\det(M)|=1}$. The key point here is that if we replace $p$ with $R_{N',M'}p$ for some $(N',M') \in {\rm GL}(n,\R) \times {\rm GL}(d,\R)$, then the corresponding $d\nu'_{\mathcal{A}}/d t$ differs from $d\nu_{\mathcal{A}}/d t$ by a factor $C'\in \R^+$. This is because we can always write $N' = aN''$ and $M' = bM''$ for some $a, b \in\R^+$ and $|\det(N'')|=|\det(M'')|=1$, under which $$R_{N,M}(R_{N',M'}p) = R_{N,M}\left( R_{N'',M''} (R_{aI,bI}p) \right) = (R_{N,M}\circ R_{N'',M''})(ab^\kappa\cdot p) = ab^\kappa R_{N''N,M''M}p,$$ and then by Lemma~\ref{lem:nu_simple_express} with $p = R_{N',M'}p$ and $f(t) = ((t^\gamma)_{1\leq |\gamma|<\kappa},R_{N',M'}p(t))$ we have
        \begin{align*}
            \frac{d \nu'_{\mathcal{A}}}{d t} &= C \Big[\inf_{\substack{|\det(N)|=1 \\ |\det(M)|=1}} \sum_{j=1}^m \norm{\left(R_{N,M}(R_{N',M'}p)\right)_j}_\kappa^2\Big]^{\frac{md}{2D}}\\
            &= C \Big[\inf_{\substack{|\det(N)|=1\\ |\det(M)|=1}} \sum_{j=1}^m \norm{(ab^\kappa R_{N''N,M''M}p)_j}_\kappa^2\Big]^{\frac{md}{2D}}\\
            &= C (ab^\kappa)^{\frac{md}{D}} \cdot \Big[\inf_{\substack{|\det(N)|=1 \\ |\det(M)|=1}} \sum_{j=1}^m \norm{(R_{N,M}p)_j}_\kappa^2\Big]^{\frac{md}{2D}}\\
            &= (ab^\kappa)^{\frac{md}{D}} \cdot \frac{d \nu_{\mathcal{A}}}{d t}.
        \end{align*}
        
        Therefore, if we assume $\{((t^\gamma)_{1\leq |\gamma|<\kappa}, p(t))\}$ is well-curved, then so is the modified $\{((t^\gamma)_{1\leq |\gamma|<\kappa},R_{N',M'}p(t))\}$ for any $(N',M') \in {\rm GL}(m,\R) \times {\rm GL}(d,\R)$. By applying Theorem~\ref{thm:gressman_thm6} to $R_{N',M'}p$, we know that 
        $$\left( \frac{1}{m} \1_m, \frac{\kappa}{d} \1_d \right) \in \mathcal{N}(R_{O_1,O_2} (R_{N',M'}p))$$ for all $(O_1,O_2) \in {\rm O}(m, \R) \times {\rm O}(d, \R)$. In particular, by taking $O_1$ and $O_2$ to be identity matrices, we get $$\left( \frac{1}{m} \1_m, \frac{\kappa}{d} \1_d \right) \in \mathcal{N}(R_{N',M'}p)$$ for all $(N',M') \in {\rm GL}(m,\R) \times {\rm GL}(d,\R)$, as desired.
    \end{proof}

    \subsection{Several versions of the Fourier extension operator}\label{subsec:E^Q_versions}\phantom{x}

    \subsubsection{Most general version}\phantom{x}
    
    The most general formulation of Fourier extension is with respect to Borel measures. For any Borel measure $\mu$ in $\R^{N}$, we define its associated Fourier extension operator as
    \begin{align*}
        E^\mu f (x) \coloneqq \int_{\R^N} e^{ix\cdot \xi} f(\xi) d\mu(\xi), \qquad x\in \R^N.
    \end{align*}
        
    In this general setting, we can connect the Fourier restriction to the Oberlin condition (see  Definition~\ref{def:oberlin_condition})\footnote{This was mentioned but not proved in \cite{gressman19}.} by the following proposition, which will be of use in the proof of $\circled{6}$ in Theorem~\ref{thm:relation_diagram}:
    \begin{proposition}\label{prop:convex_test}
        Let $p \in (1,\infty]$, $q \in [1,\infty]$, and $\mu$ be any Borel measure in $\R^{N}$. 
        Suppose $E^\mu$ is bounded from $L^p(\mu)$ to $L^q(\R^{N})$, then $\mu$ satisfies the Oberlin condition with exponent $p'/q$.
    \end{proposition}
    
    \begin{proof}
        By the duality of $L^p$-spaces, for any $F \in \mathcal{S}(\R^N)$, we have
        \begin{align*}
            \big\Vert\widehat{F}\big\Vert_{L^{p'}(\mu)} &= \sup_{\norm{g}_{L^p(\mu)}=1} \abs{\int_{\R^N} \widehat{F}g d\mu }\\
            & = \sup_{\norm{g}_{L^p(\mu)}=1} \abs{\int_{\R^N} F \widetilde{E^\mu g}}\\
            & \leq \sup_{\norm{g}_{L^p(\mu)}=1} \norm{F}_{L^{q'}(\R^N)} \norm{E^\mu g}_{L^{q}(\R^N)} \\
            & \lesssim \sup_{\norm{g}_{L^p(\mu)}=1} \norm{F}_{L^{q'}(\R^N)} \norm{g}_{L^p(\mu)} = \norm{F}_{L^{q'}(\R^N)},
        \end{align*}
        which is the ``restriction'' formulation of $E^\mu$. Here $\widetilde{E^\mu g}(x) \coloneqq E^\mu g (-x)$.
        
        For any rectangular box $T$ in $\R^N$, let $T = L(\prod_{i=1}^N[-c_i,c_i])$ for some isometry $L$, then $|T| \sim \prod_{i=1}^N c_i$. Fix a $\phi \in C_c^\infty(\R)$ with $\phi \equiv 1$ on $[-1,1]$. Let  $\Phi(\xi) \coloneqq \prod_{i=1}^N \phi(\xi_i/c_i)$, then $\Phi \in C_c^\infty(\R^N)$ with $\Phi \equiv 1$ on $\prod_{i=1}^N[-c_i,c_i]$, so $\Phi_T \coloneqq \Phi \circ L^{-1} \in C_c^\infty(\R^N)$ satisfies $\Phi_T \equiv 1$ on $T$. Plugging $F = (\Phi_T)^\vee$ into the restriction estimate above, we get $$\norm{\Phi_T}_{L^{p'}(\mu)} \lesssim \norm{(\Phi_T)^\vee}_{L^{q'}(\R^N)}.$$ 
        Furthermore, we have 
        $$\norm{(\Phi_T)^\vee}_{L^{q'}(\R^N)} = \norm{(\Phi)^\vee}_{L^{q'}(\R^N)} = \prod_{i=1}^N \norm{(\phi(\cdot/c_i))^\vee}_{L^{q'}(\R)}  \sim_\phi \prod_{i=1}^N c_i^{1/q} \sim |T|^{1/q}.$$ 
        However, note that $$\norm{\Phi_T}_{L^{p'}(\mu)} \geq \norm{\chi_T}_{L^{p'}(\mu)} = \mu(T)^{1/p'}.$$  
        Therefore, we get 
        $$\mu(T) \lesssim |T|^{p'/q}$$
        for any rectangular box $T$ in $\R^N$, with the constant independent of $T$. Finally, by Lemma~\ref{lem:Oberlin_equiv}, we conclude that $\mu$ satisfies the Oberlin condition with exponent $p'/q$, as desired.
    \end{proof}
    \begin{remark}
        Proposition~{\rm\ref{prop:convex_test}} does not include the case when $p=1$, which is actually trivial when $\mu$ is $\sigma$-finite {\rm(}which is always the case in practice{\rm)}. The proof above mostly carries through when $p=1$, with the duality justified by the fact that $\mu$ is $\sigma$-finite. The only difference is that when we arrive at $\norm{\Phi_T}_{L^{p'}(\mu)} \geq \norm{\chi_T}_{L^{p'}(\mu)}$, we may without loss of generality assume $\mu(T)>0$ {\rm(}we can always pick such a $T$ as long as $\mu$ is not trivial{\rm)}, under which we have $\norm{\chi_T}_{L^{p'}(\mu)} = 1$. Therefore, we get $1 \lesssim |T|^{1/q}$. Note that we can choose $T$ such that $\mu(T)>0$ and $|T|$ is arbitrarily small by subdividing the original $T$, so $1 \lesssim |T|^{1/q}$ cannot be true unless $q=\infty$. On the other hand, if $p=1$ and $q=\infty$, then one trivially has the boundedness of $E^\mu$ from $L^1(\mu)$ to $L^\infty(\R^N)$.
    \end{remark}
    
    \subsubsection{Fourier extension for submanifolds and its neighborhood version}\phantom{x}

    From now on, we restrict ourselves to the case when the Borel measure $\mu$ in $E^\mu$ is a surface measure. 
    
    Suppose that $S_\psi \coloneqq \{ (\xi,\psi(\xi)): \xi \in U \}$ is a graph in $\mathbb{R}^{d+n}$, where $U$ is a bounded open/closed region of $\R^d$ and $\psi:U\rightarrow\R^n$ is a smooth function. Then we define the Fourier extension operator    \begin{align}\label{def:general_ext_op}
        E_U^\psi f(x) \coloneqq \int_U e^{ix\cdot(\xi,{\psi(\xi)})} f(\xi) d\xi, \qquad x\in\R^{d+n},
    \end{align}
    as well as the (bump) smoothed Fourier extension operator
    \begin{align}\label{def:bump_extension}
        \widetilde{E}_\varphi^\psi f(x) \coloneqq \int_{\R^{d}} e^{ix\cdot(\xi,{\psi(\xi)})} f(\xi) \cdot\varphi(\xi) d\xi, \qquad x\in\R^{d+n},
    \end{align}
    where $\varphi \in C_c^\infty(\R^d)$ with $\varphi \not\equiv 0$ on $U$. Also, when $U = [0,1]^d$, we abbreviate $E_U^\psi$ to $E^\psi$. This matches our definition of $E^\mathbf{Q}$ in (\ref{s1e0}).

    Now we introduce the thickening lemma, which basically says that the $L^p \rightarrow L^q$ estimates of $E_U^\psi$ are equivalent to those of its neighborhood version. As this is well-known in the Fourier restriction community but hard to find a directly applicable version in the literature, we will largely omit the proofs and only state the version that is the most useful to us. 
    
    \begin{lemma}[Thickening]\label{lem:thickening}
        For any $\alpha\geq 0$, the following are equivalent:
    
        \noindent ${(i)}$ $$\norm{E_U^\psi f}_{L^q(B_R)} \lesssim R^\alpha \norm{f}_{L^p(U)}$$
            holds for any $f\in L^p(U)$ and $R\geq 1$.
    
        \noindent ${(ii)}$ $$\big\Vert\widehat{F}\big\Vert_{L^q(B_R)} \lesssim R^{\alpha-\frac{n}{p'}} \norm{F}_{L^p(N_{R^{-1}}(S_\psi))}$$
            holds for any $F \in L^p(N_{R^{-1}}(S_\psi))$ and $R\geq 1$.
    \end{lemma}
    The proof of Lemma~\ref{lem:thickening} is exactly the same as that of Proposition~1.27 in \cite{demeter2020}, which handles the $n=1$ case.

    \subsubsection{Smoothed versions for quadratic forms}\phantom{x}

    From now on, we will further restrict ourselves to the case when $U = [0,1]^d$, $\psi = \mathbf{Q}$, and $S_\psi = S_\mathbf{Q}$. Our main goal here is to prove the equivalence between $E^\mathbf{Q}$ and its smoothed versions in terms of (weighted) restriction theory. Before that, we need to first introduce some notations.
    
    Let $x = (x',x'') \in\R^{d+n}$ with $x' = (x_1,\dots,x_d)\in\R^d$ and $x'' = (x_{d+1},\dots,x_{d+n})\in \R^n$, $\xi = (\xi_1,\dots,\xi_d) \in \R^d$, and $\mathbf{Q}(\xi) = (Q_1(\xi),\dots,Q_n(\xi))$. Then the Hessian matrix $\nabla_\xi^2 Q_j$ of each quadratic form $Q_j(\xi)$ satisfies
      $$\quad Q_j(\xi) = \frac{1}{2}\xi^T\cdot \nabla_\xi^2 Q_j \cdot \xi.$$ Thus if we define the $d \times d$ matrix 
    \begin{equation}\label{098765}
        \overline{Q}(x'') \coloneqq \sum_{j=1}^n x_{d+j}\nabla_\xi^2Q_j,
    \end{equation}
    then we can rewrite the Fourier extension operator in (\ref{s1e0}) as
    \begin{align*}
        E^\mathbf{Q}f(x) = \int_{[0,1]^{d}} e^{i(x'\cdot\xi + \xi^T\cdot \overline{Q}(x'')\cdot\xi/2)} f(\xi) d\xi, \quad \quad x = (x',x'') \in \R^{d+n},
    \end{align*}
    and the (Gaussian) smoothed Fourier extension operator in (\ref{eq:gaussian_integral}) as
    \begin{align}
        \widetilde{E}^\mathbf{Q}f(x) = \int_{\R^d} e^{i(x'\cdot\xi + \xi^T\cdot \overline{Q}(x'')\cdot\xi/2)} f(\xi) \cdot e^{-|\xi|^2/2} d\xi, \quad\quad x = (x',x'') \in \R^{d+n}.
    \end{align}
    
    It would also be convenient to formally introduce affine transformations preserving $\mu_\mathbf{Q}$ here. For any $\xi_0 \in \R^d$, we define the affine transformation $L_{\xi_0}$ on $\R^{d+n}$ as
    \begin{align}\label{def:affine_trans}
        L_{\xi_0}(\xi,\eta) \coloneqq (\xi - \xi_0, \eta - \nabla_\xi\mathbf{Q}(\xi_0)\cdot\xi - \mathbf{Q}(\xi_0)), \qquad \forall~(\xi,\eta) \in\R^d \times \R^n. 
    \end{align}
    The effect of $L_{\xi_0}$ is to translate the data function on $\R^d$ by $-\xi_0$ on the frequency side. Correspondingly, on the physical side, we will distort everything by the adjoint transformation $L_{\xi_0}^* \coloneqq (A^T)^{-1}$, where $A \in {\rm SL}(d+n,\R)$ is the linear part of $L_{\xi_0}$. Both $L_{\xi_0}$ and $L_{\xi_0}^*$ preserve volumes in $\R^{d+n}$, and $L_{\xi_0}$ also preserves $\mu_\mathbf{Q}$.

    Now we show that weighted restriction theories of $E^\mathbf{Q}$ and $\widetilde{E}^\mathbf{Q}$ are equivalent, which justifies our usage of the smoothing trick in Shayya's argument (Subsection~\ref{subsec:shayya}):
    \begin{proposition}\label{prop:smooth_equiv}
        Let $p,q \in [1,\infty]$ and $\alpha\in (0,d+n]$. Then for any $\beta >0$ and $\eta \geq 0$, the following two statements are equivalent: 
        
        \noindent $(i)$ We have \begin{equation*}\label{2211}
                \norm{E^\mathbf{Q}f}_{L^q(B_R,H)} \lesssim A_\alpha(H)^\beta R^\eta \norm{f}_{L^p([0,1]^d)}
            \end{equation*}
            for any $\alpha$-dimensional weight $H$, $f \in L^p([0,1]^d)$, and $R\geq 1$.
    
        \noindent $(ii)$ We have \begin{equation*}\label{2212}
                \norm{\widetilde{E}^\mathbf{Q}f}_{L^q(B_R,H)} \lesssim A_\alpha(H)^\beta R^\eta \norm{f}_{L^p(\R^d)}  
            \end{equation*}    
            for any $\alpha$-dimensional weight $H$, $f \in L^p(\mathbb{R}^d)$, and $R\geq 1$. 
    \end{proposition}
    \begin{proof}
        $(ii)\Rightarrow(i)$: This is immediate by only considering those $f\in L^p([0,1]^d)$ and noting that $e^{-|\xi|^2/2} \sim 1$ over $[0,1]^d$.
        
        $(i)\Rightarrow(ii)$: Partition $\R^d$ into unit cubes $\{\nn + [0,1]^d\}_{\nn\in \Z^d}$. We claim $(i)$ implies
        $$       \norm{E^\mathbf{Q}f}_{L^q(B_R,H)} \lesssim A_\alpha(H)^\beta R^\eta (1+|\nn|)^C\norm{f}_{L^p(\nn +[0,1]^d)},     $$
        for some constant $C$. In fact, note that $L_\nn$ translates the data function on $\R^d$ from $\nn + [0,1]^d$ to $[0,1]^d$ on the frequency side, while $L_\nn^*$ preserves volumes and distorts everything on the physical side by a factor polynomially bounded by $\nn$. Thus, if $\widetilde{H} \coloneqq L_\nn^* (H)$ and $B_{\widetilde{R}} $ is the smallest ball containing $L_\nn(B_R)$, then $A_\alpha(\widetilde{H})$ and $\widetilde{R}$ will differ from $A_\alpha(H)$ and $R$ at most by a factor of $(1+|\nn|)^C$ respectively for some constant $C$. The claim is now self-evident.
        
        Let $f_\nn \coloneqq f\chi_{\nn+[0,1]^d}$, then
        \begin{align*}
            \norm{\widetilde{E}^\mathbf{Q}f}_{L^q(B_R,H)}
            \leq & \sum_{\nn\in\Z^d} \norm{\widetilde{E}^\mathbf{Q}f_\nn}_{L^q(B_R,H)}\\
            \lesssim & A_\alpha(H)^\beta R^\eta \sum_{\nn\in\Z^d} (1+|\nn|)^C \norm{f_\nn e^{-|\cdot|^2/2}}_{L^p(\R^d)}\\
            \lesssim& A_\alpha(H)^\beta R^\eta \sum_{\nn\in\Z^d} (1+|\nn|)^C e^{-|\nn|^2/2} \norm{f_\nn}_{L^p(\R^d)}\\
            \leq & A_\alpha(H)^\beta R^\eta \left( \sum_{\nn\in\Z^d} \Big((1+|\nn|)^C e^{-|\nn|^2/2} \Big)^{p'} \right)^{1/p'} \left( 
            \sum_{\nn\in\Z^d} \norm{f_\nn}_{L^p(\R^d)}^p \right)^{1/p}\\
            \lesssim & A_\alpha(H)^\beta R^\eta \norm{f}_{L^p(\R^d)}.
        \end{align*}
        And the proof is completed.
    \end{proof}

    Next, we show that restriction estimates of $E^\mathbf{Q}$, $\widetilde{E}^\mathbf{Q}$, and $\widetilde{E}_\varphi^\mathbf{Q}$ are all equivalent to each other. Besides justifying smoothing tricks, this also saves us from worrying about compatibility when we cite results in the literature later on, which may be stated for any of the three versions. In particular, the following proposition will be of use in the proof of $\circled{2}$ and $\circled{5}$ in Theorem~\ref{thm:relation_diagram}:
    \begin{proposition}\label{prop:restriction_equiv}
        Let $p,q \in [1,\infty]$, and $\varphi \in C_c^\infty(\R^d)$ with $\varphi\not\equiv0$. The following three estimates are equivalent to each other:
    
        \noindent $(i)$ $\norm{E^\mathbf{Q}f}_{L^q(\R^{d+n})} \lesssim \norm{f}_{L^p([0,1]^d)}, \,\forall f \in L^p([0,1]^d).$
    
        \noindent $(ii)$ $\norm{\widetilde{E}^\mathbf{Q} f}_{L^q(\R^{d+n})} \lesssim \norm{f}_{L^p(\R^d)}, \,\forall f \in L^p(\R^d).$
    
        \noindent $(iii)$ $\norm{\widetilde{E}_{\varphi}^\mathbf{Q} f}_{L^q(\R^{d+n})} \lesssim \norm{f}_{L^p(\R^d)}, \,\forall f \in L^p(\R^d).$
    \end{proposition}
    \begin{proof}
        The spirit is very similar to that of Proposition~\ref{prop:smooth_equiv}, so we only briefly sketch the proof here.
        
        $(i) \Rightarrow (ii)$: Partition $\R^d$ into unit cubes $\{\nn + [0,1]^d\}_{\nn\in \Z^d}$. By applying $L_\nn$ to $f$ and $L_\nn^*$ to $E_{\nn+[0,1]^d}^\mathbf{Q}f$, we see that $(i)$ implies
        % \footnote{Here we slightly abuse $E^\mathbf{Q}$, so that the integral domain is $\nn+[0,1]^d$ (or $\cup_{f}\supp f$) instead of $[0,1]^d$. We will always choose to do this when we think the context is clear and there is no potential confusion.}
        $$       \norm{E_{\nn+[0,1]^d}^\mathbf{Q}f}_{L^q(\R^{d+n})} \lesssim \norm{f}_{L^p(\nn +[0,1]^d)} $$
        with constant uniform in $\nn\in\Z^d$. By exploiting the rapid decay property of $e^{-|\xi|^2/2}$, we can sum the estimates on each $\nn + [0,1]^d$ up in a well-controlled way as we did in Proposition~\ref{prop:smooth_equiv}, which yields the desired estimate for $\widetilde{E}^\mathbf{Q}$.
        
        $(ii) \Rightarrow (iii)$: Let $\zeta = (\zeta',\zeta'') \in  \R^{d+n}$ with $\zeta' \in\R^d$ and $\zeta'' \in \R^n$. Let $\Phi(\zeta) \coloneqq \varphi(\zeta')\psi(\zeta'')$ for some fixed $\psi \in C_c^\infty(\R^n)$, such that $\Phi(\xi,\mathbf{Q}(\xi)) = \varphi(\xi)$ for all $\xi \in \R^d$, and $\widetilde{\Phi}(\zeta) \coloneqq \Phi(\zeta)e^{|\zeta'|^2/2}$. Then $\Phi, \widetilde{\Phi} \in C_c^\infty(\R^{d+n})$, and we have
        \begin{align*}
            \abs{\widetilde{E}_\varphi^\mathbf{Q}f} = \abs{(\Phi fd\mu_\mathbf{Q})^\vee} &= \abs{\left( \widetilde{\Phi}(\zeta) f(\zeta')e^{-|\zeta'|^2/2} 
                d\mu_\mathbf{Q} \right)^\vee}\\
            &= \abs{\left( \widetilde{\Phi}(\zeta)\right)^\vee * \left( f(\zeta') e^{-|\zeta'|^2/2} d\mu_\mathbf{Q} \right)^\vee}\\
            &= \abs{\left( \widetilde{\Phi}(\zeta)\right)^\vee * \widetilde{E}^\mathbf{Q}f}.
        \end{align*}
        Since $\widetilde{\Phi}  ^\vee \in \mathcal{S}(\R^{d+n}) \subset L^1(\R^{d+n})$, by Young's convolution inequality we have 
        \begin{align*}
            \norm{\widetilde{E}_\varphi^\mathbf{Q}f}_{L^q(\R^{d+n})} = \norm{\left( \widetilde{\Phi}(\zeta)\right)^\vee * \widetilde{E}^\mathbf{Q}f}_{L^q(\R^{d+n})} \lesssim \norm{\widetilde{E}^\mathbf{Q}f}_{L^q(\R^{d+n})} \lesssim \norm{f}_{L^p(\R^d)},
        \end{align*}
        as desired.
        
        $(iii) \Rightarrow (i)$: According to our assumption, there exists $\xi_0 \in \R^d$ such that $|\varphi(\xi_0)| = \max_{\xi\in\R^d}|\varphi(\xi)| \neq 0$. Without loss of generality, we may assume that $\varphi(\xi_0) > 0$. By continuity, we can then fix a dyadic $0 < c_0 < 1$ such that $\varphi> \varphi(\xi_0)/2$ on $\xi_0 + [0,c_0]^d$. Only considering those $f\in L^p(\xi_0 + [0,c_0]^d)$ and noting that $\varphi \sim 1$ over $\xi_0 + [0,c_0]^d$, we see that $(iii)$ implies  $$\norm{E_{\xi_0 + [0,c_0]^d}^\mathbf{Q}f}_{L^q(\R^{d+n})} \lesssim \norm{f}_{L^p(\xi_0 + [0,c_0]^d)}.$$ 
        Partition $[0,1]^d$ into $\{c_0\nn + [0,c_0]^d \}_{\nn\in\Z^d\cap[0, c_0^{-1})^d}$. By applying $L_{\xi_0 - c_0\nn}$ to $f$ and $L_{\xi_0 - c_0\nn}^*$ to $E^\mathbf{Q}f$, we see that $E^\mathbf{Q}: L^p(c_0\nn + [0,c_0]^d) \rightarrow L^q(\R^{d+n})$ uniformly holds for any $\nn \in \Z^d\cap [0,c_0^{-1})^d$. Since there are only finitely many pieces, by summing them up we get the desired boundedness of $E^\mathbf{Q}$ over the whole $[0,1]^d$.
    \end{proof}
    \begin{remark}\label{rmk:restriction_equiv}
        Firstly, one can see that the proof of Proposition~{\rm\ref{prop:restriction_equiv}} actually works for all {\rm TDI} manifolds {\rm(}Definition~{\rm\ref{def:TDI_mfd})}, not necessarily quadratic manifolds. Secondly, an immediate corollary of Proposition~{\rm\ref{prop:restriction_equiv}} is that for any two $\varphi_1, \varphi_2 \in C_c^\infty(\R^d)$ with $\varphi_1,\varphi_2\not\equiv0$, $\widetilde{E}_{\varphi_1}^\mathbf{Q}$ and $\widetilde{E}_{\varphi_2}^\mathbf{Q}$ have the same $L^p \rightarrow L^q$ boundedness property. Finally, there is no special role played by $[0,1]^d$ in ${\rm(}i{\rm)}$, and our proof works equally well when $E^{\mathbf{Q}}$ is replaced by  $E_B^{\mathbf{Q}}$ for any cube $B$ in $\R^d$ {\rm(}not necessarily of unit side length or parallel to the coordinate axes{\rm)}. In particular, since the form of ${\rm(}ii{\rm)}$ is symmetric and remains the same, we know that $E_B^{\mathbf{Q}}$ and ${E}^{\mathbf{Q}}$ have the same $L^p \rightarrow L^q$ boundedness property.
    \end{remark}

	\section{Uniform Fourier decay for quadratic manifolds}\label{sec2}
	
	In this section, we will provide a complete characterization of the uniform Fourier decay for quadratic manifolds of arbitrary codimensions, i.e., (\ref{thm:uniform_decay_eq}). As we have pointed out, this estimate is not only used in the proof of weighted restriction estimates in Section \ref{sec3}, but also plays a role in establishing relationships between various nondegeneracy conditions in Section~\ref{appendix:relationships}. Note that, when $n=1$, the Fourier decay of surface measure is essential and plays a role in many problems, such as Stein-Tomas estimates, Stein's spherical maximal estimates, Falconer's distance set problem, Strichartz estimates for the Schrödinger equation and so on. So we believe that (\ref{thm:uniform_decay_eq}), as one generation of the case $n=1$, will have further applications beyond those presented in this paper.

	In view of the equivalence between smooth and nonsmooth versions of Fourier extension operators (Proposition~\ref{prop:smooth_equiv}), we may without loss of generality work with $\widetilde{E}^\mathbf{Q}$ instead of $E^\mathbf{Q}$. Our main result in this subsection is as follows:
	\begin{theorem}[Uniform Fourier decay]\label{thm:uniform_decay}
		We have
		\begin{equation}\label{thm:uniform_decay_eq}
			|\widetilde{E}^\mathbf{Q}1(x)| \lesssim_\mathbf{Q} (1+|x|)^{-\frac{\mathfrak{d}_{d,1}(\mathbf{Q})}{2}}, \quad \quad x\in \R^{d+n}.
		\end{equation}
		Moreover, this bound is optimal in the sense that for any $ \gamma > \frac{\mathfrak{d}_{d,1}(\mathbf{Q})}{2}$, we cannot have $|\widetilde{E}^\mathbf{Q}f(x)| \lesssim_\mathbf{Q} (1+|x|)^{-\gamma}$ uniformly for all $x\in \R^{d+n}$.
	\end{theorem}

	Before proving Theorem~\ref{thm:uniform_decay}, we first collect some preparatory results in matrix analysis.
	
	Given any quadratic form $\mathbf{Q}$, define $\rho_\mathbf{Q}$ to be its ``uniform spectral radius'', i.e.,
	\begin{align*}
		\rho_\mathbf{Q} \coloneqq \sup_{\theta\in\mathbb{S}^{n-1}} \rho({\overline{Q}(\theta)}),
	\end{align*}
	where $\rho(A) \coloneqq \{|\lambda|: \lambda \in \C \text{ is the largest eigenvalue of } A\}$ is the classical spectral radius for any $A \in \R^{d\times d}$.
	In other words, $\rho_\mathbf{Q}$ is the supremum of the spectral radii of all (normalized) linear combinations of Hessian matrices $\{\nabla_\xi^2 Q_j\}_{j=1}^n$ associated with $\mathbf{Q}$. However, it is not clear a priori whether $\rho_\mathbf{Q}$ is finite or not, as well as how it depends on $\mathbf{Q}$. The following results answer these questions.
	
	\begin{lemma}[\cite{hj12matrix}, Corollary 6.1.5]\label{lem:single_spectral}
		If $A = [a_{ij}] \in \R^{d\times d}$, then 
		\begin{align*}
			\rho(A) \leq \min\Big\{\max_i \sum_{j=1}^d |a_{ij}|, \max_j \sum_{i=1}^d |a_{ij}|\Big\}.
		\end{align*}
	\end{lemma}
	
	Note that Lemma~\ref{lem:single_spectral} is quantitative and works for any square matrices, not just for real symmetric matrices like $\overline{Q}(\theta)$. It is an immediate corollary of the famous Gershgorin circle theorem (one may consult Theorem 6.1.1 of \cite{hj12matrix}). By using this lemma, one can easily see that:
	\begin{corollary}\label{cor:uniform_spectral}
		We have $\rho_\mathbf{Q} \leq C_\mathbf{Q}$ for some constant $C_\mathbf{Q}$ depending only on $\mathbf{Q}$.
	\end{corollary}
	\begin{proof}
		Combine the definition of $\rho_\mathbf{Q}$, Lemma~\ref{lem:single_spectral}, and the following simple fact
		$$  \Big|\sum_{j=1}^n \theta_j b_j\Big| \leq |\theta||b| = |b|,   $$
		for any $\theta \in \mathbb{S}^{n-1} $ and $b \in \R^n$.
	\end{proof}
	By carrying out the proof carefully, it is easy to write out the upper bound $C_\mathbf{Q}$ explicitly. For later convenience, we may always assume $C_\mathbf{Q} \geq 1$.
	
	The following classical result in matrix analysis is vital for us to obtain uniform bounds:
	\begin{lemma}\label{lem:perturb}
		For any matrix in ${\rm S}(d,\mathbb{R})$, its eigenvalues depend continuously on its elements, in the sense that if $A = [a_{ij}], B = [b_{ij}] \in {\rm S}(d,\mathbb{R})$, $\lambda_1 \leq \cdots \leq \lambda_d$ are eigenvalues of $A$, $\widehat{\lambda_1} \leq \cdots \leq \widehat{\lambda_d}$ are eigenvalues of $A+B$, then for any $i,j$, there is
		$$ \sum_{k=1}^d |\widehat{\lambda_k} - \lambda_k| \rightarrow 0,  \quad \quad\quad\text{~as~~~}~~ b_{ij} \rightarrow 0.$$ 
	\end{lemma}
	\begin{proof}
		Although by linear algebra, all eigenvalues of any matrix in ${\rm S}(d,\mathbb{R})$ are real numbers, it is the easiest to work in the field of complex numbers. For any $M \in \R^{d\times d}$, let $f_M(\lambda) \coloneqq \det(\lambda I - M)$ be its characteristic polynomial. Then by our assumption, $\lambda_1 \leq \cdots \leq \lambda_d$ are the roots of 
		$f_A$, while $\widehat{\lambda_1} \leq \cdots \leq \widehat{\lambda_d}$ are the roots of $f_{A+B}$. Note that the coefficients of $f_{A+B}(\lambda)$ depend continuously on $b_{ij}$. Thus by Rouché's theorem, if $f_A(\lambda_0) = 0$ for some $\lambda_0 \in \{\lambda_1, \dots, \lambda_d\}$, then for any small $\epsilon > 0$ such that $\lambda_0$ is the unique zero in $B_\epsilon(\lambda_0)$, there exists $\delta>0$ such that $f_{A+B}$ and $f_A$ have the same number of zeros inside $B_\epsilon(\lambda_0)$ (each zero is counted as many times as its multiplicity) whenever $|b_{ij}| \leq \delta, \forall \,i,j$. This just tells us that $|\widehat{\lambda_k} - \lambda_k| \rightarrow 0$ for each $k$ as $b_{ij} \rightarrow 0, \forall\, i,j$, so the proof is completed. 
	\end{proof}
	% \begin{remark}
		%     For more general and quantitative versions of Lemma~\ref{lem:perturb} (with other proof strategies), one may consult Section 6.3 of \cite{hj12matrix}. \note{may remove?}
		% \end{remark}
	
	For any $\theta\in\mathbb{S}^{n-1}$, let $\{\mu_j(\theta)\}_{j=1}^d$ be the $d$ real eigenvalues of $\overline{Q}(\theta)$ and $\Lambda_k(\theta)$ be the $k$-th largest element (if there are multiple choices, then take any one of them) in $\{|\mu_j(\theta)|\}_{j=1}^d$. Then we have:
	\begin{corollary}\label{cor:k_continuity}
		For any $k$, $\Lambda_k(\theta)$ is a continuous function on $\mathbb{S}^{n-1}$.
	\end{corollary}
	In particular, by taking $k=1$ in Corollary~\ref{cor:k_continuity} and using the compactness of $\mathbb{S}^{n-1}$, we recover Corollary~\ref{cor:uniform_spectral} qualitatively. However, the full power of Corollary~\ref{cor:k_continuity} is that we can get a strictly positive lower bound of all eigenvalues, once we throw away those degenerate ones. Let 
	$$m \coloneqq \inf_{\theta\in \mathbb{S}^{n-1}} \#\Big\{j\in \{1,\dots,d\}: \mu_j(\theta)\neq0\Big\},$$ 
	then $\Lambda_1(\theta), \dots, \Lambda_m(\theta)$ are all strictly positive continuous functions on $\mathbb{S}^{n-1}$. Thus, by the compactness of $\mathbb{S}^{n-1}$, we get $\Lambda_k(\theta) \geq c_\mathbf{Q}$ on $\mathbb{S}^{n-1}$ for any $k = 1,\dots,m$, where $c_\mathbf{Q}$ is a strictly positive constant depending only on $\mathbf{Q}$. 
	
	We can summarize what we have got so far as:
	\begin{theorem}\label{thm:U_L_bound_eigen}
		For any $\theta \in \mathbb{S}^{n-1}$ and $k=1,...,m$, we have
		$$0<c_\mathbf{Q} \leq \Lambda_k(\theta) \leq C_\mathbf{Q} <\infty.$$ 
	\end{theorem}
	
	Before proceeding, let us first point out that $m$ actually enjoys a clean algebraic characterization by language of $\mathfrak{d}_{d',n'}(\mathbf{Q})$: 
	\begin{lemma}\label{lem:alg_char_m}
		$m = \mathfrak{d}_{d,1}(\mathbf{Q})$. 
	\end{lemma}
	\begin{proof} 
		Set $x''=|x''| \theta$. Then
		\begin{align*}
			\mathfrak{d}_{d,1}(\mathbf{Q}) & = \inf_{\substack{M\in \R^{d\times d}\\ \rank(M) = d}} \inf_{\substack{M'\in \R^{n\times 1}\\ \rank(M) = 1}} {\rm NV}((\mathbf{Q}\circ M) \cdot M')\\
			& = \inf_{x'' \in \R^n\setminus\{0\}}
			\inf_{\substack{M\in \R^{d\times d}\\ \rank(M) = d}} {\rm NV}(\overline{Q}(x'') \circ M)\\
			& = 
			\inf_{x'' \in \R^n\setminus\{0\}} \rank(\overline{Q}(x''))\\
			& = \inf_{\theta \in \mathbb{S}^{n-1}} \rank(\overline{Q}(\theta))\\
			& = \inf_{\theta \in \mathbb{S}^{n-1}} \#\{j\in \{1,\dots,d\}: \mu_j(\theta)\neq0\}\\
			& = m,
		\end{align*}
		where in the third line we used Lemma A.2 in \cite{ggo23}.
	\end{proof}
    Although the proof of Lemma~\ref{lem:alg_char_m} is fairly straightforward, it is the key bridge relating Fourier decay to algebraic quantities in the literature.
    
	Now we begin the proof of Theorem~\ref{thm:uniform_decay}, which is divided into several steps.
	\begin{proposition}\label{prop:direct_compute}
		The following bound holds uniformly for all $x \in \R^{d+n}$:
		\begin{align*}
			\abs{\widetilde{E}^\mathbf{Q}1 (x)} \lesssim e^{-\frac{1}{2}\frac{|x'|^2}{1+\rho_\mathbf{Q}^2|x''|^2}} \cdot \prod_{j=1}^d (1+|x''||\mu_j(\theta)|)^{-\frac{1}{2}},
		\end{align*}
		where $\theta = x''/|x''| \in \mathbb{S}^{n-1}$. 
	\end{proposition}
	\begin{proof}
		We perform some direct computations, which is inspired by \cite{mockenhaupt96}:
		\begin{align}\label{eq:tilde_E_pointwise}
			\abs{\widetilde{E}^\mathbf{Q}1(x)} 
			& = \abs{\int_{\R^d} e^{i(x'\cdot\xi + \xi^T\cdot \overline{Q}(x'')\cdot\xi/2)} \cdot e^{-|\xi|^2/2} d\xi}\nonumber\\
			& = \abs{\int_{\R^d} e^{i x'\cdot\xi} e^{- \xi^T\cdot (I - i\overline{Q}(x'')) \cdot\xi/2} d\xi}\nonumber\\
			& = \abs{(2\pi)^{\frac{d}{2}} \cdot \det(I - i\overline{Q}(x''))^{-\frac{1}{2}} \cdot e^{-{x'}^T\cdot (I - i\overline{Q}(x''))^{-1} \cdot x'/2}} \nonumber\\
			& \sim \abs{\det(I - i\overline{Q}(x''))}^{-\frac{1}{2}} \cdot \abs{e^{-{x'}^T\cdot (I - i\overline{Q}(x''))^{-1} \cdot x'/2}},
		\end{align}
		where in the second to last line we used Lemma 2.4 in \cite{mockenhaupt96}.
		
		We estimate the two factors above separately. Firstly, by simply noting that $\{1 - i|x''|\mu_j(\theta)\}_{j=1}^d$ are eigenvalues of $I - i\overline{Q}(x'')$, we get
		\begin{align}\label{eq:det_eigen_express}
			\abs{\det(I - i\overline{Q}(x''))}^{-\frac{1}{2}} 
			=& \abs{\prod_{j=1}^d (1-i|x''|\mu_j(\theta))}^{-\frac{1}{2}} \nonumber\\
			=& \prod_{j=1}^d (1 + |x''|^2|\mu_j(\theta)|^2)^{-\frac{1}{4}} \nonumber\\
			\sim & \prod_{j=1}^d (1 + |x''||\mu_j(\theta)|)^{-\frac{1}{2}}.
		\end{align}
		Secondly, by linear algebra, there exists $O \in {\rm O}(d,\mathbb{R})$ such that
		\begin{align*}
			\overline{Q}(x'') = |x''|\cdot O^T
			\begin{bmatrix}
				\mu_1(\theta) & & \\
				& \ddots & \\
				& & \mu_d(\theta)
			\end{bmatrix} O.
		\end{align*}
		Thus
		\begin{align*}
			I - i \overline{Q}(x'') & = O^T
			\begin{bmatrix}
				1 - i|x''|\mu_1(\theta) & & \\
				& \ddots & \\
				& & 1 - i|x''|\mu_d(\theta)
			\end{bmatrix} O,\\
		\end{align*}
		and then
		\begin{align*}
			(I - i \overline{Q}(x''))^{-1} & = 
			O^T
			\begin{bmatrix}
				\frac{1}{1 - i|x''|\mu_1(\theta)} & & \\
				& \ddots & \\
				& & \frac{1}{1 - i|x''|\mu_d(\theta)}
			\end{bmatrix} O\\
			& = O^T
			\begin{bmatrix}
				\frac{1 + i|x''|\mu_1(\theta)}{1 + |x''|^2|\mu_1(\theta)|^2} & & \\
				& \ddots & \\
				& & \frac{1 + i|x''|\mu_d(\theta)}{1 + |x''|^2|\mu_d(\theta)|^2}
			\end{bmatrix} O,
		\end{align*}
		which implies
		\begin{align*}
			\abs{e^{-{x'}^T\cdot (I - i\overline{Q}(x''))^{-1} \cdot x'/2)}} & = 
			\abs{ \exp\left(- \frac{1}{2} (Ox')^T
				\begin{bmatrix}
					\frac{1 + i|x''|\mu_1(\theta)}{1 + |x''|^2|\mu_1(\theta)|^2} & & \\
					& \ddots & \\
					& & \frac{1 + i|x''|\mu_d(\theta)}{1 + |x''|^2|\mu_d(\theta)|^2}
				\end{bmatrix} (Ox') \right) }\\
			& = \exp\left(- \frac{1}{2} (Ox')^T
			\begin{bmatrix}
				\frac{1}{1 + |x''|^2|\mu_1(\theta)|^2} & & \\
				& \ddots & \\
				& & \frac{1}{1 + |x''|^2|\mu_d(\theta)|^2}
			\end{bmatrix} (Ox') \right) \\
			& \leq \exp\left(- \frac{1}{2} (Ox')^T
			\begin{bmatrix}
				\frac{1}{1 + |x''|^2\rho_\mathbf{Q}^2} & & \\
				& \ddots & \\
				& & \frac{1}{1 + |x''|^2\rho_\mathbf{Q}^2}
			\end{bmatrix} (Ox') \right)\\
			& = e^{-\frac{1}{2}\frac{|Ox'|^2}{1 + |x''|^2\rho_\mathbf{Q}^2}} \\
			& = e^{-\frac{1}{2}\frac{|x'|^2}{1 + |x''|^2\rho_\mathbf{Q}^2}}.
		\end{align*}
		And the proof is completed.
	\end{proof}
	
	Proposition~\ref{prop:direct_compute} immediately implies the following Fourier decay estimate, which is essentially Lemma 2.1.1 of \cite{banner02}:
	\begin{proposition}\label{prop:recover_banner}
		The following bound holds uniformly for all $x \in \R^{d+n}$:
		\begin{align*}
			\abs{\widetilde{E}^\mathbf{Q}1 (x)} \lesssim \prod_{j=1}^d (1+|x||\mu_j(\theta)|)^{-\frac{1}{2}},
		\end{align*}
		where $\theta = x''/|x''| \in \mathbb{S}^{n-1}$. 
	\end{proposition}
	\begin{proof}
		If $|x| \leq 10$, note $|x||\mu_j(\theta)| \leq 10C_\mathbf{Q}$, then 
		\begin{align*}
			\abs{\widetilde{E}^\mathbf{Q}1 (x)} \lesssim 1 &\lesssim_\mathbf{Q} \prod_{j=1}^d(1+10C_\mathbf{Q})^{-\frac{1}{2}}\leq \prod_{j=1}^d(1+|x||\mu_j(\theta)|)^{-\frac{1}{2}}.
		\end{align*}
		So without loss of generality, we may assume $|x| \geq 10$.
		
		By Proposition~\ref{prop:direct_compute} and Corollary~\ref{cor:uniform_spectral}, we have
		\begin{align*}
			\abs{\widetilde{E}^\mathbf{Q}1 (x)} \lesssim e^{-\frac{1}{2}\frac{|x'|^2}{1+C_\mathbf{Q}^2|x''|^2}} \cdot \prod_{j=1}^d (1+|x''||\mu_j(\theta)|)^{-\frac{1}{2}}.
		\end{align*}
		If $|x''| \leq 1$, then $|x'|\geq |x|/2$ as $|x|\geq 10$. Thus
		\begin{align*}
			\abs{\widetilde{E}^\mathbf{Q}1 (x)} & \lesssim e^{-\frac{1}{2}\frac{|x'|^2}{1+C_\mathbf{Q}^2|x''|^2}}\\ & \leq e^{-\frac{1}{8}\frac{|x|^2}{1+C_\mathbf{Q}^2}}\\
			& \lesssim_\mathbf{Q} \prod_{j=1}^d (1+C_\mathbf{Q}|x|)^{-\frac{1}{2}}\\
			& \leq \prod_{j=1}^d(1+|x||\mu_j(\theta)|)^{-\frac{1}{2}},
		\end{align*}
		where in the last line we used $|\mu_{j}(\theta)| \leq C_\mathbf{Q}$. 
		If $|x''|\geq 1$, then let $u \coloneqq |x'|/|x''| \in (0,\infty)$. We have
		\begin{align*}
			\abs{\widetilde{E}^\mathbf{Q}1 (x)} & \lesssim e^{-\frac{1}{4}\frac{|x'|^2}{C_\mathbf{Q}^2|x''|^2}} \cdot \prod_{j=1}^d \left(\frac{|x''|}{|x|} + \frac{|x''|}{|x|}|x||\mu_j(\theta)|\right)^{-\frac{1}{2}}\\
			& = \left( 
			\frac{|x|}{|x''|} \right)^{\frac{d}{2}}e^{-\frac{1}{4}\frac{|x'|^2}{C_\mathbf{Q}^2|x''|^2}} \cdot \prod_{j=1}^d \left(1 + |x||\mu_j(\theta)|\right)^{-\frac{1}{2}}\\
			& = \left( 
			1 + u^2 \right)^{\frac{d}{4}}e^{-\frac{1}{4}\frac{u^2}{C_\mathbf{Q}^2}} \cdot \prod_{j=1}^d \left(1 + |x||\mu_j(\theta)|\right)^{-\frac{1}{2}}\\
			& \lesssim_\mathbf{Q} \prod_{j=1}^d \left(1 + |x||\mu_j(\theta)|\right)^{-\frac{1}{2}}.
		\end{align*}
		So in any case, the proof is completed.
	\end{proof}
	
	Now we are finally ready to prove our main result.
	\begin{proof}[Proof of Theorem~\ref{thm:uniform_decay}] 
		Starting from Proposition~\ref{prop:recover_banner}, we have 
		\begin{align*}
			\abs{\widetilde{E}^\mathbf{Q}1 (x)} & \lesssim \prod_{j=1}^d (1+|x||\mu_j(\theta)|)^{-\frac{1}{2}}\\
			& \leq \prod_{k=1}^m (1 + \Lambda_k(\theta)|x|)^{-\frac{1}{2}}\\
			& \leq \prod_{k=1}^m (1 + c_\mathbf{Q}|x|)^{-\frac{1}{2}}\\
			& \lesssim_\mathbf{Q} (1 + |x|)^{-\frac{m}{2}}\\
			& = (1 + |x|)^{-\frac{\mathfrak{d}_{d,1}(\mathbf{Q})}{2}}
		\end{align*}
		as desired. Here in the third line we used \text{Theorem }\ref{thm:U_L_bound_eigen}, and in the last line we used \text{Lemma }\ref{lem:alg_char_m}.
		
		To see the optimality of our estimate, take $\theta_0 \in \mathbb{S}^{n-1}$ such that $m = \# \{j\in\{1,\dots,d\}: \mu_j(\theta_0)\neq 0\}$, and consider $x =(x',x'')=(0,|x''|\theta_0)$. By computations in the proof of Proposition~\ref{prop:direct_compute}, we have
		\begin{align*}
			\abs{\widetilde{E}^\mathbf{Q}f(x)} & = \abs{(2\pi)^{\frac{d}{2}} \cdot \det(I - i\overline{Q}(x''))^{-\frac{1}{2}}}\\
			& \sim \prod_{j=1}^d(1 + |x''||\mu_j(\theta_0)|)^{-\frac{1}{2}}\\
			& = \prod_{j=1}^m (1 + |x| \Lambda_j(\theta_0))^{-\frac{1}{2}}\\
			& \geq \prod_{j=1}^m (1 + |x|  C_\mathbf{Q})^{-\frac{1}{2}}\\
			& \gtrsim_\mathbf{Q} (1 + |x|)^{-\frac{m}{2}}\\
			& = (1 + |x|)^{-\frac{\mathfrak{d}_{d,1}(\mathbf{Q})}{2}},
		\end{align*}
		where in the fourth line we used \text{Theorem }\ref{thm:U_L_bound_eigen}, and in the last line we used \text{Lemma }\ref{lem:alg_char_m}.
		Suppose that we have $|\widetilde{E}^\mathbf{Q}f(x)| \lesssim_\mathbf{Q} (1+|x|)^{-\gamma}$ for some $\gamma > \frac{\mathfrak{d}_{d,1}(\mathbf{Q})}{2}$, then $(1 + |x|)^{-\frac{\mathfrak{d}_{d,1}(\mathbf{Q})}{2}} \lesssim (1+|x|)^{-\gamma}$, which fails when $|x| \gg_\mathbf{Q} 1$, a contradiction. Thus the optimality is proved. 
	\end{proof}
	
	Although Theorem~\ref{thm:uniform_decay} is already sufficient for our applications, Gaussian smoothing factors are uncommon in the literature of oscillatory integrals. For example, Lemma 2.1.1 of \cite{banner02} uses compactly supported smooth bumps ($C_c^\infty(\R^d)$) in place of Gaussian functions. By taking $\psi = \mathbf{Q}$ and $U = \supp \varphi$ in (\ref{def:bump_extension}) and using the notation (\ref{098765}), we are led to   \begin{align}\label{eq:oscillatory_integral}
		\widetilde{E}_\varphi^\mathbf{Q}f(x) \coloneqq \int_{\R^d} e^{i(x'\cdot\xi + \xi^T\cdot \overline{Q}(x'')\cdot\xi/2)} f(\xi) \cdot \varphi(\xi) d\xi, \quad\quad  x = (x',x'') \in \R^{d+n},
    \end{align}
    where $\varphi \in C_c^\infty(\R^d)$ with $\varphi \not\equiv 0$. The next result tells us that $\widetilde{E}_\varphi^\mathbf{Q}$ satisfies the same uniform Fourier decay estimate as $\widetilde{E}^\mathbf{Q}$.
	\begin{corollary}\label{cor:recover_banner}
		For any $\varphi \in C_c^\infty(\R^d)$ with $\varphi\not\equiv0$, we have 
		\begin{align*}
			\abs{\widetilde{E}_\varphi^\mathbf{Q}1 (x)} \lesssim_\mathbf{Q} (1+|x|)^{-\frac{\mathfrak{d}_{d,1}(\mathbf{Q})}{2}}, \quad \quad x\in \R^{d+n}.
		\end{align*}
		Moreover, this bound is optimal in the sense that for any $ \gamma > \frac{\mathfrak{d}_{d,1}(\mathbf{Q})}{2}$, we cannot have $|\widetilde{E}_\varphi^\mathbf{Q}f(x)| \lesssim_\mathbf{Q} (1+|x|)^{-\gamma}$ uniformly for all $x\in \R^{d+n}$.
	\end{corollary}
	\begin{proof}
		
		Define $\Phi$ and $\widetilde{\Phi}$ as in the proof of $(ii) \Rightarrow (iii)$ in Proposition~\ref{prop:restriction_equiv}. Then 
        \begin{align*}
            \abs{\widetilde{E}_\varphi^\mathbf{Q}1} = \abs{(\Phi d\mu_\mathbf{Q})^\vee} = \abs{\left( \widetilde{\Phi}(\zeta) e^{-|\zeta'|^2/2} 
                d\mu_\mathbf{Q} \right)^\vee} = \abs{\left( \widetilde{\Phi}(\zeta)\right)^\vee * \left( e^{-|\zeta'|^2/2} d\mu_\mathbf{Q} \right)^\vee}.
        \end{align*}
		Since $\widetilde{\Phi}  ^\vee \in \mathcal{S}(\R^{d+n})$, we have $\abs{\widetilde{\Phi}  ^\vee}  \lesssim_N (1 + |\cdot|)^{-N}$ for any $N \in \N^+$. Also note that $( e^{-|x'|^2/2} d\mu_\mathbf{Q})^\vee = \widetilde{E}^\mathbf{Q}1$. Therefore, we can apply \text{Theorem }\ref{thm:uniform_decay} to get
		\begin{align*}
			\abs{\widetilde{E}_\varphi^\mathbf{Q}1} & \leq \abs{\widetilde{\Phi}^\vee} * \abs{\widetilde{E}^\mathbf{Q}1} \lesssim_{\mathbf{Q}, N} (1 + |\cdot|)^{-N} * (1 + |\cdot|)^{-\frac{\mathfrak{d}_{d,1}(\mathbf{Q})}{2}} \lesssim_\mathbf{Q} (1 + |\cdot|)^{-\frac{\mathfrak{d}_{d,1}(\mathbf{Q})}{2}},
		\end{align*}
		where in the last step we used the properties of convolutions in Appendix B.1 of \cite{grafakos2014modern}.
		
		To prove the optimality of the estimate, it is not easy to compute everything directly as in the proof of Theorem~\ref{thm:uniform_decay}, as the bump $\varphi$ is arbitrary. To address this issue, we need to first do some reductions.
		
		Our strategy is to argue by contradiction. Suppose $|\widetilde{E}_\varphi^\mathbf{Q}1(x)| \lesssim_\mathbf{Q} (1+|x|)^{-\gamma}$ for some $\gamma > \frac{\mathfrak{d}_{d,1}(\mathbf{Q})}{2}$. Since $\varphi \not \equiv 0$, there exists $\xi_0 \in \R^d$ such that $\varphi(\xi_0) \neq 0$. By change of variables $\xi \mapsto \xi_0 + \xi$, we have
		\begin{align*}
			\abs{\widetilde{E}_\varphi^\mathbf{Q}1(x)} & = \abs{ \int_{\R^d} e^{i\left(x'\cdot(\xi_0 + \xi) + (\xi_0 + \xi)^T\cdot \overline{Q}(x'')\cdot(\xi_0 + \xi)/2\right)} \varphi(\xi_0 + \xi) d\xi } \\
			& = \abs{ \int_{\R^d} e^{i\left( (x' + \xi_0^T \cdot \overline{Q}(x'') )\cdot \xi + \xi^T\cdot \overline{Q}(x'')\cdot \xi/2\right)} \varphi(\xi_0 + \xi) d\xi } \\
			& \eqqcolon \abs{ \int_{\R^d} e^{i\left( \tilde{x}' \cdot \xi + \xi^T\cdot \overline{Q}(\tilde{x}'')\cdot \xi/2\right)} \tilde{\varphi}(\xi) d\xi } = \abs{\widetilde{E}_{\tilde{\varphi}}^\mathbf{Q}1(\tilde{x})},
		\end{align*}
		where $\tilde{x}' \coloneqq x' + \xi_0^T \cdot \overline{Q}(x'')$, $\tilde{x}'' \coloneqq x''$, $\tilde{x} = (\tilde{x}', \tilde{x}'')$, and $\tilde{\varphi}(\xi) \coloneqq \varphi(\xi_0 + \xi)$. Thus by our hypothesis, we have
		\begin{align*}
			\abs{\widetilde{E}_{\tilde{\varphi}}^\mathbf{Q}1(\tilde{x})} 
			& \lesssim_\mathbf{Q} (1+|x|)^{-\gamma}\\
			& = \left( 1 + \sqrt{|x'|^2 + |x''|^2} \right)^{-\gamma}\\
			& = \left( 1 + \sqrt{|\tilde{x}' - \xi_0^T \cdot \overline{Q}(\tilde{x}'')|^2 + |\tilde{x}''|^2} \right)^{-\gamma}.
		\end{align*}
		Note that there exists some constant $C \geq 1$ depending on $\xi_0$ and $\mathbf{Q}$, such that $|\xi_0^T \cdot \overline{Q}(\tilde{x}'')| \leq C |\tilde{x}''|$. Therefore, if $|\tilde{x}'| \leq 2C |\tilde{x}''|$, then $|\tilde{x}| \sim |\tilde{x}''|$, and $\abs{\widetilde{E}_{\tilde{\varphi}}^\mathbf{Q}1(\tilde{x})} \lesssim_\mathbf{Q} (1 + |\tilde{x}''|)^{-\gamma} \sim (1 + |\tilde{x}|)^{-\gamma}$; while if $|\tilde{x}'| > 2C |\tilde{x}''|$, then $|\tilde{x}' - \xi_0^T \cdot \overline{Q}(\tilde{x}'')| > |\tilde{x}'|/2$, and still $\abs{\widetilde{E}_{\tilde{\varphi}}^\mathbf{Q}1(\tilde{x})} \lesssim (1 + |\tilde{x}|)^{-\gamma}$. So we get $\abs{\widetilde{E}_{\tilde{\varphi}}^\mathbf{Q}1(\tilde{x})} \lesssim (1 + |\tilde{x}|)^{-\gamma}$ for all $\tilde{x} \in \R^{d+n}$. This means without loss of generality, we may assume $\varphi(0) \neq 0$.
		
		By Lemma~\ref{lem:alg_char_m}, there exists $\theta_0 \in \mathbb{S}^{n-1}$ such that $\mathfrak{d}_{d,1}(\mathbf{Q}) = \rank (\overline{Q}(\theta_0)) = m$. Also, by linear algebra, there exists $O \in {\rm O}(d,\R)$ such that 
		\begin{align*}
			\overline{Q}(\theta_0) = 
			O^T
			\begin{bmatrix}
				\mu_1(\theta_0) & & \\
				& \ddots & \\
				& & \mu_d(\theta_0)
			\end{bmatrix} O.
		\end{align*}
		Without loss of generality, we may assume $\mu_j(\theta_0) \neq 0$ for $j\leq m$ and $\mu_j(\theta_0) = 0$ for $j>m$. Since we assume $\varphi(0) \neq 0$, there exists $r > 0$ such that $\varphi(O^T\cdot) \neq 0$ on $[-r,r]^d$. Fix any nonnegative even function $\phi$ on $\R$ such that $\phi(0)>0$ and $\supp\phi \subset [-c_\mathbf{Q}^{1/2}r,c_\mathbf{Q}^{1/2}r]$, where $c_\mathbf{Q}$ comes from Theorem~\ref{thm:U_L_bound_eigen}. Let 
		$$\varphi_0(\xi) \coloneqq \prod_{j=1}^m \phi\Big(|\mu_j(\theta_0)|^{\frac{1}{2}}\xi_j\Big) \cdot \prod_{j=m+1}^d \phi\Big(c_\mathbf{Q}^{\frac{1}{2}}\xi_j\Big).$$ 
		Then $\supp\varphi_0 \subset [-r,r]^d$, and by further truncating $\widetilde{E}_\varphi^\mathbf{Q}1$ with respect to $\varphi_0(O\cdot)/\varphi \in C_c^\infty(\R^d)$ as in the first part of the proof, we know that 
		$$\abs{\widetilde{E}_{\varphi_0(O\cdot)}^\mathbf{Q}1(x)} \lesssim_\mathbf{Q} (1+|x|)^{-\gamma}, \quad \quad \forall ~\gamma > \frac{\mathfrak{d}_{d,1}(\mathbf{Q})}{2} = \frac{m}{2}.$$ 
		Consider $x = (x',x'') = (0, |x''|\theta_0)$, then 
	\begin{align*}
			\widetilde{E}_{\varphi_0(O\cdot)}^\mathbf{Q}1(x) & = \int_{\R^d} e^{i|x|(\xi^T\cdot \overline{Q}(\theta_0)\cdot\xi/2)} \varphi_0(O\xi) d\xi\\
			& = \int_{\R^d} e^{i\frac{|x|}{2} \sum_{j=1}^m \mu_j(\theta_0) (O\xi)_j^2} \varphi_0(O\xi) d\xi\\
			& = \int_{\R^d} e^{i\frac{|x|}{2} \sum_{j=1}^m \mu_j(\theta_0) \xi_j^2} \varphi_0(\xi) d\xi\\
			& = \prod_{j=1}^m \int_{\R} e^{i\frac{|x|}{2}\mu_j(\theta_0)\xi_j^2} \phi\Big(|\mu_j(\theta_0)|^{\frac{1}{2}} \xi_j\Big) d\xi_j \cdot \prod_{j=m+1}^d \int_{\R} \phi\Big(c_\mathbf{Q}^{\frac{1}{2}}\xi_j\Big)d\xi_j\\
			& = \Big(c_\mathbf{Q}^{-\frac{1}{2}}\int\phi\Big)^{d-m} \prod_{j=1}^m 
			|\mu_j(\theta_0)|^{-\frac{1}{2}} \cdot \prod_{j=1}^m \int_{\R} e^{\pm i\frac{|x|}{2}\xi_j^2} \phi(\xi_j) d\xi_j \\
			& \sim_{\phi,\mathbf{Q}} \prod_{j=1}^m \int_{\R} e^{\pm i\frac{|x|}{2}\xi_j^2} \phi(\xi_j) d\xi_j.
	\end{align*}
		As $\phi$ is real-valued, the ``$\pm$'' in the exponential does not affect the absolute value. So we have
		\begin{align*}
			\abs{\widetilde{E}_{\varphi_0(O\cdot)}^\mathbf{Q}1(x)} \sim_{\phi,\mathbf{Q}} \prod_{j=1}^m \abs{\int_{\R} e^{ i\frac{|x|}{2}\xi_j^2} \phi(\xi_j) d\xi_j} = \abs{\int_{\R} e^{ i\frac{|x|}{2}u^2} \phi(u) du}^m \eqqcolon \Big|I\Big( \frac{|x|}{2}\Big)\Big|^m,
		\end{align*}
		where 
		$$I(\lambda) \coloneqq \int_{\R} e^{ i\lambda u^2} \phi(u) du,    \quad\quad \forall~ \lambda \in \R.$$
		Plugging this into 
		$$\abs{\widetilde{E}_{\varphi_0(O\cdot)}^\mathbf{Q}1(x)} \lesssim_\mathbf{Q} (1+|x|)^{-\gamma}, \quad \quad \gamma > \frac{m}{2},$$ 
		we get 
		$$\Big|I\Big( \frac{|x|}{2}\Big)\Big| \lesssim_{\phi,\mathbf{Q}} (1 + |x|)^{- \frac{\gamma}{m}}, \quad \quad \forall~ |x|>0,$$ i.e., there exists $\gamma'>\frac{1}{2}$ such that 
		$$\abs{I(\lambda)} \lesssim_{\phi,\mathbf{Q}} (1 + \lambda)^{-\gamma'},\quad \quad  \forall~ \lambda>0.$$ 
		Since $\phi$ is real-valued, we know that $|I(\lambda)|$ is even, which means
		\begin{align}\label{eq:one_dim_oscillatory}
			\abs{I(\lambda)} \lesssim_{\phi,\mathbf{Q}} (1 + |\lambda|)^{-\gamma'}, \quad \forall\, \lambda\in\R.
		\end{align}
		On the other hand, by the fact that $\phi$ is even and change of variables $u \mapsto \sqrt{v}$, we get $$I(\lambda) = 2\int_0^\infty e^{ i\lambda u^2} \phi(u) du = \int_0^\infty e^{ i\lambda v} \frac{\phi(\sqrt{v})}{\sqrt{v}} dv, \quad\quad \forall\, \lambda\in\R.$$ 
		If we let 
		$$\tilde{\phi}(v) \coloneqq \chi_{[0,\infty)}(v) \cdot   \frac{\phi(\sqrt{v})}{\sqrt{v}},$$ 
		then (\ref{eq:one_dim_oscillatory}) together with $\gamma'>\frac{1}{2}$ tells us that $I(\lambda) = (\tilde{\phi})^\vee(\lambda)$ lies in $L^2(\R)$. By Plancherel's identity, we must have $\tilde{\phi} \in L^2(\R)$. However, this is impossible, as $\phi(\sqrt{v})^2/v \sim 1/v$ near $0$ (recall $\phi(0)>0$), and $1/v$ is not integrable. Thus our proof by contradiction is completed, and the Fourier decay is optimal.
	\end{proof}
	\begin{remark}
		Corollary~{\rm \ref{cor:recover_banner}} can be interpreted as a certain stability/uniformity result for the oscillatory integral {\rm (\ref{eq:oscillatory_integral})} when $f=1$, which can be rewritten as
		\begin{align*}
			I(\lambda, t) = \int_{\R^d} e^{i \lambda \phi(\xi,t)} \varphi(\xi) d\xi,
		\end{align*}
		where $\lambda = |x|$, $t = (t',t'') \in \R^d \times \R^n $ with $t' = x'/|x|, t'' = x''/|x|$, and $\phi(\xi,t) = t'\cdot\xi + \xi^T\cdot \overline{Q}(t'')\cdot\xi/2$. Corollary~{\rm \ref{cor:recover_banner}} tells us that 
		$$I(\lambda,t) \lesssim \lambda^{-\frac{\mathfrak{d}_{d,1}(\mathbf{Q})}{2}}, \quad \quad \text{uniform~in~}t\in \mathbb{S}^{d+n-1}.$$
		Our setting here closely follows Chapter {\rm 8} of \cite{gilula2016real}, which also pointed out that nondegeneracy may help us obtain stability. This is indeed the case for us: Although $\overline{Q}(\theta)$ may be more degenerate for some $\theta$ than the others, if we only care about the ``common'' degree {\rm(}$m = \mathfrak{d}_{d,1}(\mathbf{Q})${\rm)} of nondegeneracy, then we can still have a uniform estimate. However, our method relies heavily on linear algebra for quadratic forms, so does not work for manifolds of higher degrees. Also, our stability is regarding $\{\phi(\xi,t): t\in \mathbb{S}^{d+n-1}\}$, while a more widely studied kind of stability is regarding the family of phase functions whose derivatives are uniformly lower bounded. The latter is in general much harder, especially in multidimensional cases. One may consult \cite{ccw99} or \cite{pss01} for more historical remarks and results on the stability of oscillatory integrals.
	\end{remark}

	\section{The proof of Theorem \ref{th1}}\label{sec3}

	 In this section, we mainly concern the proofs of (\ref{t1e1}) and (\ref{t1e2}) in Theorem \ref{th1}.

     \subsection{Shayya's argument}\label{subsec:shayya}\phantom{x}
	
   	This subsection is devoted to prove (\ref{t1e1}). Since the last bound in (\ref{t1e1}) is the classical Agmon-H\"ormander inequality (\ref{AH ineq}), it suffices to prove the first and second bounds in (\ref{t1e1}). Our proof mainly follows Shayya's argument \cite{shayya21}. In fact, compared with (\ref{s1e1}),  Shayya studied stronger weighted estimates, which are global and include accurate information on the weight $H$, i.e.,
	\begin{equation}\label{s2e1}
		\|E^{\mathbf{Q}}f\|_{L^p(\mathbb{R}^{d+n},H)}\lesssim  A_\alpha(H)^{\frac{1}{p}}\|f\|_{L^2([0,1]^d)}.
	\end{equation}
	For $0<\alpha\leq d+n$ and $1\leq p\leq \infty$, define $P(\alpha,\mathbf{Q})$ to be the infimum of all numbers $p$ such that (\ref{s2e1}) holds for all functions $f \in L^2([0,1]^d)$ and all $\alpha$-dimensional weights $H$ on $\mathbb{R}^{d+n}$. 
	
	When $S_{\mathbf{Q}}$ is a paraboloid or a hyperbolic paraboloid with $n=1$, Shayya proved 
    		\begin{align}\label{shayya_result}
			~P(\alpha,\mathbf{Q})\leq 
			\begin{cases}
				~	2, \quad \quad\quad \quad    & 0<  \alpha < \frac{d}{2},   \\
				~	\frac{4\alpha}{d}, \quad \quad \quad\quad &\alpha\geq \frac{d}{2}. 
			\end{cases}
		\end{align}
    We will generalize this result to all quadratic manifolds as follows:
        \begin{align}\label{shayya_result2}
			~P(\alpha,\mathbf{Q})\leq 
			\begin{cases}
				~	2, \quad \quad\quad \quad    & 0<\alpha< \frac{\mathfrak{d}_{d,1}(\mathbf{Q})}{2},   \\
				~	\frac{4\alpha}{\mathfrak{d}_{d,1}(\mathbf{Q})}, \quad \quad \quad\quad &\alpha\geq \frac{\mathfrak{d}_{d,1}(\mathbf{Q})}{2}. 
			\end{cases}
		\end{align}

    We first use Theorem~\ref{thm:uniform_decay} to prove the first bound in (\ref{shayya_result2}). Its proof is similar to that in Section 6 in \cite{shayya21}, with some small modifications. For completeness, we briefly sketch the proof here.
	
	By Proposition~\ref{prop:smooth_equiv}, it suffices to show\footnote{For the optimal uniform Fourier decay rate $\frac{\mathfrak{d}_{d,1}(\mathbf{Q})}{2}$ to hold, it is important to replace the roughly truncated surface measure with a suitably smoothed one. Such a technical issue is not taken care of by Shayya \cite{shayya21}. There is no such problem if we were dealing with the sphere, which is a closed manifold.} 
	\begin{equation}\label{s2ef2}
		\|\widetilde{E}^{\mathbf{Q}}f\|_{L^p(\mathbb{R}^{d+n},H)}\lesssim  A_\alpha(H)^{\frac{1}{p}}\|f\|_{L^2([0,1]^d)}
	\end{equation}
	for all $p>2$, when $0<\alpha< \frac{\mathfrak{d}_{d,1}(\mathbf{Q})}{2}$. Set $dw(x) \coloneqq H(x)dx$, then by the layer cake representation, we have
	\begin{equation}\label{s2ef3}
		\|\widetilde{E}^{\mathbf{Q}}f\|_{L^p(\mathbb{R}^{d+n},H)}^p =p\int_0^{\|f\|_{L^1}} \lambda^{p-1} w(\{  |\widetilde{E}^{\mathbf{Q}}f|\geq \lambda  \}) d\lambda,
	\end{equation}
	where the upper limit of integration arises from a trivial estimate
	$$    \|\widetilde{E}^{\mathbf{Q}}f\|_{L^\infty(\R^{d+n})} \leq \|f\|_{L^1([0,1]^d)}.   $$
	By dividing $\widetilde{E}^{\mathbf{Q}}f$ into positive part, negative part, real part, and imaginary part, combined with the  distribution inequality, it is enough to consider the set $\{ \Re (\widetilde{E}^{\mathbf{Q}}f )\geq \lambda/4  \}$, which we denote by $G$. So we only need to estimate $w(G)$. On the one hand, by Parseval’s relation and H\"older's inequality, we have 
	\begin{align*}
		\frac{\lambda}{4} w(G) &\leq \int_G \Re (\widetilde{E}^{\mathbf{Q}}f) dw    \\
		&= \Re \int_{\R^{d+n}} \chi_G \widehat{fd\tilde{\mu}_{\mathbf{Q}}} dw     \\
		& =\Re \int_{\R^{d+n}} \widehat{\chi_G dw}  fd\tilde{\mu}_{\mathbf{Q}} \\ 
		& \leq \|\widehat{\chi_G dw}\|_{L^2(d\tilde{\mu}_{\mathbf{Q}})} \|f\|_{L^2(d\tilde{\mu}_{\mathbf{Q}})}, 
	\end{align*}
	where $f$ is equivalently regarded as a function on the surface $S_{\mathbf{Q}}$, and $d\tilde{\mu}_{\mathbf{Q}}:=e^{-|\xi|^2/2}d\mu_{\mathbf{Q}}$. On the other hand, we use Parseval’s relation again to get
	$$   \|\widehat{\chi_G dw}\|_{L^2(d\tilde{\mu}_{\mathbf{Q}})}^2= \int_{\R^{d+n}} \widehat{\chi_G dw}   \overline{\widehat{\chi_G dw}} d\tilde{\mu}_{\mathbf{Q}} = \int_{\R^{d+n}} \left( (\chi_G dw)\ast  \widehat{d\tilde{\mu}_{\mathbf{Q}}}  \right)  \chi_G dw.  $$
	Combining the above two estimates, one concludes
	\begin{equation}\label{s2ef4}
		\lambda^2 w(G)^2 \lesssim \int_{\R^{d+n}} \left( (\chi_G dw)\ast  \widehat{d\tilde{\mu}_{\mathbf{Q}}}  \right)  \chi_G dw \cdot \|f\|_{L^2([0,1]^d)}^2.
	\end{equation}

	To estimate the integral in (\ref{s2ef4}) above, we use annulus decomposition. Let $\varphi_0$ be a smooth bump function in $\mathbb{R}^{d+n}$ satisfying $\varphi_0=1$ on $B(0,1)$, and $\varphi_0=0$ outside $B(0,2)$. Also, for $l\in \mathbb{N}$, we define $\varphi_l(x)=\varphi_0(x/2^l)-\varphi_0(x/2^{l-1})$. By choosing $\varphi_0$ suitably, we can have
	$$    \sum_{l=0}^{\infty}  \varphi_l(x)\equiv 1,\quad \quad x\in \mathbb{R}^{d+n}.  $$
	Thus we obtain
	\begin{align*}
		\left| (\chi_G dw)\ast  \widehat{d\tilde{\mu}_{\mathbf{Q}}}(x) \right| &\leq \sum_{l=0}^{\infty}  \left| (\chi_G dw)\ast  (\varphi_l\widehat{d\tilde{\mu}_{\mathbf{Q}}})(x) \right| \\
		& \leq \sum_{l=0}^{\infty} \int_{\R^{d+n}} | \varphi_l(x-y)\widehat{d\tilde{\mu}_{\mathbf{Q}}}(x-y) | \chi_G(y) dw(y) \\
		& \lesssim \sum_{l=0}^{\infty}  2^{-\frac{l \cdot\mathfrak{d}_{d,1}(\mathbf{Q}) }{2}} \int_{\R^{d+n}} \chi_{B(x,2^{l+1})}(y)\chi_G(y)dw(y) \\
		& \leq \sum_{l=0}^{\infty}  2^{-\frac{l \cdot\mathfrak{d}_{d,1}(\mathbf{Q}) }{2}} w(B(x,2^{l+1})) \\
		& \lesssim   \sum_{l=0}^{\infty}  2^{-\frac{l (\mathfrak{d}_{d,1}(\mathbf{Q})-2\alpha) }{2}} A_\alpha(H) \\
		& \lesssim  A_\alpha(H).
	\end{align*}
	Here in the third line we used Theorem~\ref{thm:uniform_decay}, in the fourth line we used the definition of $\alpha$-dimensional weight $H$, and in the last line we used the assumption $0<\alpha< \frac{\mathfrak{d}_{d,1}(\mathbf{Q})}{2}$. Plugging this estimate into (\ref{s2ef4}), we get
	$$  w(G) \lesssim \lambda^{-2} A_{\alpha}(H) \|f\|_{L^2([0,1]^d)}^2.   $$
	And then (\ref{s2ef3}) becomes
	$$     \|\widetilde{E}^{\mathbf{Q}}f\|_{L^p(\mathbb{R}^{d+n},H)}^p \lesssim  \int_0^{\|f\|_{L^1}} \lambda^{p-3} d\lambda \cdot A_{\alpha}(H) \|f\|_{L^2([0,1]^d)}^2\lesssim A_{\alpha}(H) \|f\|_{L^2([0,1]^d)}^p,  $$
	provided $p>2$. This finishes the proof of (\ref{s2ef2}).

	Next we prove the second bound in (\ref{shayya_result2}). Our main tool is the weighted H\"older-type inequality built by Shayya: For any Lebesgue measurable function $F:~\mathbb{R}^{d+n}\rightarrow [0,\infty)$, define
	$$ M_{\alpha}F =\sup_{  \substack{  H: \\ 0<A_\alpha(H)<\infty  }}  \left(\int_{\R^{d+n}} F(x)^\alpha\cdot  \frac{H(x)}{A_\alpha(H)} dx \right)^{\frac{1}{\alpha}}.  $$
	Then for $0<\beta \leq \alpha\leq d+n$, there is
	\begin{equation}\label{s2ef5}
		M_\alpha F \leq M_\beta F.
	\end{equation}
	Based on (\ref{s2ef5}), Shayya showed an interesting monotone property on $P(\alpha,\mathbf{Q})$: For $0<\beta \leq \alpha\leq d+n$, there is 
	\begin{equation}\label{s2ef6}
		\frac{P(\alpha,\mathbf{Q})}{\alpha}\leq \frac{P(\beta,\mathbf{Q})}{\beta}.
	\end{equation}
	Now, if we take $0<\beta <\frac{\mathfrak{d}_{d,1}(\mathbf{Q})}{2} \leq \alpha \leq d+n$, then the first bound in (\ref{shayya_result2}) and (\ref{s2ef6}) imply
	$$     \frac{P(\alpha,\mathbf{Q})}{\alpha}\leq \frac{P(\beta,\mathbf{Q})}{\beta}\leq \frac{2}{\beta}.  $$
	Letting $\beta \rightarrow \frac{\mathfrak{d}_{d,1}(\mathbf{Q})}{2}$, we get
	\begin{equation}\label{s2ef7}
		P(\alpha,\mathbf{Q}) \leq \frac{4\alpha}{\mathfrak{d}_{d,1}(\mathbf{Q})}.
	\end{equation}
	We have finished the proof of (\ref{shayya_result2}).

	After considering the global estimate (\ref{s2e1}), we go back to  the local estimate (\ref{s1e1}). Firstly, using the first bound in (\ref{shayya_result2}), we can get
    \begin{equation}
		s(\alpha,\mathbf{Q})=0 ,\quad \quad  0<\alpha< \frac{\mathfrak{d}_{d,1}(\mathbf{Q})}{2},
	\end{equation}
	which is just the first bound in (\ref{t1e1}). Then we consider the case $\alpha\geq \mathfrak{d}_{d,1}(\mathbf{Q})/2$. Set $q=4\alpha/\mathfrak{d}_{d,1}(\mathbf{Q})$. Combining (\ref{s2ef7}) and H\"older's inequality, we obtain
	\begin{align*}
		\|E^{\mathbf{Q}}f\|_{L^2(B_R,H)}&  \leq \|E^{\mathbf{Q}}f\|_{L^{q}(B_R,H)}  \left(\int_{B_R} H(x)dx\right)^{\frac{1}{2}-\frac{1}{q}}  \\
		& \lesssim R^\epsilon A_\alpha(H)^{\frac{1}{q}} \|f\|_{L^2([0,1]^d)} \cdot (A_\alpha(H) R^\alpha)^{\frac{1}{2}-\frac{1}{q}} \\
		& = A_\alpha(H)^{\frac{1}{2}}  R^{\frac{\alpha}{2}-\frac{\alpha}{q}+\epsilon}  \|f\|_{L^2([0,1]^d)} \\
		&=  A_\alpha(H)^{\frac{1}{2}}  R^{\frac{2\alpha-\mathfrak{d}_{d,1}(\mathbf{Q})}{4}+\epsilon}  \|f\|_{L^2([0,1]^d)} .
	\end{align*}
	This means 
	\begin{equation}
		s(\alpha,\mathbf{Q})\leq \frac{2\alpha-\mathfrak{d}_{d,1}(\mathbf{Q})}{4} ,\quad \quad  \alpha\geq  \frac{\mathfrak{d}_{d,1}(\mathbf{Q})}{2},
	\end{equation}
	which is just the second bound in (\ref{t1e1}).

	\begin{remark}
		For the parabolic case, in addition to the approaches presented above, Shayya \cite{shayya21} also applied other methods, which incorporate Du-Zhang's weighted restriction estimates to further improve the upper bound of $P(\alpha,\mathbf{Q})$ in {\rm(\ref{shayya_result})} for $\alpha>\frac{d+1}{2}$. It is worth mentioning that by combining Theorem {\rm \ref{th1}}, Theorem~{\rm \ref{thm:uniform_decay}} in the last section, and the proof strategy of Shayya, we can also obtain improved bounds for quadratic manifolds in {\rm(\ref{shayya_result2})} when $\alpha>\frac{\mathfrak{d}_{d,1}(\mathbf{Q})+1}{2}$, whose details are left to interested readers.

	\end{remark}

	\subsection{The Du-Zhang method}\label{du_zhang_method}\phantom{x}
	
	This subsection is devoted to prove (\ref{t1e2}). As we have mentioned in the introduction, we will adopt the strategy of Du and Zhang \cite{dz19}. Compared with the Du-Zhang method in \cite{dz19}, apart from differences in the mode of the broad-narrow analysis, we also need to adopt the induction at the new scale for matching associated $\ell^2L^p$ decoupling for general quadratic forms in the narrow part.

	Via a similar argument as in \cite{dz19}, the estimate in (\ref{t1e2}) can be derived by the following inductive proposition. We say that a collection is essentially a dyadic constant if all the quantities in this collection are in a common dyadic interval of the form $[2^j,2^{j+1}]$, where $j$ is an integer.

	\begin{proposition}\label{main_prop:du_zhang}
		Let $d,n \geq 1$, and $\mathbf{Q}=(Q_1,...,Q_n)$ be an $n$-tuple of real quadratic forms in $d$ variables. Let $p\geq 2$, and $2\leq k\leq d+1$ be an integer. For any $\epsilon >0$, there exist constants $C_\epsilon$ and $\delta=\delta(\epsilon)$ such that the following fact holds for any $R \geq 1$ and every $f$ with ${\rm supp}f \subset [0,1]^d$. Let  $Y=\cup_{i=1}^M B_i$ be a union of lattice $K^2$-cubes in $B^{d+n}(0,R)$ with $K=R^\delta$. Suppose that 
		\begin{equation}\label{assump:dyadic_constant}
			\|E^{\mathbf{Q}} f\|_{L^p(B_i)}  \text{~is~essentially~a ~dyadic~constant~in~}   i=1,2,...,M. 
		\end{equation}
		Let $1\leq \alpha\leq d+n$ and 
		$$  \gamma:=\max_{\substack{  B(x',r)\subset B(0,R) \\ x'\in \mathbb{R}^{d+n},~r\geq K^2 }} \frac{\#\{ B_i:B_i \subset B(x',r)\}}{r^\alpha}.  $$
		Then we have
		\begin{equation}\label{eq:main_estimate}
			\|E^{\mathbf{Q}}f\|_{L^p(Y)}\leq C_\epsilon M^{\frac{1}{p}-\frac{1}{2}} \gamma^{\frac{1}{2}-\frac{1}{p}}R^{w(p,k,\alpha,\mathbf{Q})+\epsilon}\|f\|_{L^2([0,1]^d)},
		\end{equation}
		where $w(p,k,\alpha,\mathbf{Q})$ is given by 
		\begin{equation}\label{eq:key_exponent}
			w(p,k,\alpha,\mathbf{Q})= \max\left\{   \max_{1\leq m\leq d+n} \frac{m-X(\mathbf{Q},k,m)}{2m}\cdot\alpha ,~\frac{\Gamma_p^{k-2}(\mathbf{Q})}{2}+\frac{d+n-\alpha}{2p}+\frac{\alpha+n-d}{4} \right\}.
		\end{equation}   
	\end{proposition}

	\begin{proof} This proposition will be demonstrated by induction on scales. If $R\lesssim 1$, we can take a sufficiently large $C_\epsilon$, and then (\ref{eq:main_estimate}) holds trivially. From now on, we assume that Proposition \ref{main_prop:du_zhang} is true for $R$ replaced by $R/2$.
		
		Decompose $[0,1]^d$ into disjoint $K^{-1}$-cubes $\tau$, and write $f=\sum_{\tau} f_\tau$, where $f_\tau=f \chi_\tau$. For every $K^2$-cube $B$ in $Y$, we define its significant set as
		$$    \mathcal{S}_0(B) :=\Big\{ \tau: \| E^{\mathbf{Q}} f _\tau\|_{L^p(B)} \geq \frac{1}{\# \{\tau\} } \max_{\tau'}\|  E^{\mathbf{Q}} f_{\tau'}  \|_{L^p(B)} \Big\}.  $$
		We run the following algorithm. Let $\iota\geq 0$. If $\mathcal{S}_\iota(B)=\varnothing$ or $\mathcal{S}_\iota(B)$ is $\epsilon$-uniform with the controlling sequence $\{X(\mathbf{Q},k,m)\}_{m=1}^{d+n}$, then we set $\mathcal{S}_\iota(B):=\mathcal{T}(B)$ and terminate. Otherwise, there is an $m$-dimensional 
		subspace $V$ such that there are more than $\epsilon|\mathcal{S}_\iota|$ many $\tau \in \mathcal{S}_\iota$ intersecting 
		$$   Z_V:=\left\{   \xi \in \mathbb{R}^d:\dim (\pi_{V_\xi}(V)) <X(\mathbf{Q},k,m)   \right\}.   $$
		By the definition of $X(\mathbf{Q},k,m)$, we have $\dim Z_V\leq k-2$. Applying Lemma \ref{dec low var}, one gets
		$$   N_{1/K}(Z_V) \cap [0,1]^d \subset  \bigcup_{j=1}^L \bigcup_{W \in \mathcal{W}_{\iota,j}(B)} W,  $$
		where each $\mathcal{W}_{\iota,j}(B)$ denotes the collection of pairwise disjoint $1/K_j$-cubes. Then we run the algorithm to
		$$   \mathcal{S}_{\iota+1}(B)  := \mathcal{S}_\iota(B) \backslash \bigcup_{j=1}^L \bigcup_{W \in \mathcal{W}_{\iota,j}(B)} \mathcal{P}(W,K^{-1}), $$
		where $\mathcal{P}(W,K^{-1})$ denotes the partition of $W$ into disjoint $K^{-1}$-cubes $\tau$. Note this algorithm will terminate after at most $O(\log K)$ steps, which implies $\iota \lesssim \log K$. Finally, this algorithm produces the following two sets:
		$$  \mathcal{T}(B)\quad \text{~~~~~and~~~~}\quad   \bigcup_{\iota} \bigcup_{j=1}^L \mathcal{W}_{\iota,j}(B)     ,   $$
		and each $\tau \in \mathcal{S}_0(B)$ must be covered by exactly one set.
		Then by the triangle inequality, one has
		\begin{equation}\label{met1 eq2}
			\|E^{\mathbf{Q}}f\|_{L^p(B)} \leq \Big\|\sum_{\tau \in \mathcal{T}(B)} E^{\mathbf{Q}}f_\tau\Big\|_{L^p(B)}+\Big\|  \sum_{\tau \in \cup_\iota \cup_j \mathcal{W}_{\iota,j}(B) } E^{\mathbf{Q}}f_\tau \Big\|_{L^p(B)} +\Big\|   \sum_{\tau \notin\mathcal{S}_0(B)  }  E^{\mathbf{Q}}f_\tau\Big\|_{L^p(B)}.
		\end{equation}
		We say $B$ is broad if the first term dominates, otherwise we say $B$ is narrow. Denote the union of broad cubes $B $ in $ Y$ by $Y_{\text{broad}}$ and the union of narrow cubes $B $ in $ Y$ by $Y_{\text{narrow}}$. We say that we are in the broad case if $Y_{\text{broad}}$ contains $\geq M/2$ many $K^2$-cubes, and the narrow case otherwise. 
		
		\vskip0.5cm
		
		\noindent \textbf{Broad case.} For each broad cube $B$, we have
		\begin{equation*}
			\|E^{\mathbf{Q}} f \|_{L^p(B)} \lesssim \Big\|\sum_{\tau \in \mathcal{T}(B)} E^{\mathbf{Q}}f_\tau\Big\|_{L^p(B)} \lesssim K^d \max_{\tau \in \mathcal{T}(B)} \|E^{\mathbf{Q}}f_\tau\|_{L^p(B)} .
		\end{equation*}
		Set $\mathcal{T}(B)=\{ \tau_1,...,\tau_J\}$. It follows from the definition of $\mathcal{S}_0(B)$ that 
		$$    K^d \max_{\tau \in \mathcal{T}(B)} \|E^{\mathbf{Q}}f_\tau\|_{L^p(B)} \lesssim K^{2d} \prod_{j=1}^J \|E^{\mathbf{Q}}f_{\tau_j}\|^{\frac{1}{J}}_{L^p(B)} .$$
		We want to use reverse H\"older to bound the above geometric average of $L^p$-integrals by an $L^p$-integral of the geometric average. Although $E^{\mathbf{Q}}f_{\tau_j}$ does not satisfy the locally constant property on each $K^2$-cube $B$, this process still works up to a $K^{O(1)}$ loss via an additional translation argument. Readers may consult Section 3 in \cite{dz19} for details. Thus, we conclude
		\begin{align*}
			\|E^{\mathbf{Q}} f \|_{L^p(B)}  \lesssim K^{2d} \prod_{j=1}^J \|E^{\mathbf{Q}}f_{\tau_j}\|^{\frac{1}{J}}_{L^p(B)} \lesssim K^{O(1)} \Bigg\|     \prod_{j=1}^{J} |E^{\mathbf{Q}}f_{\tau_j,v_j}|^{\frac{1}{J}} \Bigg\|_{L^p(B(x_B,2))}, 
		\end{align*}
		where $x_B$ is the center of $B$, and
		$$  v_j \in B(0,K^2) \cap \mathbb{Z}^{d+n},\quad f_{\tau_j,v_j}(\xi):= f_{\tau_j}(\xi)  e^{ iv_j\cdot(\xi, \mathbf{Q}(\xi))}. $$
		We point out that the parameters $\tau_j,v_j$ and $J$ may depend on the choices of $B$. However, such dependence is in fact harmless. Since there are only $K^{O(1)}$ choices for $\tau_j,v_j$ and $J$, by pigeonholing, there exist $\tilde{\tau}_j,\tilde{v}_j,\tilde{J}$ such that the above inequality holds for at least $\geq K^{-C}M$ many broad $K^2$-cubes $B$. In the rest of the proof of the broad case, we fix $\tilde{\tau}_j,\tilde{v}_j$ and $\tilde{J}$.  Write $f_j=f_{\tilde{\tau}_j,\tilde{v}_j}$ for brevity, and denote by $\mathcal{B}$ the collection of the broad $K^2$-cubes $B$ associated with $\tilde{\tau}_j,\tilde{v}_j$ and $\tilde{J}$. Relying on these pigeonholing and (\ref{assump:dyadic_constant}), we can sum over all $B\in \mathcal{B}$ and get
		\begin{equation}
			\|E^{\mathbf{Q}} f \|_{L^p(Y)}  \lesssim K^{O(1)} \Bigg\| \prod_{j=1}^{\tilde{J}} |E^{\mathbf{Q}} f_j |^{\frac{1}{\tilde{J}}}  \Bigg\|_{L^p(\cup_{B\in \mathcal{B} } B(x_B,2))}.   
		\end{equation}
		
		To exploit the transverse condition, i.e., $\mathcal{T}(B)$ is $\epsilon$-uniform with the controlling sequence $\{X(\mathbf{Q},k,m)\}_{m=1}^{d+n}$, we apply the following $k$-linear restriction estimates for general quadratic forms.

		\begin{theorem}[\cite{ggo23}, Theorem 2.1]\label{thm:k-linear} 
			Let $\theta,J,K>0$, and $\{\tau_j \}_{j=1}^J$ be a collection of $K^{-1}$cubes in $[0,1]^d$ which is $\theta$-uniform with the controlling sequence $\{X_m\}_{m=1}^{d+n}$. Then for each $2 \leq q\leq 2J$ and $0<\epsilon\ll 1$, we have
			\begin{equation}\label{eq:k-linear}
				\Bigg\|   \prod_{j=1}^{J}   |E^{\mathbf{Q}}f_{\tau_j}|^{\frac{1}{J}} \Bigg\|_{L^q(B_R)} \leq C_{\epsilon,\theta,q,K} R^{\epsilon+\frac{(d+n)\theta}{2}}\max_{1\leq m\leq d+n}R^{\frac{m}{q}-\frac{X_m}{2}} \prod_{j=1}^{J}\|f_{\tau_j}\|^{\frac{1}{J}}_{L^2([0,1]^d)}.
			\end{equation}
		\end{theorem}
		
		Set $\theta=\epsilon$, $J=\tilde{J}$, and $X_m=X(\mathbf{Q},k,m)$ in above theorem. Via some basic calculations, we can take
		$$     q= \max_{1\leq m\leq d+n} \frac{2m}{X(\mathbf{Q},k,m)}  ,$$
		which is the endpoint of  (\ref{eq:k-linear}) in Theorem~\ref{thm:k-linear}. From the definition of $X(\mathbf{Q},k,m)$, we have
		\begin{equation}\label{xqkmrela}
			0\leq X(\mathbf{Q},k,m)\leq m,
		\end{equation}
		which means $2\leq q \leq \infty$. Although the situation $q>2\tilde{J}$ may happen, it is harmless. Since the $\theta$-uniform condition says that $\{\tau_j\}_{j=1}^{\tilde{J}}$ cannot be clustered near a low dimensional set (\ref{s3t1e1}), this condition is effective only when $\tilde{J} \gg 1$. In the current environment, $\tilde{J}\sim K^{c}$ for some constant $c$ in general. By (\ref{xqkmrela}), we get the case $q>2\tilde{J}$ will happen only when $X(\mathbf{Q},k,m)=0$ for some $1\leq m\leq d+n$, and then $q=\infty$. In this case, the estimate (\ref{eq:k-linear}) holds trivially.
		
		There are two cases according to whether $p\geq q$ or not. We first assume that $p\geq q$. Then we sort $B \in \mathcal{B}$ by the value of $ \|\prod_{j=1}^{\tilde{J}} |E^{\mathbf{Q}} f _j|^{\frac{1}{\tilde{J}}}\|_{L^\infty(B(x_B,2))}$: For each dyadic number $A$, define
		$$ \mathbb{Y}_A:=\Bigg\{B \in \mathcal{B}: \Bigg\|\prod_{j=1}^{\tilde{J}} |E^{\mathbf{Q}} f _j|^{\frac{1}{\tilde{J}}}\Bigg\|_{L^\infty(B(x_B,2))} \sim A\Bigg\} .  $$
		Let $Y_{A}$ be the union of the $K^2$-cubes $B$ in $\mathbb{Y}_{A}$. Without loss of generality, assume that $\|f\|_{L^2([0,1]^d)}=1$. And we can further assume that $R^{-C} \leq A \leq 1$ for some constant $C$, because those $\mathbb{Y}_A$'s with $A \leq R^{-C}$ only make a negligible contribution for sufficiently large $C$. So there are only $O(\log R) \leq O(K)$ choices on $A$. By dyadic pigeonholing, there exists a constant $\tilde{A}$ such that 
		\begin{equation*}
			\# \{  B: B\subset Y_{\tilde{A}}  \} \gtrsim (\log R)^{-1} \#  \mathcal{B}. 
		\end{equation*}
		Now we fix $\tilde{A}$, and denote the collection of $K^2$-cubes $B$ in $Y_{\tilde{A}}$ by $\mathcal{B}'$. This uniformization procedure allows us to use reverse H\"older to pass from $L^p$ to $L^q$. Thus, we get  
		\begin{align*}
			\|E^{\mathbf{Q}} f \|_{L^p(Y)}  &\lesssim K^{O(1)} \Bigg\| \prod_{j=1}^{\tilde{J}} |E^{\mathbf{Q}} f_j |^{\frac{1}{\tilde{J}}}  \Bigg\|_{L^p(\cup_{B\in \mathcal{B} } B(x_B,2))} \\ 
			& \leq K^{O(1)} \Bigg\| \prod_{j=1}^{\tilde{J}} |E^{\mathbf{Q}} f_j |^{\frac{1}{\tilde{J}}}  \Bigg\|_{L^p(\cup_{B\in \mathcal{B}'} B(x_B,2))}  \\
			& \sim K^{O(1)} M^{\frac{1}{p}-\frac{1}{q}} \Bigg\| \prod_{j=1}^{\tilde{J}} |E^{\mathbf{Q}}f _j|^{\frac{1}{\tilde{J}}}  \Bigg\|_{L^{q}(\cup_{B\in \mathcal{B}'} B(x_B,2))}     \\
			& \leq   K^{O(1)} M^{\frac{1}{p}-\frac{1}{q}}  \Bigg\| \prod_{j=1}^{\tilde{J}} |E^{\mathbf{Q}} f _j|^{\frac{1}{\tilde{J}}}  \Bigg\|_{L^{q}(B_R)}   \\
			& \lesssim    K^{O(1)} M^{\frac{1}{p}-\frac{1}{q}}  R^\epsilon\|f\|_{L^2([0,1]^d)},   
		\end{align*}
		where we used (\ref{assump:dyadic_constant}) in the second line, and Theorem \ref{thm:k-linear} in the last line. On the other hand, by testing two extreme values $r=K^2$ and $r=R$ in the definition of $\gamma$, one can easily see that
		$$   \gamma \geq K^{-2\alpha},\quad   M\leq \gamma R^\alpha.    $$
		Combining these two estimates, we then have
		$$   M^{\frac{1}{p}-\frac{1}{q}}  \leq  M^{\frac{1}{p}-\frac{1}{2}}  \gamma^{\frac{1}{2}-\frac{1}{q}} R^{(\frac{1}{2}-\frac{1}{q})\alpha} \leq K^{O(1)} M^{\frac{1}{p}-\frac{1}{2}}  \gamma^{\frac{1}{2}-\frac{1}{p}} R^{(\frac{1}{2}-\frac{1}{q})\alpha},  $$
		where we also used $q\geq 2$ and $p\geq q$. It follows that
		$$     \|E^{\mathbf{Q}} f \|_{L^p(Y)}   \lesssim    K^{O(1)}M^{\frac{1}{p}-\frac{1}{2}}  \gamma^{\frac{1}{2}-\frac{1}{p}} R^{(\frac{1}{2}-\frac{1}{q})\alpha}\|f\|_{L^2([0,1]^d)}. $$
		Since  
		$$    \left(\frac{1}{2}-\frac{1}{q}\right)\alpha=\max_{1\leq m\leq d+n} \frac{m-X(\mathbf{Q},k,m)}{2m}\cdot\alpha,  $$
		we have obtained the first quantity of $w(p,k,\alpha,\mathbf{Q})$ in (\ref{eq:key_exponent}).
		
		It remains to consider case $2\leq p\leq q$. Note that when $p=2$, our aim (\ref{eq:main_estimate}) becomes
		\begin{equation*}
			\|E^{\mathbf{Q}}f\|_{L^2(Y)}\leq C_\epsilon R^{w(2,k,\alpha,\mathbf{Q})+\epsilon}\|f\|_{L^2([0,1]^d)},
		\end{equation*}
		where
		\begin{equation*}
			w(2,k,\alpha,\mathbf{Q})= \max\left\{   \max_{1\leq m\leq d+n} \frac{m-X(\mathbf{Q},k,m)}{2m}\cdot\alpha ,~\frac{n}{2} \right\}\geq \frac{n}{2}.
		\end{equation*}   
		So it holds trivially by the classical Agmon-H\"ormander inequality (\ref{AH ineq}). Then the estimates (\ref{eq:main_estimate}) for $2<p<q$ can be derived via interpolating between the estimates for $p=2$ and $p=q$.

		\vskip0.5cm

		\noindent \textbf{Narrow case.}  For each narrow cube $B$, we have
		\begin{equation*}
			\|E^{\mathbf{Q}}f\|_{L^p(B)} \lesssim \Big\|  \sum_{\tau \in \cup_\iota \cup_j \mathcal{W}_{\iota,j}(B) } E^{\mathbf{Q}}f_\tau \Big\|_{L^p(B)} +\Big\|   \sum_{\tau \notin\mathcal{S}_0(B)  }  E^{\mathbf{Q}}f_\tau\Big\|_{L^p(B)}.
		\end{equation*}
		If the second term dominates, by the definition of $\mathcal{S}_0(B)$, one gets
		$$\Big\|   \sum_{\tau \notin\mathcal{S}_0(B)  }  E^{\mathbf{Q}}f_\tau\Big\|_{L^p(B)}  \leq \sum_{\tau \notin\mathcal{S}_0(B)  }   \|   E^{\mathbf{Q}}f_\tau\|_{L^p(B)} \leq \max_{\tau'}\|  E^{\mathbf{Q}} f_{\tau'}  \|_{L^p(B)} ,  $$
		which can be reduced to the first term. From now on, we will only consider the first term. By the triangle inequality, one gets
		\begin{equation}\label{add}
			\|E^{\mathbf{Q}}f\|_{L^p(B)} \lesssim \Big\|  \sum_{\tau \in \cup_\iota \cup_j W_{\iota,j}(B) } E^{\mathbf{Q}}f_\tau \Big\|_{L^p(B)} \leq  \sum_{\iota} \sum_{j=1}^{L}\Big\|  \sum_{\tau \in \mathcal{W}_{\iota,j}(B) } E^{\mathbf{Q}}f_\tau \Big\|_{L^p(B)}  .
		\end{equation}
		Recall that $\iota \lesssim \log K$ and $L=L(d)$. Without loss of generality, we assume that the contribution from $\mathcal{W}_{1,1}(B) $ dominates in (\ref{add}) for a $\gtrsim (\log R)^{-1}$ fraction of all $B 
		\subset Y_{\text{narrow}}$. Denote the union of such $B$'s by $Y_1$. Thus for each $B \subset  Y_1$, one has
		\begin{equation}\label{s3equ11}
			\|E^{\mathbf{Q}}f\|_{L^p(B)} \lesssim  \log K \Big\|  \sum_{\tau \in \mathcal{W}_{1,1}(B) } E^{\mathbf{Q}}f_\tau \Big\|_{L^p(B)} .
		\end{equation}
		
		We plan to study (\ref{add}) through the same steps as in Section 3 in \cite{dz19}: do some dyadic pigeonholing, apply lower-dimensional $\ell^2L^p$ decoupling, and use rescaling and induction on scale $R/K^2$. However, since $\mathcal{W}_{1,1}(B) $ is a collection of $1/K_1$-cubes in our environment, we will use induction on scale $R/K_1^2$ rather than $R/K^2$. More precisely, we rewrite (\ref{s3equ11}) as
		\begin{equation}\label{s3equ1}
			\|E^{\mathbf{Q}}f\|_{L^p(B)} \lesssim   \log K \Big\|  \sum_{W \in \mathcal{W}_{1,1}(B) } E^{\mathbf{Q}}f_W \Big\|_{L^p(B)}  ,
		\end{equation}
		where  each $W$ is a $1/K_1$-cube with center $c(W)$. Break $B^d(0,R)$ in the physical space into $R/K_1$-cubes $D$ with center $c(D)$, and write 
		$$\sum_{W \in \mathcal{W}_{1,1}(B) } f_W=\sum_{W,D} f_{W,D},$$ 
		where each $f_{W,D}$ is Fourier supported on $W \in \mathcal{W}_{1,1}(B) $, and essentially supported on $D$. Then  
		$$ \sum_{W \in \mathcal{W}_{1,1}(B) } E^{\mathbf{Q}}f_W=\sum_{W,D} E^{\mathbf{Q}} f_{W,D}.$$ 
		Through the stationary phase method \cite{bll17}, we can see that each $E^{\mathbf{Q}} f_{W,D}$ restricted in $B_R$ is essentially supported on 
		$$\Box_{W,D}:=\Big\{    (x,y) \in B_R  : \Big|x+\sum_{j=1}^n y_j \nabla Q_j(c(W))-c(D)    \Big|\leq R/K_1 \Big\}.$$
		Since a given $B$ lies in exactly one box $\Box_{W,D}$ for each $W$, it follows from (\ref{dec low var e2}) and $K_1\leq K$ that
		\begin{equation}\label{eeeeeeee}
			\Big\|  \sum_{W \in \mathcal{W}_{1,1}(B) } E^{\mathbf{Q}}f_W \Big\|_{L^p(B)} \lesssim  K_1^{\Gamma_p^{k-2}(\mathbf{Q})+\epsilon^2} \Big(\sum_{ \Box}  \left\| E^{\mathbf{Q}}f_\Box \right\|^2_{L^p(w_B)}\Big)^{\frac{1}{2}} , 
		\end{equation}	
		where $w_B$ is given by (\ref{notation_wb}), and $\Box:=\Box_{W,D}$ for brevity.

		Set $\tilde{R}=R/K_1^2$ and $\tilde{K}=\tilde{R}^\delta$. Tile $\Box$ by the
		rectangle boxes $S$ with $d$ short sides of length $K_1\tilde{K}^2$ pointing in the same directions as the short sides of $\Box$, and $n$ long sides of length 
		$K_1^2\tilde{K}^2$ pointing in complementary directions. In what follows, we will perform dyadic pigeonholing   to $S$ and $\Box$: 
		
		\noindent (1) For each $\Box$, sort $S \subset \Box$ that intersect $Y_{1}$ according to the value of $\|E^{\mathbf{Q}} f_{\Box}\|_{L^p(S)}$ and the number of $K^2$-cubes in $Y_{1}$ contained in it: For dyadic numbers $\eta,\beta_1$, define
		$$    \mathbb{S}_{\Box,\eta,\beta_1}:=\Big\{ S\subset \Box:S~\text{contains}\sim\eta~\text{many}~K^2\text{-cubes~in~}Y_{1},~ \|E^{\mathbf{Q}} f_{\Box}\|_{L^p(S)}\sim\beta_1  \Big\}.    $$
		Let $Y_{\Box,\eta,\beta_1}$ be the union of the rectangular boxes $S$ in $\mathbb{S}_{\Box,\eta,\beta_1}$.  	
		
		\noindent (2) For fixed $\eta$ and $\beta_1$, sort $\Box$ according to the value of $\|f_\Box\|_{L^2([0,1]^d)}$ and the number $\#\mathbb{S}_{\Box,\eta,\beta_1}$: For dyadic numbers $\beta_2,M_1$, define
		$$   \mathbb{B}_{\eta,\beta_1,\beta_2,M_1}:=\Big\{ \Box: \|f_\Box\|_{L^2([0,1]^d)}\sim \beta_2,~ \# \mathbb{S}_{\Box,\eta,\beta_1}\sim M_1\Big\}.  $$

		For each $B\subset Y_{1}$, we have
		$$\sum_{W \in \mathcal{W}_{1,1}(B) } E^{\mathbf{Q}}f_W =\sum_{\eta,\beta_1,\beta_2,M_1}\left( \sum_{\substack{\Box\in \mathbb{B}_{\eta,\beta_1,\beta_2,M_1}  \\ B\subset Y_{\Box,\eta,\beta_1}}} E^{\mathbf{Q}} f_\Box  \right). $$
		Under the hypothesis $\|f\|_{L^2([0,1]^d)}=1$, we can further assume that
		\begin{align*}
			1\leq \eta \leq K^{O(1)}, \quad R^{-C} \leq \beta_1 \leq K^{O(1)}, \quad R^{-C} \leq \beta_2 \leq 1,    \quad 1 \leq M_1 \leq R^{O(1)}
		\end{align*}
		for some constant $C$. Thus there are only $O(\log R)$ choices for each dyadic number. By dyadic pigeonholing, these exist $\eta,\beta_1,\beta_2,M_1$ depending on $B$ such that (\ref{s3equ1}) becomes
		\begin{equation}\label{absd}
			\|E^{\mathbf{Q}}f\|_{L^p(B)} \lesssim  (\log R)^5     \Bigg\| \sum_{\substack{\Box\in \mathbb{B}_{\eta,\beta_1,\beta_2,M_1}  \\ B\subset Y_{\Box,\eta,\beta_1}}} E^{\mathbf{Q}}f_\Box  \cdot   \chi_{Y_{\Box,\eta,\beta_1}}  \Bigg\|_{L^p(B)}  .
		\end{equation}
		
		Finally, we sort $B \subset  Y_{1}$. Since there are only $O(\log R)$ choices on $\eta,\beta_1,\beta_2,M_1$, by pigeonholing, we can find $\tilde{\eta},\tilde{\beta}_1,\tilde{\beta}_2,\tilde{M}_1$ such that (\ref{absd}) holds for a $\gtrsim (\log R)^{-4}$ fraction of all $B\subset Y_{1}$. Denote the union of such $B$ by $ Y'$. From now on, we abbreviate $Y_{\Box,\tilde{\eta},\tilde{\beta}_1}$ and $\mathbb{B}_{\tilde{\eta},\tilde{\beta}_1,\tilde{\beta}_2,\tilde{M}_1}$ to $Y_\Box$ and $\mathbb{B}$, respectively. Next we further sort $B \subset  Y'$ by the number $\# \{ \Box \in \mathbb{B}: B \subset Y_\Box  \}$: For each dyadic number $\mu$, define
		$$  \mathbb{Y}_\mu :=\Big\{  B \subset Y':  \#  \{ \Box \in \mathbb{B}: B\subset Y_\Box \}\sim\mu \Big\}.  $$
		Let $Y_{\mu}$ be the union of $B$'s in $\mathbb{Y}_\mu$. Using dyadic pigeonholing again, we can choose $\tilde{\mu}$ such that
		\begin{equation}\label{flr eq7}
			\#\{ B : B \subset Y_{\tilde{\mu}} \} \gtrsim (\log R)^{-1} \#\{B : B \subset Y' \}. 
		\end{equation}
		From now on, we fix $\tilde{\mu}$, and denote $Y_{\tilde{\mu}}$ by $Y''$. 
		
		For each $B\subset Y''$, it follows from (\ref{eeeeeeee}) and (\ref{absd}) that
		\begin{align*}
			\| E^{\mathbf{Q}} f \|_{L^p(B)} \lesssim&  (\log R)^5 \Big\|  \sum_{\Box \in\mathbb{B} : B \subset Y_\Box}E^{\mathbf{Q}} f_\Box \cdot \chi_{Y_{\Box}}  \Big\|_{L^p(B)}  \\
			\lesssim  & (\log R)^5  K_1^{\Gamma_p^{k-2}(\mathbf{Q})+\epsilon^2} \Big( \sum_{
				\Box \in \mathbb{B}: B \subset Y_\Box} \|E^{\mathbf{Q}} f_\Box \|^2_{L^p(w_B)}  \Big)^{\frac{1}{2}}  \\
			\lesssim   &  (\log R)^5 K_1^{\Gamma_p^{k-2}(\mathbf{Q})+\epsilon^2}\tilde{\mu}^{\frac{1}{2}-\frac{1}{p}} \Big( \sum_{\Box \in \mathbb{B}: B \subset Y_\Box} \|E^{\mathbf{Q}} f_\Box \|^p_{L^p(w_B)}  \Big)^{\frac{1}{p}}.    
		\end{align*}
		Summing over all $B\subset Y''$ and combining (\ref{assump:dyadic_constant}), there is
		\begin{align}
			\|E^{\mathbf{Q}}f \|_{L^p(Y)} &\lesssim (\log R)^6 \|E^{\mathbf{Q}} f \|_{L^p(Y'')}     \nonumber \\
			& \lesssim  (\log R)^{11} K_1^{\Gamma_p^{k-2}(\mathbf{Q})+\epsilon^2} \tilde{\mu}^{\frac{1}{2}-\frac{1}{p}} \Big(\sum_{\Box \in \mathbb{B}} \|E^{\mathbf{Q}} f _\Box\|^p_{L^p(Y_\Box)}\Big)^{\frac{1}{p}} .   \label{11}
		\end{align}
		For every individual term $ \|E^{\mathbf{Q}}f _\Box\|_{L^p(Y_\Box)}$, we can use the change of variables $$ \xi   \rightarrow c(\tau)+K_1^{-1}\xi,$$ 
		and then 
		$$  |E^{\mathbf{Q}} f _\Box(x)|   = K_1^{-\frac{d}{2}}|E^{\mathbf{Q}}g(\tilde{x})|    $$
		for some function $g$ with Fourier support in $[0,1]^d$ and $\|g\|_{L^2([0,1]^d)}=\|f_\Box\|_{L^2(\tau)}$, where the new coordinates $\tilde{x}=(\tilde{x}',\tilde{x}'')\in \R^{d}\times \R^n$ are related to the old coordinates $x=(x',x'')\in \R^{d}\times \R^n$ by
		\begin{align*}
			\begin{cases}
				~	\tilde{x}'=K_1^{-1}x'+K_1^{-1} c(\tau)^T\cdot \overline{Q}(x''),    \\
				~	\tilde{x}''=K_1^{-2}x'',
			\end{cases}
		\end{align*}  
		where $\overline{Q}(x'')$ is defined in (\ref{098765}). We set $\tilde{x}=Fx$. Therefore one gets
		\begin{equation}\label{12}
			\|E^{\mathbf{Q}} f _\Box\|_{L^p(Y_\Box)}=K_1^{\frac{d+2n}{p}-\frac{d}{2}} \|E^{\mathbf{Q}}g\|_{L^p(\tilde{Y})},
		\end{equation}
		where $\tilde{Y}$ denotes the image of $Y_\Box$ under the transformation in the physical space. Note that $\|E^{\mathbf{Q}}g\|_{L^p(\tilde{Y})}$ satisfies the inductive hypothesis of Proposition \ref{main_prop:du_zhang} under the new scale $\tilde{R}$ and parameter $\tilde{M}_1$: $\tilde{Y}=F(Y_\Box)$ consists of $\sim \tilde{M}_1$ many $\tilde{K}^2$-cubes $F(S)$ in an $\tilde{R}$-ball $F(\Box)$, and $\|E^{\mathbf{Q}}g\|_{L^p(F(S))}$ is essentially a dyadic constant in $S \subset Y_\Box$. If let
		$$   \tilde{\gamma}_1=\max_{\substack{  B(x',r)\subset B(0,\tilde{R}) \\ x'\in \mathbb{R}^{d+n},r\geq \tilde{K}^2 }} \frac{\#\{  B_i:B_i \subset B(x',r)\}}{r^\alpha}, $$
		then we can apply the inductive hypothesis to obtain
		\begin{equation}\label{13}
			\|E^{\mathbf{Q}}g\|_{L^p(\tilde{Y})} \lesssim  \tilde{M}_1^{\frac{1}{p}-\frac{1}{2}} \tilde{\gamma}_1^{\frac{1}{2}-\frac{1}{p}} \tilde{R}^{w(p,k,\alpha,\mathbf{Q})+\epsilon} \|g\|_{L^2([0,1]^d)}.
		\end{equation}
		Combining (\ref{11}), (\ref{12}) with (\ref{13}), one gets
		\begin{align}
			\|E^{\mathbf{Q}} f \|_{L^p(Y)} & \leq  K_1^{\Gamma_p^{k-2}(\mathbf{Q})+\frac{d+2n}{p}-\frac{d}{2}+2\epsilon^2}  \left( \frac{\tilde{\mu} \tilde{\gamma}_1}{\tilde{M}_1} \right)^{\frac{1}{2}-\frac{1}{p}}  \tilde{R}^{w(p,k,\alpha,\mathbf{Q})+\epsilon} \Big(\sum_{\Box \in \mathbb{B}}  \|f_\Box\|_{L^2([0,1]^d)}^p \Big)^{\frac{1}{p}}\nonumber   \\
			& \sim K_1^{\Gamma_p^{k-2}(\mathbf{Q})+\frac{d+2n}{p}-\frac{d}{2}+2\epsilon^2}  \left( \frac{\tilde{\mu} \tilde{\gamma}_1}{\tilde{M}_1 \#\mathbb{B}} \right)^{\frac{1}{2}-\frac{1}{p}}  \tilde{R}^{w(p,k,\alpha,\mathbf{Q})+\epsilon}   \|f\|_{L^2([0,1]^d)}. \label{14}
		\end{align}
		Here $(\log R)^{11}$ is absorbed into $K_1^{\epsilon^2}$ due to $ K^c \leq K_1 \leq K=R^\delta$.
		
		On the other hand, the relations between the old parameters $M,\gamma$ and new parameters $\tilde{M}_1,\tilde{\gamma}_1$ are listed as follows
		\begin{equation}\label{15}
			\frac{\tilde{\mu}}{\#\mathbb{B}}\lesssim \frac{(\log R)^6 \tilde{M}_1\tilde{\eta}}{M},\quad \quad\tilde{\gamma}_1 \tilde{\eta}\lesssim  \gamma  K_1^{\alpha+n}.  
		\end{equation}
		We omit the proof of (\ref{15}), and readers may consult Section 3 in \cite{dz19} and Section 3 in \cite{cmw24} for more details. Plugging (\ref{15}) into (\ref{14}), we get 
		$$    \| E^{\mathbf{Q}} f \|_{L^p(Y)}  \lesssim K_1^{\Gamma_p^{k-2}(\mathbf{Q})+\frac{d+n-\alpha}{p}+\frac{\alpha+n-d}{2}-2w(p,k,\alpha,\mathbf{Q}) +3\epsilon^2-2\epsilon}M^{\frac{1}{p}-\frac{1}{2}}\gamma^{\frac{1}{2}-\frac{1}{p}}R^{ 
			w(p,k,\alpha,\mathbf{Q})  +\epsilon}\|f\|_{L^2([0,1]^d)}.$$
		In order to make the induction work for the narrow case, we need the following constraint
		$$ \Gamma_p^{k-2}(\mathbf{Q})+\frac{d+n-\alpha}{p}+\frac{\alpha+n-d}{2}-2w(p,k,\alpha,\mathbf{Q})\leq 0. $$
		Via some simple calculations, it is equivalent to
		$$  w(p,k,\alpha,\mathbf{Q}) \geq  \frac{\Gamma_p^{k-2}(\mathbf{Q})}{2}+\frac{d+n-\alpha}{2p}+\frac{\alpha+n-d}{4},  $$
		which is just the second quantity of $w(p,k,\alpha,\mathbf{Q})$ in (\ref{eq:key_exponent}). 
	\end{proof}
	
	In the final part of the subsection, we give one remark on Proposition \ref{main_prop:du_zhang}. In Section 4 in \cite{cmw24}, the authors considered another simpler version of this proposition. To be more precise, instead of using the broad-narrow analysis, they ran the argument for the narrow part of Du-Zhang method directly, combined with $\ell^2L^p$ decoupling (not lower-dimensional version as above), to get 
	\begin{equation}
		\|E^{\mathbf{Q}}f\|_{L^p(Y)}\leq C_\epsilon M^{\frac{1}{p}-\frac{1}{2}} \gamma^{\frac{1}{2}-\frac{1}{p}}R^{w(p,\alpha,\mathbf{Q})+\epsilon}\|f\|_{L^2([0,1]^d)},
	\end{equation}
	where $w(p,\alpha,\mathbf{Q})$ is given by 
	\begin{equation}\label{s3ff1}
		w(p,\alpha,\mathbf{Q})= \frac{\Gamma_p^{d}(\mathbf{Q})}{2}+\frac{d+n-\alpha}{2p}+\frac{\alpha+n-d}{4} .
	\end{equation}   
	To some extent, this bound can been seen as the special case of $ w(p,k,\alpha,\mathbf{Q})$ when $k=d+2$. Though this bound seems rough, it is amazing that it is better than (\ref{eq:key_exponent}) for some special degenerate quadratic forms. (\ref{s3ff1}) will also be used in the next two sections.

	\section{Optimality of several estimates in Theorem \ref{th1}}\label{sec4}

	Having obtained the upper bounds of $s(\alpha,\mathbf{Q})$ in Theorem \ref{th1} in the previous two sections, we will discuss the optimality  of the first and third estimates and associated ranges in (\ref{t1e1}) of the section.

	To start with, we consider the first estimate in (\ref{t1e1}). Note that $s(\alpha,\mathbf{Q})=0$ is sharp. The associated range $0<  \alpha \leq \frac{\mathfrak{d}_{d,1}(\mathbf{Q})}{2}$ is due to the uniform Fourier decay estimate (\ref{thm:uniform_decay_eq}) and Shayya's argument. In the following, we will take one example to show that $s(\alpha,\mathbf{Q})>0$ when $\alpha>\frac{\mathfrak{d}_{d,1}(\mathbf{Q})}{2} $, which means that the range $0<  \alpha \leq \frac{\mathfrak{d}_{d,1}(\mathbf{Q})}{2}$ satisfying $s(\alpha,\mathbf{Q})=0$ is maximal possible. This illustrates that there is no loss in the argument of applying uniform Fourier decay and Shayya's proof for general quadratic manifolds.

	There is no harm in assuming that $d'>0$. After linear changes of variables, we may assume that $Q_1$ only depends on $\xi_1,...,\xi_{d'}$. Write frequency points in $[0,1]^d$ as $\xi =(\xi',\xi''),$ with $\xi'=(\xi_1,...,\xi_{d'})$ and $\xi''=(\xi_{d'+1},...,\xi_d)$. Write spatial points in $\mathbb{R}^{d+n}$ as
	$$  x=(x',x'',x_{d+1},...,x_{d+n}),    $$
	with $x'=(x_1,...,x_{d'})$ and $x''=(x_{d'+1},...,x_{d})$. Then we can rewrite $E^{\mathbf{Q}}f$ as
	$$   E^{\mathbf{Q}}f(x)=\int_{\R^{d+n}} e\left(   x' \cdot \xi'+x''\cdot \xi'' +x_{d+1}Q_1(\xi')+x_{d+2}Q_2(\xi)+\cdots+x_{d+n}Q_n(\xi)   \right)    f(\xi) d\xi. $$
	Using linear changes of variables again, we can diagonalize $Q_1(\xi')$, and it becomes
	\begin{align*}
		&E^{\mathbf{Q}}g(y)  \\
		=&\int_{\R^{d+n}} e\left(   y' \cdot \eta'+y''\cdot \eta'' +y_{d+1}(\eta_1^2\pm \eta_2^2\pm \cdots\pm \eta_{d'}^2)+y_{d+2}\tilde{Q}_2(\eta)+\cdots+y_{d+n}\tilde{Q}_n(\eta)   \right)    g(\eta) d\eta, 
	\end{align*}
	for some new coordinates $y$ and $\eta$, and the functions $\tilde{Q}_j$ and $g$. Now we restrict
	$$    |y''|\leq c,~ |y_{d+2}|\leq c,...,~|y_{d+n}|\leq c ,   $$
	for some small constant $c>0$ that may depend on $d$, $n$ and $\mathbf{Q}$. Note that $|\eta|\leq 1$, so it follows from the locally constant property that
	$$   |E^{\mathbf{Q}}g(y)| \sim  \left|\int_{\R^{d+n}} e\left(   y' \cdot \eta'+y_{d+1}(\eta_1^2\pm \eta_2^2\pm \cdots\pm \eta_{d'}^2)   \right)    g(\eta) d\eta\right|. $$
	On the other hand, we define an $\alpha$-dimensional weight $H$ in $\R^{d'+1}$ on the variables $y'$ and $y_{d+1}$. We also take
	$$    g(\eta)= g_1(\eta') \chi_{[0,1]^{d-d'}}(\eta'').  $$
	Based on these choices, we get
	$$  \frac{\|E^{\mathbf{Q}}f\|_{L^2(B_R^{d+n},H)}}{\|f\|_{L^2([0,1]^d)}} \sim \frac{\|E_1 g_1\|_{L^2(B_R^{d'+1},H)}}{\|g_1\|_{L^2([0,1]^{d'})}},   $$
	where $E_1 g_1$ is given by
	$$    E_1 g_1:= \int_{\R^{d'}} e\left(   y' \cdot \eta'+y_{d+1}(\eta_1^2\pm \eta_2^2\pm \cdots\pm \eta_{d'}^2)   \right)    g_1(\eta') d\eta',   $$
	which denotes the Fourier extension operator of the hyperbolic paraboloids in $\mathbb{R}^{d'+1}$. In \cite{bbcrv07}, Barcel\'{o}, Bennett, Carbery, Ruiz and Vilela studied the lower bounds of $s(\alpha,\mathbf{Q})$ for the paraboloids. By constructing a fractal example, they showed that if $\mathbf{Q}(\xi)=(\xi_1^2+\cdots+\xi_d^2)$, there is
	$$  s(\alpha,\mathbf{Q})\geq \frac{2\alpha-d}{2(d+2)} ,\quad \quad \alpha\geq \frac{d}{2}.   $$
	Note that their example is also suitable for the hyperbolic paraboloids. Therefore, as long as taking $\alpha>d'/2$, we can get
	$$    \frac{\|E^{\mathbf{Q}}f\|_{L^2(B_R^{d+n},H)}}{\|f\|_{L^2([0,1]^d)}} \sim \frac{\|E_1 g_1\|_{L^2(B_R^{d'+1},H)}}{\|g_1\|_{L^2([0,1]^{d'})}} \gtrsim R^{  \frac{2\alpha-d'}{2(d'+2)}  },      $$
	which implies
	$$   s(\alpha,\mathbf{Q})\geq \frac{2\alpha-d'}{2(d'+2)}>0.   $$

	Next, we consider the third estimate in (\ref{t1e1}). Although the bound $n/2$ is trivial, we can take one example to show this bound is optimal when $ d-\tilde{d}+n\leq \alpha \leq d+n$, where $\tilde{d}$ is defined in Theorem \ref{th1}. For convenience, we assume that $0<\tilde{d}<d$. By (1.15) of \cite{gozzk23}, there is 
	$$  \mathfrak{d}_{\tilde{d},n}(\mathbf{Q})=  \mathfrak{d}_{\tilde{d},n}(\mathbf{Q}|_H) , $$
	for some $H\subset \mathbb{R}^d$ which is a $\tilde{d}$-dimensional linear subspace. After the linear transformation, without loss of generality, we can assume 
	$$  H=\{  \xi_1=\xi_2=\cdots=\xi_{d-\tilde{d}}=0  \}.   $$
	From the definition of $\tilde{d}$, we get
	$$   \mathfrak{d}_{\tilde{d},n}(\mathbf{Q}|_H) =0, $$
	which is equivalent to 
	$$\mathbf{Q}|_H \equiv 0.$$
	This fact means that all monomials in $\mathbf{Q}(\xi)$ depending on $\xi_{d-\tilde{d}+1}$, ..., $\xi_d$ must be of the forms $\xi_i \xi_{d-\tilde{d}+1}$, ..., $\xi_i \xi_{d}$ for some $i\in \{ 1,2,...,d-\tilde{d}\}.$ Take $f$ as
	$$  f(\xi)=\chi_{[0,R^{-1}]^{d-\tilde{d}}}(\xi_1,...,\xi_{d-\tilde{d}}) \chi_{[0,1]^{\tilde{d}}}(\xi_{d-\tilde{d}+1},...,\xi_d) .      $$
	Thus
	$$ \|f\|_{L^2([0,1]^d)}=  R^{-\frac{d-\tilde{d}}{2}}.   $$
	On the other hand, by the locally constant property, there is 
	$$  |E^{\mathbf{Q}}f(x)| \gtrsim R^{-(d-d')},   $$
	when $x=(x_1,...,x_{d+n})$ belongs to
	$$ \left\{ |x_1|\leq cR,...,|x_{d-\tilde{d}}|\leq cR, |x_{d-\tilde{d}+1}|\leq c,...,|x_d|\leq c, |x_{d+1}|\leq cR,...,|x_{d+n}|\leq cR     \right\},   $$
	for some small constant $c>0$. Denote this set by $T$. Then we obtain
	$$    \frac{\|E^{\mathbf{Q}}f\|_{L^2(B_R,H)}}{\|f\|_{L^2([0,1]^d)}}\geq \frac{\|E^{\mathbf{Q}}f\|_{L^2(T,H)}}{\|f\|_{L^2([0,1]^d)}} \gtrsim  R^{-\frac{d-\tilde{d}}{2}} \left(\int_T H(x)dx  \right)^{\frac{1}{2}}.  $$
	Next, we choose a $(d-\tilde{d}+n)$-dimensional weight $H$ in $\R^{d-\tilde{d}+n}$ on the variables $x_1$, ..., $x_{d-\tilde{d}}$, $x_{d+1}$,..., $x_{d+n}$, and $H$ restricted in each 1-dimensional subspace of these variables is produced by the Lebesgue measure. Thus  the above further becomes
	$$    \frac{\|E^{\mathbf{Q}}f\|_{L^2(B_R,H)}}{\|f\|_{L^2([0,1]^d)}}\gtrsim  R^{-\frac{d-\tilde{d}}{2}} \cdot R^{\frac{d-\tilde{d}+n}{2}}=R^{\frac{n}{2}},  $$
	which implies
	$$   s(\alpha,\mathbf{Q})\geq  \frac{n}{2} ,\quad \quad d-\tilde{d}+n\leq \alpha \leq d+n.  $$

	We have used one example to show the bound $n/2$ in the range $d-\tilde{d}+n\leq \alpha \leq d+n$ is sharp. Besides, we can also  prove that this range such that $s(\alpha,\mathbf{Q})=n/2$ is maximal possible. More accurately, we will show $s(\alpha,\mathbf{Q})< n/2$ when $\alpha< d-\tilde{d}+n$. It relies on the estimates built in Subsection \ref{du_zhang_method}. In fact, (\ref{s3ff1}) is enough to obtain this conclusion:
	$$    s(\alpha,\mathbf{Q})\leq \min_{ p\geq 2 } \left(   \frac{\Gamma_p^{d}(\mathbf{Q})}{2}+\frac{d+n-\alpha}{2p}+\frac{\alpha+n-d}{4} \right).  $$
	Note this quality is just $n/2$ when $p=2$, so it suffices to show: For $\alpha< d-\tilde{d}+n$, there exists $p>2$ such that
	\begin{equation}\label{ex3e1}
		\Gamma_p^{d}(\mathbf{Q})+\frac{d+n-\alpha}{p}+\frac{\alpha+n-d}{2} <n. 
	\end{equation}
	Using the range of $\alpha$ and (\ref{dec th1 00}), one gets
	\begin{align*}
		&\Gamma_p^{d}(\mathbf{Q})+\frac{d+n-\alpha}{p}+\frac{\alpha+n-d}{2} \\
		<& \Gamma_p^{d}(\mathbf{Q})-\left( \frac{1}{2}-\frac{1}{p}\right)\tilde{d}+n    \\
		=  &   \max_{d',n'}  \Big\{ \Big(d'- \mathfrak{d}_{d',n'}(\mathbf{Q})\Big)\Big(\frac{1}{2}-\frac{1}{p}\Big) -\frac{2(n-n')}{p}    \Big\}-\left( \frac{1}{2}-\frac{1}{p}\right)\tilde{d}+n  \\
		= &   \max_{d',n'}  \Big\{ \Big(d'-\tilde{d}- \mathfrak{d}_{d',n'}(\mathbf{Q})\Big)\Big(\frac{1}{2}-\frac{1}{p}\Big) -\frac{2(n-n')}{p}    \Big\}+n.
	\end{align*}
	So we only need to show: For any $0\leq d'\leq d$ and $0\leq n'\leq n$, there exists $p>2$ such that 
	\begin{equation}\label{ex3e2}
		\Big(d'-\tilde{d}- \mathfrak{d}_{d',n'}(\mathbf{Q})\Big)\Big(\frac{1}{2}-\frac{1}{p}\Big) -\frac{2(n-n')}{p}   \leq 0. 
	\end{equation}
	
	We divide the proof of (\ref{ex3e2}) into three cases. The first case is $0\leq d'\leq \tilde{d}$ and $0\leq n'\leq n$. Then (\ref{ex3e2}) holds for all $p>2$ by noting $\mathfrak{d}_{d',n'}(\mathbf{Q})\geq 0$ for all $d'$ and $n'$. The second case is $\tilde{d}< d'\leq d$ and $n'= n$. In this case, we only need to show
	$$   d'-\tilde{d}- \mathfrak{d}_{d',n}(\mathbf{Q})\leq 0,  \quad \quad \tilde{d}< d'\leq d,    $$
	and then (\ref{ex3e2}) holds for all $p>2$. In fact, it is equivalent to
	\begin{equation}\label{s4final}
		\mathfrak{d}_{d',n}(\mathbf{Q}) \geq d'-\tilde{d}, \quad \quad \tilde{d}< d'\leq d. 
	\end{equation}
	If $d'=\tilde{d}+1$, by the definition of $\tilde{d}$, we have 
	$$     \mathfrak{d}_{d',n}(\mathbf{Q}) \geq 1,  $$
	and (\ref{s4final}) follows. If $d'>\tilde{d}+1$, we have one simple observation: As long as $\mathfrak{d}_{d',n}(\mathbf{Q})>0$, there is
	$$    \mathfrak{d}_{d',n}(\mathbf{Q})>\mathfrak{d}_{d'-1,n}(\mathbf{Q}).   $$
	Applying this observation, we can obtain (\ref{s4final}) immediately, and then (\ref{ex3e2}) follows. Finally, we consider the third case: $\tilde{d}<d'\leq d$ and $0\leq n'< n$. We can further assume $d'-\tilde{d}- \mathfrak{d}_{d',n'}(\mathbf{Q})> 0$, otherwise we can repeat the argument of the first case, and then the result follows. Next we can rewrite (\ref{ex3e2}) as follows
	\begin{equation}\label{ex3e3}
		2 < p\leq 2+\frac{4(n-n')}{d'-\tilde{d}- \mathfrak{d}_{d',n'}(\mathbf{Q})}.  
	\end{equation}
	Considering the ranges of $d',n'$ of this case, we can choose an appropriate $p$ to make (\ref{ex3e3}) hold.

	\section{Several examples of Theorem \ref{th1}}\label{sec5}

	In this section, we will deal with several examples of Theorem \ref{th1}, such as the paraboloids, hyperbolic paraboloids, all quadratic manifolds with $d+n\leq 5$, and good manifolds.

    \subsection{Paraboloids and hyperbolic paraboloids}\phantom{x}
    
	Firstly, we consider the parabolic case of Theorem \ref{th1}. The estimates in Theorem \ref{th1} will recover (\ref{his1}) for all $0 <\alpha \leq d+1$, which are the best results yet.
	
	\begin{corollary}[Paraboloids]\label{cor1}
		Suppose that $\mathbf{Q}$ denotes the paraboloids, i.e.,
		$$\mathbf{Q}(\xi)=(\xi_1^2+\xi_2^2+\cdots+\xi_d^2).$$ 
		Then we have    
		\begin{align}\label{cor1e1}
			~s(\alpha,\mathbf{Q})\leq 
			\begin{cases}
				~	0, \quad \quad\quad \quad    & 0<  \alpha \leq \frac{d}{2},   \\
				~	\frac{2\alpha-d}{4}, \quad \quad \quad\quad 
				&\frac{d}{2}<  \alpha< \frac{d+1}{2},   \\
				~	\frac{\alpha}{2(d+1)}, \quad \quad \quad\quad &\frac{d+1}{2}\leq   \alpha\leq d+1. 
			\end{cases}
		\end{align}
	\end{corollary}
	
	\begin{proof} By Definition \ref{s1d1}, we have: For each $0\leq d'\leq d$, there is
		\begin{align}
			\mathfrak{d}_{d',1}(\mathbf{Q})&= \inf_{\substack{   M \in \mathbb{R}^{d\times d} \\ {\rm rank}(M)=d'   }} \rank \big( M \cdot \nabla_\xi^2\mathbf{Q} \cdot M^T  \big) \nonumber \\
			&=  \inf_{\substack{   M \in \mathbb{R}^{d\times d} \\ {\rm rank}(M)=d'   }} \rank \big( M  \cdot M^T  \big) \nonumber \\ 
			&= \inf_{\substack{   M \in \mathbb{R}^{d\times d} \\ {\rm rank}(M)=d'   }} \rank(M) = d'. \label{s5c51eq1}
		\end{align}
		
		For $2\leq k\leq d+2$, from (\ref{dec th1 00}), we can obtain
		\begin{align*}
			\Gamma_p^{k-2}(\mathbf{Q}) &=  \max_{0\leq d'\leq k-2} \max_{0\leq n' \leq 1} \Big\{  \Big(d'- \mathfrak{d}_{d',n'}(\mathbf{Q})\Big)\Big(\frac{1}{2}-\frac{1}{p}\Big) -\frac{2(1-n')}{p}    \Big\} \\
			&=  \max_{0\leq d'\leq k-2}\max\Big\{ d'\Big( \frac{1}{2} -\frac{1}{p} \Big)-\frac{2}{p},~0    \Big\}  \\
			&= \max\Big\{ \frac{k-2}{2}-\frac{k}{p},~0    \Big\} .
		\end{align*}
		Thus we get the endpoint of $\ell^2L^p$ decoupling is $p=\frac{2k}{k-2}$, and corresponding $\Gamma_p^{k-2}(\mathbf{Q})=0$. On the other hand, repeating the arguments of Appendix C in  \cite{ggo23}, we can get  
		$$   X(\mathbf{Q},k,m) \geq m \cdot \frac{k-1}{k}.   $$
		Combining the two results and doing some calculations, (\ref{t1e2}) implies the third bound in (\ref{cor1e1}), and then Theorem \ref{th1} implies Corollary \ref{cor1}.
	\end{proof}

	Next, we consider the hyperbolic case on Theorem \ref{th1}.
	
	\begin{corollary}[Hyperbolic paraboloids]\label{cor2}
		Suppose that $\mathbf{Q}$ denotes the hyperbolic paraboloids, i.e.,
		$$\mathbf{Q}(\xi)=(\xi_1^2+\cdots+\xi_{d-m}^2-\xi_{d-m+1}^2-\cdots-\xi_d^2)$$ 
		for some $m\leq d/2$. Then we have    
		\begin{align}\label{cor2e1}
			~s(\alpha,\mathbf{Q})\leq 
			\begin{cases}
				~	0, \quad \quad\quad \quad    & 0<  \alpha \leq \frac{d}{2},   \\
				~	\frac{2\alpha-d}{4}, \quad \quad \quad\quad 
				&\frac{d}{2}<  \alpha< d+1-m,   \\
				~	\frac{1}{2}, \quad \quad \quad\quad &d+1-m\leq   \alpha\leq d+1. 
			\end{cases}
		\end{align}
		Moreover, for $m>1$ and $j\in [1,m-1]$, there is
		\begin{align}\label{cor2e2}
			~s(\alpha,\mathbf{Q})\leq 
			\begin{cases}
				~	\frac{\alpha}{2(d+1-j)}, \quad \quad\quad \quad    & \frac{(j-1)(d+1-j)}{m-1}  \leq  \alpha <\frac{j(d+1-j)}{m},   \\
				~	\frac{\alpha-j}{2(d+1-j-m)}, \quad \quad \quad\quad 
				&\frac{j(d+1-j)}{m} \leq   \alpha< \frac{j(d-j)  }{m-1}.
			\end{cases}
		\end{align}
	\end{corollary}

	\begin{proof} 
		We claim
		\begin{align}\label{s5c52eq1}
			~\mathfrak{d}_{d-\ell,1}(\mathbf{Q})=
			\begin{cases}
				~	d-2\ell, \quad \quad\quad \quad    & 0\leq   \ell\leq m,   \\
				~	d-\ell-m, \quad \quad \quad\quad 
				&m<  \ell< d-m,   \\
				~	0, \quad \quad \quad\quad &d-m\leq \ell\leq d. 
			\end{cases}
		\end{align}
		In fact, when $0\leq   \ell\leq m$, note
		\begin{equation}\label{s5c51eq2}
			\mathfrak{d}_{d-\ell,1}(\mathbf{Q}) = \inf_{L {\rm ~of~dim~}d-\ell} \mathfrak{d}_{d-\ell,1}(\mathbf{Q}|_{L}).  
		\end{equation}
        By taking 
        $$  L=\big\{  \xi_1=\xi_{d-m+1},~\xi_2=\xi_{d-m+2},\dots,~\xi_\ell=\xi_{d-m+\ell} \big\}      $$
        in (\ref{s5c51eq2}), we have 
        $$   \mathbf{Q}(\xi)|_{L}=(\xi_{\ell+1}^2+\cdots+\xi_{d-m}^2-\xi_{d-m+\ell+1}^2-\cdots-\xi_d^2), $$
        which means $  \mathfrak{d}_{d-\ell,1}(\mathbf{Q}) \leq d-2\ell  $ when $\ell \leq m$. 
		On the other hand, by Lemma 3.1 of \cite{gozzk23}, we know this upper bound is also the lower bound, so we have proved the first equation in (\ref{s5c52eq1}). When $m<  \ell \leq  d-m$, note that $\mathfrak{d}_{d-\ell,1}(\mathbf{Q}) \leq \mathfrak{d}_{d-m,1}(\mathbf{Q}) - (\ell-m) = d - \ell - m$, so it suffices to prove $\mathfrak{d}_{d-\ell,1}(\mathbf{Q}) \geq d - \ell - m$. By Definition \ref{s1d1}, we have: For each $0\leq d'\leq d$, there is
		\begin{align*}
			\mathfrak{d}_{d-\ell,1}(\mathbf{Q})&= \inf_{\substack{   M \in \mathbb{R}^{d\times d} \\ {\rm rank}(M)=d-\ell   }} \rank \big( M \cdot \nabla_\xi^2\mathbf{Q} \cdot M^T  \big) \\
			&= \inf_{\substack{   M \in \mathbb{R}^{(d-\ell)\times d} \\ {\rm rank}(M)=d-\ell   }} \rank \left( M \cdot \begin{pmatrix}
				I_{d-m} & 0\\
				0 & -I_m
			\end{pmatrix} \cdot M^T  \right)\\
			&= \inf_{\dim V = d-\ell} \dim \pi_V (T(V)),
		\end{align*}
		where $V$ is the subspace in $\R^d$ spanned by row vectors of $M$, $\pi_V$ is defined as in Definition \ref{s1d2}, and $T: \R^d \rightarrow \R^d$ with $T(x_1,\dots,x_d)=(x_1,\dots,x_{d-m}, - x_{d-m+1},-x_d)$. In the last line above we used the facts that $M\cdot w = M \cdot \pi_V(w)$ for any $w \in \R^d$, and for any basis $\{v_j\}_{j=1}^{d-\ell}$ of $V$, the map $v\in V \rightarrow (v\cdot v_j)_{j=1}^{d-\ell} \in \R^{d-\ell}$ is a linear isomorphism. 
		Let $W \coloneqq \{v \in \R^d: T(v) = v\}$, then $\dim W = d-m$, and so 
		\begin{align*}
			\dim \, T(V) \cap V &\geq \dim \, W \cap V \\
			&= \dim W + \dim V - \dim \R^d \\
			&= (d-m) + (d-\ell) - d\\ 
			&= d-\ell-m
		\end{align*}
		for any subspace $V \subset \R^d$ with $\dim V = d-\ell$. This implies that
		\begin{align*}
			\mathfrak{d}_{d-\ell,1}(\mathbf{Q}) &= \inf_{\dim V = d-\ell} \dim \pi_V (T(V))\\
			&\geq \inf_{\dim V = d-\ell} \dim \, T(V) \cap V\\
			&\geq d-\ell-m,
		\end{align*}
		as desired. So we have proved the second equation in (\ref{s5c52eq1}). Finally, the last equation in (\ref{s5c52eq1}) is obvious in view of the case $\ell=d-m$, and the proof of the claim is completed.
		
		For $m+2\leq k\leq d+2$, from (\ref{dec th1 00}), we can obtain
		\begin{align*}
			\Gamma_p^{k-2}(\mathbf{Q}) &=  \max_{0\leq d'\leq k-2} \max_{0\leq n' \leq 1} \Big\{  \Big(d'- \mathfrak{d}_{d',n'}(\mathbf{Q})\Big)\Big(\frac{1}{2}-\frac{1}{p}\Big) -\frac{2(1-n')}{p}    \Big\} \\
			&=  \max_{0\leq d'\leq k-2}\max\Big\{ d'\Big( \frac{1}{2} -\frac{1}{p} \Big)-\frac{2}{p},~\Big(d'- \mathfrak{d}_{d',1}(\mathbf{Q})\Big)\Big(\frac{1}{2}-\frac{1}{p}\Big)    \Big\}  \\
			&= \max\Big\{ \frac{k-2}{2}-\frac{k}{p},~\max_{0\leq d'\leq k-2} \Big(d'- \mathfrak{d}_{d',1}(\mathbf{Q})\Big)\Big(\frac{1}{2}-\frac{1}{p}\Big)    \Big\} \\
			&= \max\Big\{ \frac{k-2}{2}-\frac{k}{p},~m\Big(\frac{1}{2}-\frac{1}{p}\Big)  \Big\} .
		\end{align*}
		Thus we get the endpoint of $\ell^2L^p$ decoupling is $p=\frac{2(k-m)}{k-m-2}$, and corresponding $\Gamma_p^{k-2}(\mathbf{Q})=\frac{m}{k-m}$. On the other hand, via the same argument as in Appendix C in \cite{ggo23}, there is
		$$   X(\mathbf{Q},k,m) \geq m \cdot \frac{k-1}{k}.   $$
		Combining above two estimates, by (\ref{t1e2}), we conclude
		\begin{equation}\label{cor2e20}
			s(\alpha,\mathbf{Q})\leq  \min_{m+2\leq k\leq d+1} \max \left\{ \frac{\alpha}{2k},~\frac{k-d+\alpha-1}{2(k-m)}     \right\}. 
	\end{equation}
		In fact, we can find that (\ref{cor2e2}) and (\ref{cor2e20}) are equivalent through some algebraic calculations, see Section 5 in \cite{beh21} for more details. Therefore, Theorem \ref{th1} implies Corollary \ref{cor2}.
		
	\end{proof}
	
	\begin{remark}
		Note that {\rm (\ref{cor2e1})} and {\rm (\ref{cor2e2})} in Corollary {\rm \ref{cor2}} are exactly {\rm (\ref{his2})} and {\rm (\ref{his3})} from \cite{beh21}. But {\rm (\ref{his4})} is beyond the power of Theorem {\rm \ref{th1}}. When $n=1$, Barron, Erdo\u{g}an and Harris proved {\rm (\ref{his4})} by using bilinear weighted restriction estimates due to Lee \cite{lee06} when $d \geq 2$ and Vargas \cite{vargas05} independently when $d = 2$. The bilinear restriction estimates in turn rely on another way of exploiting the curvature property of hyperbolic paraboloids, which is closely related to the so-called ``rotational curvature'' condition. One may consult \cite{bll17} for a higher-codimensional generalization of the bilinear restriction method as well as the rotational curvature. It is interesting to also generalize such a bilinear method to all quadratic forms in our weighted restriction setting, which might further improve Theorem {\rm \ref{th1}}, at least in some special cases. However, the current state of the art of the bilinear method in the higher-codimensional cases relies heavily on the concept of rotational curvature, while quadratic manifolds generically fail to have nonvanishing rotational curvature. One may consult the discussions in \cite{gressman15} or Section~{\rm 5} of \cite{cmp24} for additional background on this issue. Therefore, since we are pursuing a general theory that can be applied to any quadratic form including degenerate ones, we will not proceed along the direction of bilinear restriction in this paper.
	\end{remark}
	
    \subsection{Quadratic manifolds with \texorpdfstring{$d+n \leq 5$}{d+n≤5}}\phantom{x}
    
	Now let us consider all higher-codimensional quadratic forms $\mathbf{Q}$ with $d+n\leq 5$ that do not miss any variable ($\mathfrak{d}_{d,n}(\mathbf{Q})=d$) and are linearly independent ($\mathfrak{d}_{d,1}(\mathbf{Q})>0$). The two constraints on $\mathfrak{d}_{d',n'}(\mathbf{Q})$ are not essential, since they only exclude some trivial cases. Guo and Oh \cite{go22} classified the quadratic forms for the case $d=3$ and $n=2$ by some algebra. The following version stated by using the quantities $\mathfrak{d}_{d',n'}(\mathbf{Q})$ is from Theorem 1.1 in \cite{cmw24}. 
	
    \begin{lemma}\label{addth1}
		Let $d=3$ and $n=2$. Suppose that the quadratic form $\mathbf{Q}=(Q_1,Q_2)$ satisfies   $\mathfrak{d}_{3,2}(\mathbf{Q})=3$ and $\mathfrak{d}_{3,1}(\mathbf{Q})>0$. Then

		\noindent {\rm (1)} If $\mathbf{Q}$ satisfies $\mathfrak{d}_{3,1}(\mathbf{Q})=1$ and $\mathfrak{d}_{2,2}(\mathbf{Q})=1$, then $\mathbf{Q}(\xi)\equiv (\xi_1^2,\xi_2^2+\xi_1\xi_3)$ or  $(\xi_1^2,\xi_2\xi_3)$.
		
		\noindent {\rm (2)} If $\mathbf{Q}$ satisfies $\mathfrak{d}_{3,1}(\mathbf{Q})=1$ and $\mathfrak{d}_{2,2}(\mathbf{Q})=2$, then $\mathbf{Q}(\xi)\equiv (\xi_1^2,\xi_2^2+\xi_3^2)$.
		
		\noindent {\rm (3)} If $\mathbf{Q}$ satisfies $\mathfrak{d}_{3,1}(\mathbf{Q})=2$ and $\mathfrak{d}_{2,2}(\mathbf{Q})=0$, then $\mathbf{Q}(\xi)\equiv (\xi_1 \xi_2,\xi_1\xi_3)$.
		
		\noindent {\rm (4)} If $\mathbf{Q}$ satisfies $\mathfrak{d}_{3,1}(\mathbf{Q})=2$ and $\mathfrak{d}_{2,2}(\mathbf{Q})=1$, then $\mathbf{Q}(\xi)\equiv (\xi_1\xi_2,\xi_2^2+\xi_1\xi_3)$ or $ (\xi_1\xi_2,\xi_1^2\pm\xi^2_3)$.
		
		\noindent {\rm (5)} $\mathbf{Q}$ satisfies $\mathfrak{d}_{3,1}(\mathbf{Q})=2$ and $\mathfrak{d}_{2,2}(\mathbf{Q})=2$ if and only if $\mathbf{Q}$ satisfies the ${\rm (CM)}$ condition. 
    \end{lemma}
	
	We note the classification of the case $\mathfrak{d}_{3,1}(\mathbf{Q})=2$ and $\mathfrak{d}_{2,2}(\mathbf{Q})=2$ in above lemma is still not explicit enough. Now we will further categorize this case. In fact, we will provide a complete classification of all higher-codimensional quadratic forms $\mathbf{Q}$ with $d+n\leq 5$ satisfying $\mathfrak{d}_{d,n}(\mathbf{Q})=d$ and $\mathfrak{d}_{d,1}(\mathbf{Q})>0$. 
	
	\begin{lemma}\label{addth2}
		Let $d+n\leq 5$ and $n\geq 2$. Suppose that the quadratic form $\mathbf{Q}$ satisfies $\mathfrak{d}_{d,n}(\mathbf{Q})=d$ and $\mathfrak{d}_{d,1}(\mathbf{Q})>0$. 
		
		\noindent {\rm (a)} Suppose that $d=n=2$, then
		
		{\rm (a1)} If $\mathbf{Q}$ satisfies $\mathfrak{d}_{2,1}(\mathbf{Q})=1$ and $\mathfrak{d}_{1,2}(\mathbf{Q})=0$, then $\mathbf{Q}(\xi)\equiv (\xi_1^2,\xi_1\xi_2)$.
		
		{\rm (a2)} If $\mathbf{Q}$ satisfies $\mathfrak{d}_{2,1}(\mathbf{Q})=1$ and $\mathfrak{d}_{1,2}(\mathbf{Q})=1$, then $\mathbf{Q}(\xi)\equiv (\xi_1^2,\xi_2^2)$.
		
		{\rm (a3)} If $\mathbf{Q}$ satisfies $\mathfrak{d}_{2,1}(\mathbf{Q})=2$, then $\mathbf{Q}(\xi)\equiv (\xi_1 \xi_2,\xi_1^2-\xi_2^2)$.
		
		\noindent {\rm (b)} Suppose that $d=2$ and $n=3$, then $\mathbf{Q}(\xi)\equiv (\xi_1^2,\xi_1\xi_2,\xi_2^2)$.
		
		\noindent {\rm (c)} Suppose that $d=3$ and $n=2$, then
		
		{\rm (c1)} If $\mathbf{Q}$ satisfies $\mathfrak{d}_{3,1}(\mathbf{Q})=1$ and $\mathfrak{d}_{2,2}(\mathbf{Q})=1$, then $\mathbf{Q}(\xi)\equiv (\xi_1^2,\xi_2^2+\xi_1\xi_3)$ or  $(\xi_1^2,\xi_2\xi_3)$.
		
		{\rm (c2)} If $\mathbf{Q}$ satisfies $\mathfrak{d}_{3,1}(\mathbf{Q})=1$ and $\mathfrak{d}_{2,2}(\mathbf{Q})=2$, then $\mathbf{Q}(\xi)\equiv (\xi_1^2,\xi_2^2+\xi_3^2)$.
		
		{\rm (c3)} If $\mathbf{Q}$ satisfies $\mathfrak{d}_{3,1}(\mathbf{Q})=2$ and $\mathfrak{d}_{2,2}(\mathbf{Q})=0$, then $\mathbf{Q}(\xi)\equiv (\xi_1 \xi_2,\xi_1\xi_3)$.
		
		{\rm (c4)} If $\mathbf{Q}$ satisfies $\mathfrak{d}_{3,1}(\mathbf{Q})=2$ and $\mathfrak{d}_{2,2}(\mathbf{Q})=1$, then $\mathbf{Q}(\xi)\equiv (\xi_1\xi_2,\xi_2^2+\xi_1\xi_3)$ or $ (\xi_1\xi_2,\xi_1^2\pm\xi^2_3)$.
		
		{\rm (c5)} If $\mathbf{Q}$ satisfies $\mathfrak{d}_{3,1}(\mathbf{Q})=2$ and $\mathfrak{d}_{2,2}(\mathbf{Q})=2$, then $\mathbf{Q}(\xi)\equiv (\xi_1^2+\xi_2^2,\xi_2^2+\xi_3^2)$ or $ (\xi_1^2-\xi_2^2,\xi_2^2-\xi_3^2)$ or $(\xi_1^2-\xi_2^2,\xi_1\xi_2+\xi_3^2)$.
		
		Moreover, all possible cases of higher-codimensional quadratic forms with $d+n\leq 5$ are listed above, and they are not equivalent to each other. In other words, we get a complete classification.
	\end{lemma}

	\begin{proof} When $d=1$, we note the case for $2 \leq n \leq 4$ is vacuous. When $d=n=2$, readers can see Theorem 5.1 in \cite{cmp24} for reference. When $d=2$ and $n=3$, since the dimension of the linear space spanned by all quadratic monomials in $2$ variables is exactly $3$, $\mathbf{Q}\equiv (\xi_1^2,\xi_1\xi_2,\xi_2^2)$ is the only quadratic form satisfying $\mathfrak{d}_{2,3}(\mathbf{Q})=2$ and $\mathfrak{d}_{2,1}(\mathbf{Q})>0$. When $d=3$ and $n=2$, combined with Lemma \ref{addth1}, we only need to consider the case $\mathfrak{d}_{3,1}(\mathbf{Q})=2$ and $\mathfrak{d}_{2,2}(\mathbf{Q})=2$.
		
		We pick $M\in \mathbb{R}^{3\times 3}$ and $M'\in \mathbb{R}^{1\times 2}$ such that the equality in (\ref{s2def2e1}) with $d'=3,n'=1$ is achieved. By linear transformations, we may suppose that $M=I$ and $M'=(1,0)$. Then $\mathfrak{d}_{3,1}(\mathbf{Q})=2$ implies that $Q_1$ depends on 2 variables. We use linear transformations again to diagonalize $Q_1$, then
		$$ (Q_1(\xi),Q_2(\xi))\equiv (\xi_1^2\pm\xi_2^2, b_{11}\xi_1^2 +b_{22}\xi_2^2+b_{33}\xi_3^2+2b_{12}\xi_1\xi_2+2b_{13}\xi_1\xi_3+2b_{23}\xi_2\xi_3).     $$
		We divide it into two cases: $Q_1(\xi)=\xi_1^2+\xi_2^2$ and $Q_1(\xi)=\xi_1^2-\xi_2^2$. 
		
		We begin with the first case, i.e., $Q_1(\xi)=\xi_1^2+\xi_2^2$. We first prove $b_{33}\neq 0$. Otherwise, suppose $b_{33}= 0$, then
		$$ (Q_1(\xi),Q_2(\xi))\equiv (\xi_1^2+\xi_2^2, b_{11}\xi_1^2 +2b_{12}\xi_1\xi_2+2b_{13}\xi_1\xi_3+2b_{23}\xi_2\xi_3).     $$
		However, it is easy to see that $\det(x_1Q_1+x_2Q_2)$ has a factor $x_2^2$, which contradicts our assumption that $\mathbf{Q}$ satisfies the (CM) condition. 
		
		Note that $b_{33}\neq 0$, by a change of variables in $\xi_3$, we obtain
		\begin{align*}
			(Q_1(\xi),Q_2(\xi))&\equiv (\xi_1^2+\xi_2^2, b_{11}\xi_1^2 +b_{22}\xi_2^2+2b_{12}\xi_1\xi_2+\xi_3^2)    \\
			&\equiv (\xi_1^2+\xi_2^2, b_{22}\xi_2^2+2b_{12}\xi_1\xi_2+\xi_3^2).
		\end{align*}
		Here we are slightly abusing the notation $b_{ij}$, whose exact meaning may change after linear transformations. We choose to do so for convenience as we only care about the form of the expression. 
		If $b_{22}=0$, then $b_{12}\neq 0$, since otherwise it becomes the case (c2), which contradicts the (CM) assumption. So 
		\begin{equation}\label{addt2e1}
			(Q_1(\xi),Q_2(\xi))\equiv  (\xi_1^2+\xi_2^2, \xi_1\xi_2+\xi_3^2) \equiv  (\xi_1^2+\xi_2^2, \xi_2^2+\xi_3^2).  
		\end{equation}
		If $b_{22}\neq 0$, we can write
		$$    (Q_1(\xi),Q_2(\xi))\equiv (\xi_1^2+\xi_2^2, \xi_2^2+2b_{12}\xi_1\xi_2+\xi_3^2). $$
		By adding a multiple of $Q_1$ to $Q_2$, we form a complete square in $Q_2$. Using a change of variables again, we obtain
		$$    (Q_1(\xi),Q_2(\xi))\equiv (f(\xi_1,\xi_2), \xi_2^2-\xi_3^2) $$
		for some quadratic form $f(\xi_1,\xi_2)$, and it is reduced to the second case.
		
		Now we consider the second case, i.e., $Q_1(\xi)=\xi_1^2-\xi_2^2$. By a change of variables, we can assume $Q_1(\xi)=\xi_1\xi_2$, and then
		$$ (Q_1(\xi),Q_2(\xi))\equiv (\xi_1\xi_2, b_{11}\xi_1^2 +b_{22}\xi_2^2+b_{33}\xi_3^2+2b_{13}\xi_1\xi_3+2b_{23}\xi_2\xi_3).     $$
		Similar to the argument in the first case, we have $b_{33}\neq 0$, and thus
		$$ (Q_1(\xi),Q_2(\xi))\equiv (\xi_1\xi_2, b_{11}\xi_1^2 +b_{22}\xi_2^2+\xi_3^2+2b_{13}\xi_1\xi_3+2b_{23}\xi_2\xi_3).     $$
		Via additional elementary transformation, we get
		$$ (Q_1(\xi),Q_2(\xi))\equiv (\xi_1\xi_2, b_{11}\xi_1^2 +b_{22}\xi_2^2+\xi_3^2).     $$
		We will actually show that $b_{11}\neq 0$ and $b_{22}\neq0$.
		
		If $b_{11}=b_{22}=0$, one gets
		$$ (Q_1(\xi),Q_2(\xi))\equiv (\xi_1\xi_2, \xi_3^2),     $$
		which is the case (c1), a contradiction. If $b_{11}\neq 0$ and $b_{22}=0$, one gets
		$$ (Q_1(\xi),Q_2(\xi))\equiv (\xi_1\xi_2, \xi_1^2 +\xi_3^2),     $$
		which is the case (c4), still a contradiction. The cases  $b_{11}= 0$ and $b_{22}\neq 0$ are the same. 
		
		Thus we have proved $b_{11}\neq 0$ and $b_{22}\neq0$. In this setting, we have either
		\begin{equation}\label{addt2e2}
			(Q_1(\xi),Q_2(\xi))\equiv (\xi_1\xi_2, \xi_1^2 +\xi_2^2+\xi_3^2)\equiv (\xi_1^2+\xi_2^2,\xi_2^2+\xi_3^2), 
		\end{equation}
		or 
		\begin{equation}\label{addt2e3}
			(Q_1(\xi),Q_2(\xi))\equiv (\xi_1\xi_2, \xi_1^2 +\xi_2^2-\xi_3^2)\equiv (\xi_1^2-\xi_2^2,\xi_2^2-\xi_3^2), 
		\end{equation}
		or
		\begin{equation}\label{addt2e4}
			(Q_1(\xi),Q_2(\xi))\equiv (\xi_1\xi_2, \xi_1^2 -\xi_2^2+\xi_3^2)\equiv (\xi_1^2-\xi_2^2,\xi_1\xi_2+\xi_3^2).
		\end{equation}
		Note that (\ref{addt2e2}) is just (\ref{addt2e1}), and we have completed the classification of the case (c5). In addition, for (\ref{addt2e2}), we can easily check $\mathfrak{d}_{2,1}(\mathbf{Q})=0$ and $\mathfrak{d}_{1,2}(\mathbf{Q})=1$. For (\ref{addt2e3}) and (\ref{addt2e4}), we can easily check that $\mathfrak{d}_{2,1}(\mathbf{Q})=0$ and $\mathfrak{d}_{1,2}(\mathbf{Q})=0$.  
		
		To show that we get a complete classification, first note that by our proof these are indeed the only possible cases, so it remains to show that they are not equivalent to each other. Such an issue is not discussed in \cite{go22}.
		
		If two quadratic forms have different $(d,n)$, or have the same $(d,n)$ but different $\mathfrak{d}_{d',n'}(\mathbf{Q})$, then they are not equivalent. This is because $\mathfrak{d}_{d',n'}(\mathbf{Q})$ is an invariant under the equivalence. By this fact, combined with the calculations on the case (c5), we only need to show that $(\xi_1^2,\xi_2^2+\xi_1\xi_3)$ and  $(\xi_1^2,\xi_2\xi_3)$ in (c1) are not equivalent, $(\xi_1\xi_2,\xi_2^2+\xi_1\xi_3)$ and $ (\xi_1\xi_2,\xi_1^2\pm\xi^2_3)$ in (c4) are not equivalent, and $(\xi_1^2-\xi_2^2,\xi_2^2-\xi_3^2)$ and $(\xi_1^2-\xi_2^2,\xi_1\xi_2+\xi_3^2)$ in (c5) are not equivalent.
		
		A very useful property that we will use is that, for $\mathbf{Q} = (Q_1,Q_2)$, the distribution of the zero set
		$$\Lambda \coloneqq \{(x:y) \in \mathbb{RP}^1 : \det(xQ_1+yQ_2) = 0\}$$ 
		is an invariant under the equivalence. Here by ``distribution'' we mean the number of distinct zeros as well as the multiplicities of these zeros. This property can be proved by simply noting that change of variables only multiplies the value of $\det$ by a nonzero constant, and linear transformation of $\mathbf{Q}$ only distorts the zero set in $\mathbb{RP}^1$ by an invertible linear isomorphism. 
		
		For $(\xi_1^2,\xi_2^2+\xi_1\xi_3)$ and  $(\xi_1^2,\xi_2\xi_3)$ in (c1), the $\Lambda$ of the former is given by $$2y^3 = 0 \Leftrightarrow (x:y) = (1:0),$$ 
		while the $\Lambda$ of the latter is given by 
		$$2xy^2 = 0 \Leftrightarrow (x:y) = (1:0), (0:1).$$
		So they are not equivalent.
		
		Similarly, for $(\xi_1^2-\xi_2^2,\xi_2^2-\xi_3^2)$ and $(\xi_1^2-\xi_2^2,\xi_1\xi_2+\xi_3^2)$ in (c5), the $\Lambda$ of the former is given by 
		$$8x(x-y)y = 0 \Leftrightarrow (x:y) = (1:0), (1:1), (0:1),$$ 
		while the $\Lambda$ of the latter is given by 
		$$2y(4x^2+y^2) = 0 \Leftrightarrow (x:y) = (1:0).$$ 
		So they are not equivalent.
		
		Next, for $(\xi_1\xi_2,\xi_2^2+\xi_1\xi_3)$ and $ (\xi_1\xi_2,\xi_1^2\pm\xi_3^2)$ in (c4), the $\Lambda$ of the former is given by 
		$$2y^3 = 0 \Leftrightarrow (x:y) = (1:0),$$ 
		while the $\Lambda$ of the latter is given by 
		$$\pm 2yx^2 = 0 \Leftrightarrow (x:y) = (1:0), (0:1).$$ 
		So they are not equivalent. 
		
		Finally, we still need to show that $(\xi_1\xi_2,\xi_1^2 + \xi_3^2)$ and $ (\xi_1\xi_2,\xi_1^2 - \xi_3^2)$ are not equivalent, and all the tricks introduced before do not work any more. Therefore, we proceed with a direct computation. Assume by contradiction that $(\xi_1\xi_2,\xi_1^2 + \xi_3^2)$ and $ (\xi_1\xi_2,\xi_1^2 - \xi_3^2)$ are equivalent, then there must exist some invertible linear transformation
		$$(\widetilde{\xi}_1, \widetilde{\xi}_2, \widetilde{\xi}_3) = (a_1\xi_1 + b_1\xi_2 + c_1\xi_3, a_2\xi_1 + b_2\xi_2 + c_2\xi_3, a_3\xi_1 + b_3\xi_2 + c_3\xi_3)$$
		such that 
		$$(\widetilde{\xi}_1\widetilde{\xi}_2, \widetilde{\xi}_1^2 + \widetilde{\xi}_3^2) = (A_1\xi_1\xi_2 + B_1(\xi_1^2 - \xi_3^2), A_2\xi_1\xi_2 + B_2(\xi_1^2 - \xi_3^2))$$ 
		with $A_1B_2-A_2B_1 \neq 0$. By comparing the coefficients of $\xi_1^2$ and $\xi_3^2$ in the second factor $\widetilde{\xi}_1^2 + \widetilde{\xi}_3^2$, we see that $a_1^2 + a_3^2 = B_2$ and $c_1^2 + c_3^2 = -B_2$, which means $a_1^2 + a_3^2 + c_1^2 + c_3^2 = 0$ and thus we must have $a_1 = a_3 = c_1 = c_3 = 0$. However, this implies that the $3 \times 3$ matrix
		\begin{align*}
			\begin{pmatrix}
				a_1 & b_1 & c_1 \\
				a_2 & b_2 & c_2 \\
				a_3 & b_3 & c_3
			\end{pmatrix}
			=
			\begin{pmatrix}
				0 & b_1 & 0 \\
				a_2 & b_2 & c_2 \\
				0 & b_3 & 0
			\end{pmatrix}
		\end{align*}
		is not invertible, a contradiction. So $(\xi_1\xi_2,\xi_1^2 + \xi_3^2)$ and $ (\xi_1\xi_2,\xi_1^2 - \xi_3^2)$ are not equivalent. And the proof is completed.
		
	\end{proof}
	
	Based on this classification, combined with Theorem \ref{th1}, we can obtain the following weighted restriction results. In particular, for the case $d=3$ and $n=2$, Theorem \ref{th1} can improve the results of \cite{cmw24}.

	\begin{corollary}\label{th3}
		Let $d+n\leq 5$ and $n\geq 2$. Suppose that the quadratic form $\mathbf{Q}$ satisfies $\mathfrak{d}_{d,n}(\mathbf{Q})=d$ and $\mathfrak{d}_{d,1}(\mathbf{Q})>0$.

		\noindent {\rm (a)} Suppose that $d=n=2$, then

		{\rm (a1)} If $\mathbf{Q}(\xi)\equiv (\xi_1^2,\xi_1\xi_2)$, we have
		\begin{equation}\label{kk1}
			~s(\alpha,\mathbf{Q}) \leq \begin{cases}
				~	0, \quad \quad&\alpha\in(0,\frac{1}{2}],   \\
				~	\frac{2\alpha-1}{4}, \quad \quad&\alpha\in(\frac{1}{2},2],  \\
				~	\frac{\alpha+1}{4}, \quad \quad&\alpha\in(2,3],   \\
				~	1, \quad \quad& \alpha\in(3,4].
			\end{cases}
		\end{equation}

		{\rm (a2)} If $\mathbf{Q}(\xi)\equiv (\xi_1^2,\xi_2^2)$, we have
		\begin{equation}\label{kk2}
			~s(\alpha,\mathbf{Q}) \leq \begin{cases}
				~	0, \quad \quad&\alpha\in(0,\frac{1}{2}],   \\
				~	\frac{2\alpha-1}{4}, \quad \quad&\alpha\in(\frac{1}{2},\frac{3}{2}],  \\
				~	\frac{\alpha}{3}, \quad \quad&\alpha\in(\frac{3}{2},2],   \\
				~	\frac{\alpha+2}{6}, \quad \quad& \alpha\in(2,4].
			\end{cases}
		\end{equation}
		
		{\rm (a3)} If $\mathbf{Q}(\xi)\equiv (\xi_1 \xi_2,\xi_1^2-\xi_2^2)$, we have
		\begin{equation}\label{kk3}
			~s(\alpha,\mathbf{Q}) \leq \begin{cases}
				~	0, \quad \quad&\alpha\in(0,1],   \\
				~	\frac{\alpha-1}{2}, \quad \quad&\alpha\in(1,2],  \\
				~	\frac{\alpha}{4}, \quad \quad& \alpha\in(2,4].
			\end{cases}
		\end{equation}

		\noindent {\rm (b)} Suppose that $d=2$ and $n=3$. If $\mathbf{Q}(\xi)\equiv (\xi_1^2,\xi_1\xi_2,\xi_2^2)$, we have
		\begin{equation}\label{kk4}
			~s(\alpha,\mathbf{Q}) \leq \begin{cases}
				~	0, \quad \quad&\alpha\in(0,\frac{1}{2}],   \\
				~	\frac{2\alpha-1}{4}, \quad \quad&\alpha\in(\frac{1}{2},2],  \\
				~	\frac{3\alpha}{8}, \quad \quad& \alpha\in(2,\frac{16}{5}]   \\
				~    \frac{\alpha+4}{6}  , \quad \quad& \alpha\in(\frac{16}{5},5] .
			\end{cases}
		\end{equation}

		\noindent {\rm (c)} Suppose that $d=3$ and $n=2$.
		
		{\rm (c1)} If $\mathbf{Q}(\xi)\equiv (\xi_1^2,\xi_2^2+\xi_1\xi_3)$ or  $(\xi_1^2,\xi_2\xi_3)$, we have
		\begin{equation}\label{kk5}
			~s(\alpha,\mathbf{Q}) \leq \begin{cases}
				~	0, \quad \quad&\alpha\in(0,\frac{1}{2}],   \\
				~	\frac{2\alpha-1}{4}, \quad \quad&\alpha\in(\frac{1}{2},\frac{3}{2}],  \\
				~	\frac{\alpha}{3}, \quad \quad&\alpha\in(\frac{3}{2},2],   \\
				~	\frac{\alpha+2}{6} , \quad \quad& \alpha\in(2,4],  \\
				~	1, \quad  \quad& \alpha\in(4,5].
			\end{cases}
		\end{equation}

		{\rm (c2)} If $\mathbf{Q}(\xi)\equiv (\xi_1^2,\xi_2^2+\xi_3^2)$, we have
		\begin{equation}\label{kk6}
			~s(\alpha,\mathbf{Q}) \leq \begin{cases}
				~	0, \quad \quad&\alpha\in(0,\frac{1}{2}],   \\
				~	\frac{2\alpha-1}{4}, \quad \quad&\alpha\in(\frac{1}{2},\frac{3}{2}],  \\
				~	\frac{\alpha}{3}, \quad \quad& \alpha\in(\frac{3}{2},\frac{9}{5}],   \\		~	\frac{\alpha+3}{8} , \quad \quad&  \alpha\in(\frac{9}{5},5].
			\end{cases}
		\end{equation}

		{\rm (c3)} If $\mathbf{Q}(\xi)\equiv (\xi_1 \xi_2,\xi_1\xi_3)$, we have
		\begin{equation}\label{kk7}
			~s(\alpha,\mathbf{Q}) \leq \begin{cases}
				~	0, \quad \quad&\alpha\in(0,1],   \\
				~	\frac{\alpha-1}{2}, \quad \quad& \alpha\in(1,3],   \\
				~	1, \quad  \quad& \alpha\in(3,5].
			\end{cases}
		\end{equation}

		{\rm (c4)} If $\mathbf{Q}(\xi)\equiv (\xi_1\xi_2,\xi_2^2+\xi_1\xi_3)$ or $ (\xi_1\xi_2,\xi_1^2\pm\xi^2_3)$, we have
		\begin{equation}
			~s(\alpha,\mathbf{Q}) \leq \begin{cases}
				~	0, \quad \quad&\alpha\in(0,1],   \\
				~	\frac{\alpha-1}{2}, \quad \quad&\alpha\in(1,2],  \\
				~	\frac{\alpha}{4}, \quad \quad& \alpha\in(2,4],   \\
				~	1, \quad  \quad&  \alpha\in(4,5].
			\end{cases}
		\end{equation}
		
		{\rm (c5)} If $\mathbf{Q}(\xi)\equiv (\xi_1^2+\xi_2^2,\xi_2^2+\xi_3^2)$, we have
		\begin{equation}\label{th3e7}
			~s(\alpha,\mathbf{Q}) \leq \begin{cases}
				~	0, \quad \quad&\alpha\in(0,1],   \\
				~	\frac{\alpha-1}{2}, \quad \quad&\alpha\in(1,2],  \\
				~	\frac{\alpha}{4}, \quad \quad& \alpha\in(2,3],  \\
				~	\frac{\alpha+3}{8} , \quad \quad&  \alpha\in(3,5].
			\end{cases}
		\end{equation}

		{\rm (c6)} If $\mathbf{Q}(\xi)\equiv (\xi_1^2-\xi_2^2,\xi_2^2-\xi_3^2)$ or $(\xi_1^2-\xi_2^2,\xi_1\xi_2+\xi_3^2)$, we have
		\begin{equation}\label{th3e8}
			~s(\alpha,\mathbf{Q}) \leq \begin{cases}
				~	0, \quad \quad&\alpha\in(0,1],   \\
				~	\frac{\alpha-1}{2}, \quad \quad&\alpha\in(1,2],  \\
				~	\frac{\alpha}{4}, \quad \quad& \alpha\in(2,4],   \\
				~	1, \quad  \quad&  \alpha\in(4,5].
			\end{cases}
		\end{equation}
		
	\end{corollary}

	\begin{proof} The proof of this corollary is mainly some tedious calculations. Nonetheless, for completeness, we still record the detailed proof of one case, such as (a2). Suppose that $\mathbf{Q}(\xi) = (\xi_1^2,\xi_2^2)$. We note
		\begin{equation}
			\mathfrak{d}_{1,1}(\mathbf{Q})=0,~\mathfrak{d}_{1,2}(\mathbf{Q})=1,~\mathfrak{d}_{2,1}(\mathbf{Q})=1,~\mathfrak{d}_{2,2}(\mathbf{Q})=2.    
		\end{equation}
		These facts together with (\ref{t1e1}) in Theorem \ref{th1} give the first and second bounds in (\ref{kk2}). Besides, we can obtain $\Gamma_p^0(\mathbf{Q})=0$, and
		\begin{align}
			\Gamma_p^1(\mathbf{Q}) &=  \max_{0\leq d'\leq 1} \max_{0\leq n' \leq 2} \Big\{  \Big(d'- \mathfrak{d}_{d',n'}(\mathbf{Q})\Big)\Big(\frac{1}{2}-\frac{1}{p}\Big) -\frac{2(2-n')}{p}    \Big\}  \nonumber \\
			&=  \max_{0\leq n' \leq 2}\Big\{ \Big(1- \mathfrak{d}_{1,n'}(\mathbf{Q})\Big)\Big(\frac{1}{2}-\frac{1}{p}\Big) -\frac{2(2-n')}{p}     \Big\} \nonumber \\
			&= \max\Big\{ \frac{1}{2}-\frac{3}{p},~0    \Big\} ,  \label{cor6e2}
		\end{align}
		and
		\begin{align}
			\Gamma_p^2(\mathbf{Q}) &=  \max_{0\leq d'\leq 2} \max_{0\leq n' \leq 2} \Big\{  \Big(d'- \mathfrak{d}_{d',n'}(\mathbf{Q})\Big)\Big(\frac{1}{2}-\frac{1}{p}\Big) -\frac{2(2-n')}{p}    \Big\}  \nonumber\\
			&=     \max\Big\{\frac{1}{2}-\frac{3}{p},~0 ,~ \max_{0\leq n'\leq 2} \Big\{ \Big(2- \mathfrak{d}_{2,n'}(\mathbf{Q})\Big)\Big(\frac{1}{2}-\frac{1}{p}\Big) -\frac{2(2-n')}{p}   \Big\}    \Big\} \nonumber \\
			&= \max\Big\{ 1-\frac{6}{p},~0    \Big\} .   \label{cor6e3}
		\end{align}
		On the other hand, we start to calculate the values of $X(\mathbf{Q},k,m)$ for all $2\leq k\leq 3$ and $0\leq m \leq 4$. Recall that $V_\xi$ denotes the tangent space of the graph  $S_{\mathbf{Q}}$ at the point $(\xi,\mathbf{Q}(\xi))$, so
		\begin{equation*}
			V_{\xi}=  \begin{pmatrix}
				1 & 0 & 2\xi_1 &0 \\
				0 & 1 & 0 & 2\xi_2
			\end{pmatrix}  .  
		\end{equation*}
		We also use $V_\xi^{\perp}$ to denote the normal space of the graph  $S_{\mathbf{Q}}$ at the point $(\xi,\mathbf{Q}(\xi))$, so
		\begin{equation*}
			V_{\xi}^{\perp}=  \begin{pmatrix}
				2\xi_1 & 0 & -1 &0 \\
				0 & 2\xi_2 & 0 & -1
			\end{pmatrix}  .  
		\end{equation*}
		Let's first state the final calculation results: When $m=3$, there are
		\begin{align}
			&\sup_{\dim V=3} \dim \{ \xi\in \mathbb{R}^2 : \dim (\pi_{V_\xi} (V)) \geq 3  \}=0,  \label{c6e1}\\   
			&\sup_{\dim V=3} \dim \{ \xi\in \mathbb{R}^2 : \dim (\pi_{V_\xi} (V))\leq 2   \}=2,  \label{c6e2}\\  
			&\sup_{\dim V=3} \dim \{ \xi\in \mathbb{R}^2 : \dim (\pi_{V_\xi} (V))\leq 1   \}=1,  \label{c6e3}\\  
			&\sup_{\dim V=3} \dim \{ \xi\in \mathbb{R}^2 : \dim (\pi_{V_\xi} (V)) \leq 0   \}=0. \label{c6e4}
		\end{align}
		When $m=2$, there are
		\begin{align}
			&\sup_{\dim V=2} \dim \{ \xi\in \mathbb{R}^2 : \dim (\pi_{V_\xi} (V)) \geq 3  \}=0,  \label{c6e5}\\   
			&\sup_{\dim V=2} \dim \{ \xi\in \mathbb{R}^2 : \dim (\pi_{V_\xi} (V))\leq 2   \}=2,  \label{c6e6}\\  
			&\sup_{\dim V=2} \dim \{ \xi\in \mathbb{R}^2 : \dim (\pi_{V_\xi} (V))\leq 1   \}=2, \label{c6e7} \\  
			&\sup_{\dim V=2} \dim \{ \xi\in \mathbb{R}^2 : \dim (\pi_{V_\xi} (V)) \leq 0   \}=0.\label{c6e8}
		\end{align}
		When $m=1$, there are
		\begin{align}   
			&\sup_{\dim V=1} \dim \{ \xi\in \mathbb{R}^2 : \dim (\pi_{V_\xi} (V))\geq 2   \}=0, \label{c6e9} \\  
			&\sup_{\dim V=1} \dim \{ \xi\in \mathbb{R}^2 : \dim (\pi_{V_\xi} (V))\leq 1   \}=2, \label{c6e10} \\  
			&\sup_{\dim V=1} \dim \{ \xi\in \mathbb{R}^2 : \dim (\pi_{V_\xi} (V)) \leq 0   \}=1.\label{c6e11}
		\end{align}
		Combining these estimates with Definition \ref{s1d2}, we have
		\begin{align}
			&X(\mathbf{Q},2,1)=0,~ X(\mathbf{Q},2,2)=1,~X(\mathbf{Q},2,3)=1,~X(\mathbf{Q},2,4)=2, \label{c6e12}\\
			&X(\mathbf{Q},3,1)=1,~ X(\mathbf{Q},3,2)=1,~X(\mathbf{Q},3,3)=2,~X(\mathbf{Q},3,4)=2.  \label{c6e13}
		\end{align}
		Here the case for $m=4$ is obvious by noting $\dim (\pi_{V_\xi} (V))=2$ for all $\xi \in \mathbb{R}^2$. It suffices to prove (\ref{c6e1})-(\ref{c6e11}). We will only prove the case $m=3$, i.e., (\ref{c6e1})-(\ref{c6e4}), and the arguments for other cases are similar. Let $V$ be a subspace with dimension $3$. By the rank inequality, we have
		$$     1 \leq  \dim (\pi_{V_\xi} (V)) \leq 2.  $$
		This fact implies (\ref{c6e1}), (\ref{c6e2}) and (\ref{c6e4}) immediately. So we still need to prove (\ref{c6e3}). In view of Lemma 5.3 in \cite{ggo23}, (\ref{c6e3}) is equivalent to
		\begin{equation}\label{c6e14}
			\sup_{\dim W=1} \dim \{ \xi\in \mathbb{R}^2 : \dim (\pi_{V_\xi^{\perp}} (W))\leq 0   \}=1.   
		\end{equation}
		Let $W =\text{span} \{(a,b,c,d)\}$, then 
		\begin{equation*}
			\dim (\pi_{V_\xi^{\perp}} (W))= \text{rank}_{\mathbb{R}} \begin{pmatrix}
				2\xi_1 & 0 & -1 &0 \\
				0 & 2\xi_2 & 0 & -1
			\end{pmatrix}  
			\begin{pmatrix}
				a \\
				b \\
				c \\
				d
			\end{pmatrix}=
			\text{rank}_{\mathbb{R}} \begin{pmatrix}
				2a\xi_1-c \\
				2b\xi_2-d
			\end{pmatrix} . 
		\end{equation*}
		Then the condition $ \dim (\pi_{V_\xi^{\perp}} (W))\leq 0$ implies 
		\begin{equation}\label{c6e20}
			2a\xi_1-c= 2b\xi_2-d=0.   
		\end{equation}
		Noting that $a,b,c,d$ are not zero simultaneously, (\ref{c6e20}) denotes a submanifold in $\mathbb{R}^2$ with dimensions not exceeding 1 (can reach 1), and then (\ref{c6e3}) follows. After calculating the quantities $X(\mathbf{Q},k,m)$ and $\Gamma_p^{k-2}(\mathbf{Q})$, we can use (\ref{t1e2}) to get 
		\begin{align}
			~s(\alpha,\mathbf{Q})\leq 
			\begin{cases}
				~	\frac{\alpha}{2}, \quad \quad\quad \quad    & 0<  \alpha \leq 1,   \\
				~	\frac{\alpha+2}{6}, \quad \quad \quad\quad 
				&1<  \alpha   \leq 4,
			\end{cases}
		\end{align}
		which gives the final bound in (\ref{kk2}). So Theorem \ref{th1} implies (\ref{kk2}) except the third bound $\alpha/3$. 
		
		The bound $\alpha/3$ in (\ref{kk2}) can be derived by a Stein-Tomas-type inequality. To be more precise, we can easily check that $\mathbf{Q}= (\xi_1^2,\xi_2^2)$ satisfies the (CM) condition. By Theorem 2.14 in \cite{mockenhaupt96}, we have the following sharp estimate:
		$$   \|E^{\mathbf{Q}}f\|_{L^6(B_R)} \lesssim \|f\|_{L^2([0,1]^d)}.   $$
		Using H\"older's inequality, one concludes
		$$     \|E^{\mathbf{Q}}f\|_{L^2(B_R,H)} \lesssim R^{\frac{\alpha}{3}} \|f\|_{L^2([0,1]^d)},  $$
		which gives the third bound in (\ref{kk2}). We also point out that the estimate (\ref{s3ff1}) derived by (\ref{cor6e3}) is 
		$$    ~s(\alpha,\mathbf{Q})\leq \frac{\alpha+2}{6}, $$
		which is not better than that from (\ref{t1e2}).

	\end{proof}

	\begin{remark}\label{rmk:ST_characterization}
		From the proof above, we see that in certain ranges, the bounds derived by Theorem {\rm \ref{th1}} are worse than those in Corollary {\rm \ref{th3}} {\rm(}or Theorem {\rm 1.3} in \cite{cmw24}{\rm)}, such as the third bound in {\rm (\ref{kk2})}, {\rm (\ref{kk4})}, {\rm (\ref{kk5})} and {\rm (\ref{kk6})}. These bounds were obtained by sharp Stein-Tomas-type inequalities as above. Therefore, it is also an interesting problem to excavate such type of argument for general quadratic forms, which should further improve Theorem {\rm \ref{th1}}, at least in some cases. For example, when $d=3$ and $n=2$, this has been completely solved in \cite{cmw24} by using the broad-narrow analysis, adaptive decoupling and adaptive induction on scales.
        
        However, due to possible degenerate structures, attaining sharp Stein-Tomas-type inequalities for general quadratic forms is still an open problem. Another obstacle is to find the correct nondegeneracy condition. Although the algebraic quantities $\mathfrak{d}_{d',n'}(\mathbf{Q})$ are enough to characterize the properties $s(\alpha,\mathbf{Q})=0$ and $s(\alpha,\mathbf{Q})=n/2$\footnote{This is because in {\rm (\ref{t1e1})}, the first and third estimates, as well as their associated ranges, are sharp.}, they cannot fully characterize the endpoints of Stein-Tomas-type inequalities. For example, $(\xi_1\xi_2,\xi_2^2+\xi_1\xi_3)$ and $ (\xi_1\xi_2,\xi_1^2\pm\xi^2_3)$ possess the same $\mathfrak{d}_{d',n'}(\mathbf{Q})$ for all $0\leq d'\leq d$ and $0\leq n'\leq n$, but they have different sharp Stein-Tomas-type inequalities {\rm(}see Section {\rm1} in \cite{cmw24}{\rm)}. 
        
        Based on these considerations, we do not intend to study Stein-Tomas inequalities systematically in this paper. However, in Section {\rm \ref{appendix:relationships}}, we will propose a conjecture about the complete geometric/algebraic characterization of quadratic manifolds satisfying ``best possible'' Stein-Tomas-type inequalities, as well as prove one direction of the implication, which may facilitate future works along this direction. 
	\end{remark}

    \subsection{Good manifolds}\phantom{x}
    
	Finally, we consider good quadratic forms (Definition~\ref{def:good_mfd}). The following corollary of Theorem~\ref{th1} relies on results obtained in our study of Theorem~\ref{thm:relation_diagram}, such as Proposition~\ref{prop:diag_not_str}.

	\begin{corollary}\label{cor:good_mfd}
		Suppose that $\mathbf{Q}$ is a good quadratic form and $d \geq 4$\footnote{This additional assumption is introduced purely to ensure that all the ranges that appear make sense, especially $\lfloor\frac{d+1}{2}\rfloor + 3 \leq k \leq d+1$. However, since when $n=2$, all the cases $d \leq 3$ have been explicitly computed in Corollary \ref{th3}, such an assumption is not a big deal.}, then    
		\begin{align}\label{cor4e1}
			~s(\alpha,\mathbf{Q})\leq 
			\begin{cases}
				~	0, \quad \quad\quad \quad    & 0<  \alpha \leq \frac{d-1}{2},   \\
				~	\frac{2\alpha-d+1}{4}, \quad \quad \quad\quad 
				&\frac{d-1}{2}<  \alpha< \frac{d+3}{2}.
			\end{cases}
		\end{align}
		Moreover, when $\frac{d-1}{2}<\alpha\leq d+2$, we have
		$$  s(\alpha,\mathbf{Q})\leq \min_{3\leq k\leq d+1}  \max\Big\{   \frac{\alpha}{k+1} ,~ \mu(k,\alpha)\Big\},  $$
		where $\mu(k,\alpha)$ is given by
		\begin{align*}
			~\mu(k,\alpha)=
			\begin{cases}
				~	\frac{2k+\alpha-d-2}{2k}, \quad \quad\quad \quad    & 3\leq k\leq \lfloor\frac{d+1}{2}\rfloor+2,   \\
				~	\min\big\{ \frac{k+\alpha-d}{2(k-\lfloor\frac{d+1}{2}\rfloor)},~ \frac{2\lfloor\frac{d+1}{2}\rfloor+\alpha+2-d}{2(\lfloor\frac{d+1}{2}\rfloor+2)}  \big\}, \quad \quad \quad\quad 
				&\lfloor\frac{d+1}{2}\rfloor + 3 \leq k \leq d+1.
			\end{cases}
		\end{align*}
	\end{corollary}

	\begin{proof} From Lemma 5.1 in \cite{ggo23}, we have
		\begin{equation}\label{cor4pe1}
			\mathfrak{d}_{d-m,2}(\mathbf{Q})  \geq d-m , \quad  \mathfrak{d}_{d-m,1}(\mathbf{Q})  \geq d-2m-1.      
		\end{equation}
		We note $\mathfrak{d}_{d,1}(\mathbf{Q})  \geq d-1$. In fact, using one basic computation, we have $\mathfrak{d}_{d,1}(\mathbf{Q}) = d-1$, and then (\ref{cor4e1}) follows. 
		
		By the proof of Proposition~\ref{prop:diag_not_str} and Remark~\ref{rmk:d_even_partial}, we have $\mathfrak{d}_{\frac{d+1}{2},1}(\mathbf{Q}) =0$ when $d$ is odd, and $\mathfrak{d}_{\frac{d}{2},1}(\mathbf{Q}) =0$ when $d$ is even. In summary, we have $\mathfrak{d}_{\lfloor\frac{d+1}{2}\rfloor,1}(\mathbf{Q}) =0$, where $\lfloor a\rfloor$ represents the maximum integer not greater than $a$. We now use all above results to calculate $\ell^2L^p$ decoupling exponents:
		\begin{align}
			\Gamma_p^{k-2}(\mathbf{Q}) &=  \max_{0\leq d'\leq k-2} \max_{0\leq n' \leq 2} \Big\{  \Big(d'- \mathfrak{d}_{d',n'}(\mathbf{Q})\Big)\Big(\frac{1}{2}-\frac{1}{p}\Big) -\frac{2(2-n')}{p}    \Big\} \nonumber \\
			&=     \max\Big\{\frac{k-2}{2}-\frac{k+2}{p},~0 ,~ \max_{0\leq d'\leq k-2} \Big\{ \Big(d'- \mathfrak{d}_{d',1}(\mathbf{Q})\Big)\Big(\frac{1}{2}-\frac{1}{p}\Big) -\frac{2}{p}   \Big\}    \Big\}. \label{cor4pe2} 
		\end{align}
		If $k-2\leq \lfloor\frac{d+1}{2}\rfloor$, then (\ref{cor4pe2}) becomes
		\begin{align*}
			\Gamma_p^{k-2}(\mathbf{Q}) &=     \max\Big\{\frac{k-2}{2}-\frac{k+2}{p},~0 ,~ \max_{0\leq d'\leq k-2} \Big\{d'\Big(\frac{1}{2}-\frac{1}{p}\Big) -\frac{2}{p}   \Big\}    \Big\} \\
			&= \max\Big\{\frac{k-2}{2}-\frac{k}{p},~0  \Big\}.
		\end{align*}
		Thus we get the endpoint of $\ell^2L^p$ decoupling and the corresponding exponent:
		$$ p=\frac{2k}{k-2},\quad \quad    \Gamma_p^{k-2}(\mathbf{Q})=0. $$
		If $k-2> \lfloor\frac{d+1}{2}\rfloor$, then (\ref{cor4pe2}) becomes
		\begin{align*}
			\Gamma_p^{k-2}(\mathbf{Q}) &\leq  \max\Big\{\frac{k-2}{2}-\frac{k+2}{p},~0 ,\max_{0\leq d'\leq \lfloor\frac{d+1}{2}\rfloor} \Big\{d'\Big(\frac{1}{2}-\frac{1}{p}\Big) -\frac{2}{p}   \Big\} ,\\
			&\quad \quad \quad \quad \quad \quad \quad \quad \quad \quad \quad \quad~\max_{\lfloor\frac{d+1}{2}\rfloor+1\leq d'\leq k-2} \Big\{\Big(d-d'+1 \Big)\Big(\frac{1}{2}-\frac{1}{p}\Big) -\frac{2}{p}   \Big\}    \Big\} \\
			&=  \max\Big\{\frac{k-2}{2}-\frac{k+2}{p},~\frac{\lfloor\frac{d+1}{2}\rfloor}{2}-\frac{\lfloor\frac{d+1}{2}\rfloor+2}{p},~\frac{\lceil\frac{d-1}{2}\rceil}{2}-\frac{\lceil\frac{d-1}{2}\rceil+2}{p},~0   \Big\}    \\
			&=   \max\Big\{\frac{k-2}{2}-\frac{k+2}{p},~\frac{\lfloor\frac{d+1}{2}\rfloor}{2}-\frac{\lfloor\frac{d+1}{2}\rfloor+2}{p},~0   \Big\}  ,
		\end{align*}
		where $\lceil a\rceil$ represents the smallest integer not less than $a$.
		Thus we get the endpoints of $\ell^2L^p$ decoupling and the corresponding exponents:
		$$ p=\frac{2(\lfloor\frac{d+1}{2}\rfloor+2)}{\lfloor\frac{d+1}{2}\rfloor}, \quad\quad  \Gamma_p^{k-2}(\mathbf{Q})=0;    $$
		and
		$$    p=\frac{2(k-\lfloor\frac{d+1}{2}\rfloor)}{k-\lfloor\frac{d+1}{2}\rfloor-2}, \quad\quad  \Gamma_p^{k-2}(\mathbf{Q})=\frac{2\lfloor\frac{d+1}{2}\rfloor-k+2}{k-\lfloor\frac{d+1}{2}\rfloor}.   $$
		On the other hand, by Proposition 5.2 in \cite{ggo23}, we have
		$$   X(\mathbf{Q},k,m) \geq m \cdot \frac{k-1}{k+1}.   $$
		Using all estimates we obtained above, (\ref{t1e2}) becomes
		\begin{align*}
			s(\alpha,\mathbf{Q})&\leq \min_{ \substack{ 3\leq k\leq d+1 \\ p\geq 2 } } \max\Big\{   \frac{\alpha}{k+1} ,~\frac{\Gamma_p^{k-2}(\mathbf{Q})}{2}+\frac{d+2-\alpha}{2p}+\frac{\alpha+2-d}{4} \Big\}    \\
			&=   \min_{ 3\leq k\leq d+1  }  \max\Big\{   \frac{\alpha}{k+1} ,~\min_{  p\geq 2  }\Big\{\frac{\Gamma_p^{k-2}(\mathbf{Q})}{2}+\frac{d+2-\alpha}{2p}+\frac{\alpha+2-d}{4}\Big\} \Big\}  \\
			&\leq   \min \Big\{ \min_{3\leq k\leq \lfloor\frac{d+1}{2}\rfloor+2} \max\Big\{  \frac{\alpha}{k+1} , \frac{2k+\alpha-d-2}{2k} \Big\} ,   \\
			&\quad \quad \quad \quad \quad \quad  \min_{\lfloor\frac{d+1}{2}\rfloor+3\leq k  \leq d} \max\Big\{  \frac{\alpha}{k+1} ,    \frac{k+\alpha-d}{2(k-\lfloor\frac{d+1}{2}\rfloor)}\Big\} , ~   \frac{2\lfloor\frac{d+1}{2}\rfloor+\alpha+2-d}{2(\lfloor\frac{d+1}{2}\rfloor+2)}\Big\}.
		\end{align*}
		
	\end{proof}

	\section{The proof of Theorem~\ref{thm:relation_diagram} and additional remarks}\label{appendix:relationships}

    The goal of this section is to present the proof of Theorem~\ref{thm:relation_diagram}, which is quite lengthy: there are so many arrows that we need to justify. So we will split the proof into $12$ subsections. As we have mentioned, not all parts of the proof belong to us, and we will cite references at the proper time along the way when something has already been done in the literature. We also devote one additional subsection to remarks on general higher-codimensional Fourier restriction theory, which can be viewed as a natural extension of the discussions in this section and may inspire future work.

    Without loss of generality, we may always assume $\mathbf{Q}$ is nontrivial, i.e., $\mathfrak{d}_{d,1}(\mathbf{Q}) \geq 1$, throughout this section. Otherwise $S_\mathbf{Q}$ would be degenerate and compressed into a low-dimensional vector space, so Theorem~\ref{thm:relation_diagram} is vacuous.

    \addtocontents{toc}{\protect\setcounter{tocdepth}{0}}
    
    \subsection{Proof of \texorpdfstring{$\protect\circled{1}$}{①}}\label{subsec:1_proof}\phantom{x}
	
	``$\Longrightarrow$'' is due to Proposition~2.3 in \cite{mockenhaupt96}, while ``$\Longleftarrow$'' when $n=2$ is essentially due to discussions on page~17 of \cite{mockenhaupt96}. However, when $n \geq 3$, ``$\Longleftarrow$'' is not shown anywhere in the literature. Here we will provide a proof of it by performing a direct computation, part of which is similar to that in the proof of Proposition~2.3 in \cite{mockenhaupt96}. Even so, some annoying issues need to be taken care of to deal with the  ``endpoint'' cases. Therefore, for the reader's convenience as well as for demonstrating the additional difficulties, we begin by first rewriting the proof of ``$\Longrightarrow$'' in \cite{mockenhaupt96} by our notation in a rigorous way.
	
	By (\ref{eq:tilde_E_pointwise}) at the beginning of the proof of Proposition~\ref{prop:direct_compute}, we have
	
	\begin{align}\label{eq:tilde_E_Lp}
		\int_{\R^{d+n}} \abs{\widetilde{E}^\mathbf{Q}1(x)}^p dx & = \int_{\R^{d+n}} \abs{(2\pi)^{\frac{d}{2}} \cdot \det(I - i\overline{Q}(x''))^{-\frac{1}{2}} \cdot e^{-{x'}^T\cdot (I - i\overline{Q}(x''))^{-1} \cdot x'/2}}^p dx\nonumber\\
		& \sim \int_{\R^n} \abs{\det(I-i\overline{Q}(x''))}^{-\frac{p}{2}} \int_{\R^d}\abs{e^{-{x'}^T\cdot (I - i\overline{Q}(x''))^{-1} \cdot x'}}^{\frac{p}{2}} dx'dx''.
	\end{align}
	Note that $(I - i\overline{Q}(x''))^{-1} = (I + \overline{Q}(x'')^2)^{-1}(I + i\overline{Q}(x''))$, so we have
	\begin{align*}
		\int_{\R^d}\abs{e^{-{x'}^T\cdot (I - i\overline{Q}(x''))^{-1} \cdot x'}}^{\frac{p}{2}} dx' & = \int_{\R^d}\abs{e^{-{x'}^T\cdot (I + \overline{Q}(x'')^2)^{-1}(I + i\overline{Q}(x'')) \cdot x'}}^{\frac{p}{2}} dx'\\
		& = \int_{\R^d}\abs{e^{-{x'}^T\cdot (I + \overline{Q}(x'')^2)^{-1} \cdot x'}}^{\frac{p}{2}} dx'\\
		& = \det(I + \overline{Q}(x'')^2)^{\frac{1}{2}} \cdot \int_{\R^d}e^{-\frac{p}{2}|x'|^2} dx'\\
		& \sim \det(I + \overline{Q}(x'')^2)^{\frac{1}{2}}.
	\end{align*}
	Besides, for any $A\in \R^{d\times d}$, by $\det(I + A^2) = \det(I+iA)\det(I-iA)$ and $\det(I+iA) = \overline{\det(I-iA)}$, we see that $\det(I + A^2) = \abs{\det(I-iA)}^2$ (always nonnegative). For our purposes, we take $A = \overline{Q}(x'')$. Thus, by plugging these estimates back into (\ref{eq:tilde_E_Lp}), we get
	\begin{align}\label{eq:Lp_identity_moc96}
		\int_{\R^{d+n}} \abs{\widetilde{E}^\mathbf{Q}1(x)}^p dx & \sim \int_{\R^n} \abs{\det(I-i\overline{Q}(x''))}^{-\frac{p}{2}+1} dx'' \nonumber\\
		& \sim \int_{\R^n} \left[\prod_{j=1}^d (1 + |x''||\mu_j(\theta)|)\right]^{-\frac{p}{2} + 1} dx''\nonumber\\
		& = \int_{\mathbb{S}^{n-1}} \int_0^\infty \frac{r^{n-1}}{\left[\prod_{j=1}^d (1 + r|\mu_j(\theta)|)\right]^{\frac{p}{2} - 1}} dr d\sigma(\theta),
    \end{align}
	where in the second line we applied (\ref{eq:det_eigen_express}) with $\theta = x''/|x''|$, and in the last line we used polar coordinates transformation. Recall that we only care about the case when $p > \frac{2(d+n)}{d}$, which means the exponent $\frac{p}{2}-1 > \frac{n}{d} > 0$. As a side remark, the computations we have done so far follow \cite{mockenhaupt96} and will be used in the proofs of both ``$\Longrightarrow$'' and ``$\Longleftarrow$''.
	
	For ``$\Longrightarrow$'', Mockenhaupt \cite{mockenhaupt96} proceeded by first obtaining a good lower bound for the inner integral in (\ref{eq:Lp_identity_moc96}). Note that 
	$$\prod_{j=1}^d (1 + r|\mu_j(\theta)|)=\gamma_0(\theta) + \gamma_1(\theta) r + \cdots + \gamma_d(\theta) r^d,$$ where each coefficient $\gamma_k(\theta) \geq 0$, and
	$$  \gamma_0(\theta) = 1, \quad \gamma_d(\theta) = \prod_{j=1}^d \abs{\mu_j(\theta)} = \abs{\det(\overline{Q}(\theta))}.    $$
	Since each $\gamma_k(\theta)$ is a polynomial in $\{\abs{\mu_j(\theta)}\}_{j=1}^d$ which is uniformly bounded above by some constant $C_\mathbf{Q}$ (by Corollary~\ref{cor:uniform_spectral}), we know that there exists some constant $M_\mathbf{Q}$ such that $\gamma_k(\theta) \leq M_\mathbf{Q}$, $\forall\, k=1,\dots,d$. Now for each $\theta \in \mathbb{S}^{d-1}$, by considering $r \geq  \widetilde{C}_\mathbf{Q}/\gamma_d(\theta)$ \footnote{We interpret this as $r\geq \infty$ if $\gamma_d(\theta) = 0$, i.e., we simply ignore contribution to (\ref{eq:Lp_identity_moc96}) from such $\theta$'s.} where  $\widetilde{C}_\mathbf{Q} \coloneqq \max\{C_\mathbf{Q}, M_\mathbf{Q}\}$, we have $r \geq \widetilde{C}_\mathbf{Q}/C_\mathbf{Q} \geq 1$, which implies 
	$$   \gamma_d(\theta)r^d \geq \gamma_d(\theta)r \geq \widetilde{C}_\mathbf{Q} \geq C_\mathbf{Q} \geq 1 = \gamma_0(\theta)   $$
	and
	$$  \gamma_d(\theta)r^d \geq M_\mathbf{Q}r^k,\quad \quad \forall~ k = 1,\dots,d-1. $$
	Therefore, we can bound the inner integral in (\ref{eq:Lp_identity_moc96}) as follows: For each fixed $\theta \in \mathbb{S}^{n-1}$, we have
	\begin{align*}
		\int_0^\infty \frac{r^{n-1}}{\left[\prod_{j=1}^d (1 + r|\mu_j(\theta)|)\right]^{\frac{p}{2} - 1}} dr & = \int_0^\infty \frac{r^{n-1}}{(\gamma_0(\theta) + \gamma_1(\theta) r + \cdots + \gamma_d(\theta) r^d)^{\frac{p}{2} - 1}} dr\\
		& \geq \int_0^\infty \frac{r^{n-1}}{(\gamma_0(\theta) + M_\mathbf{Q}r + \cdots + M_\mathbf{Q}r^{d-1} + \gamma_d(\theta) r^d)^{\frac{p}{2} - 1}} dr\\
		& \geq \int_{\frac{\widetilde{C}_\mathbf{Q}}{\gamma_d(\theta)}}^\infty \frac{r^{n-1}}{(\gamma_0(\theta) + M_\mathbf{Q}r + \cdots + M_\mathbf{Q}r^{d-1} + \gamma_d(\theta) r^d)^{\frac{p}{2} - 1}} dr\\
		& \geq \int_{\frac{\widetilde{C}_\mathbf{Q}}{\gamma_d(\theta)}}^\infty \frac{r^{n-1}}{\big((d+1)\gamma_d(\theta) r^d\big)^{\frac{p}{2} - 1}} dr\\
		& \gtrsim \gamma_d(\theta)^{1-\frac{p}{2}} \int_{\frac{\widetilde{C}_\mathbf{Q}}{\gamma_d(\theta)}}^\infty r^{n-1-(\frac{p}{2}-1)d} dr\\
		& \sim_\mathbf{Q} \gamma_d(\theta)^{(\frac{p}{2}-1)(d-1) - n},
	\end{align*}
	where the last line holds due to $n-1-\big(\frac{p}{2}-1\big)d <-1$.
	Plugging this back into (\ref{eq:Lp_identity_moc96}), we get
	\begin{align*}
		\int_{\R^{d+n}} \abs{\widetilde{E}^\mathbf{Q}1(x)}^p dx & \gtrsim_\mathbf{Q} \int_{\mathbb{S}^{n-1}} \gamma_d(\theta)^{(\frac{p}{2}-1)(d-1) - n} d\sigma(\theta)\\
		& = \int_{\mathbb{S}^{n-1}} \abs{\det(\overline{Q}(\theta))}^{(\frac{p}{2}-1)(d-1) - n} d\sigma(\theta).
	\end{align*}
	So $\widetilde{E}^\mathbf{Q}1 \in L^p$ holds for any $p > \frac{2(d+n)}{d}$ will imply that (\ref{eq:CM}) holds for any index $0< \gamma < \frac{d}{n}$, i.e., the (CM) condition holds. And the proof of ``$\Longrightarrow$'' is completed.
	
	To prove ``$\Longleftarrow$'', our strategy is to bound $\int_{\R^{d+n}} \abs{\widetilde{E}^\mathbf{Q}1(x)}^p dx$ above by something like $\int_{\mathbb{S}^{d-1}} \abs{\det(\overline{Q}(\theta))}^{-\gamma} d\sigma(\theta)$. Our starting point is still (\ref{eq:Lp_identity_moc96}). However, additional difficulties arise. One naive way to proceed is to simply bound the inner integral above by $\int_0^\infty r^{d-1}/(1 + \gamma_d(\theta) r^d)^{\frac{p}{2} - 1}$, where the terms of intermediate degrees in the denominator are thrown away. Unfortunately, after some computation everything comes down to the endpoint (CM) condition (\ref{eq:CM_endpoint}), which is in general stronger than the (CM) condition (\ref{eq:CM}), so we are missing the endpoint case for ``$\Longleftarrow$''. 
	
	To remedy this defect, when $n=2$, Mockenhaupt \cite{mockenhaupt96} takes advantage of the fact that for any $\theta$ there is at least one nonzero $\mu_j(\theta)$, which ultimately yields an $\epsilon$-gain of the vanishing order of zeros of the one-variable function $\overline{Q}(\theta)$ in $\theta\in\mathbb{S}^1$. A crucial feature that makes his arguments work is that when $n=2$, the (CM) condition is equivalent to a certain condition on the vanishing order of zeros of $\overline{Q}(\theta)$. At first glance, this seems to be an essential barrier to handle cases when $n\geq 3$, as the zero set of $|\det(\overline{Q}(\theta))|$ can be a complicated algebraic variety, and there are no simple criteria in terms of vanishing orders to measure how singular $|\det(\overline{Q}(\theta))|^{-\gamma}$ is around this zero set. However, thanks to the properties of the eigenvalues $\{\mu_j(\theta)\}_{j=1}^d$ established in Section~\ref{sec2}, we are able to extend Mockenhaupt's argument to the $n\geq 3$ case through some modifications.
	
	The key point here is that we cannot crudely throw away all intermediate terms in the previous naive approach: much more delicate analysis is necessary to handle the endpoint cases. Recall that $\{\Lambda_k(\theta)\}_{k=1}^d$ are $\{|\mu_j(\theta)|\}_{j=1}^d$ in order, and by Corollary~\ref{cor:k_continuity}, they are all continuous functions on $\mathbb{S}^{n-1}$. Also, by Theorem~\ref{thm:U_L_bound_eigen}, Lemma~\ref{lem:alg_char_m}, and the assumption that $\mathfrak{d}_{d,1}(\mathbf{Q}) \geq 1$, we always have
	\begin{equation}\label{asdfgh}
		\Lambda_1(\theta) \sim_\mathbf{Q} 1, \quad \quad \theta \in \mathbb{S}^{n-1}.  
	\end{equation}
	Then for each fixed $\theta\in\mathbb{S}^{n-1}$, we have
	\begin{align*}
		\int_0^\infty \frac{r^{n-1}}{\left[\prod_{j=1}^d (1 + r|\mu_j(\theta)|)\right]^{\frac{p}{2} - 1}} dr & =  \int_0^\infty \frac{r^{n-1}}{\left[\prod_{k=1}^d (1 + r\Lambda_k(\theta))\right]^{\frac{p}{2} - 1}} dr\\
		& = \int_0^\infty \frac{r^{n-1}}{(\gamma_0(\theta) + \gamma_1(\theta) r + \cdots + \gamma_d(\theta) r^d)^{\frac{p}{2} - 1}} dr\\
		& \leq \int_0^\infty \frac{r^{n-1}}{(\gamma_0(\theta) + \gamma_{d-1}(\theta) r^{d-1}  + \gamma_d(\theta) r^d)^{\frac{p}{2} - 1}} dr,
	\end{align*}
	where 
	$$   \gamma_0(\theta) = 1,\quad \gamma_{d-1}(\theta) = \sum_{j=1}^d \prod_{k\neq j}\Lambda_k(\theta),\quad \gamma_d(\theta) = \prod_{k=1}^d \Lambda_k(\theta) = \abs{\det(\overline{Q}(\theta))}.    $$
	As a side remark, without loss of generality, we may always assume $d \geq 2$, which ensures the existence of the $r^{d-1}$ term, as $d=1$ implies $n=1$ and $\mathbf{Q}$ must be a parabola, for which the equivalence $\circled{1}$ trivially holds. By the AM-GM inequality and (\ref{asdfgh}), we have\footnote{Indeed, there is no unique way of doing this, and even choosing to work with the term $\gamma_{d-1}(\theta)$ is not a must. The only important thing is that the profit from $\Lambda_1(\theta)$ is exploited in some sense to gain a little bit of the exponent of $\abs{\det(\overline{Q}(\theta))}$.}
	\begin{align*}
		\gamma_{d-1}(\theta) & \geq \sum_{j=2}^d \prod_{k\neq j}\Lambda_k(\theta)\\
		& \geq (d-1) \left(\prod_{j=2}^d \prod_{k\neq j}\Lambda_k(\theta)\right)^{\frac{1}{d-1}}\\
		& = (d-1)\Lambda_1(\theta)\prod_{k=2}^d \Lambda_k(\theta)^{\frac{d-2}{d-1}}\\
		& \sim_\mathbf{Q} \prod_{k=1}^d \Lambda_k(\theta)^{\frac{d-2}{d-1}}\\
		& = \abs{\det(\overline{Q}(\theta))}^{\frac{d-2}{d-1}},
	\end{align*}
	which implies
	\begin{align*}
		\int_0^\infty \frac{r^{n-1}}{\left[\prod_{j=1}^d (1 + r|\mu_j(\theta)|)\right]^{\frac{p}{2} - 1}} dr & \lesssim_\mathbf{Q} 
		\int_0^\infty \frac{r^{n-1}}{\left(
			1 + \abs{\det(\overline{Q}(\theta))}^{\frac{d-2}{d-1}} r^{d-1}  + \abs{\det(\overline{Q}(\theta))} r^d \right)^{\frac{p}{2} - 1}} dr\\
		& \leq 
		\int_0^\infty \frac{r^{n-1}}{\left(
			1 + \abs{\det(\overline{Q}(\theta))}^{1-\frac{\epsilon}{d-1}} r^{d-\epsilon} \right)^{\frac{p}{2} - 1}} dr,
	\end{align*}
	for any $\epsilon\in (0,1)$. In the second line above, we used a simple fact from calculus that $a + b \geq a^\epsilon b^{1-\epsilon}$ for any $a,b\geq 0$ and $\epsilon\in (0,1)$. Define $r_0$ satisfying
	$$  \abs{\det(\overline{Q}(\theta))}^{1-\frac{\epsilon}{d-1}} r_0^{d-\epsilon} = 1,   $$
	then we can further bound the integral above by
	\begin{align*}
		&\int_0^{r_0} r^{n-1} dr + \abs{\det(\overline{Q}(\theta))}^{-(1-\frac{\epsilon}{d-1})(\frac{p}{2}-1)} \int_{r_0}^\infty r^{n-1-(d-\epsilon)(\frac{p}{2}-1)} dr \\
		\sim\,\, & r_0^n + \abs{\det(\overline{Q}(\theta))}^{-(1-\frac{\epsilon}{d-1})(\frac{p}{2}-1)} r_0^{n-(d-\epsilon)(\frac{p}{2}-1)}\\
		=\,\,& r_0^n = \abs{\det(\overline{Q}(\theta))}^{-\frac{n}{d-\epsilon}(1-\frac{\epsilon}{d-1})},
	\end{align*}
	as long as $\epsilon\in(0,1)$ is taken sufficiently small (depending on $p$) such that $n-1-(d-\epsilon)(\frac{p}{2}-1) < -1$, i.e., $\epsilon < d-n/(\frac{p}{2}-1)$ ($p>\frac{2(d+n)}{d}$ ensures its positivity). 
	
	Plugging the estimate just obtained back into (\ref{eq:Lp_identity_moc96}), we get
	\begin{align*}
		\int_{\R^{d+n}} \abs{\widetilde{E}^\mathbf{Q}1(x)}^p dx \lesssim_\mathbf{Q} \int_{\mathbb{S}^{n-1}} \abs{\det(\overline{Q}(\theta))}^{-\frac{n}{d-\epsilon}(1-\frac{\epsilon}{d-1})} d\sigma(\theta)
	\end{align*}
	for any $0 < \epsilon < d-n/(\frac{p}{2}-1)$. But note that 
	$$\frac{n}{d-\epsilon}\Big(1-\frac{\epsilon}{d-1}\Big) = \frac{n}{d} - \frac{n\epsilon}{d(d-1)(d-\epsilon)} \in \Big(0,\frac{n}{d}\Big),$$
	so by the (CM) condition the right hand side is finite, and $\widetilde{E}^{\mathbf{Q}}1 \in L^p$ for any $p > \frac{2(d+n)}{d}$. Thus we have proved ``$\Longleftarrow$'' in $\circled{1}$.
	
	\phantom{x}
	
    \subsection{Proof of \texorpdfstring{$\protect\circled{2}$}{②}}\label{subsec:2_proof}\phantom{x}

    Note that the best Stein-Tomas in Theorem~\ref{thm:relation_diagram} is formulated in terms of $E^\mathbf{Q}$ instead of $\widetilde{E}^\mathbf{Q}$ or $\widetilde{E}_\varphi^\mathbf{Q}$ (as defined in (\ref{eq:gaussian_integral}) and (\ref{eq:oscillatory_integral})), although the latter two are the most commonly used in the literature of Stein-Tomas-type inequalities. However, by Proposition~\ref{prop:restriction_equiv}, estimates for all these versions of Fourier extension operators are indeed equivalent to each other. This allows us to cite results in the literature without worrying about compatibility.
    
	For the proof of 
	$\circled{2}$, ``$\Longrightarrow$'' is due to Corollary~2.12 in \cite{mockenhaupt96}. Moreover, Theorem~2.11 in \cite{mockenhaupt96} also shows that the endpoint (CM) condition implies the endpoint best Stein-Tomas-type  inequality (i.e., $L^2\rightarrow L^{\frac{2(d+2n)}{d}}$). The main technique is complex interpolation of analytic family of operators. As a side remark, in Conjecture~4.1.2 in \cite{banner02}, it is conjectured that the (CM) condition together with the assumption that $n < d$ implies the endpoint best Stein-Tomas-type  inequality. By Theorem~2.14 in \cite{mockenhaupt96}, this conjecture is true when $n=2$.
	
	As for the reverse implication ``$\Longleftarrow$'', when $n=2$, Proposition~3.1 in \cite{christ82} proved\footnote{Although the statement of the proposition itself only says that the best ``endpoint'' Stein-Tomas-type inequality (i.e., $L^2 \rightarrow L^{\frac{2(d+4)}{d}}$) implies the (CM) condition, which requires slightly stronger assumption, its proof actually works well enough to provide what we want. As a side remark, something more general is proved: There is no need to confine ourselves to quadratic forms, and the result actually holds for any smooth manifolds. In this setting, by using Hessian matrices, the (CM) condition can also be defined in a similar way as in Definition~\ref{def:CM}.} that the best Stein-Tomas-type 
	inequality (i.e., $L^2 \rightarrow L^{\frac{2(d+4)}{d}+}$) implies the (CM) condition. The approach is based on deriving certain local ``normal forms'' for manifolds, followed by a Knapp-type infinitesimal homogeneity argument. Proving the existence of the normal forms (Proposition~3.2 in \cite{christ82}) is not an easy job and requires heavy manipulation of change of coordinates. There are other ways of proving ``$\Longleftarrow$'' when $n=2$. As we will soon see, $\circled{6}$ combined with ``$\Longleftarrow$'' of $\circled{7}$ when $n=2$ provides an alternative proof (although still not easy, see explanations in the $\circled{7}$ part later on), which relies on geometric invariant theory.
	
	However, when $n \geq 3$, ``$\Longleftarrow$'' is not true in general, even if we strengthen our assumption to endpoint best Stein-Tomas-type inequalities. Here we provide several counterexamples:
	\begin{itemize}
		\item When $d=5$ and $n=3$, 
		$$\mathbf{Q}(\xi) = (\xi_1^2 + \xi_3^2 + \xi_5^2, 2(\xi_1\xi_2 + \xi_3\xi_4), \xi_2^2 + \xi_4^2 +\xi_5^2)$$ 
		satisfies the endpoint best Stein-Tomas-type inequality, see Theorem~2.18 in \cite{mockenhaupt96} for the proof. However, it does not satisfy the (CM) condition, because Example~2.10 in \cite{mockenhaupt96} tells us that 
		$$\int_{\mathbb{S}^{2}}  |\det(\overline{Q}(\theta))|^{-\gamma}d\sigma(\theta) = \infty, \quad \quad \forall ~\gamma \geq \frac{1}{2},$$
		while for the (CM) condition to hold we need 
		$$\int_{\mathbb{S}^{2}}  |\det(\overline{Q}(\theta))|^{-\gamma}d\sigma(\theta) < \infty , \quad \quad \forall~\gamma < \frac{3}{5}.$$
		Another way to see the failure of the (CM) condition is through the equivalence in $\circled{1}$: Theorem~2.18 in \cite{mockenhaupt96} also tells us that $\widetilde{E}^\mathbf{Q}1 \not\in L^p$ whenever $p \leq \frac{10}{3}$, while for the (CM) condition to hold we need $\widetilde{E}^\mathbf{Q}1 \in L^p$ for all $p > \frac{16}{5}$.
		
		\item When $d=2$ and $n=3$, 
		$$\mathbf{Q}(\xi) = (\xi_1^2, 2\xi_1\xi_2, \xi_2^2)$$ 
		satisfies the endpoint best Stein-Tomas-type inequality, see Proposition~16.2 in \cite{christ82} or Theorem~2.15 in \cite{mockenhaupt96} for the proof. However, it does not satisfy the (CM) condition, because  
		$$\int_{\mathbb{S}^{2}}  |\det(\overline{Q}(\theta))|^{-\gamma}d\sigma(\theta) = \int_{\mathbb{S}^{2}}  |4(\theta_1\theta_3-\theta_2^2)|^{-\gamma}d\sigma(\theta) = \infty,\quad \quad \forall~\gamma \geq 1,$$
		while for the (CM) condition to hold we need 
		$$\int_{\mathbb{S}^{2}}  |\det(\overline{Q}(\theta))|^{-\gamma}d\sigma(\theta) < \infty, \quad \quad \forall~\gamma < \frac{3}{2}.$$
		Another way to see the failure of the (CM) condition is through the equivalence in $\circled{1}$: Theorem~2.15 in \cite{mockenhaupt96} also tells us that $\widetilde{E}^\mathbf{Q}1 \in L^p$ if and only if $p > 6$, while for the (CM) condition to hold we need $\widetilde{E}^\mathbf{Q}1 \in L^p$ if and only if $p > 5$.
		
		\item When $d=3$ and $n=5$, 
		$$\mathbf{Q}(\xi) = (2\xi_1\xi_2, 2\xi_1\xi_3, 2\xi_2\xi_3, \xi_2^2 - \xi_1^2, \xi_3^2 - \xi_1^2)$$ 
		satisfies the endpoint best Stein-Tomas-type inequality, see Proposition~6.1.2 in \cite{banner02} for the proof. However, it does not satisfy the (CM) condition, because in the proof of Proposition~6.1.2 in \cite{banner02} it is also shown that $\{\theta\in\mathbb{S}^4 : \overline{Q}(\theta)=0\}$ is a smooth submanifold of codimension one in $\mathbb{S}^4$, which implies 
		$$\int_{\mathbb{S}^{4}}  |\det(\overline{Q}(\theta))|^{-\gamma}d\sigma(\theta) = \infty,\quad \quad \forall~\gamma \geq 1,$$
		while for the (CM) condition to hold we need 
		$$\int_{\mathbb{S}^{4}}  |\det(\overline{Q}(\theta))|^{-\gamma}d\sigma(\theta) < \infty,\quad \quad \forall~\gamma < \frac{5}{3}.$$
		One may consult Corollary~2.6 in \cite{mockenhaupt96} for  more general discussions in this spirit.
	\end{itemize}
	
	\phantom{x}
	
    \subsection{Proof of \texorpdfstring{$\protect\circled{3}$}{③}}\label{subsec:3_proof}\phantom{x}
	
	By the definition of good manifolds, in what follows we always restrict ourselves to diagonal forms $\mathbf{Q}(\xi) = (a_1\xi_1^2 + \cdots + a_d\xi_d^2,b_1\xi_1^2 + \cdots + b_d\xi_d^2)$ of codimension $2$. When $d=2$, the ``$\Longrightarrow$'' is actually ``$\Longleftrightarrow$'', because in this case the only good manifold is $\mathbf{Q} \equiv (\xi_1^2, \xi_2^2)$, and so is the only strongly nondegenerate diagonal form.
	
	For any even $d \geq 4$, we claim that ``$\Longrightarrow$'' holds, while ``$\Longleftarrow$'' is not true in general. To verify that a good manifold $\mathbf{Q}$ is strongly nondegenerate, we need to check that $\mathfrak{d}_{d-m,n'}(\mathbf{Q}) \geq \frac{n'}{2}d - m$ for every $m$ with $0\leq m \leq d$ and every $n'$ with $0 \leq n' \leq 2$. The case $n'=0$ is trivial by noting $\mathfrak{d}_{d-m,0}(\mathbf{Q}) = 0$. If $n'=1$, since $\mathfrak{d}_{d-m,n'}(\mathbf{Q}) \geq 0$, without loss of generality we may assume $0 \leq m < \frac{d}{2}$, which implies $m \leq \frac{d}{2} - 1$ (note that $\frac{d}{2} \in \N^+$ for even $d$). By Lemma~5.1 in \cite{ggo23}, we have $\mathfrak{d}_{d-m,1}(\mathbf{Q}) \geq d-2m-1$, which is further bounded below by $\frac{d}{2}-m$, as $m \leq \frac{d}{2} - 1$. If $n'=2$, then by Lemma~5.1 in \cite{ggo23}, we directly have $\mathfrak{d}_{d-m,2}(\mathbf{Q}) \geq d-m$ for all $0\leq m \leq d$, as desired. Therefore, ``$\Longrightarrow$'' holds.
	
	To show that ``$\Longleftarrow$'' is not true in general, one may consider 
	$$\mathbf{Q}(\xi) = (\xi_1^2 + \cdots + \xi_{\frac{d}{2}}^2, \xi_{\frac{d}{2}+1}^2 + \cdots + \xi_d^2),$$
	which is clearly not good. To check strongly nondegeneracy, similarly to the discussions in the last paragraph, it suffices to show $\mathfrak{d}_{d-m,1}(\mathbf{Q}) \geq \frac{d}{2} - m$ and $\mathfrak{d}_{d-m,2}(\mathbf{Q}) \geq d-m$ for every $m$ with $0\leq m \leq d$. For the former, observe that 
	\begin{equation}\label{123321123321}
		\mathfrak{d}_{d-m,1}(\mathbf{Q}) = \inf_{\theta = (s,t)\in \mathbb{S}^1} \mathfrak{d}_{d-m,1}(\overline{Q}(\theta)),   
	\end{equation}
	where 
	$$\overline{Q}(\theta) = s(\xi_1^2 + \cdots + \xi_{\frac{d}{2}}^2) + t(\xi_{\frac{d}{2}+1}^2 + \cdots + \xi_d^2).$$
	If $s=0$ or $t=0$, then it is easy to see that $\mathfrak{d}_{d-m,1}(\overline{Q}(\theta)) \geq \frac{d}{2} - m$ (much similar to the paraboloid case in Corollary~\ref{cor1}), as desired. So it remains to consider the case when $s\neq 0$ and $t\neq 0$. If $s$ and $t$ have the same sign, then $\overline{Q}(\theta) \equiv \xi_1^2 + \cdots + \xi_d^2$, which by the paraboloid case in Corollary~\ref{cor1} satisfies $\mathfrak{d}_{d-m,1}(\overline{Q}(\theta)) \geq d-m > \frac{d}{2}-m$. If $s$ and $t$ have different signs, then by the hyperbolic paraboloid case in Corollary~\ref{cor2} (with $m=\frac{d}{2}$), we have $\mathfrak{d}_{d-m,1}(\overline{Q}(\theta)) = d-2m \geq \frac{d}{2}-m$ (note that without loss of generality we may assume $0\leq m\leq \frac{d}{2}$), as desired. Thus in any case we have verified $\mathfrak{d}_{d-m,1}(\mathbf{Q}) \geq \frac{d}{2}-m$. For the latter, observe that the proof of Lemma~5.1 in \cite{ggo23} carries through without any modification. To be more precise, suppose to the contrary $\mathfrak{d}_{d-m,2}(\mathbf{Q}) < d-m$, then they showed that there exists a nonzero vector $\vec{w} = (w_1,\dots,w_d)\in\R^d$ such that $\sum_{i=1}^d a_i (w_i)^2=0$ and $\sum_{i=1}^d b_i (w_i)^2=0$, where $a_i$'s and $b_i$'s are the coefficients in the diagonal form. But in our case this is exactly 
	$$\sum_{i=1}^{\frac{d}{2}} (w_i)^2=0, \quad \sum_{\frac{d}{2}+1}^d (w_i)^2=0,$$
	which implies $\vec{w}=0$, a contradiction. So the proof of strongly nondegeneracy is completed.
	
	For any odd $d \geq 3$, we claim that ``$\Longrightarrow$'' is not true in general, while ``$\Longleftarrow$'' is vacuous. To see that ``$\Longrightarrow$'' fails in general, one may consider 
	$$\mathbf{Q}(\xi) = \Big(\sum_{i=1}^d \xi_i^2, \sum_{i=1}^{\frac{d-1}{2}} i \xi_i^2 - \sum_{i=\frac{d+3}{2}}^{d} (d + 1 - i) \xi_{i}^2\Big).$$
	Note that the coefficient matrix of this $\mathbf{Q}$ is
	\begin{align*}
		\begin{pmatrix}
			a_1 & \cdots & a_d\\
			b_1 & \cdots & b_d
		\end{pmatrix}
		=
		\begin{pmatrix}
			1 & \cdots & 1 & 1 & 1 & \cdots & 1\\
			1 & \cdots & \frac{d-1}{2} & 0 & -\frac{d-1}{2} & \cdots & -1
		\end{pmatrix},
	\end{align*}
	we easily see that it is good. But by considering (\ref{123321123321}) with $m=\frac{d-1}{2}$, we have $$\mathfrak{d}_{\frac{d+1}{2},1}(\mathbf{Q}) \leq \mathfrak{d}_{\frac{d+1}{2},1}(\overline{Q}((0,1))) = \mathfrak{d}_{\frac{d+1}{2},1} \Big(\sum_{i=1}^{\frac{d-1}{2}} i \xi_i^2 - \sum_{i=\frac{d+3}{2}}^{d} (d + 1 - i) \xi_{i}^2 \Big) =0,$$
	which is much similar to the hyperbolic paraboloid case in Corollary~\ref{cor2}. So $\mathfrak{d}_{\frac{d+1}{2},1}(\mathbf{Q}) \geq \frac{1}{2}$ does not hold, and $\mathbf{Q}$ is not strongly nondegenerate.
	
	To show that ``$\Longleftarrow$'' is vacuous, we only need to prove:
	\begin{proposition}\label{prop:diag_not_str}
		For any odd $d\geq 3$, all diagonal forms $\mathbf{Q}(\xi) = (Q_1(\xi),Q_2(\xi)) \coloneqq (a_1\xi_1^2 + \cdots + a_d\xi_d^2,b_1\xi_1^2 + \cdots + b_d\xi_d^2)$ cannot be strongly nondegenerate.
	\end{proposition}
	\begin{proof}
		The key idea of the proof is to find a linear combination of $Q_1$ and $Q_2$ that is the most degenerate by tracing the signs of the coefficients. Firstly, as $\mathbf{Q}$ does not omit any variables, by perturbing $Q_1$ using $Q_2$, we may without loss of generality assume that $a_i\neq 0$ for any $1\leq i \leq d$. Let $n_+$ and $n_-$ be the numbers of positive and negative numbers in $\{a_i\}_{i=1}^d$ respectively, then $n_+ + n_- =d$. Also, note that $d$ is odd implies $n_+ \neq n_-$, so we may further assume that $n_+ > n_-$ (otherwise we replace $Q_1$ with $-Q_1$). Next, let $n_+(\lambda)$, $n_-(\lambda)$, and $n_0(\lambda)$ be the numbers of positive, negative, and zero numbers in $\{b_i+\lambda a_i\}_{i=1}^d$ (the coefficients of $Q_2 + \lambda Q_1$), respectively. Note that for $\lambda \gg 1$, the signature status of $\{b_i+\lambda a_i\}_{i=1}^d$ is the same as that of $\{a_i\}_{i=1}^d$, which means 
		$$   n_+(\lambda_L) = n_+, \quad n_-(\lambda_L) = n_-, \quad  n_0(\lambda_L)=0   $$
		for some large $\lambda_L$. Similarly, for $\lambda \ll -1$, the signature status of $\{b_i+\lambda a_i\}_{i=1}^d$ is the opposite of that of $\{a_i\}_{i=1}^d$, which means 
		$$  n_+(\lambda_S) = n_-, \quad   n_-(\lambda_S) = n_+, \quad n_0(\lambda_S)=0 $$
		for some small $\lambda_S$. 
		
		Roughly speaking, our goal is to find some $\lambda \in [\lambda_S,\lambda_L]$ such that $|n_+ - n_-|$ is as small as possible, which is believed to be the most degenerate scenario. To achieve this, observe that there can be at most finitely many $\lambda \in [\lambda_S,\lambda_L]$ such that $n_0(\lambda) > 0$, because any such $\lambda$ must be one of $\{-b_i/a_i\}_{i=1}^d$. Besides, by the intermediate value principle, for any subinterval $[\lambda_s,\lambda_l] \subset [\lambda_S,\lambda_L]$ satisfying $n_+(\lambda_s) \neq n_+(\lambda_l)$ or $n_-(\lambda_s) \neq n_-(\lambda_l)$, there must be at least one $\lambda \in [\lambda_s,\lambda_l]$ such that $n_0(\lambda) > 0$. In particular, recall that 
		$$  n_+(\lambda_S) = n_- < n_+ = n_+(\lambda_L), \quad  n_-(\lambda_S) = n_+ > n_- =  n_-(\lambda_L),   $$
		thus by taking $[\lambda_s,\lambda_l] = [\lambda_S,\lambda_L]$, there must be at least one $\lambda \in [\lambda_S,\lambda_L]$ such that $n_0(\lambda) > 0$. By our assumption that $n_0(\lambda_S) = n_0(\lambda_L)=0$, all such $\lambda$'s cannot be equal to $\lambda_S$ and $\lambda_L$, and so we can arrange them in order as 
		$$  (\lambda_0 \eqqcolon \lambda_S <) \, \lambda_1 < \cdots < \lambda_k \, (< \lambda_L \coloneqq \lambda_{k+1})   $$
		for some $k\in\N^+$. Note that for any $j = 0, \dots, k$, we have $n_0(\lambda) \equiv 0$ over $(\lambda_j, \lambda_{j+1})$, so applying our previous arguments to any subinterval $[\lambda_s,\lambda_l] \subset (\lambda_j, \lambda_{j+1})$ yields $n_+(\lambda_s) = n_+(\lambda_l)$ and $n_-(\lambda_s) = n_-(\lambda_l)$. In other words, on each $(\lambda_j,\lambda_{j+1})$ ($j=0,\dots, k$), not only $n_0(\lambda) \equiv0$, but also both $n_+(\lambda)$ and $n_-(\lambda)$ are constant, so we may define $n_{j,+}$ and $n_{j,-}$ to be these constant values of $n_+(\lambda)$ and $n_-(\lambda)$, respectively. Note that for each $j = 0,\dots,k$, $n_{j,+} + n_{j,-} = d$ with $d$ odd, so we have either $n_{j,+} > n_{j,-}$ or $n_{j,+} < n_{j,-}$. However, as $n_0(\lambda_S) = n_0(\lambda_L)=0$, by continuity at the two extreme endpoints $\lambda_S$ and $\lambda_L$, we know that 
		$$   n_{0,+} = n_+(\lambda_S) = n_-, \quad n_{0,-} = n_-(\lambda_S) = n_+,    $$
		$$   n_{k,+} = n_+(\lambda_L) = n_+, \quad  n_{k,-} = n_-(\lambda_L) = n_- ,$$
		which by $n_+>n_-$ implies 
		$$  n_{0,+} < n_{0,-},\quad n_{k,+} > n_{k,-}.  $$
		Therefore, by the intermediate value principle again, there must exist some $j_0 \in \{1,\dots,k+1\}$ such that 
		$$  n_{j_0-1,+} < n_{j_0-1,-}, \quad  n_{j_0,+} > n_{j_0,-}. $$
		Then $\lambda = \lambda_{j_0} \in [\lambda_S,\lambda_L]$ is the critical moment at which $Q_2 + \lambda Q_1$ is expected to be the most degenerate. In what follows we will show that this indeed forces 
		$$   \mathfrak{d}_{d-m,1}(Q_2 + \lambda_{j_0} Q_1) < \frac{d}{2} - m $$
		for some $0\leq m\leq d$, and thus $\mathfrak{d}_{d-m,1}(\mathbf{Q}) < \frac{d}{2} - m$ by definition, which implies that $\mathbf{Q}$ cannot be strongly nondegenerate as desired.
		
		For convenience, let $\widetilde{Q} \coloneqq Q_2 + \lambda_{j_0} Q_1$. By the definition of $\lambda_{j_0}$, we have $n_0(\lambda_{j_0}) > 0$. Also, without loss of generality, we may assume that $n_+(\lambda_{j_0}) \leq n_-(\lambda_{j_0})$ (the other case when $n_+(\lambda_{j_0}) \geq n_-(\lambda_{j_0})$ can be dealt with in a completely symmetric way). When $\lambda$ goes down from slightly above $\lambda_{j_0}$ to touching $\lambda_{j_0}$, each of $\{b_i+\lambda a_i\}_{i=1}^d$ varies monotonically, and we wish to track how their behavior on $\lambda_{j_0}$ is related to that slightly above $\lambda_{j_0}$. Define 
		$$ u \coloneqq \#\{1 \leq i \leq d: b_i+\lambda_{j_0} a_i = 0, a_i > 0\}   $$
		and 
		$$   v \coloneqq \#\{1 \leq i \leq d: b_i+\lambda_{j_0} a_i = 0, a_i < 0\} .$$
		As $a_i \neq 0$ for any $1\leq i \leq d$, we have $u+v = n_0(\lambda_{j_0})$, which further implies $n_+(\lambda_{j_0}) + n_-(\lambda_{j_0}) + u + v = d$. Moreover, it is not hard to observe that $n_+(\lambda_{j_0}) + u = n_{j_0,+}$ and $n_-(\lambda_{j_0}) + v = n_{j_0,-}$, which implies $n_+(\lambda_{j_0}) + u > n_-(\lambda_{j_0}) + v$ as $n_{j_0,+} > n_{j_0,-}$.  Therefore, we can summarize all the relations between the parameters as follows:
		\[
		\begin{cases}
			0 \leq n_-(\lambda_{j_0}) - n_+(\lambda_{j_0}) \leq u-v-1,\\
			n_-(\lambda_{j_0}) + n_+(\lambda_{j_0}) = d-u-v.
		\end{cases}
		\]
		Thus 
		$$n_+(\lambda_{j_0}) \geq \frac{(d-u-v) - (u-v-1)}{2} = \frac{d+1}{2} - u.$$ 
		
		We claim $\mathfrak{d}_{\frac{d+1}{2},1}(\widetilde{Q}) = 0$. First assume $u\leq \frac{d+1}{2}$. Then there are at least $\frac{d+1}{2} - u$ disjoint pairs of positive and negative coefficient numbers in $\{b_i+\lambda_{j_0}a_i\}_{i=1}^d$, which implies 
		$$  \mathfrak{d}_{d-(\frac{d+1}{2} - u), 1}(\widetilde{Q}) \leq (d-u-v) - 2\Big(\frac{d+1}{2} - u\Big)= u - v - 1, $$ much similar to the hyperbolic paraboloid case in Corollary~\ref{cor2}. Thus we have 
		$$\mathfrak{d}_{\frac{d+1}{2},1}(\widetilde{Q}) \leq \max\Big\{0, \mathfrak{d}_{\frac{d-1}{2}+u,1}(\widetilde{Q}) - (u-1)\Big\} \leq \max\{0,-v\} = 0.$$
		It remains to consider $u > \frac{d+1}{2}$. In this case, $n_0(\lambda_{j_0}) \geq u \geq \frac{d+3}{2}$ ($d$ odd), which implies 
		$$\mathfrak{d}_{d,1}(\widetilde{Q}) \leq d - \frac{d+3}{2} = \frac{d-3}{2}.$$
		Thus we have 
		$$\mathfrak{d}_{\frac{d+1}{2},1}(\widetilde{Q}) \leq \max\Big\{0,\mathfrak{d}_{d,1}(\widetilde{Q}) - \frac{d-1}{2}\Big\} = 0.$$ 
		Therefore, in any case, we get $\mathfrak{d}_{\frac{d+1}{2},1}(\widetilde{Q}) = 0$, and so $\mathfrak{d}_{\frac{d+1}{2},1}(\mathbf{Q}) = 0$. However, for $\mathbf{Q}$ to be strongly nondegenerate, we need $\mathfrak{d}_{\frac{d+1}{2},1}(\mathbf{Q}) \geq \frac{1}{2}$. So the proof is completed.
	\end{proof}
	
	\begin{remark}\label{rmk:d_even_partial}
		In fact, via a similar strategy, we can show that $\mathfrak{d}_{\frac{d}{2},1}(\mathbf{Q}) = 0$ when $d$ is even. When we run the proof of Proposition~{\rm\ref{prop:diag_not_str}} for $n$ even, the major problem is that we do not necessarily have $n_+ \neq n_-$ at the beginning. However, if $n_+ = n_-$, then by $n_+ + n_- = d$, we have $n_+ = n_- = \frac{d}{2}$, which immediately implies $\mathfrak{d}_{\frac{d}{2},1}(\mathbf{Q}) = 0$. Therefore, we may always assume $n_+ \neq n_-$. Similarly, in the process of locating $\lambda_0$ by the intermediate value principle, we can always assume $n_{j,+} \neq n_{j,-}$, and all arguments carry through. The only difference is the computation at the final stage. Since $d$ is even, we have $n_+(\lambda_{j_0}) \geq \frac{d+2}{2} - u$, which implies 
		$$  \mathfrak{d}_{d-(\frac{d+2}{2} - u), 1}(\widetilde{Q}) \leq (d-u-v) - 2\Big(\frac{d+2}{2} - u\Big)= u - v - 2.   $$
		Thus we have 
		$$\mathfrak{d}_{\frac{d}{2},1}(\widetilde{Q}) \leq \max\Big\{0, \mathfrak{d}_{\frac{d-2}{2}+u,1}(\widetilde{Q}) - (u-1)\Big\} \leq \max\{0,-v-1\} = 0.$$
		It remains to consider $u > \frac{d+2}{2}$. In this case, $n_0(\lambda_{j_0}) \geq u \geq \frac{d+4}{2}$ {\rm(}$d$ even{\rm)}, which implies 
		$$\mathfrak{d}_{d,1}(\widetilde{Q}) \leq d - \frac{d+4}{2} = \frac{d-4}{2}.$$
		Thus we have 
		$$\mathfrak{d}_{\frac{d}{2},1}(\widetilde{Q}) \leq \max\Big\{0,\mathfrak{d}_{d,1}(\widetilde{Q}) - \frac{d}{2}\Big\} = 0.$$ 
		Therefore, in any case, we get $\mathfrak{d}_{\frac{d}{2},1}(\widetilde{Q}) = 0$, and so $\mathfrak{d}_{\frac{d}{2},1}(\mathbf{Q}) = 0$ as desired.
	\end{remark}
	
	\phantom{x}
	
    \subsection{Proof of \texorpdfstring{$\protect\circled{4}$}{④}}\label{subsec:4_proof}\phantom{x}
	
	This is due to Corollary~1.3 in \cite{gozzk23}.
	
	\phantom{x}
	
    \subsection{Proof of \texorpdfstring{$\protect\circled{5}$}{⑤}}\label{subsec:5_proof}\phantom{x}
	
	If one only cares about local Stein-Tomas-type inequalities, then $\circled{5}$ is a kind of folklore but is hardly written down in detail in the literature. When $n=1$, the only strongly nondegenerate quadratic form is the paraboloid, for which ``$\Downarrow$'' has been proved by Pinney \cite{pinney2021}. Here we will focus on all strongly nondegenerate quadratic forms, but the work carries through without any difficulty. 
	
	However, to get global Stein-Tomas-type inequalities, we will need an epsilon removal lemma for general higher-codimensional manifolds. There has been a version of it in the literature:
    
    \begin{lemma}[Epsilon removal, Theorem~4.1 in \cite{go22}]\label{lem:epsilon_removal}
        Let $q\geq p\geq 2$, $\varphi \in C_c^\infty(\R^d)$ with $\varphi\not\equiv0$. Assume for some $\gamma > 0$, we have the uniform Fourier decay
        \begin{align*}
            \abs{\widetilde{E}_\varphi^\psi 1 (x)} \lesssim_{\psi, \gamma} (1+|x|)^{-\gamma}, \quad \quad x\in \R^{d+n}.
        \end{align*} 
        Also, assume for every $0< \alpha \ll 1$, we have the local extension estimate
        \begin{align*}
            \norm{\widetilde{E}_\varphi^\psi f}_{L^q(B_R)} \lesssim_\alpha R^\alpha \norm{f}_{L^p(\R^d)}
        \end{align*}
        for any $R\geq 1$, any ball $B_R$ in $\R^{d+n}$, and any $f\in L^p(\R^d)$. Then for any $r > q$, we have the global extension estimate 
        \begin{align*}
            \norm{\widetilde{E}_\varphi^\psi f}_{L^r(\R^{d+n})} \lesssim \norm{f}_{L^p(\R^d)}
        \end{align*}
        for any $f\in L^p(\R^d)$.
    \end{lemma}

    \begin{remark}
        It is worth mentioning that although Guo and Oh \cite{go22} only state the theorem when $\varphi \in C_c^\infty(\R^d)$ with $\supp\varphi \subset [-1, 2]^d$ and $\varphi\equiv 1$ on $[0,1]^d$, the same proof actually works for general $\varphi \in C_c^\infty(\R^d)$ with $\varphi \not\equiv0$ without any changes.
    \end{remark}
    
    Unfortunately, Lemma~\ref{lem:epsilon_removal} only works for a  smoothed version of extension operators, and it is unclear to what extent it can be applied to general quadratic manifolds. Our main contribution here is to obtain a version of the epsilon removal lemma (Corollary~\ref{cor:quad_ep_removal}) that works directly for all $E^\mathbf{Q}$. Since in the recent study of Fourier restriction people tend to work with  $E^\mathbf{Q}$ instead of its smoothed version, our lemma should be more convenient to apply in many cases. Besides, with the help of the uniform Fourier decay estimates (Corollary~\ref{cor:recover_banner}), we can actually conclude that all nontrivial quadratic manifolds enjoy epsilon removal:
    \begin{corollary}\label{cor:quad_ep_removal}
        Let $q \geq p \geq 2$. Suppose $\mathbf{Q}$ is an $n$-tuple of linearly independent real quadratic forms in $d$ variables {\rm(}i.e., $\mathfrak{d}_{d,1}(\mathbf{Q}) \geq 1${\rm)}. If for every $0<\alpha \ll 1$, we have the local extension estimate
        \begin{align*}
            \norm{E^\mathbf{Q} f}_{L^q(B_R)} \lesssim_\alpha R^\alpha \norm{f}_{L^p([0,1]^d)}
        \end{align*}
        for any $R\geq 1$, any ball $B_R$ in $\R^{d+n}$, and any $f\in L^p([0,1]^d)$. Then for any $r > q$, we have the global extension estimate 
        \begin{align*}
            \norm{E^\mathbf{Q} f}_{L^r(\R^{d+n})} \lesssim \norm{f}_{L^p([0,1]^d)}
        \end{align*}
        for any $f\in L^p([0,1]^d)$.
    \end{corollary}
    \begin{proof}
        Fix any $\varphi \in C_c^\infty(\R^d)$ with $\supp\varphi \subset [-1, 1]^d$, $|\varphi|\leq 1$, and $\varphi\equiv 1$ on $[-\frac{1}{2},\frac{1}{2}]^d$. Then by $\widetilde{E}_\varphi^\mathbf{Q} f = E^\mathbf{Q} (f\varphi)$ and the assumed local extension estimate for $E^\mathbf{Q}$, for every  $0<\alpha\ll 1$, we have
        \begin{align*}
            \norm{\widetilde{E}_\varphi^\mathbf{Q} f}_{L^q(B_R)} = \norm{E^\mathbf{Q} (f\varphi)}_{L^q(B_R)}
            \lesssim \norm{f\varphi}_{L^p([0,1]^d)} \leq \norm{f}_{L^p([0,1]^d)}
        \end{align*}
        for each $R\geq 1$, each ball $B_R$ in $\R^{d+n}$, and each $f\in L^p(\R^d)$.
        
        On the other hand, by Corollary~\ref{cor:recover_banner}, we have the uniform Fourier decay estimate
        \begin{align*}
            \abs{\widetilde{E}_\varphi^\mathbf{Q}1 (x)} \lesssim_\mathbf{Q} (1+|x|)^{-\frac{\mathfrak{d}_{d,1}(\mathbf{Q})}{2}}, \quad \quad x\in \R^{d+n}.
        \end{align*}
        By our assumption that $\mathfrak{d}_{d,1}(\mathbf{Q}) \geq 1$, this implies
        \begin{align*}
            \abs{\widetilde{E}_\varphi^\mathbf{Q}1 (x)} \lesssim_\mathbf{Q} (1+|x|)^{-\frac{1}{2}}, \quad \quad x\in \R^{d+n}.
        \end{align*}
        
        Thus the assumptions in Lemma~\ref{lem:epsilon_removal} are satisfied, and by conclusions there, for any $r > q$, we have the global estimate 
        \begin{align*}
            \norm{\widetilde{E}_\varphi^\mathbf{Q} f}_{L^r(\R^{d+n})} \lesssim \norm{f}_{L^p(\R^d)}, \qquad \forall \,f\in L^p(\R^d).
        \end{align*}
        Finally, by the equivalence $(i) \Leftrightarrow (iii)$ proved in Proposition~\ref{prop:restriction_equiv}, for any $r > q$, we have the global estimate
        \begin{align*}
            \norm{E^\mathbf{Q} f}_{L^r(\R^{d+n})} \lesssim \norm{f}_{L^p([0,1]^d)}, \qquad \forall\, f\in L^p([0,1]^d),
        \end{align*}
        as desired.
    \end{proof}	
	
    \begin{remark}
        As a historical remark, the first type of the epsilon removal lemma is due to Tao \cite{tao99}, which deals with the case when $S_\psi$ is a sphere and $p=q$. Since then, his revolutionary result has been extended to various other settings. One may also consult Section~{\rm3} of \cite{pinney2021} for another exposition of epsilon removal, which sticks to the extension formulation {\rm(}i.e., does not involve any restriction formulation originally used by Tao{\rm)}. Moreover, we point out that in the proof of the epsilon removal lemma, it is important that we have a ``uniform'' Fourier decay estimate in hand: For estimates lacking uniformity, such as that in Lemma~{\rm2.1.1} of \cite{banner02}, the proof will not go through any more.
    \end{remark}

	Now we are ready to begin the proof of $\circled{5}$. 	By Lemma~\ref{lem:thickening}, to prove 
	$$\norm{E^\mathbf{Q} f}_{L^{p_0}(B_{R})} \lesssim R^{\alpha} \norm{f}_{L^2([0,1]^d)},$$
	it suffices to prove 
	$$\big\|\widehat{F}\big\|_{L^{p_0}(B_{R})} \lesssim R^{\alpha-\frac{n}{2}} \norm{F}_{L^2(N_{R^{-1}}(S_{\mathbf{Q}}))}.$$
	When $\mathbf{Q}$ is strongly nondegenerate, by Corollary~1.3 in \cite{gozzk23}, the decoupling inequality (\ref{eq:decoupling_neighborhood}) holds for $q=2$, $2 \leq p < \infty$, and 
	$$\Gamma_p^d(\mathbf{Q}) = \max\Big\{0, d\big(\frac{1}{2}-\frac{1}{p}\big)-\frac{2n}{p}\Big\}.$$
	In particular, by taking $p = p_0 \coloneqq \frac{2(d+2n)}{d}$, we know that\footnote{Here we simply replace $\R^{d+n}$ by the smaller $B_{R}$ on the left hand side of (\ref{eq:decoupling_neighborhood}). Besides, for simplicity, we will omit cumbersome explanations of inequalities from now on.}
	\begin{equation}\label{eq:decoupling_str_nondegenerate}
		\| \widehat{F}\|_{L^{p_0}(B_{R})} \lesssim_\epsilon R^{\epsilon} \Big( \sum_\tau \big\|\widehat{F_\tau}\big\|_{L^{p_0}(\R^{d+n})}^2  \Big)^{\frac{1}{2}}.
	\end{equation} 
    Using H\"older's inequality, we have for each $\tau$ that 
	\begin{align*}
		\big\|\widehat{F_\tau}\big\|_{L^{p_0}(\R^{d+n})} &\leq \big\|\widehat{F_\tau}\big\|_{L^{\infty}(\R^{d+n})}^{\frac{2n}{d+2n}} \big\|\widehat{F_\tau}\big\|_{L^{2}(\R^{d+n})}^{\frac{d}{d+2n}}\\
		& \leq \norm{F_\tau}_{L^1(N_{R^{-1}}(S_{\mathbf{Q}}|_\tau))}^{\frac{2n}{d+2n}}
		\norm{F_\tau}_{L^2(N_{R^{-1}}(S_{\mathbf{Q}}|_\tau))}^{\frac{d}{d+2n}}\\
		& \leq |N_{R^{-1}}(S_{\mathbf{Q}}|_\tau)|^{\frac{n}{d+2n}} \norm{F_\tau}_{L^2(N_{R^{-1}}(S_{\mathbf{Q}}|_\tau))}\\
		& \sim R^{-\frac{n}{2}} \norm{F_\tau}_{L^2(N_{R^{-1}}(S_{\mathbf{Q}}|_\tau))}.
	\end{align*}
	Using the above inequality to bound the right hand side of (\ref{eq:decoupling_str_nondegenerate}), we have
	\begin{align*}
		\big\|\widehat{F}\big\|_{L^{p_0}(B_R)} & \lesssim_\epsilon R^{\epsilon} \Big( \sum_\tau \big\|\widehat{F_\tau}\big\|_{L^{p_0}(\R^{d+n})}^2  \Big)^{\frac{1}{2}}\\
		& \lesssim R^{-\frac{n}{2}+\epsilon}  \Big( \sum_\tau \norm{F_\tau}_{L^2((N_{R^{-1}}(S_{\mathbf{Q}}|_\tau))}^2 \Big)^{\frac{1}{2}}\\
		& = R^{-\frac{n}{2}+\epsilon} \norm{F}_{L^2((N_{R^{-1}}(S_{\mathbf{Q}}))}.
	\end{align*}
	As mentioned, by Lemma~\ref{lem:thickening}, this implies 
	$$\norm{E^\mathbf{Q} f}_{L^{p_0}(B_{R})} \lesssim_\epsilon R^{\epsilon} \norm{f}_{L^2([0,1]^d)}, \quad \quad \forall
	\, \epsilon>0.$$ 
	Finally, By Corollary~\ref{cor:quad_ep_removal}, for any $r > p_0$, we have 
	$$\norm{E^\mathbf{Q} f}_{L^r(\R^{d+n})} \lesssim \norm{f}_{L^2([0,1]^d)},$$ 
	i.e., $E^\mathbf{Q}: L^2 \rightarrow L^{\frac{2(d+2n)}{d}+}$, as desired. And the proof of $\circled{5}$ is completed.
	
% \begin{remark}
% It is worth mentioning that we cannot expect to derive sharp Stein-Tomas-type estimates for all quadratic forms, even if we already have the sharp $\ell^2L^p$ decoupling theorems in hand. For example, even when $n=1$, any hyperbolic paraboloid serves as a counterexample: The additional loss mainly comes from the geometric fact that the hyperbolic paraboloid contains a line. However, if we change the mode of frequency decomposition in \cite{gozzk23}, then it is possible to remedy such a defect in some cases. For example, when $d=2$ and $n=1$, \cite{gmo24} developed a new way to divide hypersurfaces into a finitely-overlapping collection of rectangular caps, whose corresponding (sharp) $\ell^2L^4$ decoupling inequality indeed implies sharp Stein-Tomas inequalities (see Section~1.2 of \cite{gmo24}). Unfortunately, their approach does not readily extend to higher dimensional cases (see Appendix~B of \cite{gmo24}). On the other hand, higher codimensional quadratic manifolds often resemble hyperbolic paraboloids in many ways, so there are still difficulties towards this direction. The key point of $\circled{5}$ is that if we impose stronger enough nondegeneracy condition on the quadratic manifold such that there is no ``flat'' lower-dimensional part, then everything would be fine. In a nutshell, sharp $\ell^2L^p$ decoupling does not necessarily imply sharp Stein-Tomas-type inequalities, while best $\ell^2L^p$ decoupling indeed imply best Stein-Tomas-type inequalities.
% \end{remark}
	
	In the end, we provide several counterexamples showing that the reverse implication ``$\Uparrow$'' is not true in general:
	\begin{itemize}
		\item All three examples presented at the end of Subsection~\ref{subsec:2_proof} satisfy the endpoint best Stein-Tomas-type inequality, but they are all not strongly nondegenerate. This is because 
		\begin{itemize}
			\item When $d=5$ and $n=3$, consider 
			$$\mathbf{Q}(\xi) = (\xi_1^2 + \xi_3^2 + \xi_5^2, 2(\xi_1\xi_2 + \xi_3\xi_4), \xi_2^2 + \xi_4^2 +\xi_5^2).$$ 
			We have $\mathfrak{d}_{3,2}(\mathbf{Q}) \leq 1$ and $\mathfrak{d}_{2,2}(\mathbf{Q})=0$, while we need $\mathfrak{d}_{3,2}(\mathbf{Q}) \geq \frac{4}{3}$ and $\mathfrak{d}_{2,2}(\mathbf{Q}) \geq \frac{1}{3}$ for it to be strongly nondegenerate.
			
			\item When $d=2$ and $n=3$, consider 
			$$\mathbf{Q}(\xi) = (\xi_1^2, 2\xi_1\xi_2, \xi_2^2).$$ 
			We have $\mathfrak{d}_{1,2}(\mathbf{Q}) = 0$, while we need $\mathfrak{d}_{1,2}(\mathbf{Q}) \geq  \frac{1}{3}$ for it to be strongly nondegenerate.
			
			\item When $d=3$ and $n=5$, consider 
			$$\mathbf{Q}(\xi) = (2\xi_1\xi_2, 2\xi_1\xi_3, 2\xi_2\xi_3, \xi_2^2 - \xi_1^2, \xi_3^2 - \xi_1^2).$$ We have $\mathfrak{d}_{1,4}(\mathbf{Q})=0$ and $\mathfrak{d}_{2,2}(\mathbf{Q})=0$, while we need $\mathfrak{d}_{1,4}(\mathbf{Q}) \geq  \frac{2}{5}$ and $\mathfrak{d}_{2,2}(\mathbf{Q}) \geq  \frac{1}{5}$ for it to be strongly nondegenerate.
		\end{itemize}
		
		\item Recall that all the examples above do not satisfy the (CM) condition. However, there are also many examples which indeed satisfy the (CM) condition (thus by $\circled{2}$, also satisfy the best Stein-Tomas-type inequalities), while are not strongly nondegenerate:
		\begin{itemize}
			\item When $n=1$, all hyperbolic paraboloids $\mathbf{Q}$ (as defined in Corollary~\ref{cor2}) are not strongly nondegenerate, because $\mathfrak{d}_{d-1,1}(\mathbf{Q}) = d-2 < d-1$. Meanwhile, they all clearly satisfy the (CM) condition (the $n=1$ case explained in Definition~\ref{def:CM}).
			
			\item When $d=4$ and $n=2$, for 
			$$\mathbf{Q}(\xi) = (\xi_1^2 - \xi_2^2, \xi_3^2 - \xi_4^2),$$ 
			we have $\mathfrak{d}_{2,2}(\mathbf{Q}) = 0$, while we need $\mathfrak{d}_{2,2}(\mathbf{Q}) \geq 2$ for it to be strongly nondegenerate. Meanwhile, it satisfies the (CM) condition, because $\det(\overline{Q}(\theta)) = 16\theta_1^2\theta_2^2$ do not admit any root of multiplicity larger than $\frac{d}{2}(=2)$ on $\mathbb{S}^1$. Another way to see the validity of the (CM) condition is through the equivalence in $\circled{1}$: By regarding $\mathbf{Q}$ as the tensor product of two hyperbolic paraboloids in $\R^3$, we see that $\widetilde{E}^\mathbf{Q}1 \in L^p$ all $p > \frac{2(d+n)}{d} (=3)$.
			
			\item When $d=3$ and $n=3$, for 
			$$\mathbf{Q}(\xi) = (\xi_1\xi_2, \xi_2\xi_3, \xi_3\xi_1),$$
			we have $\mathfrak{d}_{1,3}(\mathbf{Q}) = 0$ and $\mathfrak{d}_{2,2}(\mathbf{Q}) = 0$, while we need $\mathfrak{d}_{1,3}(\mathbf{Q}) \geq 1$ and $\mathfrak{d}_{2,2}(\mathbf{Q}) \geq 1$ for it to be strongly nondegenerate. Meanwhile, it satisfies the (CM) condition, because 
			$$\int_{\mathbb{S}^{2}}  |\det(\overline{Q}(\theta))|^{-\gamma}d\sigma(\theta) = \int_{\mathbb{S}^{2}}  |2\theta_1\theta_2\theta_3|^{-\gamma}d\sigma(\theta) < \infty, \quad \quad \forall~\gamma < 1.$$ Though we do not find existing analysis of $\widetilde{E}^\mathbf{Q}1$ in the literature, by $\circled{1}$, we know $\widetilde{E}^\mathbf{Q}1\in L^p$ for all $p > 2$, which is best possible (up to the endpoint).
		\end{itemize}
		As a side remark, we do not know whether strong nondegeneracy implies the (CM) condition: we can neither prove it nor find a counterexample. We choose not to formally post it as a conjecture like Conjecture~\ref{conj:S-T} for two reasons: Firstly, the evidence does not seem to be abundant and strong enough for us; secondly, this problem does not seem to be as important as Conjecture~\ref{conj:S-T} at the current stage.
	\end{itemize}
	
	\phantom{x}
	
    \subsection{Proof of \texorpdfstring{$\protect\circled{6}$}{⑥}}\label{subsec:6_proof}\phantom{x}
	
	Let us first explain why $\circled{6}$ is nontrivial. Recall that Proposition~\ref{prop:convex_test} connects the Fourier restriction to the Oberlin condition. If we naively apply Proposition~\ref{prop:convex_test} with $N=d+n$, $\mu = \mu_\mathbf{Q}^*$, $E^\mu = E^{\mathbf{Q}}$, $p=2$, and $q > \frac{2(d+2n)}{d}$, then we immediately know that $\mu_\mathbf{Q}^*$ satisfies the Oberlin condition with any exponent $\alpha < \frac{d}{D}$ (for quadratic manifolds, the homogeneous dimension $D = d+2n$), which is almost what we want in view of Theorem~\ref{thm:gressman_thm1}. However, the major problem of such an intuitive argument is that we are actually missing the endpoint case $\alpha=\frac{d}{D}$, so strictly speaking we cannot directly apply Theorem~\ref{thm:gressman_thm1} to deduce well-curvedness. Indeed, there are no existing results on whether or not the Oberlin condition is ``closed'', i.e., if a measure satisfies the Oberlin condition with all exponents $\alpha<\frac{d}{D}$, then it also satisfies the Oberlin condition with the exponent $\alpha=\frac{d}{D}$.\footnote{In contrast, if we had assumed the endpoint best Stein-Tomas $E^\mathbf{Q}: L^2 \rightarrow L^{\frac{2(d+2n)}{d}}$, then by Proposition~\ref{prop:convex_test} we would know that $\mu_\mathbf{Q}^*$ satisfies the Oberlin condition with $\alpha = \frac{d}{D}$, which by (2) of Theorem~\ref{thm:gressman_thm1} would imply that the affine measure $\mu_{\mathcal{A}}$ of $S_\mathbf{Q}$ satisfies $\mu_{\mathcal{A}} \gtrsim \mu_\mathbf{Q}^*$, so is everywhere nonzero, as desired.}
    
    To rectify these drawbacks and make everything rigorous, we will argue by contradiction via the Newton-type polyhedra characterization Theorem~\ref{thm:gressman_thm6} of well-curvedness. Roughly speaking, if (\ref{eq:newton_convex_hull}) is somehow violated, then by the separating hyperplane theorem it must be violated at a certain quantitative degree, and such a quantitative gap would neutralize the absence of endpoint results. Now we provide the details of the proof.
	
	Suppose $\mathbf{Q}$ is not well-curved, and we want to prove that $E^\mathbf{Q}: L^p \rightarrow L^q$ is unbounded for $p=2$ and some $q > \frac{2(d+2n)}{d}$. By Theorem~\ref{thm:gressman_thm6} (take $m=n$, $\kappa=2$, $p = \mathbf{Q}$, and $f(\xi) = (\xi, \mathbf{Q}(\xi))$), there exists $(O_1,O_2) \in {\rm O}(n,\R) \times {\rm O}(d, \R)$ such that (\ref{eq:newton_convex_hull}) does not hold. Let 
	$\mathbf{Q}' = (Q_1',\dots,Q_n') \coloneqq R_{O_1,O_2} \mathbf{Q}$, then we have $\left( \frac{1}{n} \1_n, \frac{2}{d} \1_d \right) \not\in \mathcal{N}(\mathbf{Q}')$. Define $\widetilde{O}\in {\rm O}(d+n,\R)$ by 
	$\zeta \cdot \widetilde{O} \coloneqq (\zeta'\cdot O_2, \zeta''\cdot O_1)$, where $\zeta = (\zeta',\zeta'')\in \R^d\times\R^n$. Note that for $x\in \mathbb{R}^{d+n}$, we have\footnote{When we use ``$\cdot$'' between two vectors in $\R^N$, we mean inner product;  when we use ``$\cdot$'' between a vector and a matrix, we mean matrix multiplication; when we directly put a matrix before a vector, the vector should always be understood as a column vector. In other words, $A^T x = x \cdot A$ for any matrix $A$ and row vector $x$.}
	\begin{align*}
		E^{\mathbf{Q}'}f(x) &= \int_{[0,1]^d} e^{ix\cdot(\xi,\mathbf{Q}'(\xi))}f(\xi)d\xi\\
		&= \int_{[0,1]^d} e^{ix\cdot(\xi,\mathbf{Q}(\xi\cdot O_2^T)\cdot O_1)}f(\xi)d\xi\\
		&= \int_{O_2([0,1]^d)} e^{ix\cdot (\eta\cdot O_2,\mathbf{Q}(\eta)\cdot O_1)} f(\eta \cdot O_2)d\eta\\
		&= \int_{O_2([0,1]^d)} e^{ix\cdot [(\eta,\mathbf{Q}(\eta))\cdot O]} f(\eta \cdot O_2)d\eta\\
		&= \int_{O_2([0,1]^d)} e^{i(Ox)\cdot (\eta,\mathbf{Q}(\eta))} f(\eta \cdot O_2)d\eta\\
		&= {E}_{O_2([0,1]^d)}^{\mathbf{Q}} \tilde{f} (Ox),
	\end{align*}
	where $\eta = \xi \cdot O_2^T$ and $\tilde{f}(\eta) \coloneqq f(\eta\cdot O_2)$. This immediately implies that for any $p,q\in [1,\infty]$, the $L^p \rightarrow L^q$ boundedness property of $E^{\mathbf{Q}'}$ and ${E}_{O_2([0,1]^d)}^{\mathbf{Q}}$ are the same. On the other hand, by Remark~\ref{rmk:restriction_equiv}, we know that the $L^p \rightarrow L^q$ boundedness property of ${E}_{O_2([0,1]^d)}^{\mathbf{Q}}$ and ${E}^{\mathbf{Q}}$ are also the same. Therefore, to prove $E^\mathbf{Q}$ is unbounded, it suffices to prove $E^{\mathbf{Q}'}$ is unbounded.
	
	Let $V(\mathbf{Q}') \coloneqq \{(e_j,\gamma): \partial^\gamma Q_j'(\xi)|_{\xi=0} \neq 0\}$, then $\mathcal{N}(\mathbf{Q}')$ is the convex hull of $V(\mathbf{Q}')$. Note that any $(e_j,\gamma)\in V(\mathbf{Q}')$ must satisfy $|\gamma|=2$, and for each $j$, we have $Q_j'(\xi) = \sum_{(e_j,\gamma)\in V(\mathbf{Q}')} c_{\gamma}\xi^\gamma$ for some nonzero real coefficients $c_\gamma$. Since $\left( \frac{1}{n} \1_n, \frac{2}{d} \1_d \right) \not\in \mathcal{N}(\mathbf{Q}')$, by the separating hyperplane theorem, there exist some $(\omega, \lambda) \in \R^n \times \R^d$ such that
	\begin{align}\label{eq:hahn-banach}
		(\omega, \lambda) \cdot (e_j, \gamma) > (\omega, \lambda) \cdot \left( \frac{1}{n} \1_n, \frac{2}{d} \1_d \right) + c_0, \qquad \forall\, (e_j, \gamma) \in V(\mathbf{Q}')
	\end{align}
	for some constant $c_0>0$ depending only on $\mathbf{Q}$. 
	Note that $\frac{2}{d} \1_d \cdot \1_d = 2 = \gamma \cdot \1_d$, so we can add a multiple of $\1_d$ to $\lambda$ (if necessary) so that $\lambda = (\lambda_1, \dots, \lambda_d) \in [0, \infty)^d$ and (\ref{eq:hahn-banach}) remain true.
	
	Now the key point is to consider the Knapp-type example corresponding to the part of $S_{\mathbf{Q}'}$ above $Z(\delta) \coloneqq \prod_{i=1}^d[0,\delta^{\lambda_i}]$ for $0 < \delta \ll 1$, where $\lambda \in [0, \infty)^d$ guarantees $Z(\delta) \subset [0,1]^d$. For any $j$, there exists $(e_j,\gamma_j)\in V(\mathbf{Q}')$ such that $\lambda\cdot\gamma_j = \min\{\lambda\cdot\gamma: \,(e_j,\gamma)\in V(\mathbf{Q}')\}$. Then we have
	\[\sup_{\xi\in Z(\delta)} |Q_j'(\xi)| 
	\lesssim \sup_{\xi\in Z(\delta)} \sum_{(e_j,\gamma)\in V(\mathbf{Q}')} \xi^\gamma 
	= \sum_{(e_j,\gamma)\in V(\mathbf{Q}')} \delta^{\lambda \cdot \gamma}
	\lesssim \delta^{\min\{\lambda\cdot\gamma: \,(e_j,\gamma)\in V(\mathbf{Q}')\}} = \delta^{\lambda\cdot\gamma_j}.
	\]
	Therefore, there exists some constant $C>0$ depending only on $\mathbf{Q}$, such that the rectangular box $T \coloneqq Z \times \prod_{j=1}^n [-C \delta^{\lambda\cdot\gamma_j}, C \delta^{\lambda\cdot\gamma_j}]$ contains $\{(\xi,\mathbf{Q}(\xi)): \xi \in Z\}$. And we can bound $|T|$ from above as follows:
	\begin{align*}
		|T| \lesssim \prod_{i=1}^d \delta^{\lambda_i} \times \prod_{j=1}^n \delta^{\lambda\cdot\gamma_j} = \delta^{\sum_{i=1}^d \lambda_i + \sum_{j=1}^n \lambda\cdot\gamma_j}.
	\end{align*}
	
	Besides, by applying (\ref{eq:hahn-banach}) to each $(e_j,\gamma_j)$ and taking the arithmetic mean of all the $n$ inequalities, we obtain
	\begin{align*}
		(\omega, \lambda) \cdot \left(\frac{1}{n} \1_n, \frac{1}{n} \sum_{j=1}^n\gamma_j \right) > (\omega, \lambda) \cdot \left( \frac{1}{n} \1_n, \frac{2}{d} \1_d \right) + c_0.
	\end{align*}
	Note that the inner product with $\omega$ cancels out on both sides above, so we get 
	\begin{align*}
		\frac{1}{n} \sum_{j=1}^n \lambda\cdot\gamma_j > \frac{2}{d} \sum_{i=1}^d \lambda_i +c_0,
	\end{align*}
	which implies
	\begin{align*}
		|T| \lesssim \delta^{\frac{d+2n}{d} \sum_{i=1}^d \lambda_i + nc_0}.
	\end{align*}
	
	Suppose $\mu_\mathbf{Q}^*$ satisfies the Oberlin condition with exponent $\alpha \geq 0$, then by testing it on $T$, we must have $\mu_\mathbf{Q}^*(T) \lesssim |T|^\alpha$. By noting that $\mu_\mathbf{Q}^*(T) \geq |Z| = \delta^{\sum_{i=1}^d \lambda_i}$ and combining it with the upper bound of $|T|$ above, we get
	\begin{align*}
		\delta^{\sum_{i=1}^d \lambda_i} \lesssim \delta^{\alpha(\frac{d+2n}{d} \sum_{i=1}^d \lambda_i + nc_0)}.
	\end{align*}
	By taking $\delta \rightarrow 0$, this becomes
	\begin{align*}
		\sum_{i=1}^d \lambda_i &\geq \alpha \left(\frac{d+2n}{d} \sum_{i=1}^d \lambda_i + nc_0\right)\\
		\Longleftrightarrow \qquad \alpha&\leq \frac{\sum_{i=1}^d \lambda_i}{\left(\frac{d+2n}{d} \sum_{i=1}^d \lambda_i + nc_0\right)}.
	\end{align*}
	Note that this upper bound for $\alpha$ is strictly smaller than $\frac{d}{d+2n}$ as $c_0>0$, so there always exists some $\alpha < \frac{d}{d+2n}$ such that the above relation is violated, which in turn contradicts our assumption that $\mu_\mathbf{Q}^*$ satisfies the Oberlin condition with exponent $\alpha \geq 0$. In other words, $\mu_\mathbf{Q}^*$ does not satisfy the Oberlin condition with exponent $p'/q$ for some $q > \frac{p'(d+2n)}{d}$. By applying Proposition~\ref{prop:convex_test} with $N=d+n$, $\mu = \mu_\mathbf{Q}^*$, and $E^\mu = E^{\mathbf{Q}'}$, we know that $E^{\mathbf{Q}'}: L^p\rightarrow L^q$ is unbounded for some $q > \frac{p'(d+2n)}{d}$, and so is $E^\mathbf{Q}$. In particular, by taking $p=2$, we get $E^\mathbf{Q}: L^2\rightarrow L^q$ is unbounded for some $q > \frac{2(d+2n)}{d}$, which contradicts our assumption that $E^\mathbf{Q}$ satisfies the best Stein-Tomas inequality. So the proof by contradiction is completed, and $\circled{6}$ is true.
	
	For the reverse implication ``$\Uparrow$'', the $n=2$ case can be verified affirmatively by combining $\circled{2}$ and Theorem~2 in \cite{dmv22}, which says that $\circled{7}$ is equivalent when $n=2$. Therefore, Conjecture~\ref{conj:S-T} holds true when $n=2$. However, we do not know whether it is true or not in general when $n \geq 3$.
	
	\phantom{x}
	
    \subsection{Proof of \texorpdfstring{$\protect\circled{7}$}{⑦}}\label{subsec:7_proof}\phantom{x}
	
	``$\Longrightarrow$'' is an immediate corollary of $\circled{2}$ and $\circled{6}$, the $n=2$ case of which has been previously proved by Theorem~2 in \cite{dmv22}. We first briefly summarize their proof. The method in \cite{dmv22} relies heavily on the nice fact that the (CM) condition can be reformulated as $\overline{Q}(\theta)$ has no root of multiplicity $> \frac{d}{2}$ on $\mathbb{S}^1$, which is in turn equivalent to a certain semistability of $\overline{Q}$ via a simple application of the Hilbert-Mumford criterion for ${\rm SL}(2,\C)$-action. Such a semistability of $\overline{Q}$ immediately leads to a certain semistability of $\mathbf{Q}$, which is in turn equivalent to the nonvanishing property of invariant polynomials via the so-called ``Fundamental Theorem of GIT''. Starting from this point, they can show that $\mathbf{Q}$ is well-curved by directly computing the density of the affine measure and applying the so-called ``First Fundamental Theorem for ${\rm SL}(2)$-invariants'' on generators. 
	
	Compared to their proof, our alternative proof avoids additional black boxes from geometric invariant theory, and is more of the analytic ($\circled{2}$ only relies on complex interpolation of analytic family of operators) and geometric ($\circled{6}$ is basically some Knapp-type infinitesimal homogeneity argument) natures. Also, our proof seems to be more robust, in the sense that it also works for any $n \geq 3$, which is not dealt with at all in \cite{dmv22}. Besides, our proof is ``stronger'' than that in \cite{dmv22} in the sense that we can naturally insert the hierarchy of best Stein-Tomas in between the (CM) condition and well-curvedness. Such an observation is new even when $n=2$.
	
	For the reverse implication ``$\Longleftarrow$'', the $n=2$ case has been verified affirmatively by Theorem~2 in \cite{dmv22}. They showed that the path from the nonvanishing property of invariant polynomials to well-curvedness of $\mathbf{Q}$ used in ``$\Longrightarrow$'' can actually be reversed, thanks to Lemma~2 in \cite{gressman19}, which provides a way to estimate the density of the affine measure in terms of the homogeneous generators of the invariant polynomials. Thus they actually get an equivalent characterization of well-curvedness of $\mathbf{Q}$ (Lemma~5 in \cite{dmv22}). Under all these conversions, everything boils down to a problem in linear algebra, which can be solved by a careful analysis of matrices (Section~4.2 and 4.3 in \cite{dmv22}).
	
	However, when $n \geq 3$, ``$\Longleftarrow$'' is not true in general. Here we provide several
	counterexamples: 
	\begin{itemize}
		\item All three examples presented at the end of Subsection~\ref{subsec:2_proof} are well-curved by $\circled{6}$ and the fact that they all satisfy the best Stein-Tomas-type inequality. However, we already know in $\circled{2}$ that the (CM) condition fails for all of them.
		\item When $d=n=4$, consider 
		$$\mathbf{Q}(\xi) = (\xi_1\xi_2, \xi_2\xi_3, \xi_3\xi_4, \xi_4\xi_1).$$
		It is fairly straightforward to check that $f(\xi) = (\xi, \mathbf{Q}(\xi))$ satisfies (\ref{model_poly_form}) with $c=c'=2\neq 0$, so by (2) in Theorem~\ref{thm:gressman_thm2}, we know that the affine measure of $S_\mathbf{Q}$ is a nonzero constant times the push-forward of Lebesgue measure via $f$ (i.e., $\mu$ defined in the proof of $\circled{6}$), which in particular implies $\mathbf{Q}$ is well-curved. However, $\mathbf{Q}$ does not satisfy the (CM) condition, because direct computation shows that  
		$$\int_{\mathbb{S}^3} |\det(\overline{Q}(\theta))|^{-\gamma} d\sigma(\theta) = \int_{\mathbb{S}^3} |(\theta_1\theta_3 - \theta_2\theta_4)^2|^{-\gamma} d\sigma(\theta) = \infty,\quad \quad \forall~\gamma\geq \frac{1}{2},$$
		while for the (CM) condition to hold we need 
		$$\int_{\mathbb{S}^3} |\det(\overline{Q}(\theta))|^{-\gamma} d\sigma(\theta) < \infty,\quad \quad \forall~\gamma < 1.$$
	\end{itemize}
	
	Finally, we point out that when $n=2$, $\circled{2}\,\circled{6}\,\circled{7}$ are all ``$\Longleftrightarrow$'' and form a cycle.

	\phantom{x}
	
    \subsection{Proof of \texorpdfstring{$\protect\circled{8}$}{⑧}}\label{subsec:8_proof}\phantom{x}

	Our proof strategy is to argue by contradiction, and the key ingredient is Theorem~\ref{thm:gressman_thm6_variant}. Suppose $\mathbf{Q}$ is not nondegenerate, then there exist $0 \leq d' \leq d$ and $0 \leq n' \leq n$ such that $\mathfrak{d}_{d',n'}(\mathbf{Q}) < \frac{n'}{n}d - 2(d-d')$. Since any nontrivial $\mathbf{Q}$ satisfies $\mathfrak{d}_{d,n}(\mathbf{Q}) = d$, we must have either $0 \leq d' < d$ or $0 \leq n' < n$, which implies $\frac{n'}{n}d - 2(d-d') < d'$. For any $0 \leq m_1 \leq m_2$, let $\DD_{m_1,m_2}$ be the set of $m_2\times m_2$ diagonal matrices with $m_1$ diagonal elements $1$ and $m_2 - m_1$ diagonal elements $0$. Note that we can rewrite the definition of $\mathfrak{d}_{d',n'}(\mathbf{Q})$ in (\ref{s2def2e1}) as
	\begin{align*}
		\mathfrak{d}_{d',n'}(\mathbf{Q}) & = \inf_{\substack{   M \in \mathbb{R}^{d\times d} \\ {\rm rank}(M)=d'   }}  \inf_{\substack{   N \in \mathbb{R}^{n\times n'} \\ {\rm rank}(N)=n'   }} {\rm NV}(R_{N,M}\mathbf{Q})\\
		& = \inf_{\substack{   M \in {\rm GL}(d,\R) \\ D_1 \in \DD_{d',d}   }}  \inf_{\substack{   N \in {\rm GL}(n,\R) \\ D_2 \in \DD_{n',n}   }} {\rm NV}(R_{N D_2, M D_1}\mathbf{Q})\\
		& = \inf_{\substack{   D_2 \in \DD_{d',d} \\ D_1 \in \DD_{n',n}   }}  \inf_{\substack{   M \in {\rm GL}(d,\R) \\ N \in {\rm GL}(n,\R)   }} {\rm NV}\left( R_{D_2,D_1}(R_{N, M}\mathbf{Q}) \right).
	\end{align*}
	Since ${\rm NV}$ only take values in $\N$, the infimum must be attained, i.e., there exist $(D_2, D_1) \in \DD_{n',n} \times \DD_{d',d}$ and $(N,M) \in {\rm GL}(n,\R)\times {\rm GL}(d,\R)$ such that $${\rm NV}\left( R_{D_2,D_1}(R_{N, M}\mathbf{Q}) \right) = \mathfrak{d}_{d',n'}(\mathbf{Q}) < \frac{n'}{n}d - 2(d-d').$$ 
	Let 
	$$\I_1 \coloneqq \{1\leq i \leq d: (D_1)_{ii} = 1\}, \quad \J \coloneqq \{1\leq j \leq n: (D_2)_{jj} = 1\},$$ and $\mathbf{Q}' = (Q_1', \dots, Q_n') \coloneqq R_{N, M}\mathbf{Q} $. Then $\#\I_1 = d'$ and $\#\J = n'$. On the other hand, since we assume $\mathbf{Q}$ is well-curved, by Theorem~\ref{thm:gressman_thm6_variant} (take $m=n$, $\kappa=2$, $p = \mathbf{Q}$, and $f(\xi) = (\xi, \mathbf{Q}(\xi))$), we know that $\left( \frac{1}{n} \1_n, \frac{2}{d} \1_d \right) \in \mathcal{N}(\mathbf{Q}')$. Our goal comes down to prove that this Newton-type polyhedra condition contradicts $${\rm NV}\left( R_{D_2,D_1}(\mathbf{Q}') \right) < \frac{n'}{n}d - 2(d-d').$$ Once this goal is achieved, the proof of $\circled{8}$ by contradiction would be completed.
	
	By $\left( \frac{1}{n} \1_n, \frac{2}{d} \1_d \right) \in \mathcal{N}(\mathbf{Q}')$, there exist nonnegative real numbers $\{\theta_{j,\gamma}\}_{(e_j,\gamma)\in V(\mathbf{Q}')}$ such that $\sum_{(e_j,\gamma)\in V(\mathbf{Q}')}\theta_{j,\gamma} = 1$, and 
	\begin{align}\label{eq:convex_identity}
		\sum_{1\leq j\leq n}\sum_{\gamma: (e_j,\gamma)\in V(\mathbf{Q}')} \theta_{j,\gamma} (e_j, \gamma) = \left( \frac{1}{n} \1_n, \frac{2}{d} \1_d \right).
	\end{align}
	For each $1 \leq j \leq n$, by taking an inner product of both sides of (\ref{eq:convex_identity}) with $(e_j,0)$, we get $\sum_{\gamma: (e_j,\gamma)\in V(\mathbf{Q}')} \theta_{j,\gamma} = \frac{1}{n}$. Let $\{\tilde{e}_1,\dots,\tilde{e}_d\}$ be the standard basis of $\R^d$. Define the index set $\I_2 \text{ to be those } i \in \I_1 \text{ such that there exists } (e_j,\gamma)\in V(\mathbf{Q}') \text{ with } j\in\J \text{ and } \gamma\cdot \tilde{e}_i \neq 0, \text{ and whenever this happens, we have } \forall i'\not\in\I_1, \gamma\cdot \tilde{e}_{i'} = 0$. Let $\I_3 \coloneqq \I_1 \setminus \I_2$. Note that $\I_2$ corresponds to variables $\xi_i$ that remains in $\mathbf{Q}'$ after deleting those $Q'_j$ with $j\not\in \J$ and letting $\xi_i=0$ for $i\not\in \I_1$, so $\#\I_2 = {\rm NV}\left( R_{D_2,D_1}(\mathbf{Q}') \right)$. On the other hand, $\I_3$ corresponds to variables $\xi_i$ that satisfies at least one of the following two cases: $(i)$ it only appears in those $Q'_j$ with $j\not\in \J$; $(ii)$ it appears in some $Q'_j$ with $j \in \J$ and are not directly sent to $0$, but are still deleted as they appears only in crossed terms paired with those $\xi_{i'}$ that are indeed sent to $0$. By our assumption, we have $\#\I_2 < \frac{n'}{n}d - 2(d-d')$, which implies that
	$$\#\I_3 = \#\I_1 - \#\I_2 > d' - \Big(\frac{n'}{n}d - 2(d-d')\Big) = \frac{n-n'}{n}d + (d-d') (> 0)$$
	by our previous computation. Now the key point is to compare the contribution of both sides of (\ref{eq:convex_identity}) over $\I_3$. More precisely, by taking an inner product with $\sum_{i \in \I_3} (0,\tilde{e}_i)$ on both sides of (\ref{eq:convex_identity}), we get
	\begin{align}\label{eq:convex_identity_sum}
		\sum_{i \in \I_3} \sum_{1\leq j\leq n}\sum_{\gamma:(e_j,\gamma)\in V(\mathbf{Q}')} \theta_{j,\gamma} \gamma \cdot \tilde{e}_i = \sum_{i \in \I_3} \frac{2}{d} \1_d \cdot \tilde{e}_i.
	\end{align}
	For the right hand side of (\ref{eq:convex_identity_sum}), we have
	\begin{align}\label{eq:convex_identity_RHS}
		\sum_{i \in \I_3} \frac{2}{d} \1_d \cdot \tilde{e}_i = \frac{2}{d}\#\I_3 > \frac{2}{d}\Big( \frac{n-n'}{n}d + (d-d') \Big) = \frac{2}{n}(n-n') + \frac{2}{d}(d-d').
	\end{align}
	For the left hand side of (\ref{eq:convex_identity_sum}), we split it into two sums in view of Case $(i)$ and Case $(ii)$ in our intuitive interpretation of $\I_3$:
	\begin{align*}
		&\,\phantom{iii}\sum_{i \in \I_3} \sum_{1\leq j \leq n}\sum_{\gamma:(e_j,\gamma)\in V(\mathbf{Q}')} \theta_{j,\gamma} \gamma \cdot \tilde{e}_i\\
		&= \sum_{i \in \I_3} \sum_{j \not\in \J} \sum_{\gamma:(e_j,\gamma)\in V(\mathbf{Q}')} \theta_{j,\gamma} \gamma \cdot \tilde{e}_i 
		+ \sum_{i \in \I_3} \sum_{j \in \J} \sum_{\gamma:(e_j,\gamma)\in V(\mathbf{Q}')} \theta_{j,\gamma} \gamma \cdot \tilde{e}_i\\
		&\eqqcolon \Pi_1 + \Pi_2.
	\end{align*}
	As for $\Pi_1$, we interchange the order of summation and apply the trivial estimate $\sum_{i \in \I_3} \gamma \cdot \tilde{e}_i \leq |\gamma| = 2$, which gives us:
	\begin{align*}
		\Pi_1 = \sum_{j \not\in \J} \sum_{\gamma:(e_j,\gamma)\in V(\mathbf{Q}')} \theta_{j,\gamma} \sum_{i \in \I_3} \gamma \cdot \tilde{e}_i
		\leq 2 \sum_{j\not\in \J} \sum_{\gamma:(e_j,\gamma)\in V(\mathbf{Q}')} \theta_{j,\gamma}.
	\end{align*}
	Recall that $\#\J = n'$ and $\sum_{\gamma: (e_j,\gamma)\in V(\mathbf{Q}')} \theta_{j,\gamma} = \frac{1}{n}$ for any $1\leq j \leq n$ (and in particular for any $j\not\in\J$), so the above upper bound for $\Pi_1$ becomes
	\begin{align}\label{eq:Pi_1}
		\Pi_1 \leq \frac{2}{n}\#\{1\leq j \leq n: j \not\in \J\} = \frac{2}{n}(n-n').
	\end{align}
	It remains to control $\Pi_2$. The key observation here is that by our previous explanation in Case $(ii)$ (which is the case for $\Pi_2$), for each $j\in\J$ and $(e_j,\gamma)\in V(\mathbf{Q}')$ such that $\gamma\cdot\tilde{e}_i \neq 0$ for some $i\in\I_3$, there must exist some $i' \not\in \I_1$ such that $\gamma\cdot\tilde{e}_{i'} \neq 0$. In this case, $\gamma$ corresponds to the crossed term $\xi_i\xi_{i'}$ ($i\in\I_3$, $i' \not\in \I_1$ means $i \neq i'$) and $\gamma\cdot\tilde{e}_i = \gamma\cdot\tilde{e}_{i'} = 1$, which implies that $\sum_{i\in\I_3}\gamma\cdot\tilde{e}_i = \sum_{i'\not\in\I_1} \gamma\cdot\tilde{e}_{i'}$ (as $\gamma\cdot\tilde{e}_{i''}=0$ whenever $i''\neq i,i'$). Therefore, by interchanging the order of summation in $\Pi_2$ , we have
	\begin{align*}
		\Pi_2 & = \sum_{j \in \J} \sum_{\gamma:(e_j,\gamma)\in V(\mathbf{Q}')} \theta_{j,\gamma} \sum_{i \in \I_3} \gamma \cdot \tilde{e}_i\\
		& = \sum_{j \in \J} \sum_{\gamma:(e_j,\gamma)\in V(\mathbf{Q}')} \theta_{j,\gamma} \sum_{i'\not\in\I_1} \gamma\cdot\tilde{e}_{i'}\\
		& \leq \sum_{1\leq j \leq n} \sum_{\gamma:(e_j,\gamma)\in V(\mathbf{Q}')} \theta_{j,\gamma} \sum_{i'\not\in\I_1} \gamma\cdot\tilde{e}_{i'}\\
		& = \sum_{i'\not\in\I_1} \sum_{1\leq j \leq n} \sum_{\gamma:(e_j,\gamma)\in V(\mathbf{Q}')} \theta_{j,\gamma} \gamma\cdot\tilde{e}_{i'}.
	\end{align*}
	However, by taking an inner product with $\sum_{i' \not\in \I_1} (0,\tilde{e}_{i'})$ on both sides of (\ref{eq:convex_identity}), we get
	\begin{align*}
		\sum_{i' \not\in \I_1} \sum_{1\leq j\leq n}\sum_{\gamma:(e_j,\gamma)\in V(\mathbf{Q}')} \theta_{j,\gamma} \gamma \cdot \tilde{e}_{i'} = \sum_{i' \not\in \I_1} \frac{2}{d} \1_d \cdot \tilde{e}_{i'},
	\end{align*}
	whose left hand side matches the upper bound for $\Pi_2$ above, so we get
	\begin{align}\label{eq:Pi_2}
		\Pi_2 \leq \sum_{i' \not\in \I_1} \frac{2}{d} \1_d \cdot \tilde{e}_{i'} = \frac{2}{d}(d-\#\I_1) = \frac{2}{d}(d-d').
	\end{align}
	Now by combining the upper bound (\ref{eq:Pi_1}) for $\Pi_1$ and the upper bound (\ref{eq:Pi_2}) for $\Pi_2$, we can conclude that the left hand side of (\ref{eq:convex_identity_sum}) $\leq \frac{2}{n}(n-n') + \frac{2}{d}(d-d')$, which contradicts (\ref{eq:convex_identity_RHS}). Hence our initial assumption that $\mathbf{Q}$ is not nondegenerate cannot hold, and the proof of $\circled{8}$ is completed. 
	
	However, when $n \geq 2$, ``$\Uparrow$'' is not true in general. Here we provide all counterexamples in the case when $d=3$ and $n=2$: Any $$\mathbf{Q}(\xi) \equiv (\xi_1\xi_2, \xi_2^2 + \xi_1\xi_3) ~\text{or}~ (\xi_1\xi_2,\xi_1^2 \pm \xi_1\xi_3)$$ is nondegenerate by (c4) in Lemma~\ref{addth2}, since $\mathfrak{d}_{3,1}(\mathbf{Q})=2 \geq \frac{3}{2}$ and $\mathfrak{d}_{2,2}(\mathbf{Q})=1 \geq 1$. There are many ways to see that they are all not well-curved. The easiest way is by using $\circled{7}$, which tells us that well-curvedness is equivalent to the (CM) condition when $n=2$. However, by (e) in Lemma~\ref{addth1}, they must not satisfy the (CM) condition. It is easy to see that they are the only counterexamples when $d=3$ and $n=2$. If one does not want to resort to the nontrivial black box $\circled{7}$, then one may also choose to apply the criteria established in \cite{gressman19}. For simplicity, below we only work with $\mathbf{Q}(\xi) \equiv (\xi_1\xi_2, \xi_2^2 + \xi_1\xi_3)$ to illustrate the ideas, but the same methods equally apply to $\mathbf{Q}(\xi) \equiv (\xi_1\xi_2,\xi_1^2 \pm \xi_1\xi_3)$. An initial observation is that by Corollary~\ref{cor:quad_equiv_curved}, instead of considering all equivalent quadratic forms, we only need to check $\mathbf{Q}(\xi) = (\xi_1\xi_2, \xi_2^2 + \xi_1\xi_3)$ without loss of generality.
	\begin{itemize}
		\item Newton-type polyhedra approaches: By Theorem~\ref{thm:gressman_thm6} (take $d=3$, $m=2$, $\kappa=2$, $p = \mathbf{Q}$, and $f(\xi) = (\xi, \mathbf{Q}(\xi))$), to show that $\mathbf{Q}$ is not well-curved, it suffices to show that there exists $(O_1,O_2) \in {\rm O}(n, \R) \times {\rm O}(d, \R)$ such that $\left( \frac{1}{2}, \frac{1}{2}, \frac{2}{3}, \frac{2}{3}, \frac{2}{3} \right) \not\in \mathcal{N}(R_{O_1,O_2} \mathbf{Q})$. In fact, if we take $O_1$ and $O_2$ to be identity matrices, then $\left( \frac{1}{2}, \frac{1}{2}, \frac{2}{3}, \frac{2}{3}, \frac{2}{3} \right) \not\in \mathcal{N}(\mathbf{Q})$, which is the convex hull of $V(\mathbf{Q}) = \{(1,0,1,1,0),(0,1,0,2,0),(0,1,1,0,1)\}$. This is because if 
		$$\theta_1 (1,0,1,1,0) + \theta_2 (0,1,0,2,0) + \theta_3 (0,1,1,0,1) = \left( \frac{1}{2}, \frac{1}{2}, \frac{2}{3}, \frac{2}{3}, \frac{2}{3} \right),$$
		then $\theta_1 = \frac{1}{2}$ (first position) and $\theta_3 = \frac{2}{3}$ (last position), so $\theta_2 (0,1,0,2,0) = \left( 0,-\frac{1}{6},-\frac{1}{2}, \frac{1}{6},0 \right)$, which is impossible.
		\item Sublevel set approaches: By Section~3.2 of \cite{gressman19}, given a parametrization $f$ of the form (\ref{canonical_poly_form})\footnote{For our purposes we only need this form, but the technique actually works for general $f$.}, if for $s_1,\dots,s_m \in\R^d$, define the polynomial $$\mathcal{P}(s_1,\dots,s_m) \coloneqq \det \begin{bmatrix}
			p_1(s_1) & \cdots & p_1(s_m)\\
			\vdots & \ddots & \vdots\\
			p_m(s_1) & \cdots & p_m(s_m)
		\end{bmatrix}$$
		and the associated sublevel set $Z \coloneqq \{(s_1,\dots,s_m): |\mathcal{P}(s_1,\dots,s_m)|\leq 1\}$, then we have ($\nu_{\mathcal{A}}$ is a constant by Proposition~\ref{prop:TDI_affine_const}):
		\begin{align}\label{eq:sublevel_oberlin}
			\left( \frac{d\nu_{\mathcal{A}}}{dt} \right)^{-1} \sim \sup_{M\in {\rm GL}(d,\R)} \Big\{ |MK|^{\frac{m}{D}} : (MK)^m \subset Z \Big\},
		\end{align}
		where $K$ is the unit ball of any norm on $\R^d$, $|MK|$ denotes the Lebesgue measure, and $(MK)^m$ is the $m$-fold product of $MK$. Now for $f(\xi) = (\xi, \mathbf{Q}(\xi))$, we have 
		$$\mathcal{P}(u_1,u_2,u_3,v_1,v_2,v_3) = u_1u_2(v_2^2 - v_1v_3) - v_1v_2(u_2^2 - u_1u_3).$$
		The boxes $K_\epsilon = [-\frac{\epsilon^3}{4}, \frac{\epsilon^3}{4}] \times [-\epsilon^{-1},\epsilon^{-1}] \times [-\epsilon^{-5}, \epsilon^{-5}]$ have volume tending to infinity as $\epsilon \rightarrow 0^+$, but $K_\epsilon \times K_\epsilon$ is contained in $Z$ for all $\epsilon >0$. Thus (\ref{eq:sublevel_oberlin}) implies $\nu_{\mathcal{A}} = 0$ ($\mu_{\mathcal{A}} = 0$), and $S_\mathbf{Q}$ is not well-curved.
	\end{itemize}
	
	For a counterexample for ``$\Uparrow$'' when $n \geq 3$, one may take $$\mathbf{Q}(\xi) = (\xi_3^2, \xi_1\xi_2, \xi_2^2 + \xi_1\xi_3)$$ with $d=n=3$. To check that $\mathbf{Q}$ is nondegenerate, the key point is to show $\mathfrak{d}_{3,3}(\mathbf{Q}) = 3$, $\mathfrak{d}_{2,3}(\mathbf{Q}) \geq 1$, $\mathfrak{d}_{3,2}(\mathbf{Q}) \geq 2$, and $\mathfrak{d}_{3,1}(\mathbf{Q}) \geq 1$:
	\begin{itemize}
		\item $\mathfrak{d}_{3,3}(\mathbf{Q}) = 3$: Suppose to the contrary that $\mathfrak{d}_{3,3}(\mathbf{Q}) \leq 2$, then we must have $\mathbf{Q} \equiv (\xi_1^2, \xi_1\xi_2, \xi_2^2)$ (as all components are linearly independent), which implies that there exist at least two different $(\theta_1:\theta_2:\theta_3)\in\R\mathbb{P}^2$ such that $\theta_1\xi_3^2 + \theta_2\xi_1\xi_2 + \theta_3(\xi_2^2 + \xi_1\xi_3)$ is a perfect square $(a\xi_1+b\xi_2+c\xi_3)^2$. However, this is impossible: Since there is no $\xi_1^2$ term, we must have $a=0$, which forces $\theta_2 = \theta_3 = 0$, i.e., $(\theta_1:\theta_2:\theta_3)=(1:0:0)$ is the only solution.
		
		\item $\mathfrak{d}_{2,3}(\mathbf{Q}) \geq 1$: Suppose to the contrary that $\mathfrak{d}_{2,3}(\mathbf{Q}) = 0$, then we must have $\mathbf{Q}(\xi) \equiv (\xi_1L_1(\xi), \xi_1L_2(\xi), \xi_1L_3(\xi))$ where each $L_i(\xi)$ is a linear form in $\xi$. This implies that there is a nontrivial common factor of all components of $\mathbf{Q}$, which is clearly impossible.
		
		\item $\mathfrak{d}_{3,2}(\mathbf{Q}) \geq 2$: Suppose to the contrary that $\mathfrak{d}_{3,2}(\mathbf{Q}) \leq 1$, then $\mathbf{Q} \equiv (\xi_1^2, 0 , *)$, where $*$ can be any quadratic form in $\xi$. This implies that there exists a nontrivial linear combination of all components of $\mathbf{Q}$ that equals $0$, which is clearly impossible.
		
		\item $\mathfrak{d}_{3,1}(\mathbf{Q}) \geq 1$: Suppose to the contrary that $\mathfrak{d}_{3,1}(\mathbf{Q}) = 0$, then similarly, there exists a nontrivial linear combination of all components of $\mathbf{Q}$ that equals $0$, which is clearly impossible.
	\end{itemize}
	To check that $\mathbf{Q}$ is not well-curved, we can apply Newton-type polyhedra approaches as before. At the end of the day, things are reduced to disprove the existence of a convex combination 
	$$\theta_1(1,0,0,0,0,2) + \theta_2(0,1,0,1,1,0) + \theta_3(0,0,1,0,2,0) + \theta_4(0,0,1,1,0,1) = \left(\frac{1}{3},\frac{1}{3},\frac{1}{3},\frac{2}{3},\frac{2}{3},\frac{2}{3}\right).$$
	The first position yields $\theta_1 = \frac{1}{3}$, while the second position yields $\theta_2 = \frac{1}{3}$, so we get $\theta_3(0,0,1,0,2,0) + \theta_4(0,0,1,1,0,1) = \left(0,0,\frac{1}{3},\frac{1}{3},\frac{1}{3},0\right)$, which is clearly impossible. Therefore, $\mathbf{Q}$ is nondegenerate but not well-curved.
	
	\phantom{x}

    \subsection{Proof of \texorpdfstring{$\protect\circled{9}$}{⑨}}\label{subsec:9_proof}\phantom{x}
	
	This is due to Corollary~1.4 in \cite{gozzk23}.
	
    \subsection{Proof of \texorpdfstring{$\protect\circled{10}$}{⑩}}\label{subsec:10_proof}\phantom{x}
	
	By Theorem~\ref{thm:uniform_decay}, we have the uniform Fourier decay estimate $$|\widetilde{E}^\mathbf{Q}1(x)| \lesssim_\mathbf{Q} (1+|x|)^{-\frac{\mathfrak{d}_{d,1}(\mathbf{Q})}{2}},$$
	and $\mathfrak{d}_{d,1}(\mathbf{Q}) = d$ represents the best decay 
	$$|\widetilde{E}^\mathbf{Q}1(x)| \lesssim_\mathbf{Q} (1+|x|)^{-\frac{d}{2}},$$ 
	which implies that
	\begin{align*}
		\big\|\widetilde{E}^\mathbf{Q}1\big\|_{L^p(\R^{d+n})}^p \lesssim_\mathbf{Q} \int_{\R^{d+n}}(1+|x|)^{-\frac{d}{2}p}dx.
	\end{align*}
	The integral above is finite if $p>\frac{2(d+n)}{d}$, as desired. 
	
	An alternative approach is to apply $\circled{1}$. Since $\mathfrak{d}_{d,1}(\mathbf{Q}) = d$ implies that $\overline{Q}(\theta)$ is nonsingular for any $\theta\in\mathbb{S}^{n-1}$, the integrand $\det(\overline{Q}(\theta))$ in the (CM) condition (\ref{eq:CM}) is nonvanishing on $\mathbb{S}^{n-1}$. So $\mathbf{Q}$ trivially satisfies the (CM) condition, which implies $\widetilde{E}^\mathbf{Q}1 \in L^{\frac{2(d+n)}{d}+}$ by ``$\Leftarrow$'' of $\circled{1}$.
	
	The reverse implication ``$\Uparrow$'' is certainly not true in general. Still in view of $\circled{1}$, for $\widetilde{E}^\mathbf{Q}1 \in L^{\frac{2(d+n)}{d}+}$ to hold, we only need some control over the singularity of $|\det(\overline{Q})|^{-\gamma}$ near its zero set, while for $\mathfrak{d}_{d,1}(\mathbf{Q}) = d$ we have to completely avoid all zeros. For instance, all $\mathbf{Q}$'s in (c5) of Lemma~\ref{addth2} serve as counterexamples. For an example when $n\geq 3$, one may take $\mathbf{Q}(\xi) = (\xi_1\xi_2, \xi_2\xi_3, \xi_3\xi_1)$ with $d=n=3$.
	
	\begin{remark}
		For a thorough discussion on historical background and the significance of the condition $\mathfrak{d}_{d,1}(\mathbf{Q}) = d$, one may consult Appendix~{\rm\ref{appendix:par_d,1}}.
	\end{remark}

	\begin{remark}
		Recall at the end of Subsection~{\rm\ref{subsec:5_proof}} that the {\rm(CM)} condition does not imply strong nondegeneracy. Since we have just seen that $\mathfrak{d}_{d,1}(\mathbf{Q}) = d$ is stronger than the {\rm(CM)} condition, it is quite natural to ask whether $\mathfrak{d}_{d,1}(\mathbf{Q}) = d$ implies strong nondegeneracy. Unfortunately, the answer is negative. When $n=1$, all hyperbolic paraboloids satisfy $\mathfrak{d}_{d,1}(\mathbf{Q})=d$, but are not strongly nondegenerate. For higher-codimensional examples, by the explicit construction described in Remark~{\rm\ref{rmk:partial_d_1_construct}} of Appendix~{\rm\ref{appendix:par_d,1}}, when $d=4$ and $n=3$, $$\mathbf{Q}(\xi) = (\xi_1\xi_4 - \xi_2\xi_3, \xi_1\xi_3 - \xi_2\xi_4, \xi_1^2 + \xi_2^2 - \xi_3^2 - \xi_4^2)$$ satisfies $\mathfrak{d}_{4,1}(\mathbf{Q}) = 4$. However, by deleting the last component of $\mathbf{Q}$ and setting $\xi_1 = \xi_2 = 0$, we get $\mathfrak{d}_{2,2}(\mathbf{Q}) = 0 <\frac{2}{3}$, which means $\mathbf{Q}$ is not strongly nondegenerate. For a counterexample when $n=2$, one may simply restrict the previous $\mathbf{Q}$ to its first two components, i.e., take $$\mathbf{Q}(\xi) = (\xi_1\xi_4 - \xi_2\xi_3, \xi_1\xi_3 - \xi_2\xi_4)$$ when $d=4$ and $n=2$. Note that we still have $\mathfrak{d}_{4,1}(\mathbf{Q})=4$ and $\mathfrak{d}_{2,2}(\mathbf{Q}) = 0 < 2$.
		
		Conversely, strong nondegeneracy also does not imply $\mathfrak{d}_{d,1}(\mathbf{Q}) = d$ in general. When $d=4$ and $n=2$, one may take $$\mathbf{Q}(\xi) = (\xi_1^2 + \xi_2^2, \xi_3^2 + \xi_4^2).$$ When $d=3$ and $n=3$, one may take $$\mathbf{Q}(\xi) = (\xi_1^2 + \xi_3^2, 2\xi_1\xi_2, \xi_2^2 + \xi_3^2)$$
        or 
        $$\mathbf{Q}(\xi) = (\xi_1^2 + \xi_2\xi_3, \xi_2^2 + \xi_1\xi_3, \xi_3^2 + \xi_1\xi_2).$$ Verification of these examples is left to the interested readers. Therefore, there is no direct relationship between $\mathfrak{d}_{d,1}(\mathbf{Q}) = d$ and strong nondegeneracy {\rm(}best $\ell^2L^p$ decoupling{\rm)}. However, by Corollary~{\rm\ref{cor:algebra_partial}}, we know $\mathfrak{d}_{d,1}(\mathbf{Q}) = d$ indeed implies that $\mathbf{Q}$ is nondegenerate {\rm(}best $\ell^pL^p$ decoupling{\rm)}.
	\end{remark}

    \subsection{Proof of \texorpdfstring{$\protect\circled{11}$}{⑪}}\label{subsec:11_proof}\phantom{x}
	
	Recall the classical fact that for Borel sets $A$, $\dim_F A \leq \dim_H A$ (Hausdorff dimension). One may consult \cite{mattila15} for its proof.
	
	As $S_\mathbf{Q}$ is a $d$-dimensional smooth manifold (with boundary), we clearly have $\dim_H S_\mathbf{Q} = d$. Therefore, by Definition~\ref{def:salem}, to prove $\circled{11}$, it suffices to show:
	\begin{proposition}\label{prop:dim_F}
		$\dim_F S_\mathbf{Q} = \mathfrak{d}_{d,1}(\mathbf{Q})$.
	\end{proposition}
	\begin{proof}
		Fix any $\varphi_0\in C_c^\infty(\R^d)$ with $\varphi_0 \geq 0$, $\varphi_0 \not\equiv 0$, and $\supp\varphi_0 \subset [0,1]^d$. Consider the Borel measure $\mu$ defined by 
        $$d\mu(\zeta) \coloneqq \varphi_0(\zeta') d\mu_\mathbf{Q}(\zeta), \quad \quad \zeta = (\zeta',\zeta'')\in \R^d\times\R^n. $$
        Then $\mu \in \mathcal{M}(S_\mathbf{Q})$, and $$|\widehat{\mu}(x)| = \abs{\widetilde{E}_{\varphi_0}^\mathbf{Q}1(-x)} \lesssim_\mathbf{Q} (1+|x|)^{- \frac{\mathfrak{d}_{d,1}(\mathbf{Q})}{2}   }$$ 
		by Corollary~\ref{cor:recover_banner}. This implies $\dim_F S_\mathbf{Q} \geq \mathfrak{d}_{d,1}(\mathbf{Q})$. 
		
		To show $\dim_F S_\mathbf{Q} \leq \mathfrak{d}_{d,1}(\mathbf{Q})$, we argue by contradiction. Suppose $\dim_F S_\mathbf{Q} > \mathfrak{d}_{d,1}(\mathbf{Q})$, then by the definition, there must exist some $\mu \in \mathcal{M}(S_\mathbf{Q})$ such that 
		$$|\widehat{\mu}(x)| \lesssim |x|^{-\gamma}, \quad \quad \text{for ~some~}\gamma > \frac{\mathfrak{d}_{d,1}(\mathbf{Q})}{2}.$$ Note that $\mu$ is finite implies $|\widehat{\mu}(x)| \lesssim (1+|x|)^{-\gamma}$ for the same $\gamma$. By truncating $\mu$ with a suitable smooth bump as in the proof of Corollary~\ref{cor:recover_banner}, we may without loss of generality assume $\supp\mu \subset [\epsilon,1-\epsilon]^d \times \R^n$ for some $\epsilon>0$. Thanks to the TDI property of quadratic manifolds, we can apply a smoothing trick similar to that in the proof of Proposition~\ref{prop:TDI_affine_const}. It is worth mentioning that in Section~6 of \cite{fhk22fourier}, similar ideas have been used to determine the Fourier dimension of light cones, where rotational symmetry is exploited instead. 
		
		For any $\nu \in \mathcal{M}(\R^d)$, we define
		$$E_\nu^\mathbf{Q}1(x) \coloneqq \int_{\R^d} e^{i(x'\cdot\xi + \xi^T\cdot \overline{Q}(x'')\cdot\xi/2)} d\nu(\xi), \quad\quad  x = (x',x'') \in \R^{d+n}.$$
		In particular, let $\nu \in \mathcal{M}([\epsilon,1-\epsilon]^d)$ be the pullback of $\mu$ under the map $\xi \mapsto (\xi,\mathbf{Q}(\xi))$, i.e., 
		$$\int_{\R^d} f(\zeta') d\nu(\zeta') = \int_{\R^{d+n}} f(\zeta') d\mu(\zeta),    \quad \quad \zeta = (\zeta',\zeta'')\in \R^d\times\R^n$$ 
		for any $f \in C_c^\infty(\R^d)$, then $\widehat{\mu}(-x) = E_\nu^\mathbf{Q}1(x)$. 
		
		By change of variables (as in the proof of Corollary~\ref{cor:recover_banner}, except that $\varphi(\xi)d\xi$ is replaced with $d\nu(\xi)$), we see that 
		$$\abs{E_\nu^\mathbf{Q}1(x)} \lesssim (1+|x|)^{-\gamma}, \quad \quad \forall ~x\in \R^{d+n}$$ 
		implies that 
	\begin{align}\label{eq:uniform_decay_xi_0}
	       \abs{E_{\nu_{\xi_0}}^\mathbf{Q}1(x)} \lesssim (1+|x|)^{-\gamma}, \quad \quad \forall ~x\in \R^{d+n},
	\end{align}
		where $\nu_{\xi_0}$ is the translation of $\nu$ by $\xi_0 \in \R^d$, i.e., 
		$$\int_{\R^d} f(\eta) d\nu_{\xi_0}(\eta) = \int_{\R^d} f(\eta-\xi_0) d\nu(\eta)$$ 
		for any nonnegative Borel function $f$ on $\R^d$. Also, the bound ``$\lesssim$'' in (\ref{eq:uniform_decay_xi_0}) is uniform for all $\xi_0$ in a compact region of $\R^d$. Now since $\nu \neq 0$, we can always find $\varphi_1 \in C_c^\infty(\R^d)$ with $\varphi_1 \geq 0$, $\supp\varphi_1 \subset B_{\epsilon/3}$, and $\nu * \varphi_1 \neq 0$.\footnote{Otherwise by taking an approximation of the identity and Theorem~2.6 of \cite{mattila15}, we know that $\nu$, as the weak limit of zero measures, must also be zero itself.} Let $\varphi \coloneqq \nu * \varphi_1$, then $\varphi \in C_c^\infty([0,1]^d)$, $\varphi \geq 0$, and $\varphi \neq 0$. Moreover, we have
		\begin{align*}
			\widetilde{E}_\varphi^\mathbf{Q}1(x) 
			& = \int_{\R^d} e^{i(x'\cdot\xi + \xi^T\cdot \overline{Q}(x'')\cdot\xi/2)} \varphi(\xi) d\xi\\
			& = \int_{\R^d} e^{i(x'\cdot\xi + \xi^T\cdot \overline{Q}(x'')\cdot\xi/2)} (\nu * \varphi_1)(\xi) d\xi\\
			& = \int_{\R^d} e^{i(x'\cdot\xi + \xi^T\cdot \overline{Q}(x'')\cdot\xi/2)} \int_{\R^d} \varphi_1(\xi-\eta)d\nu(\eta) d\xi\\
			& = \int_{\R^d} \int_{\R^d} e^{i(x'\cdot(\xi+\eta) + (\xi+\eta)^T\cdot \overline{Q}(x'')\cdot(\xi+\eta)/2)}  \varphi_1(\xi)d\nu(\eta) d\xi\\
			& = \int_{\R^d} \int_{\R^d} e^{i(x'\cdot(\eta+\xi) + (\eta+\xi)^T\cdot \overline{Q}(x'')\cdot(\eta+\xi)/2)}  d\nu(\eta) \varphi_1(\xi) d\xi\\
			& = \int_{\R^d} \int_{\R^d} e^{i(x'\cdot\eta + \eta^T\cdot \overline{Q}(x'')\cdot\eta/2)}  d\nu_{-\xi}(\eta) \varphi_1(\xi) d\xi\\
			& = \int_{\R^d} E_{\nu_{-\xi}}^\mathbf{Q}1(x) \varphi_1(\xi) d\xi.
		\end{align*}
	On the other hand, by (\ref{eq:uniform_decay_xi_0}) and the related uniformity observation, we have
		$$\abs{E_{\nu_{-\xi}}^\mathbf{Q}1(x)} \lesssim (1+|x|)^{-\gamma}, \quad \quad \forall~\xi \in \supp\varphi_0,$$ 
		thus we obtain
		\begin{align*}
			\abs{\widetilde{E}_\varphi^\mathbf{Q}1(x)} \leq \int_{\supp\varphi_1} \abs{E_{\nu_{-\xi}}^\mathbf{Q}1(x)} \varphi_1(\xi) d\xi \lesssim_{\varphi_1} (1+|x|)^{-\gamma}.
		\end{align*}
		However, by the optimality part of Corollary~\ref{cor:recover_banner}, we know that this cannot hold for any $\gamma > \frac{\mathfrak{d}_{d,1}(\mathbf{Q})}{2}$. So we arrive at a contradiction, and the proof of Proposition~\ref{prop:dim_F} (thus $\circled{11}$) is completed.
	\end{proof}

        \begin{remark}
		Firstly, by taking $\tilde{\mu}$ to be the pushforward measure of $\varphi(\xi)d\xi$ in the proof, we have $\tilde{\mu} \in \mathcal{M}(S_\mathbf{Q})$ and $|\tilde{\mu}(x)| \lesssim |x|^{-\gamma}$ for all $x\in\R^{d+n}$. This means that to test the Fourier dimension of $S_\mathbf{Q}$, we only need to focus on smooth $\mu \in \mathcal{M}(S_\mathbf{Q})$. Secondly, via exactly the same proof, one can easily see that $\dim_F \{(\xi,\mathbf{Q}(\xi)): \xi \in U\} = \mathfrak{d}_{d,1}(\mathbf{Q})$ for any set $U \subset \R^d$ containing an interior point {\rm(}not necessarily $U = [0,1]^d${\rm)}. Finally, the smoothing trick in this subsection works equally well for all {\rm TDI} manifolds.
	\end{remark} 
    
	\begin{remark}
		When $n=1$ and $S_\mathbf{Q}$ is a paraboloid or hyperbolic paraboloid, we trivially have $\dim_F S_\mathbf{Q} = d$, thus $\mathbf{Q}$ is Salem by $\dim_F S_\mathbf{Q} \leq \dim_H S_\mathbf{Q} = d$ and $\dim_F S_\mathbf{Q} \geq \mathfrak{d}_{d,1}(\mathbf{Q}) = d$. In fact, this result holds for all hypersurfaces with nonvanishing Gaussian curvature, and no smoothing trick is required. However, when $n \geq 2$, the trivial upper bound $\dim_F S_\mathbf{Q} \leq \dim_H S_\mathbf{Q} = d$ no longer matches the lower bound, as $\mathfrak{d}_{d,1}(\mathbf{Q}) < d$ generically {\rm(}see Remark~{\rm\ref{rmk:rho_growing_order})}, and we need to perform a more delicate analysis. This is a fundamental difference between the codimension $1$ case and higher-codimensional cases.
	\end{remark}

    \addtocontents{toc}{\protect\setcounter{tocdepth}{2}}

	\subsection{Optimality of the exponents}\label{subsec:best_proof}\phantom{x}
	
	In the previous subsections, we have almost completed the proof of Theorem~\ref{thm:relation_diagram}, which is the main result in this section. However, there are still some leftover problems, such as the optimality of $L^{\frac{2(d+n)}{d}+}$ and $L^{\frac{2(d+2n)}{d}+}$, which will be dealt with in this subsection.
	
	Our proof of optimality strongly relies on the machinery of Oberlin affine curvature, and is very flexible to be extended to higher-codimensional submanifolds of degree $\geq 3$.
	
	\cprotect\subsubsection{Optimality of $\frac{2(d+2n)}{d}$}\phantom{x}
	
	First, let us see why the threshold $\frac{2(d+2n)}{d}$ in the Stein-Tomas-type inequality $E^\mathbf{Q}: L^2 \rightarrow L^{\frac{2(d+2n)}{d}+}$ cannot be lowered. Suppose to the contrary that $E^\mathbf{Q}: L^2 \rightarrow L^{q+}$ for some $q < \frac{2(d+2n)}{d}$, then there must exist some $1 \leq q_0 < \frac{2(d+2n)}{d}$ such that $E^\mathbf{Q}: L^2 \rightarrow L^{q_0}$. Now by applying Proposition~\ref{prop:convex_test} with $N=d+n$, $\mu = \mu_\mathbf{Q}^*$, $E^\mu = E^{\mathbf{Q}}$, $p=2$, and $q = q_0$, we immediately know that $\mu_\mathbf{Q}^*$ satisfies the Oberlin condition with exponent $\alpha = \frac{2}{q_0} > \frac{d}{d+2n} = \frac{d}{D}$. However, by (1) of Theorem~\ref{thm:gressman_thm1}, this means that $\mu_\mathbf{Q}^*$ must be a zero measure, contradicting the fact that $\mu_\mathbf{Q}^*$ is the (nonzero) surface measure of $S_\mathbf{Q}$. Therefore, the threshold $\frac{2(d+2n)}{d}$ is optimal.
	
	In exactly the same way, one can easily prove the following more general result, whose proof we omit: 
	\begin{proposition}\label{prop:opt_ST_index}
		Let $U$ be a bounded open/closed region of $\R^d$. Then for any smooth function $\psi: U \rightarrow \R^n$, $p \in[1,\infty]$, and $q \in [1,\frac{p'D}{d})$ {\rm(}$D = D(d,n)$ is the homogeneous dimension{\rm)}, the associated Fourier extension operator $E_U^\psi$ {\rm(}recall {\rm(\ref{def:general_ext_op})} for its definition{\rm)} cannot be bounded from $L^p(U)$ to $L^q(\R^{d+n})$. In other words, $E_U^\psi: L^p \rightarrow L^{\frac{p'D}{d}+}$ {\rm(}once established{\rm)} is the best, in the sense that the threshold $\frac{p'D}{d}$ cannot be replaced by any strictly smaller exponent.
	\end{proposition}
	For example, the moment curve $\{\psi(t) = (t,t^2,\dots, t^N): t\in [0,1]\}$ in $\R^N$ has a homogeneous dimension $D = 1+2+\cdots+N = \frac{N(N+1)}{2}$, which by Proposition~\ref{prop:opt_ST_index} implies that $E^\psi: L^p \rightarrow L^{\frac{p'N(N+1)}{2}+}$ (once established) is the best. As is well known, this is indeed the case.

	\cprotect\subsubsection{Optimality of $\frac{2(d+n)}{d}$}\phantom{x}
	
	Next, let us see why the threshold $\frac{2(d+n)}{d}$ in $\widetilde{E}^\mathbf{Q} 1 \in L^{\frac{2(d+n)}{d}+}$ cannot be lowered. 
    % In fact, we will show the same thing for $E_\varphi^\mathbf{Q}$, which is slightly stronger. 
    Our proof closely follows the framework developed in \cite{gizzk23}. Suppose to the contrary that $\widetilde{E}^\mathbf{Q}1 \in L^{p+}$ for some $p < \frac{2(d+n)}{d}$, then there must exist some $1 \leq p_0 < \frac{2(d+n)}{d}$ such that $\widetilde{E}^\mathbf{Q} 1 \in L^{p_0}$. Via exactly the same arguments in the $(ii) \Rightarrow (iii)$ part of the proof of Proposition~\ref{prop:restriction_equiv} (with $p=\infty$ and $f \equiv 1$ there), we immediately know that for any $\varphi\in C_c^\infty(\R^d)$ with $\varphi\not\equiv0$, we have $\widetilde{E}_\varphi^\mathbf{Q} 1 \in L^{p_0}$ for some $1 \leq p_0 < \frac{2(d+n)}{d}$. In particular, we may fix one $\varphi_0\in C_c^\infty(\R^d)$ such that $\varphi_0\equiv0$ outside $[\frac{1}{8},\frac{7}{8}]^d$, $\varphi_0\equiv 1$ inside $[\frac{1}{4},\frac{3}{4}]^d$, and $\varphi_0 \in [0,1]$ in between.
	
	For any (small) dyadic number $\delta\in(0,1)$, we cover $[0,1]^d$ with dyadic cubes $\{[-\delta,\delta]^d + \delta\nn\}_{\nn\in\Z^d\cap[0,\delta^{-1}]^d}$. Since $\mathbf{Q}$ is a quadratic form, there must exist some constant $C$ (depending only on $\mathbf{Q}$) such that the rectangular box $\Box_\delta \coloneqq [-\delta,\delta]^d\times [-C\delta^2, C\delta^2]^n$ contains $\{(\xi,\mathbf{Q}(\xi)):\xi\in [-\delta,\delta]^d\}$. Fix a $\phi \in C_c^\infty(\R)$ with $\supp\phi \subset [-2,2]$, $\phi\equiv 1$ on $[-1,1]$, $\phi\in[0,1]$ in between, and $\sum_{k\in \Z}\phi(\cdot+3k)\equiv 1$ over $\R$ (thus $\sum_{k\in \Z}\phi(\cdot+k)\equiv 3$ over $\R$). Define 
	$$\Phi_\delta(\xi) \coloneqq \prod_{i=1}^d\phi\Big(\frac{\xi_i}{\delta}\Big)\prod_{j=1}^n\phi\Big(\frac{\xi_{d+j}}{C\delta^2}\Big),$$
	then $\Phi_\delta \in C_c^\infty(\R^{d+n})$ such that $\Phi_\delta\equiv0$ outside $2\Box_\delta$, $\Phi_\delta \equiv 1$ inside $\Box_\delta$, and $\Phi_\delta \in [0,1]$ on $2\Box_\delta\setminus \Box_\delta$. For any $\nn\in\Z^d\cap[0,\delta^{-1}]^d$, define $\Phi_{\delta,\nn} \coloneqq \Phi_\delta \circ L_{\delta\nn}$. Then it is easy to see that $\Phi_{\delta,\nn} \equiv 0$ outside $2L_{\delta\nn}^{-1}(\Box_\delta)$, $\Phi_{\delta,\nn} \equiv 1$ inside $L_{\delta\nn}^{-1}(\Box_\delta)$, and $\Phi_{\delta,\nn} \in [0,1]$ on $2L_{\delta\nn}^{-1}(\Box_\delta)\setminus L_{\delta\nn}^{-1}(\Box_\delta)$. Also, $\{\Phi_{\delta,\nn}\}_{\nn\in\Z^d\cap[0,\delta^{-1}]^d}$ is finitely overlapping, and 
	$$\Phi_\delta^\mathbf{Q} \coloneqq \sum_{\nn\in\Z^d\cap[0,\delta^{-1}]^d} \Phi_{\delta,\nn} \equiv 3,\quad \quad \text{on~}\Big\{(\xi,\mathbf{Q}(\xi)):\xi\in\Big[\frac{1}{8},\frac{7}{8}\Big]^d\Big\},$$
	which means 
	$$\widetilde{E}_{\varphi_0}^\mathbf{Q}1 = (\varphi_0 d\sigma)^\vee = \frac{1}{3}\Big(\Phi_\delta^\mathbf{Q} \cdot \varphi_0 d\sigma\Big)^\vee = \frac{1}{3}(\Phi_\delta^\mathbf{Q})^\vee * \widetilde{E}_{\varphi_0}^\mathbf{Q}1, \quad \quad \forall~\delta>0.$$
	By H\"older's inequality, this implies that $$\big\|\widetilde{E}_{\varphi_0}^\mathbf{Q}1\big\|_{L^\infty(\R^{d+n})} \lesssim \big\|(\Phi_\delta^\mathbf{Q})^\vee\big\|_{L^{p_0'}(\R^{d+n})} \big\|\widetilde{E}_{\varphi_0}^\mathbf{Q}1\big\|_{L^{p_0}(\R^{d+n})},$$
	where $\big\|\widetilde{E}_{\varphi_0}^\mathbf{Q}1\big\|_{L^{p_0}(\R^{d+n})} <\infty$ by our assumption. We claim that 
	$$\big\|(\Phi_\delta^\mathbf{Q})^\vee\big\|_{L^{p_0'}(\R^{d+n})} \rightarrow 0\quad \quad \text{as~}\delta \rightarrow 0.$$ This will imply $\big\|\widetilde{E}_{\varphi_0}^\mathbf{Q}1\big\|_{L^\infty(\R^{d+n})} = 0$, i.e., $\widetilde{E}_{\varphi_0}^\mathbf{Q}1 \equiv 0$. However, $\varphi_0\not\equiv0$ directly implies that $\widetilde{E}_{\varphi}^\mathbf{Q}1 \not\equiv 0$, so we arrive at a contradiction, and the proof is completed.
	
	To show the above claim, first observe that
	\begin{align*}
		\big\|(\Phi_\delta^\mathbf{Q})^\vee\big\|_{L^{2}(\R^{d+n})} & =  \big\|\Phi_\delta^\mathbf{Q} \big\|_{L^{2}(\R^{d+n})}\\
		& = \Big\| \sum_{\nn\in\Z^d\cap[0,\delta^{-1}]^d} \Phi_{\delta,\nn} \Big\|_{L^{2}(\R^{d+n})}\\
		& \lesssim \Big(\sum_{\nn\in\Z^d\cap[0,\delta^{-1}]^d} \big\| \Phi_{\delta,\nn} \big\|_{L^{2}(\R^{d+n})}^2 \Big)^{\frac{1}{2}}\\
		& \sim \delta^{-\frac{d}{2}}\cdot \big\| \Phi_\delta \big\|_{L^{2}(\R^{d+n})} \sim \delta^{n}.
	\end{align*}
	On the other hand, we also have
	\begin{align*}
		\big\|(\Phi_\delta^\mathbf{Q})^\vee\big\|_{L^{1}(\R^{d+n})} & \leq   \sum_{\nn\in\Z^d\cap[0,\delta^{-1}]^d} \big\| (\Phi_{\delta,\nn})^\vee \big\|_{L^{1}(\R^{d+n})}\\
		& \sim \delta^{-d} \cdot \big\| (\Phi_\delta)^\vee \big\|_{L^{1}(\R^{d+n})}  \sim \delta^{-d}.
	\end{align*}
	Interpolation gives that $\big\|(\Phi_\delta^\mathbf{Q})^\vee\big\|_{L^{r}(\R^{d+n})} \rightarrow 0$ as $\delta \rightarrow 0$ for any $r > \frac{2(d+n)}{d+2n}$, and in particular for $r = p_0'$, as desired.

    Our proof actually shows that $p = \frac{2(d+n)}{d}$ is the best possible for $\widetilde{E}_{\varphi_0}^\mathbf{Q} 1 \in L^{p+}$ to hold, not just $\widetilde{E}^\mathbf{Q} 1 \in L^{p+}$. Moreover, there is no special role played by $\varphi_0$, except that it is not identically zero. In particular, by replacing $\varphi_0$ with $\chi_{[0,1]^d}$ in the proof, we get $p = \frac{2(d+n)}{d}$ is the best possible for ${E}^\mathbf{Q} 1 \in L^{p+}$ to hold.
    
    In fact, one can say something much stronger, as we actually do not need any a priori regularity assumption of the input of the extension operators. Also, the framework behind the proof is very robust, as it basically only involves some Knapp-type analysis: One can imagine that such a method can be carried out for any higher-codimensional submanifold, not just quadratic manifolds. In general, one can prove the following unified result:
	\begin{proposition}\label{prop:opt_surf_meas_Lp_index}
		Let $U$ be a bounded open/closed region of $\R^d$, and $\psi: U \rightarrow \R^n$ be any smooth function that extends to a neighborhood of $\overline{U}$. If $f\in L^p$ with $1\leq p<\frac{D + d}{d}$ {\rm(}$D = D(d,n)$ is the homogeneous dimension{\rm)} satisfies that $\widehat{f}$ is supported on $\{(\xi,\psi(\xi)): \xi\in U\}$, then $f$ is identically zero. In particular, $E^\psi1 \in L^{\frac{D+d}{d}+}$ and $\widetilde{E}_\varphi^\psi 1 \in L^{\frac{D+d}{d}+}$, where $\varphi \in C_c^\infty(\R^d)$ with $\varphi\not\equiv 0$ on $U$, {\rm(}once established{\rm)} is the best, in the sense that the threshold $\frac{D+d}{d}$ cannot be replaced by any strictly smaller exponent.
	\end{proposition}
	Since the ideas behind Proposition~\ref{prop:opt_surf_meas_Lp_index} are exactly the same as those in the quadratic case, we only provide a sketch of its proof.
	\begin{proof}[Proof sketch of Proposition~\ref{prop:opt_surf_meas_Lp_index}]
		Without loss of generality, we may assume $U = [0,1]^d$. Cover $[0,1]^d$ with dyadic cubes $\{[-\delta,\delta]^d + \delta\nn\}_{\nn\in\Z^d\cap[0,\delta^{-1}]^d}$, and for each $\nn\in\Z^d\cap[0,\delta^{-1}]^d$, let $K_{\delta,\nn}$ be the convex hull of $\{(\xi,\psi(\xi)): \xi\in [0,1]^d \cap ([-\delta,\delta]^d + \delta\nn) \}$. By the proof of (1) of Theorem~\ref{thm:gressman_thm1} in \cite{gressman19} and compactness, one can see that $|K_{\delta,\nn}| \leq C \delta^D$ with a constant $C$ uniform in $\nn$ and $\delta$ (depending only on $\psi$). For each $K_{\delta,\nn}$, by the arguments in $(iii) \Rightarrow (ii) \Rightarrow (i)$ of Lemma~\ref{lem:Oberlin_equiv}, there exists a rectangular box $T_{\delta,\nn}$ with $K_{\delta,\nn} \subset T_{\delta,\nn}$ and $|T_{\delta,\nn}| \leq C_{d,n} |K_{\delta,\nn}|$. Thus we can carefully choose a partitioning of unity $\{\Phi_{\delta,\nn}\}$ nicely adapted to $\{T_{\delta,\nn}\}$ such that $\big\| \Phi_{\delta,\nn} \big\|_{L^{2}(\R^{d+n})} \lesssim \delta^{\frac{D}{2}}$ and $\big\| (\Phi_{\delta,\nn})^\vee \big\|_{L^{1}(\R^{d+n})} \lesssim 1$ hold uniformly. Also, $\{\Phi_{\delta,\nn}\}$ is finitely overlapping and $\Phi_\delta^\psi \coloneqq \sum_{\nn}\Phi_{\delta,\nn} \equiv 1$ over $\{(\xi,\psi(\xi)):\xi\in[0,1]^d\}$. Similarly to the quadratic case, now $f = f * (\Phi_\delta^\psi)^\vee$ implies $\norm{f}_{L^\infty} \leq \norm{f}_{L^p} \big\|(\Phi_\delta^\psi)^\vee \big\|_{L^{p'}}$, and we only need to check that $\big\|(\Phi_\delta^\psi)^\vee \big\|_{L^{p'}(\R^{d+n})} \rightarrow 0$ as $\delta \rightarrow 0$ whenever $1 \leq p < \frac{D+d}{d}$. But via similar computation as before, we have $\big\|(\Phi_\delta^\mathbf{Q})^\vee\big\|_{L^{2}(\R^{d+n})} \lesssim \delta^{\frac{D-d}{2}}$ and $\big\|(\Phi_\delta^\mathbf{Q})^\vee\big\|_{L^{1}(\R^{d+n})} \lesssim \delta^{-d}$, which by interpolation yield what we want.
	\end{proof}
	
	As a historical remark on Proposition~\ref{prop:opt_surf_meas_Lp_index}, our formulation (as well as that in \cite{gizzk23}) for general $f$ is called ``the $p$-thin problem'', and previous techniques are mainly inspired by spectral synthesis and of an Agmon-Hörmander-type flavor. More precisely, a subset $X$ of $\R^N$ is said to be ``$p$-thin''\footnote{We may always assume $p\geq 1$.}, if for any $u \in \mathcal{S}'(\R^N)$ with $\widehat{u} \in L^p(\R^N)$ and $\supp\, u \subset X$, we have $u=0$. This type of results has many applications in other fields, and one may consult \cite{guo1993p} for uniqueness problems of wave equations, \cite{agranovsky2004p} for uniqueness problems of convolution equations, and \cite{ikromov1997convergence} or \cite{bgzzk2024stationary} for Tarry's problem in number theory.
	
	Now we provide some higher-codimensional examples demonstrating the power of Proposition~\ref{prop:opt_surf_meas_Lp_index}. In Theorem~1 of \cite{agranovsky2004p}, Agranovsky--Narayanan proved that any $d$-dimensional smooth submanifold of $\R^{d+n}$ is $p$-thin if $p\leq \frac{2(d+n)}{d}$. Note that $D(d,n) \geq d+2n$ for any $d$ and $n$, so our Proposition~\ref{prop:opt_surf_meas_Lp_index} recovers the main result in \cite{agranovsky2004p} up to the endpoint. Moreover, when $d \geq n$, the threshold $\frac{2(d+n)}{d}$ is sharp, see Theorem~2 of \cite{agranovsky2004p}. Another example is the $d$-dimensional Parsell-Vinogradov manifold of degree $k$, $\{(\xi^\gamma)_{1\leq |\gamma| \leq k}: \xi\in[0,1]^d\}$ in $\R^N$, where $N = \binom{d+k}{k}-1$. Ikromov \cite{ikromov1997convergence} proved that the Fourier transform of its surface measure, i.e., $E^\psi1$ with $\psi(\xi) = (\xi^\gamma)_{2\leq |\gamma| \leq k}$ in our notation, is not in $L^p$ for any $p < \binom{d+k}{d+1} + 1$. The idea behind his proof is to show that $E^\psi1$ is ``large'' on a disjoint union of rectangular boxes with the help of oscillatory integral estimates. Note that $D = \frac{kd}{d+1}\binom{d+k}{d}$ (one may consult Figure~1 in \cite{gressman19}), so our Proposition~\ref{prop:opt_surf_meas_Lp_index} yields a threshold $\frac{D+d}{d} = \binom{d+k}{d+1} + 1$, which in particular recovers Ikromov's result via a simpler and more geometric method. Also, we actually get the $p$-thin property, which is stronger than merely about $E^\psi1$. 
	
	\begin{remark}
		However, we should point out that in terms of $E^\psi1 \in L^p$, the threshold $\binom{d+k}{d+1}+1$ is in general not sharp: Basu--Guo--Zhang--Zorin-Kranich \cite{bgzzk2024stationary} proved that when $d=2$, the sharp threshold should be $\binom{d+k}{d+1}+2$, which is a bit larger. For the lower bound of the threshold, they refined the arguments in \cite{ikromov1997convergence} by exploiting the ``rotational symmetry'' of Parsell-Vinogradov manifolds to find more rectangular boxes on which $E^\psi1$ is large; for the upper bound of the threshold, they developed a novel stationary set estimate for oscillatory integrals, which relies on tools from real algebraic geometry and model theory. On the other hand, when $d=1$ {\rm(}moment curve{\rm)}, the threshold $\binom{d+k}{d+1}+1$ is indeed sharp. Such a phenomenon that Proposition~{\rm\ref{prop:opt_surf_meas_Lp_index}} may or may not be sharp is very important in understanding what a desired Fourier restriction theory for general higher-codimensional submanifolds should be like, and we will turn to this issue seriously in the final part of this subsection.
	\end{remark}

    \subsection{Final remarks on higher-codimensional Fourier restriction theory}\label{subsec:final_remarks}\phantom{x}
    
	In this final subsection, we will discuss and make conjectures on what should be the decisive factors dominating the higher-codimensional Fourier restriction theory. To begin with, consider the following natural question:
	\begin{question}\label{ques:desired_restriction}
		For given $d$ and $n$, what is the best $L^p \rightarrow L^q$ Fourier extension estimate that a $d$-dimensional submanifold of $\R^{d+n}$ can have?
	\end{question}
	This question is a bit vague as we need to make sense of what we mean by ``best''. One way to define it is by using the right-angled trapezoid $\TT(\Sigma)$ formed by all $\left(\frac{1}{p},\frac{1}{q}\right)$ such that $L^p \rightarrow L^q$ Fourier extension estimate holds for a submanifold $\Sigma$. If there exists a $d$-dimensional submanifold $\Sigma_0$ in $\R^{d+n}$ such that $\TT(\Sigma) \subset \TT(\Sigma_0)$ for any (other) $d$-dimensional submanifold $\Sigma$ in $\R^{d+n}$, then we say that $\TT(\Sigma_0)$ is the best $L^p \rightarrow L^q$ Fourier restriction estimate for given $d$ and $n$. Of course, it is not known a priori whether or not such a $\Sigma_0$ exists. Question~\ref{ques:desired_restriction} is inevitable if one wants to formulate a Fourier restriction conjecture for general higher-codimensional surfaces, which seems a trivial task (in view of the $n=1$ case) but in fact requires deep understanding of both geometric and Fourier analytic aspects of higher-codimensional surfaces.
	
	In the study of Fourier restriction theory, to find necessary conditions for Fourier extension estimates, one usually considers two types of examples: One is Knapp-type homogeneity examples, which yields the oblique leg of $\TT$; the other is the $L^p$ bound of the Fourier transform of the surface measure\footnote{Whether or not it is smoothed.}, which yields the top base of $\TT$. 
	
	In our case, Knapp-type homogeneity examples correspond to the Oberlin condition, so it is reasonable to conjecture that well-curvedness determines the best oblique leg of $\TT$. Indeed, by Proposition~\ref{prop:opt_ST_index}, it is reasonable to conjecture that for well-curved surfaces, the oblique leg of $\TT$ should always lie on $\frac{p'}{q}=\frac{d}{D}$. This 
	is also one reason why we believe Conjecture~\ref{conj:S-T} should be true: The pair $\left(\frac{1}{p},\frac{1}{q}\right)$ for sharp Stein-Tomas-type inequality always lies on the oblique leg of $\TT$, at least for all known examples. Moreover, by 
	(3) of Theorem~\ref{thm:gressman_thm2}, for any $d$ and $n$, there always exists a well-curved $d$-dimensional submanifold of $\R^{d+n}$, so it is reasonable to make the following conjecture, which is a cousin of Conjecture~\ref{conj:S-T}:
	\begin{conjecture}\label{conj:oblique_leg}
		For Question~{\rm\ref{ques:desired_restriction}}, the best oblique leg of $\TT(\Sigma)$ should lie on $\frac{p'}{q}=\frac{d}{D}$, which is achieved if and only if $\Sigma$ is well-curved.
	\end{conjecture}
	
	However, as for the Fourier transform of the surface measure, tricky issues happen. A naive way is to simply apply Proposition~\ref{prop:opt_surf_meas_Lp_index}, which tells us that the top base of $\TT$ cannot lie above $\frac{1}{q} = \frac{d}{D+d}$. Thus we obtain a very general necessary condition:
	\begin{proposition}\label{prop:opt_index}
		For any $d$-dimensional submanifold $\Sigma$ in $\R^{d+n}$, we have $\TT(\Sigma) \subset \TT_0$, where $\TT_0$ is the right-angled trapezoid with top base on $\frac{1}{q} = \frac{d}{D+d}$, bottom base on $\frac{1}{q}=0$, oblique leg on $\frac{p'}{q}=\frac{d}{D}$, and perpendicular leg on $\frac{1}{p} = 0$. 
	\end{proposition}
	Define the ``critical endpoint'' of a right-angled trapezoid to be the intersection point of its top base and its oblique leg. Then in Proposition~\ref{prop:opt_index}, the critical point of $\TT_0$ is $(\frac{1}{p},\frac{1}{q}) = (\frac{d}{D+d}, \frac{d}{D+d})$, which exactly lies on a $45$ degree line through the origin. 
	In view of Conjecture~\ref{conj:oblique_leg} and the case when $n=1$ (hypersurfaces with nonvanishing Gaussian curvature) or $d=1$ (curves with nonvanishing torsion), it is reasonable to ask whether or not any well-curved $\Sigma$ satisfies $\TT(\Sigma) = \TT_0$. 
    
    Unfortunately, this is not true in general. Typical counterexamples are offered by the $d$-dimensional Parsell-Vinogradov manifold of degree $k$ introduced previously, which is well-curved by the example following equation (20) in \cite{gressman19}. As we have seen before, by \cite{bgzzk2024stationary}, when $d=2$, the conjectured top base of $\TT$ should be strictly below $\frac{1}{q} = \frac{d}{D+d}$. Also, when $k=2$, Bak--Lee \cite{bl04} proved that although the oblique leg of $\TT$ matches that of $\TT_0$ ($\frac{p'}{q}=\frac{1}{d+2}$ by $D = d(d+2)$), the top base of $\TT$ ($\frac{1}{q} = \frac{1}{2(d+1)}$) is strictly below that of $\TT_0$ ($\frac{1}{q} = \frac{1}{d+3}$ by $D = d(d+2)$) when $d \geq 2$. For counterexamples that are not Parsell-Vinogradov, one may take any of the examples presented in Subsection~\ref{subsec:7_proof}, which are well-curved quadratic manifolds, but do not satisfy the (CM) condition. This is because by $\circled{1}$ of Theorem~\ref{thm:relation_diagram}, for quadratic manifolds, the best $L^p$ bound ($L^{\frac{2(d+n)}{d}}$) of the Fourier transform of the surface measure is completely determined by the (CM) condition.
	
	In a nutshell, for any given $d$ and $n$, although we expect that the oblique leg of $\TT_0$ (or the best Stein-Tomas-type inequality $L^2 \rightarrow L^{\frac{2D}{d}+}$) is always achieved (by well-curved submanifolds), the top base of $\TT_0$ (i.e., $E^\psi1\in L^{\frac{D+d}{d}+}$) may not be achieved (by any submanifolds). However, this at least means that there should always exist $\Sigma_0$ such that $\TT(\Sigma) \subset \TT(\Sigma_0)$ for any other $\Sigma$, so that there exists a ``best'' $L^p\rightarrow L^q$ estimate and Question~\ref{ques:desired_restriction} should be well-posed. Besides, we should always expect the critical endpoint of $\TT(\Sigma)$ for a well-curved $\Sigma$ to lie on a line passing through the origin but with $\leq 45$ degree. These are indeed the cases for all known examples.
	
	For quadratic manifolds, which is the main focus of the current paper,  the gap between $\TT(\Sigma)$ and $\TT_0$ for any well-curved $\Sigma$ should exactly correspond to the gap between the (CM) condition and well-curvedness ($\circled{7}$ of the diagram): For general $d$ and $n$, well-curvedness can always be achieved, while the (CM) condition may not. In other words, if we were to pursue a complete picture of Fourier restriction theory, i.e., to get a good answer to Question~\ref{ques:desired_restriction}, the notion of well-curvedness developed in \cite{gressman19} should not be enough, even for quadratic manifolds (we will give an explicit counterexample at the end of this subsection). One needs more delicate estimates of the form $E^\psi1 \in L^p$, as has been done in \cite{bgzzk2024stationary} for $2$-dimensional Parsell-Vinogradov manifolds of
	degree $k$ and in \cite{mockenhaupt96} for several quadratic manifolds. Such a phenomenon cannot be seen when $d=1$ or $n=1$, and only appears when submanifolds are of ``intermediate'' dimension and codimension. However, we are currently not sure what the form of a general theory would be like to fully answer Question~\ref{ques:desired_restriction}.
	
	In view of the previous discussions, it seems reasonable to post the following question, which seems more approachable and should shed light on Question~\ref{ques:desired_restriction}:
	\begin{question}\label{ques:exist_CM}
		For which $d$ and $n$ does there always exist some $n$-tuple of real quadratic forms $\mathbf{Q}$ in $d$ variables that satisfy the {\rm (CM)} condition?
	\end{question}
	This question has been partially answered by Mockenhaupt in Corollary~2.6 of \cite{mockenhaupt96}, which says that for the (CM) condition to be achieved by some $\mathbf{Q}$, if $d$ is odd, or if $d$ is even and there exists $\theta \in \mathbb{S}^{n-1}$ with $\overline{Q}(\theta)$ positive definite, then $d\geq n$. Here we prove a kind of converse of Mockenhaupt's observation:
	\begin{proposition}\label{prop:exist_CM}
		For any $d$ and $n$ with $d \geq n$, there exists a $n$-tuple of real quadratic form $\mathbf{Q}$ in $d$ variables that satisfies the {\rm (CM)} condition.
	\end{proposition}
	\begin{proof}
		The construction is exactly the same as that in the proof of Theorem~2 in \cite{agranovsky2004p}. Consider $\mathbf{Q}(\xi) = (\sum_{i=1}^d \lambda_{ij} \xi_i^2)_{j=1}^n$, where $\lambda_{ij}$ are to be chosen later. It is easy to see that
		\begin{align*}
			\int_{\mathbb{S}^{n-1}} |\det(\overline{Q}(\theta))|^{-\gamma} d\sigma(\theta) 
			& = \int_{\mathbb{S}^{n-1}} \left|2^d\prod_{i=1}^d \left(\sum_{j=1}^n \lambda_{ij} \theta_j\right)\right|^{-\gamma} d\sigma(\theta)\\
			& \eqqcolon \int_{\mathbb{S}^{n-1}} \left|2^d\prod_{i=1}^d L_i(\theta)\right|^{-\gamma} d\sigma(\theta),
		\end{align*}
		where $L_i(\theta) \coloneqq \sum_{j=1}^n \lambda_{ij} \theta_j$ for any $i$. Now by Lemma~1 in \cite{agranovsky2004p}, when $d\geq n$, we can always choose $\lambda_{ij}$ such that any $n$ of $\{(\lambda_{ij})_{j=1}^n\}_{i=1}^d$ form a linearly independent set in $\R^n$. This ensures that no more than $n-1$ of the linear forms $\{L_i\}_{i=1}^d$ can vanish simultaneously on $\mathbb{S}^{n-1}$: Otherwise, $L_i(x) = 0$ for some $x\in \R^{n}$ and $n$ many linear forms implies $x=0$, due to linear independence, which contradicts $x\in\mathbb{S}^{n-1}$.
		
		Suppose $\prod_{i=1}^d L_i(\theta_0) = 0$ for some $\theta_0 \in \mathbb{S}^{n-1}$. Then without loss of generality, we may assume $L_i(\theta_0) = 0$ for $1\leq i \leq m$ with $m\leq n-1$, while $L_i(\theta_0) \neq 0$ for $m+1\leq i \leq n$. Via a suitable diffeomorphism, the integrability of $\left|\prod_{i=1}^d L_i(\theta)\right|^{-\gamma}$ in a small neighborhood of $\theta_0$ (such that $L_i$ is bounded away from $0$ for $m+1\leq i \leq n$) is the same as that of $\prod_{j=1}^r |x_j|^{-\gamma}$ in a small neighborhood of $0$, which holds true whenever $\gamma < 1$. Since $\frac{n}{d} \leq 1$, we in particular know that $|\det(\overline{Q}(\theta))|^{-\gamma}$ is integrable in this small neighborhood of $\theta_0$ whenever $\gamma<\frac{n}{d}$. Arguing similarly at other zeros of $\prod_{i=1}^d L_i$ and using the compactness of $\mathbb{S}^{n-1}$, we conclude that $|\det(\overline{Q}(\theta))|^{-\gamma}$ is integrable over the whole $\mathbb{S}^{n-1}$ whenever $\gamma<\frac{n}{d}$, i.e., $\mathbf{Q}$ satisfies the (CM) condition.
	\end{proof}
	
	By combining Mockenhaupt's observation with Proposition~\ref{prop:exist_CM}, we get a complete answer to Question~\ref{ques:exist_CM} when $d$ is odd, i.e., for odd $d$ there exists some $\mathbf{Q}$ satisfying the (CM) condition if and only if $d \geq n$. And the only remaining case is when $d$ is even and $d < n$. This case can be subtle, and we do not know what a complete answer should be like.
	
	Anyway, Proposition~\ref{prop:exist_CM} together with previous comments following Proposition~\ref{prop:opt_index} motivates us to propose the following conjecture, which partially tells us in what cases we can expect Hypotheses~2.2 in \cite{mockenhaupt96} to hold true:
	\begin{conjecture}\label{conj:opt_quad_restr}
		For Question~{\rm\ref{ques:desired_restriction}} when $d \geq n$, the best $\TT(\Sigma) = \TT_0$, which is achieved if and only if $\Sigma$ satisfies the {\rm (CM)} condition\footnote{Recall that for general smooth submanifolds, we can use Hessian matrices to define the {\rm (CM)} condition in a similar way as in Definition~\ref{def:CM}.}. Here $\TT_0$ is the right-angled trapezoid with top base on $\frac{1}{q} = \frac{d}{2(d+n)}$, bottom base on $\frac{1}{q}=0$, oblique leg on $\frac{p'}{q}=\frac{d}{d+2n}$, and perpendicular leg on $\frac{1}{p} = 0$.
	\end{conjecture}
	A resolution of Conjecture~\ref{conj:opt_quad_restr}, even in quadratic cases, would represent great progress towards a complete understanding of higher-codimensional Fourier restriction theory. Note that the special case when $n=1$ is just the classical Fourier restriction conjecture, which remains unsolved so far.
	
	Finally, let us look at an interesting special case of Proposition~\ref{prop:exist_CM}. When $d=5$ and $n=3$, there exists $\mathbf{Q}$ satisfying the (CM) condition, and we may explicitly take $$\mathbf{Q}(\xi) = (\xi_1^2 + \xi_4^2 + 2\xi_5^2, \xi_2^2 + 2\xi_4^2 + \xi_5^2, \xi_3^2 + \xi_4^2 + \xi_5^2)$$ by the construction in the proof. By $\circled{7}$ of Theorem~\ref{thm:relation_diagram}, this $\mathbf{Q}$ is also well-curved. However, recall the example $$\mathbf{Q}(\xi) = (\xi_1^2 + \xi_3^2 + \xi_5^2, 2(\xi_1\xi_2 + \xi_3\xi_4), \xi_2^2 + \xi_4^2 +\xi_5^2)$$ (still when $d=5$ and $n=3$) in Subsection~\ref{subsec:7_proof}, which is well-curved but fails to satisfy the (CM) condition. This means that for given $d$ and $n$, it is possible that some well-curved quadratic manifolds satisfy the (CM) condition, while others are not: For the former, we expect $\TT$ equals to $\TT_0$; while for the latter, we expect $\TT$ to be a proper subset of $\TT_0$. Such a phenomenon is not accidental. For example, one may also take $\mathbf{Q}(\xi) = (\xi_1^2,\xi_2^2,\xi_3^2,\xi_4^2)$ and $\mathbf{Q}(\xi) = (\xi_1\xi_2,\xi_2\xi_3,\xi_3\xi_4,\xi_4\xi_1)$  when $d=n=4$: The former is (CM), while the latter is well-curved but not (CM), as shown in Subsection~\ref{subsec:7_proof}. 
    
    In other words, for given $d$ and $n$, it is possible that well-curved quadratic manifolds do not enjoy the same $L^p\rightarrow L^q$ Fourier extension estimate, which confirms in a more concrete way that the notion of well-curvedness is not enough to completely answer Question~\ref{ques:desired_restriction}: We need more delicate information to fully determine the Fourier restriction theory for general higher-codimensional submanifolds. For the previous examples, this more delicate information comes from the (CM) condition which reflects not only the dimension of the algebraic variety $\{\theta\in\mathbb{S}^2: \det(\overline{Q}(\theta))=0\}$ but also how singular $|\det(\overline{Q}(\theta))|^{-\gamma}$ is around this variety. But as we have mentioned before, in order to completely answer Question~\ref{ques:desired_restriction}, one may need (in the first place) to get a better understanding of estimates of the form $E^\mathbf{Q}1 \in L^p$ for general $\mathbf{Q}$ that may or may not satisfy the (CM) condition.

    \appendix
	
    \section{Properties of \texorpdfstring{$\mathfrak{d}_{d,1}(\mathbf{Q})$}{∂\_\{d,1\}(Q)}}\label{appendix:par_d,1}
	
	The algebraic quantity $\mathfrak{d}_{d,1}(\mathbf{Q})$ plays a vital role in many different research topics:
	\begin{itemize}
		\item Fourier decay: The sharp order of uniform Fourier decay in Theorem~\ref{thm:uniform_decay} is $\frac{\mathfrak{d}_{d,1}(\mathbf{Q})}{2}$.
		
		\item Weighted Fourier restriction: In our Theorem~\ref{th1}, $\mathfrak{d}_{d,1}(\mathbf{Q})$ takes effect in several ways. Firstly, it appears in the second estimate of $s(\alpha,\mathbf{Q})$; secondly, it appears in the threshold $\frac{\mathfrak{d}_{d,1}(\mathbf{Q})}{2}$ dividing the first two ranges of $\alpha$; thirdly, this threshold is sharp for the first estimate to hold.
		
		\item Geometric measure theory: By 
		the proof of $\circled{11}$ in our Theorem~\ref{thm:relation_diagram}, the Fourier dimension of $S_\mathbf{Q}$ equals $\mathfrak{d}_{d,1}(\mathbf{Q})$, and $S_\mathbf{Q}$ is a Salem set if and only if $\mathfrak{d}_{d,1}(\mathbf{Q}) = d$.
		
		\item Radon-like transforms: Convolution with the surface measure of a quadratic manifold associated with $\mathbf{Q}$ has the best mapping properties from $L^2$ into the scale of $L^2$-Sobolev spaces if and only if $\mathfrak{d}_{d,1}(\mathbf{Q}) = d$. This is because for general Radon-like transforms, such optimal Sobolev regularity is characterized by the nonvanishing of the Phong-Stein rotational curvature, which is equivalent to the invertibility of certain Monge–Ampère-type matrix. By computing everything in our special translation-invariant setting involving $\mathbf{Q}$, one can easily see that the invertibility is equivalent to $\mathfrak{d}_{d,1}(\mathbf{Q}) = d$. We omit the details, and one may consult \cite{gressman15} for more background on this point.
	\end{itemize}
	
	Therefore, we are sufficiently motivated to explore the properties of $\mathfrak{d}_{d,1}(\mathbf{Q})$.
	
	In this appendix, we will provide 
	a short survey discussing the connection between $\mathfrak{d}_{d,1}(\mathbf{Q})$ and other fields such as applied linear algebra, algebraic geometry and algebraic topology. Such a connection in its essence is not new and we do not claim any originality here. However, we feel that the connection deserves to be explained thoroughly for several reasons: Firstly, the importance of $\mathfrak{d}_{d,1}(\mathbf{Q})$ is not fully recognized (or at least has not been investigated in a systematic and serious way) among people working on Fourier restriction theory of higher-codimensional quadratic forms; secondly, $\mathfrak{d}_{d,1}(\mathbf{Q})$ has only appeared in the literature outside harmonic analysis in an implicit way, so translating everything into our setting may ease future reference; thirdly, the research on the related algebraic problems seems to be sluggish for about a decade, so highlighting their importance in analysis may help to bring them back to the fore.
	
	There are many different ways of raising a question about $\mathfrak{d}_{d,1}(\mathbf{Q})$. The most naive question is, given $\mathbf{Q}$, how can we explicitly compute the exact value of $\mathfrak{d}_{d,1}(\mathbf{Q})$? This problem has already been solved in Appendix A of \cite{ggo23}, where an algorithm is designed (even for general $\mathfrak{d}_{d',n'}(\mathbf{Q})$) by using tools from real algebraic geometry. However, the algorithm is complicated and very hard to manipulate by hand (even for small $d'$ and $n'$), so it makes sense to ask for general rules of $\mathfrak{d}_{d,1}(\mathbf{Q})$ instead of exact values. In this appendix, we will mainly focus on the following type of question:
	\begin{question}\label{ques:partial_d_1}
		Given $r \in \N^+$, under what conditions of $d$ and $n$ {\rm(}may depend on $r${\rm)} do we always have $\mathfrak{d}_{d,1}(\mathbf{Q}) < r$ for all 
		$n$-tuple real quadratic forms $\mathbf{Q}$ in $d$ variables?
	\end{question}
	As we shall soon see, even for this seemingly much simpler question, we are still far from getting a complete answer.
	
	In the study of the so-called linear preserver problems \cite{beasley1997rank} in applied linear algebra, people care about the following problem:
	\begin{question}\label{ques_max_dim}
		Let $V \coloneqq {\rm S}(d, \R)$. For any given integer $0 \leq r \leq d$, try to find the maximal possible dimension $D$ of a subspace $L$ of $V$ satisfying one of the following three types of subspaces:

        \noindent {\rm (1)} For any $ A \in L$, $\rank(A) \leq r$;

        \noindent {\rm (2)} For any $ A \in L\setminus{\{0\}}$, $\rank(A) = r$;

        \noindent {\rm (3)} For any $ A \in L\setminus{\{0\}}$, $\rank(A) \geq r$.
	\end{question}
	All three types above have been used in \cite{beasley1997rank} and systematically studied later on. Besides ${\rm S}(d,\R)$, here $V$ can also be taken to be complex Hermitian matrices or real skew-symmetric matrices. For all these $V$'s, one may consult \cite{causin2007real} for very sharp results of type (2), which were obtained by using the theory of stable vector bundles and homotopy of the classical groups (topological K-theory). The reason why we mainly care about ${\rm S}(d,\R)$ in Question~\ref{ques_max_dim} is that Hessian matrices associated with quadratic forms are always real symmetric. However, if one cares about the algebraic problem in its own right, then there is no need to confine oneself to the $\R$ or $\C$ case, and one can ask the same question by replacing ${\rm S}(d,\R)$ by ${\rm S}(d,\F)$ for a general field $F$. In this most general setting, \cite{pazzis2016affine} contains very sharp results of type (1). Unfortunately, the type most relevant to our Question~\ref{ques:partial_d_1} is type (3), which is far from being fully understood. More precisely, if $n > D$, then $\mathfrak{d}_{d,1}(\mathbf{Q}) < r$ (i.e., if $n \geq D + 1$, then $\mathfrak{d}_{d,1}(\mathbf{Q}) \leq r - 1$) for all $n$-tuple quadratic forms $\mathbf{Q}$ in $d$ variables. This can be easily seen by unfolding the definition of $\mathfrak{d}_{d,1}(\mathbf{Q})$. Therefore, we will focus on type (3) from now on. Moreover, due to the complexity of the problem, we will often be content with upper bounds of $D$, which still allows us to obtain interesting properties of $\mathfrak{d}_{d,1}(\mathbf{Q})$. For example, if $D \leq D_0$ holds for some $D_0$, then $\mathfrak{d}_{d,1}(\mathbf{Q}) < r$ for all $\mathbf{Q}$ whenever $n>D_0$.
	\begin{remark}
		It is also possible to translate {\rm(1)} and {\rm(2)} in Question~{\rm\ref{ques_max_dim}} into certain statements regarding $\mathfrak{d}_{d,1}(\mathbf{Q})$. However, the statements do not seem as interesting as that corresponding to Question~{\rm\ref{ques:partial_d_1}}.
	\end{remark}
	\begin{remark}
		If ${\rm S}(d,\R)$ is replaced with ${\rm S}(d,\C)$, then {\rm(3)} has been completely determined {\rm(}$D = \binom{n-r+2}{2}${\rm)} in \cite{fl99} by using results on the degree of determinantal varieties in \cite{harris1984}. Such a method does not work for ${\rm S}(d,\R)$ because $\R$ is not algebraically closed.
	\end{remark}
	
	When $r=d$, since $\rank(A) \geq d$ is equivalent to $A \in {\rm GL}(d,\R)$, things are reduced to the following classical result in \cite{alp65}:
	\begin{theorem}\label{thm:Radon_Hurwitz}
		The maximal possible dimension $D$ of a subspace of ${\rm S}(d,\R)$ where each nonzero matrix is invertible equals $\rho\big(\frac{d}{2}\big) + 1$, where $\rho(u)$ is the Radon-Hurwitz number of $u$:
		\begin{align*}
			\rho(u) =
			\begin{cases}
				0, &{\rm if ~} u \notin \N;\\
				2^a + 8b, &{\rm if~ } u = 2^{a+4b}(2c+1) {\rm~ with~ } a,b,c \in\N, 0 \leq a \leq 3.
			\end{cases}
		\end{align*}
	\end{theorem}
	\begin{remark}\label{rmk:partial_d_1_construct}
		In \cite{alp65}, the upper bound for $D$ is proved by first reducing to the $\R^{8d \times 8d}$ case through tensor product with Cayley numbers and then applying the results in \cite{adams1962}, which relates the problem to linearly independent vector fields on the unit sphere. For the lower bound, \cite{alp65} reduces it to the $\R^{\frac{d}{2}\times\frac{d}{2}}$ case, and then still applies known results in \cite{adams1962}. For a more direct proof of the lower bound by explicitly constructing the subspace in ${\rm S}(d,\R)$, one may consult Proposition {\rm3.2.9} in \cite{causin2007real}, which is actually stronger in that it works for all ``constant rank'' cases {\rm(}not necessarily full rank{\rm)}. The key point is, for every even $d$, by using ``Radon-Hurwitz system'' of order $\frac{d}{2}$, we can find $\{M_1,\dots,M_{\rho(\frac{d}{2})-1}\} \subset {\rm O}\big(\frac{d}{2},\R\big)$ satisfying the ``Clifford relations'':
		\[
		M_iM_j + M_jM_i = 0, \,\forall i\neq j, \quad M_i^2 = -I, \quad M_i^T = -M_i.
		\]
		Then we can complete the construction by taking 
		\begin{align*}
			\begin{pmatrix}
				0 & M_i\\
				M_i^T & 0
			\end{pmatrix},\, 1 \leq i \leq \rho\Big(\frac{d}{2}\Big)-1, \quad 
			\begin{pmatrix}
				0 & I\\
				I & 0
			\end{pmatrix},
			\begin{pmatrix}
				I & 0\\
				0 & -I
			\end{pmatrix}.
		\end{align*}
	\end{remark}
	
	Here and henceforth, in the statements we omit ``for all $n$-tuple quadratic forms $\mathbf{Q}$ in $d$ variables'', but the reader should keep such a setting in mind. By the connection to $\mathfrak{d}_{d,1}(\mathbf{Q})$ as previously explained, Theorem~\ref{thm:Radon_Hurwitz} immediately yields:
	\begin{corollary}\label{cor:Radon_Hurwitz}
		If $n \geq \rho(\frac{d}{2}) + 2$, then $\mathfrak{d}_{d,1}(\mathbf{Q}) \leq d-1$.
	\end{corollary}
	
	\begin{remark}\label{rmk:rho_growing_order}
		By digesting the definition one can easily see that $\rho(\frac{d}{2})$ grows approximately like $\log_2 d$, which is much smaller than $d$ for large $d$, so the intuition behind Corollary~{\rm\ref{cor:Radon_Hurwitz}} is that $\mathfrak{d}_{d,1}(\mathbf{Q}) \leq d-1$ ``almost always'' happens for general $d$ and $n$. In other words, $\mathfrak{d}_{d,1}(\mathbf{Q}) = d$ is such a strong nondegeneracy requirement that it is rarely met.
	\end{remark}
	
	Once $\mathfrak{d}_{d,1}(\mathbf{Q}) \leq d-1$ is well understood, it is natural to ask when  $\mathfrak{d}_{d,1}(\mathbf{Q}) \leq d-2$ or $\mathfrak{d}_{d,1}(\mathbf{Q}) \leq d-3$ hold. By the results in \cite{fl99} and \cite{fll02}, we are able to obtain such properties of $\mathfrak{d}_{d,1}(\mathbf{Q})$ in a rather straightforward way. We summarize all the conclusions as follows:
	\begin{theorem}\label{thm:partial_d_1_property}
		For any $k,d\in\N$ with $k < d$, define\footnote{This is just the degree of the projective variety formed by all matrices in $\text{S}(d,\C)$ of rank $k$ or less (see  \cite{harris1984} for the proof of this fact).} $\delta_{k,d} \coloneqq \prod_{j=0}^{d-k-1}\binom{d+j}{d-k-j}/\binom{2j+1}{j}$. Then the following properties hold:

        \noindent {\rm (1)} If $\delta_{k,d}$ is odd and $n \geq \binom{d-k+1}{2}+1$, then $\mathfrak{d}_{d,1}(\mathbf{Q}) \leq k$.

        \noindent {\rm (2)} {\rm(}Consequence of {\rm(1))} If $d>q\geq 1$ and $d\equiv \pm q (\mathrm{mod}\,\, 2^{\lceil \log_2(2q)\rceil})$, then $\delta_{d-q,d}$ is odd, and so $n \geq \binom{q+1}{2}+1$ implies $\mathfrak{d}_{d,1}(\mathbf{Q}) \leq d-q$.

        \noindent {\rm (3)} {\rm(}Consequence of {\rm(2))} If $d \equiv 2 (\mathrm{mod}\,\,4), n\geq 4$, then $\mathfrak{d}_{d,1}(\mathbf{Q}) \leq d-2$.

        \noindent {\rm (4)} {\rm(}Consequence of {\rm(2))} If $d \equiv 3,5 (\mathrm{mod}\,\,8), n\geq 7$, then $\mathfrak{d}_{d,1}(\mathbf{Q}) \leq d-3$. {\rm(}This actually implies that if $d \equiv 2,4 (\mathrm{mod}\,\,8), n\geq 7$, then $\mathfrak{d}_{d,1}(\mathbf{Q}) \leq d-2$.{\rm)}

        \noindent {\rm (5)} If $d \equiv 0 (\mathrm{mod}\,\,4), n\geq d+1$, then $\mathfrak{d}_{d,1}(\mathbf{Q}) \leq d-2$.

        \noindent {\rm (6)} If $\xi_1^2 + \cdots + \xi_j^2 \in \mathbf{Q}(\xi)$ for some $1\leq j \leq d$ up to linear transformation, then $d \geq 4, n \geq \sigma(d) +2$ implies $\mathfrak{d}_{d,1}(\mathbf{Q}) \leq d-2$. Here
			\begin{align*}
				\sigma(d) = 
				\begin{cases}
					2, &{\rm~ if~ } d \not\equiv 0,\pm 1 (\mathrm{mod}\,\, 8);\\
					\rho(4b), &{\rm ~if~ } d = 8b,8b\pm1 {\rm~ with~ } b \in\N^+.
				\end{cases}
			\end{align*}
	\end{theorem}
	One can see that the properties of $\mathfrak{d}_{d,1}(\mathbf{Q})$ may depend on some number theoretic characteristics of $d$ and $n$.
	
	Moreover, \cite{fll02} shows that
	\begin{proposition}
		Suppose $d$ is odd. Then there exists a $(d+1)$-dimensional subspace $L$ of ${\rm S}(d,\R)$ such that for any $ A\in L\setminus\{0\}$, we have $\rank(A) \geq d-1$. 
	\end{proposition}
	This means that there is no way to expect a ``logarithmic law'' for ``$\mathfrak{d}_{d,1}(\mathbf{Q}) \leq d-2$'' as that of ``$\mathfrak{d}_{d,1}(\mathbf{Q})\leq d-1$'' as explained in Remark~\ref{rmk:rho_growing_order}. 
	
	As a final remark, we emphasize that all the results in this appendix indicate the complexity of $\mathfrak{d}_{d',n'}(\mathbf{Q})$: Even for $\mathfrak{d}_{d,1}(\mathbf{Q})$, the complexity will be at least as that of (3) in Question~\ref{ques_max_dim} (or Question~\ref{ques:partial_d_1}), which is a wide-open algebraic problem. However, it is still possible to obtain good results for small $d$ and $n$. For example, when $d+n\leq 5$, by direct computation, we have a complete classification of all $\mathbf{Q}$'s (Lemma~\ref{addth2}).
	
	All the properties of $\mathfrak{d}_{d,1}(\mathbf{Q})$ can be easily translated into the corresponding results in any of the four contexts mentioned at the beginning, and we omit these trivial discussions.

\subsection*{Acknowledgements} 
The authors would like to thank Philip T. Gressman and Yumeng Ou for helpful comments on an early manuscript.

\bibliographystyle{plain}
\bibliography{mybibfile}

\begin{thebibliography}{10}

\bibitem{adams1962}
J.~F. Adams.
\newblock Vector fields on spheres.
\newblock {\em Ann. of Math.}, 75(3):603--632, 1962.

\bibitem{alp65}
J.~F. Adams, P.~D. Lax, and R.~S. Phillips.
\newblock On matrices whose real linear combinations are nonsingular.
\newblock {\em Proc. Amer. Math. Soc}, 16(2):318--322, 1965.

\bibitem{agranovsky2004p}
M.~L. Agranovsky and E.~K. Narayanan.
\newblock ${L}^p$-integrability, supports of {F}ourier transforms and uniqueness for convolution equations.
\newblock {\em J. Fourier Anal. Appl.}, 10:315--324, 2004.

\bibitem{bll17}
J.-G. Bak, J.~Lee, and S.~Lee.
\newblock Bilinear restriction estimates for surfaces of codimension bigger than 1.
\newblock {\em Anal. PDE}, 10(8):1961--1985, 2017.

\bibitem{bl04}
J.-G. Bak and S.~Lee.
\newblock Restriction of the {F}ourier transform to a quadratic surface in $\mathbb{R}^n$.
\newblock {\em Math. Z.}, 247(2):409--422, 2004.

\bibitem{banner02}
A.~D. Banner.
\newblock Restriction of the {F}ourier transform to quadratic submanifolds.
\newblock {\em Ph.D. thesis, Princeton University}, 2002.

\bibitem{bbcrv07}
J.~A. Barcel\'{o}, J.~M. Bennett, A.~Carbery, A.~Ruiz, and M.~C. Vilela.
\newblock Some special solutions of the {S}chr\"{o}dinger equation.
\newblock {\em Indiana Univ. Math. J.}, 56(4):1581--1593, 2007.

\bibitem{beh21}
A.~Barron, M.~B. Erdo\u{g}an, and T.~L.~J. Harris.
\newblock Fourier decay of fractal measures on hyperboloids.
\newblock {\em Trans. Amer. Math. Soc.}, 374(2):1041–1075, 2021.

\bibitem{bgzzk2024stationary}
S.~Basu, S.~Guo, R.~Zhang, and P.~Zorin-Kranich.
\newblock A stationary set method for estimating oscillatory integrals.
\newblock {\em J. Eur. Math. Soc.}, 2024.

\bibitem{beasley1997rank}
L.~B. Beasley and R.~Loewy.
\newblock Rank preservers on spaces of symmetric matrices.
\newblock {\em Linear Multilinear Algebra}, 43(1-3):63--86, 1997.

\bibitem{cmp24}
Z.~Cao, C.~Miao, and Y.~Pang.
\newblock Sharp restriction estimates for several degenerate higher co-dimensional quadratic surfaces.
\newblock {\em arXiv:2404.09020}, 2024.

\bibitem{cmw24}
Z.~Cao, C.~Miao, and Z.~Wang.
\newblock Fourier decay of fractal measures on surfaces of co-dimension two in $\mathbb{R}^5$.
\newblock {\em J. Funct. Anal.}, 286(8):110378, 2024.

\bibitem{ccw99}
A.~Carbery, M.~Christ, and J.~Wright.
\newblock Multidimensional van der {C}orput and sublevel set estimates.
\newblock {\em J. Amer. Math. Soc.}, 12(4):981--1015, 1999.

\bibitem{causin2007real}
A.~Causin.
\newblock Real linear spaces of matrices.
\newblock {\em Ph.D. thesis, Sapienza University of Rome}, 2007.

\bibitem{christ82}
M.~Christ.
\newblock Restriction of the {Fourier} transform to submanifolds of low codimension.
\newblock {\em Ph.D. thesis, University of Chicago}, 1982.

\bibitem{pazzis2016affine}
C.~de~Seguins~Pazzis.
\newblock Affine spaces of symmetric or alternating matrices with bounded rank.
\newblock {\em Linear Algebra Appl.}, 504:503--558, 2016.

\bibitem{demeter2020}
C.~Demeter.
\newblock {\em Fourier restriction, decoupling, and applications}.
\newblock Cambridge University Press, 2020.

\bibitem{dmv22}
S.~Dendrinos, A.~Mustata, and M.~Vitturi.
\newblock A restricted $2 $-plane transform related to {F}ourier restriction for surfaces of codimension $2$.
\newblock {\em Anal. PDE}, 18(2):475--526, 2025.

\bibitem{d20}
X.~Du.
\newblock Upper bounds for {Fourier} decay rates of fractal measures.
\newblock {\em J. Lond. Math. Soc.}, 102(3):1318--1336, 2020.

\bibitem{dgl17}
X.~Du, L.~Guth, and X.~Li.
\newblock A sharp {Schr{\"o}dinger} maximal estimate in $\mathbb{R}^{2}$.
\newblock {\em Ann. Math.}, 186(2):607--640, 2017.

\bibitem{dz19}
X.~Du and R.~Zhang.
\newblock Sharp {$L^2$} estimate of {Schr\"{o}dinger} maximal function in higher dimensions.
\newblock {\em Ann. Math.}, 189(3):837--861, 2019.

\bibitem{fll02}
D.~Falikman, S.~Friedland, and R.~Loewy.
\newblock On spaces of matrices containing a nonzero matrix of bounded rank.
\newblock {\em Pacific J. Math.}, 207(1):157--176, 2002.

\bibitem{fhk22fourier}
J.~M. Fraser, T.~L.~J. Harris, and N.~G. Kroon.
\newblock On the {F}ourier dimension of $(d, k)$-sets and {K}akeya sets with restricted directions.
\newblock {\em Math. Z.}, 301(3):2497--2508, 2022.

\bibitem{fl99}
S.~Friedland and R.~Loewy.
\newblock Spaces of symmetric matrices containing a nonzero matrix of bounded rank.
\newblock {\em Linear Algebra Appl.}, 287(1-3):161--170, 1999.

\bibitem{ggo23}
S.~Gan, L.~Guth, and C.~Oh.
\newblock Restriction estimates for quadratic manifolds of arbitrary codimensions.
\newblock {\em arXiv.2308.06427v1}, 2023.

\bibitem{gilula2016real}
Maxim Gilula.
\newblock {\em A real analytic approach to estimating oscillatory integrals}.
\newblock PhD thesis University of Pennsylvania, 2016.

\bibitem{grafakos2014modern}
L.~Grafakos.
\newblock {\em Modern {F}ourier analysis}, volume 250.
\newblock Springer, 2014.

\bibitem{gressman15}
P.~T. Gressman.
\newblock ${L}^{p}$-nondegenerate {R}adon-like operators with vanishing rotational curvature.
\newblock {\em Proc. Amer. Math. Soc.}, 143(4):1595--1604, 2015.

\bibitem{gressman19}
P.~T. Gressman.
\newblock On the {O}berlin affine curvature condition.
\newblock {\em Duke Math. J.}, 168(11):2075--2126, 2019.

\bibitem{gressman24}
P.~T. Gressman.
\newblock Generalized sublevel estimates for form-valued functions and related results for {R}adon-like transforms.
\newblock {\em arXiv:2407.18860}, 2024.

\bibitem{guo1993p}
K.~Guo.
\newblock On the $p$-thin problem for hypersurfaces of $\mathbb{R}^n$ with zero {G}aussian curvature.
\newblock {\em Canad. Math. Bull.}, 36(1):64--73, 1993.

\bibitem{gizzk23}
S.~Guo, A.~Iosevich, R.~Zhang, and P.~Zorin-Kranich.
\newblock ${L}^p$ integrability of functions with {F}ourier support on a smooth space curve.
\newblock {\em arXiv:2311.11529}, 2023.

\bibitem{go22}
S.~Guo and C.~Oh.
\newblock Fourier restriction estimates for surfaces of co-dimension two in $\mathbb{R}^5$.
\newblock {\em J. Anal. Math.}, 148(2):471--499, 2022.

\bibitem{gozzk23}
S.~Guo, C.~Oh, R.~Zhang, and P.~Zorin-Kranich.
\newblock Decoupling inequalities for quadratic forms.
\newblock {\em Duke Math. J.}, 172(2):387--445, 2023.

\bibitem{gzk20}
S.~Guo and P.~Zorin-Kranich.
\newblock Decoupling for certain quadratic surfaces of low co-dimensions.
\newblock {\em J. Lond. Math. Soc.}, 102(1):319--344, 2020.

\bibitem{harris1984}
J.~Harris and L.~W. Tu.
\newblock On symmetric and skew-symmetric determinantal varieties.
\newblock {\em Topology}, 23(1), 1984.

\bibitem{hj12matrix}
R.~A. Horn and C.~R. Johnson.
\newblock {\em Matrix analysis}.
\newblock Cambridge university press, 2012.

\bibitem{ikromov1997convergence}
I.~A. Ikromov.
\newblock On the convergence exponent of trigonometric integrals.
\newblock {\em Tr. Mat. Inst. Steklova}, 218(Anal. Teor. Chisel i Prilozh)(0):179--189, 1997.

\bibitem{john2014extremum}
F.~John.
\newblock Extremum problems with inequalities as subsidiary conditions.
\newblock {\em Traces and emergence of nonlinear programming}, pages 197--215, 2014.

\bibitem{lee06}
S.~Lee.
\newblock Bilinear restriction estimates for surfaces with curvatures of different signs.
\newblock {\em Trans. Amer. Math. Soc.}, 358(8):3511–3533, 2006.

\bibitem{mattila87}
P.~Mattila.
\newblock Spherical averages of {Fourier} transforms of measures with finite energy; dimension of intersections and distance sets.
\newblock {\em Mathematika}, 34(2):207--228, 1987.

\bibitem{mattila15}
P.~Mattila.
\newblock Fourier {A}nalysis and {H}ausdorff {D}imension.
\newblock {\em Cambridge Univ. Press}, 2015.

\bibitem{mockenhaupt96}
G.~Mockenhaupt.
\newblock Bounds in {Lebesgue} spaces of oscillatory integral operators.
\newblock {\em Habilitationsschrift, Universit\"at Siegen}, 1996.

\bibitem{oberlin00}
D.~M. Oberlin.
\newblock Convolution with measures on hypersurfaces.
\newblock {\em Math. Proc. Cambridge Philos. Soc.}, 129(3):517--526, 2000.

\bibitem{pss01}
D.~H. Phong, E.~M. Stein, and J.~Sturm.
\newblock Multilinear level set operators, oscillatory integral operators, and newton polyhedra.
\newblock {\em Math. Ann.}, 319:573--596, 2001.

\bibitem{pinney2021}
G.~Pinney.
\newblock A decoupling proof of the {T}omas restriction theorem.
\newblock {\em arXiv:2112.04111}, 2021.

\bibitem{shayya21}
B.~Shayya.
\newblock Fourier restriction in low fractal dimensions.
\newblock {\em Proc. Edinb. Math. Soc.}, 64(2):373--407, 2021.

\bibitem{stein93}
E.~M. Stein and T.~S. Murphy.
\newblock {\em Harmonic analysis: real-variable methods, orthogonality, and oscillatory integrals}, volume~3.
\newblock Princeton University Press, 1993.

\bibitem{tao99}
T.~Tao.
\newblock The {B}ochner-{R}iesz conjecture implies the restriction conjecture.
\newblock {\em Duke Math. J.}, 96(2):363--375, 1999.

\bibitem{vargas05}
A.~Vargas.
\newblock Restriction theorems for a surface with negative curvature.
\newblock {\em Math. Z.}, 249(1):97–111, 2005.

\bibitem{ww24}
H.~Wang and S.~Wu.
\newblock Restriction estimates using decoupling theorems and two-ends {F}urstenberg inequalities.
\newblock {\em arXiv:2411.08871v1}, 2024.

\bibitem{wolff99}
T.~H. Wolff.
\newblock Decay of circular means of {Fourier} transforms of measures.
\newblock {\em Int. Math. Res. Not.}, 1999(10):547--567, 1999.

\end{thebibliography}

\end{document}